\documentclass[10pt]{amsart}

\usepackage{amsmath,amsthm,graphicx,amscd,multirow}
\usepackage{ amssymb} 
\usepackage{ mathrsfs} 
\usepackage{ amscd} 
\usepackage{ enumitem}
\usepackage{ scalerel}

\usepackage{ color}

\allowdisplaybreaks[4]
\usepackage{ascmac}
\usepackage[arrow,matrix]{xy}

\DeclareMathOperator*{\bigboxtimes}{\scalerel*{\boxtimes}{\sum}}

\newcommand{\Hom}{\mathop{\mathrm{Hom}}\nolimits}
\newcommand{\diag}{\mathop{\mathrm{diag}}\nolimits}
\newcommand{\id}{\mathop{\mathrm{id}}\nolimits}

\newcommand{\sw}{\mathsf{w}}

\newcommand{\cA}{\mathcal{A}}
\newcommand{\cD}{\mathcal{D}}
\newcommand{\cE}{\mathcal{E}}

\newcommand{\cI}{\mathcal{I}}
\newcommand{\cJ}{\mathcal{J}}
\newcommand{\cK}{\mathcal{K}}
\newcommand{\cL}{\mathcal{L}}
\newcommand{\cO}{\mathcal{O}}

\newcommand{\cS}{\mathcal{S}}
\newcommand{\cT}{\mathcal{T}}
\newcommand{\cU}{\mathcal{U}}
\newcommand{\cV}{\mathcal{V}}
\newcommand{\cW}{\mathcal{W}}

\newcommand{\rGL}{\mathrm{GL}}
\newcommand{\rop}{\mathrm{op}}
\newcommand{\rpre}{\mathrm{pre}}
\newcommand{\rfin}{\mathrm{fin}}
\newcommand{\rdom}{\mathrm{dom}}
\newcommand{\rlex}{\mathrm{lex}}
\newcommand{\rconj}{\mathrm{conj}}
\newcommand{\rtriv}{\mathrm{triv}}
\newcommand{\rRe}{\mathrm{Re}}
\newcommand{\rre}{\mathrm{re}}
\newcommand{\rim}{\mathrm{im}}
\newcommand{\rpos}{\mathrm{pos}}
\newcommand{\rneg}{\mathrm{neg}}
\newcommand{\rAd}{\mathrm{Ad}}
\newcommand{\rad}{\mathrm{ad}}

\newcommand{\rB}{\mathrm{B}}
\newcommand{\rC}{\mathrm{C}}

\newcommand{\rG}{\mathrm{G}}
\newcommand{\rI}{\mathrm{I}}
\newcommand{\rM}{\mathrm{M}}
\newcommand{\rN}{\mathrm{N}}

\newcommand{\rP}{\mathrm{P}}
\newcommand{\rR}{\mathrm{R}}
\newcommand{\rS}{\mathrm{S}}
\newcommand{\rT}{\mathrm{T}}
\newcommand{\rU}{\mathrm{U}}
\newcommand{\rW}{\mathrm{W}}

\newcommand{\ra}{\mathrm{a}}
\newcommand{\rb}{\mathrm{b}}
\newcommand{\rc}{\mathrm{c}}
\newcommand{\rd}{\mathrm{d}}
\newcommand{\rf}{\mathrm{f}}
\newcommand{\rg}{\mathrm{g}}

\newcommand{\rrq}{\mathrm{q}}
\newcommand{\rr}{\mathrm{r}}
\newcommand{\rs}{\mathrm{s}}
\newcommand{\rv}{\mathrm{v}}

\newcommand{\mA}{\mathbf{A}}
\newcommand{\mC}{\mathbf{C}}
\newcommand{\mE}{\mathbf{E}}
\newcommand{\mN}{\mathbf{N}}
\newcommand{\mQ}{\mathbf{Q}}
\newcommand{\mR}{\mathbf{R}}
\newcommand{\mW}{\mathbf{W}}
\newcommand{\mX}{\mathbf{X}}
\newcommand{\mY}{\mathbf{Y}}
\newcommand{\mZ}{\mathbf{Z}}
\newcommand{\me}{\mathbf{e}}
\newcommand{\mv}{\mathbf{v}}

\newcommand{\ggl}{\mathfrak{gl}}
\newcommand{\gS}{\mathfrak{S}}
\newcommand{\ga}{\mathfrak{a}}
\newcommand{\g }{\mathfrak{g}}

\newcommand{\gn}{\mathfrak{n}}

\newcommand{\gr}{\mathfrak{r}}
\newcommand{\gp}{\mathfrak{p}}
\newcommand{\gu}{\mathfrak{u}}

\newcommand{\bepsilon}{{\boldsymbol \epsilon}}

\newcommand{\blambda}{{\boldsymbol \lambda}}

\newcommand{\bmu}{{\boldsymbol \mu}}

\newcommand{\be}{{\boldsymbol e}}

\newcommand{\bi}{{\boldsymbol i}}

\newcommand{\bu}{{\boldsymbol u}}

\newtheorem{thm}{Theorem}[section]
\newtheorem{prop}[thm]{Proposition}
\newtheorem{cor}[thm]{Corollary}
\newtheorem{lem}[thm]{Lemma}
\newtheorem*{Main}{Main Result}

\theoremstyle{definition}
\numberwithin{equation}{section}
\newtheorem{dfn}[thm]{Definition}

\newtheorem{rem}[thm]{Remark}
\newtheorem*{rmk}{Remark}

\title[Uniform Integrality of critical values]{Uniform Integrality of critical values of the Rankin--Selberg $L$-function for ${\rm GL}_{n}\times {\rm GL}_{n-1}$}
\author[T.~Hara, T.~Miyazaki and K.~Namikawa]{Takashi Hara, Tadashi Miyazaki and Kenichi Namikawa}

\address{ Department of Mathematics, College of Liberal Arts,  Tsuda University, 2-1-1 Tsuda-machi, Kodaira, Tokyo 187-8577, Japan }
\email{t-hara@tsuda.ac.jp}

\address{ Department of Mathematics, College of Liberal Arts and Sciences,  Kitasato University, 1-15-1 Kitazato, Minami-Ku, Sagamihara, Kanagawa 252-0373, Japan }
\email{miyaza@kitasato-u.ac.jp}

\address{ Department of Mathematics, School of System Design and Technology, Tokyo Denki University, 5 Senju Asahi-cho, Adachi-ku, Tokyo 120-8551, Japan }
\email{namikawa@mail.dendai.ac.jp}


\begin{document}

\begin{abstract}
After introducing the notion of uniform integrality of critical values of the
 Rankin--Selberg $L$-functions for $\mathrm{GL}_{n}\times \mathrm{GL}_{n-1}$,
 we study it when the base field is totally imaginary. For this purpose,
 we adopt specific models with highest weight representations of the general linear groups, construct the Eichler--Shimura classes for $\mathrm{GL}_{n}$ and $\mathrm{GL}_{n-1}$
  in an explicit manner,
 and then evaluate the cohomological cup product of them, by making the best use of Gel'fand--Tsetlin basis.

\end{abstract}

\maketitle

\setcounter{tocdepth}{2}
\tableofcontents

\section{Introduction} \label{sec:Intro}

\subsection{Background}
Study of {\em special values} of automorphic $L$-functions is one of the most interesting and important research topics in number theory. Among these, significant progress is achieved on the {\em rationality problem of critical values of Rankin--Selberg $L$-functions for $\mathrm{GL}_n\times \mathrm{GL}_{n-1}$}. For the case $\mathrm{GL}_2\times \mathrm{GL}_1$ over the rational number field $\mathbf{Q}$, which is the most classical (elliptic modular) case,  
it is initiated by Manin \cite{man72} and Shimura \cite{shi76,shi77}.    
Taking the philosophical guideline based on Deligne's conjecture on critical values of motivic Hasse--Weil $L$-functions \cite[Conjecture 1.8]{del79} into accounts, 
one naturally expects that rationality of critical values should be deduced from suitable {\em cohomological interpretation} of integral expressions of the critical values (so-called the {\em global zeta integrals}) under consideration. 
Indeed, Hida introduces in \cite{hid94} appropriate notion of periods, which are called {\em Hida's canonical periods} at the present, as an analogue of the motivic periods $c^{\pm}(\mathcal{M})$ introduced by Deligne in \cite{del79},
and deduces rationality of critical values of Rankin--Selberg $L$-functions for $\mathrm{GL}_2\times \mathrm{GL}_1$ over general number fields with respect to his canonical periods.
In \cite{rs08}, Raghuram and Shahidi define the {\em Whittaker period} $p^{\rb}(\pi^{(n)},\boldsymbol{\epsilon}^{(n)})$ for a cohomological irreducible cuspidal automorphic representation $\pi^{(n)}$ of $\mathrm{GL}_n$ over a general number field\footnote{The symbol $\boldsymbol{\epsilon}^{(n)}\in \{\pm\}^{r_1}$ denotes certain admissible signatures, where $r_1$ is the number of the real places of the base field. Since we assume that the base field is totally imaginary (thus $r_1=0$) in the present article, we need not take the admissible signatures into accounts, and thus drop the symbol $\boldsymbol{\epsilon}^{(n)}$ from the notation of the Whittaker periods hereafter.}, which generalises the notion of Hida's canonical periods. One of the ultimate purposes of the present article is to generalise Hida's result on rationality of critical values \cite[Theorem 8.1]{hid94}
to the Rankin--Selberg $L$-function for $\mathrm{GL}_n \times \mathrm{GL}_{n-1}$ over a {\em totally imaginary} number field $F$,
with respect to Raghuram and Shahidi's Whittaker periods $p^{\rb}(\pi^{(n)})$ and $p^{\rb}(\pi^{(n-1)})$. Although this attempt is enforced by Raghuram himself, we present a more precise and explicit result than \cite[Theorem 1.1]{rag16}, as we shall see later.

Beyond rationality, {\em $p$-adic nature of critical values} (for a prime number $p$) is also a very significant topic from the arithmetic viewpoint. Such study rewinds to the remarkable discovery of {\em Kummer's congruence} among Bernoulli numbers: for any positive integer $k\in \mN$, if positive even integers $m_1$ and $m_2$ (neither divisible by $p-1$) satisfy $m_1\equiv m_2 \pmod{(p-1)p^{k-1}}$, the congruence
\begin{align} \label{eq:Kummer}
  (1-p^{m_1-1})\dfrac{B_{m_1}}{m_1} \equiv (1-p^{m_2-1})\dfrac{B_{m_2}}{m_2} \pmod{p^{k}}
\end{align}
holds where $B_m$ (for $m\in \mN_0$) denotes the $m$-th Bernoulli number. Since the ratio $-B_m/m$ is regarded as the special value $\zeta(1-m)$ of the Riemann's zeta function $\zeta(s)$ at $s=1-m$ with $m\in 2\mN$, the congruence \eqref{eq:Kummer} clarifies that the set of critical values of $\zeta(s)$ (modified at the $p$-Euler factors) $\{(1-p^{m-1})\zeta(1-m)\mid m\in 2\mN\}$ behaves somewhat {\em $p$-adically continuously} with respect to $m$, and even implies existence of certain {\em $p$-adic continuous} (or even {\em $p$-adic analytic/meromorphic}) functions interpolating (parts of) these values; indeed, as is broadly known, Kubota and Leopoldt have succeeded in construction of so-called {\em $p$-adic Dirichlet $L$-functions} in \cite{kl64}, and many other constructions of them have been developed by Coleman, Fresnel, Iwasawa and so on. Soon after that, the attempt to construct $p$-adic $L$-functions is extended to the Hecke $L$-functions of elliptic modular forms (or the Rankin--Selberg $L$-functions for $\mathrm{GL}_2\times \mathrm{GL}_1$ over the rational number field $\mQ$ in our context) by Manin \cite{man73}, Mazur and Swinnerton-Dyer \cite{msd74}, Amice and V\'elu \cite{am75}, and Vi\v{s}ik \cite{vis76}. We also refer the readers to \cite{mtt86} and \cite{kit94} as comprehensive literature on this topic. Similarly to the Riemann zeta function case, {\em Kummer-type congruences} among critical values (explicated in \cite[Lemma~4.6]{kit94}) play, explicitly or implicitly, crucial roles in the construction of $p$-adic $L$-functions for elliptic modular forms. Therefore if one attempts to construct the $p$-adic $L$-functions for more general automorphic $L$-functions---for instance, for the Rankin--Selberg $L$-functions for $\mathrm{GL}_n\times \mathrm{GL}_{n-1}$---, it seems reasonable to try to verify the Kummer-type congruences among their critical values as the first step.  
Of course, it is nonsense to consider congruences among transcendental complex numbers, and thus the rationality results of critical values such as \cite[Theorem~1.1]{rag16} are crucial, but not sufficient, for one to consider congruence relations among the critical values. Indeed, if he or she wants to consider congruences modulo a power of a prime ideal $\mathfrak{P}$, these critical values (divided by certain periods) are required to be {\em integral at $\mathfrak{P}$}. Further it is preferable that each critical value should be divided by the {\em same} period, for otherwise it would become so difficult to compare the distinct critical values. Because of these kinds of reasons, we here propose the following problem:
\begin{quotation}
for a fixed prime ideal $\mathfrak{P}$ lying above $p$, is each critical value divided by the {\em same} period is {\em simultaneously integral} at $\mathfrak{P}$  (up to a minor multiple)? 
\end{quotation}
If the answer of the problem is affirmative, we say that the critical values are  {\em uniformly $(\mathfrak{P}\text{-})$integral}\footnote{For the Riemann zeta function $\zeta (s)$, the uniform $p$-integrality of critical values is nothing but the following well-known fact (also due to Kummer): for every positive even integer $m$ not divisible by $p-1$, the ratio $B_m/m$ is $p$-integral.}. The second main goal of the present article is to establish precise formulation of uniform integrality of the critical values of the Rankin--Selberg $L$-function for $\mathrm{GL}_n \times \mathrm{GL}_{n-1}$ over a totally imaginary number field $F$, and verify it under certain mild conditions.   For this purpose, we have to normalise the Whittaker period $p^{\mathrm{b}}(\pi^{(n)})$ in a {\em $p$-optimal} way (see Definition~\ref{dfn:whittper} for details).

\subsection{Main result and comparison with previous research}

Now let us quickly introduce our main result in a rough form. We always  assume that the base field $F$ is a {\em totally imaginary} number field. For a natural number $n$ greater than or equal to $2$, let $\pi^{(n)}$ (resp.\ $\pi^{(n-1)}$) be a cohomological irreducible cuspidal automorphic representation of $\mathrm{GL}_n(F_{\mA})$ (resp.\ $\mathrm{GL}_{n-1}(F_{\mA})$). We use the symbol $\blambda=(\lambda_{\sigma,j})_{\sigma\in I_F, 1\leq j\leq n}$ and $\bmu=(\mu_{\sigma,j})_{\sigma\in I_F, 1\leq j\leq n-1}$ for the highest weight associated with $\pi^{(n)}$ and $\pi^{(n-1)}$ respectively (see Section~\ref{subsec:settings} for details), where $I_F$ denotes the set of all the embeddings of $F$ into the complex number field $\mC$. We always assume that there exists an integer $m_0\in \mZ$ such that $\blambda^\vee$ interlaces $\bmu+m_0$; that is, we assume that the inequalities \eqref{eq:interlace_condition} hold for certain $m_0\in \mZ$. Suppose that $\pi^{(n-1)}_v$ is spherical at any finite place $v$ of $F$, and let $\mathfrak{N}$ be an integral ideal of $F$ which is maximal among all ideals $\mathfrak{A}$ such that $(\pi^{(n)})^{\mathcal{K}_{n,1}(\mathfrak{A})}$ contains a nonzero vector; here $\mathcal{K}_{n,1}(\mathfrak{A})$ is the mirahoric subgroup defined as \eqref{eq:mirahoric}. Let $\mathfrak{r}(\pi^{(n)},\pi^{(n-1)})$ be the ring of integers of the composite field $\mQ(\pi^{(n)},\pi^{(n-1)}):=\mQ(\pi^{(n)})\mQ(\pi^{(n-1)})$, where $\mQ(\pi^{(n)})$ (resp.\ $\mQ(\pi^{(n-1)})$) denotes the field of rationality of $\pi^{(n)}$ (resp.\ $\pi^{(n-1)}$) defined as a subfield of $\mC$ (see Section~\ref{subsec:WhittakerPeriods} for details). We fix an isomorphism $\boldsymbol{i}\colon \mC \xrightarrow{\, \sim \,} \mC_p$ between the complex number field $\mC$ and the $p$-adic completion $\mC_p$ of a fixed algebraic closure of the $p$-adic number field $\mQ_p$. Let $\mathfrak{P}_0$ be the prime ideal of $\mathfrak{r}(\pi^{(n)},\pi^{(n-1)})$ induced from the field embedding $\mQ(\pi^{(n)}, \pi^{(n-1)})\subset \mC \xrightarrow{\, \boldsymbol{i}\,} \mC_p$. The localisation of $\mathfrak{r}(\pi^{(n)},\pi^{(n-1)})$ at $\mathfrak{P}_0$ is denoted as $\mathfrak{r}(\pi^{(n)},\pi^{(n-1)})_{(\mathfrak{P}_0)}$. Under these settings, our main result is stated as follows:

\begin{Main}[rough form of Theorem~$\ref{thm:UniformIntegrality}$]
Fix $\varepsilon\in \{\pm\}$ and let $p$ be a prime number satisfying the inequality
\begin{align} \label{eq:pbound}
p>\max\{\lambda_{\sigma,1}-\lambda_{\sigma,n}+n-2 \mid \sigma\in I_F\}. 
\end{align} 
Suppose that the discriminant of $F$ is prime to $\mathfrak{N}$. 
Then, for {\em every} critical point $\frac{1}{2}+m$ 
  $($see Section~$\ref{subsec:settings}$ for the definition of critical points$)$ 
of the Rankin--Selberg $L$-function $L(s,\pi^{(n)}\times \pi^{(n-1)})$, the complex value
\begin{align} \label{eq:ratio_integral}
\mathcal{C}(m,\pi^{(n)} \times \pi^{(n-1)})   \frac{ L(\frac{1}{2} +m, \pi^{(n)} \times \pi^{(n-1)}) }{  p^{\rb}(\pi^{(n)}) p^{\rb} (\pi^{(n-1)})  } 
\end{align}
is indeed contained in $\mathfrak{r}(\pi^{(n)},\pi^{(n-1)})_{(\mathfrak{P}_0)}$, where $\mathcal{C}(m,\pi^{(n)}\times \pi^{(n-1)})$ is an explicit complex constant $($depending on $m)$ and $p^{\rb}(\pi^{(n)})$ $($resp.\ $p^{\rb}(\pi^{(n-1)}))$ denotes the {\em $p$-optimal Whittaker period} of $\pi^{(n)}$ $($resp.\ $\pi^{(n-1)})$ defined as in Definition~$\ref{dfn:whittper}$. Furthermore, for each $\alpha \in \mathrm{Aut}(\mathbf{C})$, we have 
\begin{align} \label{eq:equivariance_integral}
  \alpha \left(  \mathcal{C}(m,\pi^{(n)} \times \pi^{(n-1)}) 
                      \frac{ L(\frac{1}{2} +m, \pi^{(n)} \times \pi^{(n-1)}) }{  p^{\rb}(\pi^{(n)})p^{\rb} (\pi^{(n-1)})  } \right)     =  \mathcal{C}(m, {}^\alpha \pi^{(n)} \times {}^\alpha\pi^{(n-1)}) 
      \frac{ L(\frac{1}{2} +m, {}^\alpha\pi^{(n)} \times {}^\alpha\pi^{(n-1)}) }{  p^{\rb}({}^\alpha\pi^{(n)})p^{\rb} ( {}^\alpha\pi^{(n-1)})  }
\end{align}
where ${}^\alpha \pi^{(n)}$ $($resp. ${}^\alpha \pi^{(n-1)})$ is the $\alpha$-twist of $\pi^{(n)}$ $($resp.\ $\pi^{(n-1)})${\rm ;} see Section~$\ref{subsec:whittvec}   $ for details.
\end{Main}

\begin{rmk}
 The main result above insists that {\em all} the critical values of $L(s, \pi^{(n)}\times \pi^{(n-1)})$ (multiplied with certain constants) are {\em simultaneously} $\mathfrak{P}_0$-integral after they are divided by the {\em same} period $p^{\rm b}(\pi^{(n)})p^{\rm b}(\pi^{(n-1)})$; namely, the critical values of $L(s,\pi^{(n)}\times \pi^{(n-1)})$ are {\em uniformly {\rm (}$\mathfrak{P}_0$-{\rm )}integral} with respect to the product of the $p$-optimal Whittaker periods $p^{\rm b}(\pi^{(n)})p^{\rm b}(\pi^{(n-1)})$. Note that, if $p$ is sufficiently large, the critical values of $L(s,\pi^{(n)}\times \pi^{(n-1)})$ are always uniformly integral for the trivial reason that $L(s,\pi^{(n)}\times \pi^{(n-1)})$ admits only finitely many critical points. One crucial point of our main result is that we succeed in proposing an explicit lower bound \eqref{eq:pbound} of $p$. We think that this bound is quite reasonable; it is determined only by the highest weight associated with $\pi^{(n)}$. Also note that the constant $\mathcal{C}(m,{}^\alpha \pi^{(n)}\times {}^\alpha \pi^{(n-1)})$ appearing in \eqref{eq:ratio_integral} and \eqref{eq:equivariance_integral} is explicitly described as \eqref{eq:constant}, 
     which seems harmless (consisting of terms concerning the discriminant of $F$ and powers of $-1$, $2$ and $\pm \sqrt{-1}$) though it depends on the critical points. 
\end{rmk}

To prove the main result above, we mainly adopt Raghuram's strategy in \cite{rag16}, and develop {\em cohomological interpretation} of the Rankin--Selberg integrals. 
A remarkable milestone of this topic would be Kazhdan, Mazur and Schmidt's work \cite{kms00} on construction of $p$-adic distributions associated to the Rankin--Selberg $L$-functions for $\mathrm{GL}_n \times \mathrm{GL}_{n-1}$ over $\mathbf{Q}$, in which they develop a cohomological method called {\em generalised modular symbol method}. One of the crucial points in their result is to construct appropriate cohomology classes from cuspidal automorphic forms of $\pi^{(n)}$ and $\pi^{(n-1)}$, and express their ``cup product'' as (a sum of) the product of local Rankin--Selberg integrals. 
Non-archimedean local Rankin--Selberg integrals are studied thoroughly by Shintani \cite{shin76}, Jacquet, Piatetski-Shapiro and Shalika \cite{jpss81} and others,  
but computation of {\em archimedean} local Rankin--Selberg integrals at each critical point still remains as a hard problem. 
Kazhdan and Mazur expect  nontriviality of the archimedean local Rankin--Selberg integrals in 1970's (see also \cite[p.98, Question]{kms00}), 
which  has been recently verified by B.~Sun \cite{sun17}.  
Based on the method of \cite{kms00}, Mahnkopf \cite{mah05} and Raghuram \cite{rag10,rag16} have proved that, each critical value is rational after it is divided by the product of the archimedean local Rankin--Selberg integrals (which {\em do depend} on the critical points) and the Whittaker periods.   
We will improve their results and prove that each simple ratio \eqref{eq:ratio_integral} is rational (and even integral) by {\em explicitly evaluating} the archimedean local Rankin--Selberg integrals, effectively utilising {\em Gel'fand--Tsetlin basis}. Note that our result is consistent with Deligne's conjecture \cite{del79}  predicting the behaviour of the (motivic) periods under the Tate twists.

Recently Li, Liu and Sun have  verified in \cite{lls} an explicit relation among the archimedean local Rankin--Selberg integrals at distinct critical points, which implies that each critical value is rational after it is divided by the product of the Whittaker periods (which {\em do not depend} on the critical points). Thus their result is also consistent with Deligne's conjecture under the Tate twists. Furthermore, contrary to our result, their result can be applied to the case where the base field admits real archimedean places. Yet we would like to emphasise here that they do not carry out computation of the archimedean local Rankin--Selberg integrals to the end; they avoid the whole calculation of them by cleverly normalising generators of the $(\mathfrak{g},K)$-cohomology groups (see \cite[Section~6.1]{lls} for details). The authors strongly believe that our explicit computation of the archimedean local Rankin--Selberg integrals, especially determination of the constant $\mathcal{C}(m,{}^\alpha \pi^{(n)}\times {}^\alpha \pi^{(n-1)})$, would be useful when we discuss the $p$-adic nature of the critical values.
  
As we have already mentioned, we verify not only rationality of the critical values but also {\em uniform integrality} of them. To this end, we should equip the cohomology groups under consideration  with nice {\em integral structures}. We here utilise Gel'fand--Tsetlin basis again, which are indispensable ingredients in the present article. One of the greatest advantages to introduce integral structure on the cohomology groups utilising Gel'fand--Tsetlin basis is that the branching rule map for the pair $(\rGL_n,\rGL_{n-1})$, which plays a fundamental role in the cohomological interpretation of the zeta integrals, clearly preserves the integral structure introduced in the present article by construction.

As one of typical applications of thorough study on integral structures of cohomology groups and explicit archimedean zeta calculus,  
 we here mention the recent work of Januszewski \cite{jan} about 
 a construction of $p$-adic Rankin--Selberg $L$-functions for $\mathrm{GL}_n\times \mathrm{GL}_{n-1}$. 
Januszewski constructs there a $p$-adic measure  associated with $\rGL_n\times \rGL_{n-1}$, and gives its interpolation formula in \cite[Theorem A]{jan}. However, his interpolation formula still includes archimedean local Rankin--Selberg integrals depending on each critical points without clarification, similarly to \cite[Theorem 1.1]{rag16}.   
In particular uniform integrality of critical values of Rankin--Selberg $L$-functions for $\mathrm{GL}_n\times \mathrm{GL}_{n-1}$ does not follow from his result.  
According to Coates and Perrin-Riou's general conjecture    
  on the existence of $p$-adic $L$-functions for pure motives (see \cite[Principal Conjecture]{cp89} or \cite[Principal Conjecture]{coa89}; recall that the automorphic $L$-functions of the general linear groups are conjectured by Clozel to be motivic $L$-functions), the effects appearing when one chooses different critical points should be explicitly described in terms of {\em modified Euler factors at archimedean places}.     
But the relation between archimedean local Rankin--Selberg integrals introduced in \cite{jan} and the modified Euler factors at archimedean places explicitly defined as in \cite[Section~2]{cp89} and \cite[Section~1]{coa89} has not been studied yet in general; see also \cite[Remark 6.2]{jan}.   
Because of this reason, the relation among his pointwisely constructed $p$-adic measures with respect to the $p$-adic Tate twists verified in \cite[Theorem 5.4]{jan} 
   do not immediately imply the desired Kummer type congruences of critical values of Rankin--Selberg $L$-functions for $\mathrm{GL}_n\times \mathrm{GL}_{n-1}$. In the present article, due to effective use of Gel'fand-Tsetlin basis, 
we write down in an explicit manner both integral structures on appropriate cohomology groups
   and archimedean Rankin--Selberg integrals under considerations, and thus we are able to deduce much more precise $p$-adic nature on critical values of Rankin--Selberg $L$-functions for $\mathrm{GL}_n\times \mathrm{GL}_{n-1}$, namely the uniform integrality. 
We expect that our sophisticated cohomological interpretations of Rankin--Selberg integrals and their explicit evaluation developed in the present article will be useful for further study of $p$-adic properties of the critical values, such as Kummer type congruences.

Finally, Li, Liu and Sun's result \cite{lls} or ours implies that,  assuming the existence of the conjectural motives $\mathcal{M}[\pi^{(n)}]$ associated with the cohomological irreducible cuspidal automorphic representation $\pi^{(n)}$, 
Deligne's motivic periods $c^\pm( \mathcal{M}[\pi^{(n)}] \otimes \mathcal{M}[\pi^{(n-1)}] )$ should coincide with 
the product of a power of $2\pi\sqrt{-1}$,  the product $p^\mathrm{b}(\pi^{(n)}) p^\mathrm{b}(\pi^{(n-1)})$ of the Whittaker periods 
and the archimedean local Rankin--Selberg integrals up to algebraic multiples. 
This suggests that thorough study of the archimedean local Rankin--Selberg integrals is also required to clarify the expected properties of each Whittaker period $p^{\rb}(\pi^{(n)})$    
in terms of the motivic periods of  $\mathcal{M}[\pi^{(n)}]$. 
See \cite{hn} for the study of this direction.

\subsection{Strategy and contents of the article}

Here let us explain our strategy a little more specifically. We basically adopt the strategy of Raghuram in \cite{rag16}, as we have already noted, but we shall sophisticate his result based on {\em explicit calculation} of the archimedean local Rankin--Selberg integrals (Theorem~\ref{thm:archzetaformula}). What enables us to complete our computations is the notion of {\em rational Gel'fand--Tsetlin basis}. For a totally imaginary number field $F$, a (complex) representation $\widetilde{V}(\blambda^\vee)=\bigotimes_{\sigma \in I_F} V_{\lambda_\sigma^\vee}$ of $\rGL_n(F)$ with highest weight $\blambda^\vee=(\lambda_\sigma^\vee)_{\sigma\in I_F}$ appears as the coefficients of cohomology groups related with $\pi^{(n)}$ 
 (see Appendix~\ref{sec:localsystem} for details on notation). On each $V_{\lambda_\sigma^\vee}$ and a $\rGL_n(\mC)$-invariant hermitian inner product $(\cdot,\cdot)_{\lambda_\sigma^\vee}$, Gel'fand and Tsetlin constructed in \cite{Gelfand_Tsetlin_001} an orthonormal basis $\{\zeta_M\}_{M\in \rG(\lambda_\sigma^\vee)}$
     (see Section~\ref{subsec:GT_basis} for the notation), with respect to which the action of several elements of the Lie algebra $\mathfrak{gl}_{n}$ of $\rGL_n(\mC)$ is described in an explicit manner. In \cite{im}, Ishii and the second-named author introduce a modification $\{\xi_M\}_{M\in \rG(\lambda_\sigma^\vee)}$ of $\{\zeta_M\}_{M\in \rG(\lambda_\sigma^\vee)}$, which we call the rational Gel'fand--Tsetlin basis. Using this, we can explicitly construct a generator $[\pi^{(n)}_\infty]_{\pm}$ of a certain $(\mathfrak{g}_{n\mC},\widetilde{K}_n)$-cohomology group, and furthermore, can evaluate the pairing $\widetilde{\cI}^{(m)}([\pi^{(n)}_\infty]_{\varepsilon},[\pi^{(n-1)}_\infty]_{-\varepsilon})$ (this corresponds to accomplishment of the calculation of the archimedean local Rankin--Selberg integrals explained above; see Theorem~\ref{thm:archzetaformula} for details). 
We can then define an {\em Eichler--Shimura cohomology class} $\delta(\pi^{(n)})$ in the cuspidal cohomology group $H^{{\rm b}_{n,F}}_{\mathrm{cusp}}(Y_{\cK}^{(n)},\widetilde{\cV}(\blambda^\vee))$ by using $[\pi^{(n)}_\infty]_{\pm}$. The Whittaker period is defined as the ratio of $\delta(\pi^{(n)})$ to a certain rational cohomology class $\eta(\pi^{(n)})$ of $H^{{\rm b}_{n,F}}_{\mathrm{cusp}}(Y_{\cK}^{(n)},\widetilde{\cV}(\blambda^\vee))$ in \cite{rs08}. To establish uniform integrality of the critical values, we should modify the Whittaker period into the {\em $p$-optimal} one, and thus we have to equip the cohomology group $H^{{\rm b}_{n,F}}_{\mathrm{cusp}}(Y_{\cK}^{(n)},\widetilde{\cV}(\blambda^\vee))$ with a suitable {\em $p$-integral structure}. One of the significant discovery made in the present article is to find that we can utilise the rational Gel'fand--Tsetlin basis to equip $\widetilde{V}(\blambda^\vee)$ with a natural integral structure (Proposition~\ref{prop:xiM_integral}). More precisely, if $E^*$ is a number field containing the normal closure of $F$ over $\mQ$, we can define an $\cO^*$-integral representation $\widetilde{V}(\blambda^\vee)_{\cO^*}^{(p)}$, which is a {\em self dual lattice} in $\widetilde{V}(\blambda^\vee)_{\cE^*}^{(p)}:=\widetilde{V}(\blambda^\vee)_{\cO^*}^{(p)}\otimes_{\cO^*}\cE^*$ by construction; here $\cE^*$ is the closure of the image of the composite map $E^*\subset \mC\xrightarrow{\, \boldsymbol{i}\,}\mC_p$, and $\cO^*$ denotes the ring of integers of $\cE^*$ (see Appendix~\ref{sec:localsystem} for details). We can construct a $p$-adic local system $\widetilde{\cV}(\blambda^\vee)_{\cO^*}^{(p)}$, which equips $H^{{\rm b}_{n,F}}_{\mathrm{cusp}}(Y_{\cK}^{(n)},\widetilde{\cV}(\blambda^\vee))$ with a suitable $p$-integral structure and enables us to define the $p$-optimal Whittaker period $p^{\rm b}(\pi^{(n)})$ (see Definition~\ref{dfn:whittper} for details). Furthermore the integral structure which we define behaves well under cohomological cup product operations, and thus we can obtain the desired uniform integrality of the critical values (Theorem~\ref{thm:UniformIntegrality}). 

\medskip
Let us briefly explain the organisation of the present article. The content is divided into two parts, and Section~\ref{sec:main_result} is the main part. Most of crucial ideas and constructions, as well as the precise statements of our main results, are contained here, and thus we expect that one can grasp the outline of the present article by only reading Section~\ref{sec:main_result}. After introducing general settings on cohomological automorphic representations in Section~\ref{subsec:settings}, we explain cohomological interpretation of the Rankin--Selberg zeta integrals in Section~\ref{subsec:cupproduct}. Section~\ref{subsec:finite_dimensional} is devoted to preparation of the settings on finite dimensional representations of the general linear group $\rGL_n$, containing the introduction of the rational Gel'fand--Tsetlin basis \eqref{eq:def_Qbasis} and its basic properties. We then construct an explicit generator $[\pi^{(n)}_\infty]_\varepsilon$ of the $(\mathfrak{g}_{n\mC},\widetilde{K}_n)$-cohomology group in Section~\ref{subsec:explicit_generators} utilising the rational Gel'fand--Tsetlin basis, and state the first main theorem (Theorem~\ref{thm:archzetaformula}) which gives the explicit evaluation of $\widetilde{\cI}^{(m)}([\pi^{(n)}_\infty]_{\varepsilon},[\pi^{(n-1)}_\infty]_{-\varepsilon})$. This part is one of the technical hearts of the present article, but the precise proof of Theorem~\ref{thm:archzetaformula} is postponed to Section~\ref{sec:calcoeff}. In Section~\ref{subsec:uniformintegrality}, we first define the $p$-optimal Whittaker periods $p^{\rm b}(\pi^{(n)})$ (Definition~\ref{dfn:whittper}), and then state and prove uniform integrality of the critical values (Theorem~\ref{thm:UniformIntegrality}). We here basically follow the arguments in \cite{rag16}, but the $p$-integral structure of the cohomology groups introduced in Appendix~\ref{subsec:ratint} plays a crucial role.

The other sections propose details and supplements to Section~\ref{sec:main_result}. We introduce our normalisation of Haar measures and several notions concerning Lie algebras in Section~\ref{sec:measures}. Section~\ref{sec:HWrepGLn} gives supplementary explanation to the contents of Section~\ref{subsec:finite_dimensional}. Here we give the proofs to all the propositions which have not been verified in Section~\ref{subsec:finite_dimensional}. In Section~\ref{sec:arch_whittaker}, we give the explicit archimedean zeta integral formula (Proposition~\ref{prop:ArchInt}) for the archimedean Whittaker functions defined in Section~\ref{subsec:C_def_ps}. This section contains many supplementary materials to \cite{im}, and we deduce Proposition~\ref{prop:ArchInt} by using them to relate  our Whittaker functions with those introduced in \cite{im}. Section~\ref{sec:calcoeff} is devoted to the proof of Theorem~\ref{thm:archzetaformula}, which is also a computational heart of the present article. In Appendix~\ref{sec:localsystem}, we summarise rational/integral structure of the finite dimensional representations of $\rGL_n$, the associated local systems, and the rational/integral structure of several cohomology groups.

\subsection{Basic notation}
\label{subsec:notation}

We here introduce several notation which is used throughout the present article. 
The symbols $\mZ$, $\mQ$, $\mR$ and $\mC$ denote 
the ring of rational integers, 
the rational number field, the real number field and the complex number field 
respectively. Let $\mN_0$ (resp.\ $\mN$) be the set of {\em non-negative} integers (resp.\ {\em positive} integers). 
For $z\in \mC$, let $\overline{z}$ and $|z|$ denote
the complex conjugate and the absolute value of $z$, respectively. 
Twice the ordinary Lebesgue measure on $\mC \simeq \mR^2$ is denoted as ${\rm d}_\mC z:=2{\rm d}x\,{\rm d}y$ ($z=x+\sqrt{-1} y$).  
As usual, we define a meromorphic function 
$\Gamma_{\mC}(s)$ of $s$ by 
$\Gamma_\mathbf{C}(s):=2(2\pi)^{-s}\Gamma (s)$,  
where $\Gamma (s)$ is the Gamma function. 

For a prime number $p$, let $\mC_p$ denote the $p$-adic completion of a fixed algebraic closure $\overline{\mQ}$ of $\mQ$. Throughout the present article, we fix an isomorphism of fields $\bi\colon \mC\xrightarrow{\, \sim \,} \mC_p$. 

In  the present article, $F$ always denotes a {\em totally imaginary} number field. The ring of integers of $F$ is denoted by $\gr_F$. 
Let $\Sigma_{F}$ be the set of all places of $F$ and $I_F$ 
the set of all embeddings of $F$ into $\mC$. 
Let $\Sigma_{F,\infty}$ (resp.\ $\Sigma_{F,\rfin}$) 
denote the set of archimedean (resp.\ finite) places of $F$. 
For each $v\in \Sigma_{F,\infty}$, we specify one of the conjugate pair of complex embeddings $F\hookrightarrow \mC$ which determines $v$, and regard $\Sigma_{F,\infty}$ as a subset of $I_F$; in other words, we (non-canonically) identify $I_F$ with $\{v,\bar{v} \mid v\in \Sigma_{F,\infty} \}$ where $\bar{v}$ denotes the complex conjugate of $v$. 
For every $v\in \Sigma_F$, the topological completion of $F$ at $v$ is denoted by $F_v$. 
For $v\in \Sigma_{F,\rfin}$, the ring of integers of $F_v$ is denoted by $\gr_{F,v}$. 
We let $F_{\mA}$ denote the ring of ad\`eles of 
$F$, and use the symbol $F_{\mA,\rfin}$ 
for its subring of finite ad\`eles. 
The archimedean part of $F_{\mA}$ is denoted as 
$F_{\mA,\infty} = F\otimes_{\mQ}\mR =
\prod_{v\in \Sigma_{F,\infty}}F_v$. We use similar notation
for every ad\`elic object.

Let $\psi_{\mathbf{Q}, \varepsilon}\colon  \mathbf{Q} \backslash \mathbf{Q}_{\mathbf{A}} \to \mathbf{C}^\times$ be the standard additive character; 
namely, for each $\varepsilon \in \{ \pm\}$, its archimedean part is defined as $\psi_{\mathbf{Q}, \varepsilon, \infty}(x) = \exp( \varepsilon 2\pi \sqrt{-1} x)$ for $x\in \mathbf{R}$, whereas for each $p\in \Sigma_{\mQ,\mathrm{fin}}$, the $p$-component is normalised so that $\psi_{\mQ,\varepsilon,p}\vert_{\mathbf{Z}_p}=1$ and $\psi_{\mQ,\varepsilon,p}(p^{-r})= \exp(-\varepsilon 2\pi \sqrt{-1} p^{-r})$ ($r\in \mN$) hold. 
Define an additive character $\psi_\varepsilon \colon F\backslash F_\mathbf{A} \to \mathbf{C}^\times$ as $\psi_\varepsilon = \psi_{\mathbf{Q}, \varepsilon} \circ \mathrm{Tr}_{F/\mathbf{Q}}$. As usual, we decompose $\psi_\varepsilon =\bigotimes_{v\in \Sigma_F}\psi_{\varepsilon ,v}$ with 
$\psi_{\varepsilon,v}:=\psi_\varepsilon |_{F_v} $. We set $\psi_{\varepsilon ,\infty}:=\bigotimes_{v\in \Sigma_{F,\infty}}\psi_{\varepsilon,v}$ 
and $\psi_{\varepsilon ,\rfin}:=\bigotimes_{v\in \Sigma_{F,\rfin}}\psi_{\varepsilon,v}$.

For positive integers $n$ and $n'$, let $O_{n,n'}$ and $1_n$ respectively denote the zero matrix of size $n\times n'$ and the unit matrix of size $n$. 
We use the symbol $\gS_n$ for the symmetric group 
of degree $n$, and define the action of $\gS_n$ on $\mC^n$ by 
\begin{align}
\label{eq:def_actSn}
&\sigma z :=(z_{\sigma^{-1}(1)},z_{\sigma^{-1}(2)},
\dots ,z_{\sigma^{-1}(n)})&
\quad \text{for }\sigma \in \gS_n \text{ and }
z =(z_1,z_2,\dots ,z_n)\in \mC^n.
\end{align}
Moreover, we set 
$z+s:=(z_1+s,z_2+s,\dots ,z_n+s)$ for $z =(z_1,z_2,\dots ,z_n)\in \mC^n$ and $s\in \mC$.

The symbol $\rGL_{n/F}$ denotes the general linear group of degree $n$, 
considered as an algebraic group defined over $F$. 
Let $\rB_n$ be the Borel subgroup of $\rGL_{n/F}$ consisting of all 
upper triangular matrices. 
The unipotent radical of $\rB_n$ is denoted as $\rN_n$, which consists of 
all upper triangular matrices with every diagonal entry equal to $1$. 
The symbol  $\rT_n$ denotes the diagonal torus of $\rB_n$, that is, the subgroup of $\rB_n$ consisting of all diagonal matrices. Let $K_n$ and $\widetilde{K}_n$ denote the subgroups of $\rGL_n(F_{\mA,\infty})$ respectively defined as $\displaystyle \prod_{v\in \Sigma_{F,\infty}} {\rm U}(n)$ and $\displaystyle \prod_{v\in \Sigma_{F,\infty}} \mC^\times {\rm U}(n)$, where ${\rm U}(n)$ denotes the unitary group of degree $n$. Note that $K_n$ is a maximal compact subgroup of $\rGL_n(F_{\mA,\infty})$. 

For a Lie algebra $\mathfrak{g}$ defined over $\mR$, 
the complexification of $\mathfrak{g}$ is denoted as $\mathfrak{g}_\mathbf{C}:=\g \otimes_\mR \mC$. 
For a Lie algebra $\mathfrak{g}$ defined over $\mC$, 
the universal enveloping algebra of $\mathfrak{g}$ is denoted as $\cU (\mathfrak{g})$. 
We use the symbol $\mathfrak{gl}_n=\mathrm{Lie}(\rGL_n(\mC))$ to denote the complex general linear Lie algebra of degree $n$. Write the Cartan decomposition of $\mathfrak{gl}_n$ as $\ggl_n=\gu(n)\oplus \gp_n$ where $\mathfrak{u}(n)$ denotes the Lie algebra of $\mathrm{U}(n)$. By putting $\gp_{n\mC}:=\gp_n\otimes_\mR \mC$, we obtain the decomposition of $\ggl_{n\mC}$ as 
\begin{align} \label{eq:CartanDecomp}
\mathfrak{gl}_{n\mC} = \gu(n)_{\mC} \oplus \gp_{n\mC}. 
\end{align} 
Define an isomorphism of Lie algebras ${\rm gl}_n \colon \mathfrak{gl}_{n\mC} \to \mathfrak{gl}_n \oplus \mathfrak{gl}_n$ to be  
\begin{align*}
   {\rm gl}_n ( X\otimes c ) = (cX, c\overline{X}) \qquad (X\in \mathfrak{gl}_n, c\in {\mathbf C}). 
\end{align*}
Let  $E_{i,j}\in \mathfrak{gl}_n$ denotes the matrix whose $(i, j)$-component is $1$ and the other components are $0$. 
Then define elements $E^{\mathfrak{gl}_n}_{i,j}$, $\widetilde{E}^{\mathfrak{gl}_n}_{i,j}$, $E^{\gu (n)}_{i,j}$ and $E^{\gp_n}_{i,j}$ of  $\mathfrak{gl}_{n\mC}$ to be 
\begin{align*}
    E^{\mathfrak{gl}_n}_{i,j} &=   {\rm gl}^{-1}_n( E_{i,j}, 0 ) = \frac{1}{2} \left\{  E_{i,j} \otimes 1 - (\sqrt{-1}E_{i,j}) \otimes \sqrt{-1}  \right\},  &
    \widetilde{E}^{\mathfrak{gl}_n}_{i,j} &=   {\rm gl}^{-1}_n( 0, E_{i,j} ) = \frac{1}{2} \left\{  E_{i,j} \otimes 1 + (\sqrt{-1}E_{i,j}) \otimes \sqrt{-1}  \right\},    \\
  E^{ \mathfrak{u}(n) }_{i,j}  &= E^{\mathfrak{gl}_n}_{i, j} - \widetilde{E}^{\mathfrak{gl}_n}_{j, i},   &
  E^{\gp_n }_{i,j}  &= E^{\mathfrak{gl}_n}_{i, j} + \widetilde{E}^{\mathfrak{gl}_n}_{j, i}.     
\end{align*}
Then $\{E^{\gu (n)}_{i,j}\}_{1\leq i,j\leq n}$ and $\{E^{\gp_n}_{i,j}\}_{1\leq i,j\leq n}$ respectively form 
bases of $\gu (n)_{\mC}$ and $\gp_{n\mC}$.
Let $A_n$ be the subgroup of $\rGL_n(\mC)$ consisting of all the diagonal matrices with positive real entries, 
and $Z_n$ the center of ${\rm GL}_n(\mC)$ which consists of all the scalar matrices. 
Define $\mathfrak{a}_n$ (resp.\ $\mathfrak{z}_{\mathfrak{a}_n}$) 
to be the Lie algebra of $A_n$ (resp.\ $Z_n \cap A_n$).
We also define $\mathfrak{p}^0_{n}$ to be the subspace of $\mathfrak{p}_n$ consisting of all the elements with trace zero, and set $\gp^0_{n\mC}:=\gp^0_{n}\otimes_{\mR}\mC$.
Then  $\mathfrak{z}_{\mathfrak{a}_n\mC}$ is a one dimensional space spanned by $E^{\gp_n}_{1,1} + E^{\gp_n}_{2,2} + \cdots + E^{\gp_n}_{n,n}$, 
and $\gp_{n\mC}^0$ has a basis
\begin{align}  \label{eq:basis_p0}
 \{E^{\gp_n}_{i,i} - E^{\gp_n}_{i+1,i+1} \}_{1\leq i\leq n-1} \cup \{E^{\gp_n}_{j,k}\}_{1\leq j,k\leq n, j\neq k}.
\end{align} 
By this explicit description of the basis of $\mathfrak{z}_{\mathfrak{a}_n\mC}$ and $\mathfrak{p}^0_{n\mC}$, one readily observes that $\mathfrak{p}_{n\mC}$ is decomposed as
\begin{align}\label{eq:decomp_p0}
 \mathfrak{p}_{n\mC}=\mathfrak{z}_{\mathfrak{a}_n\mC} \oplus \mathfrak{p}^0_{n\mC}.
\end{align}
The symbol $\mathfrak{p}^\vee_{n\mC}$  denotes
the dual space of $\mathfrak{p}_{n\mC}$, equipped with   
the dual basis  $\{  E^{\gp_n\vee}_{i, j}  \}_{1\leq i, j \leq n}$    
of $\{  E^{\gp_n}_{i, j}  \}_{1\leq i, j \leq n}$. Similarly $\mathfrak{z}^\vee_{\mathfrak{a}_n{\mathbf C}}$ and 
  $\gp^{0\vee}_{ n{\mathbf C}}$ respectively denote the dual spaces of $\mathfrak{z}_{\mathfrak{a}_n{\mathbf C}}$ and  
  $\mathfrak{p}^{0}_{ n{\mathbf C}}$. In the present article, based on the decomposition (\ref{eq:decomp_p0}), 
we regard $\gp^{0\vee}_{n\mC}$ as a subspace of $\gp^{\vee}_{n\mC}$ 
by extending each $\omega \in \gp^{0\vee}_{n\mC}$ to $\gp_{n\mC}$ so that 
$\omega (X)=0$ for any $X\in \mathfrak{z}_{\mathfrak{a}_n\mC}$.

We define ${\mathfrak g}_n$ (resp.\ ${\mathfrak k}_n$) as the Lie algebra of ${\rm GL}_n(F_{\mathbf{A}, \infty})$ (resp.\ $K_n:=\prod_{v\in \Sigma_{F, \infty}}\mathrm{U}(n)$), which is canonically identified with $\bigoplus_{v\in \Sigma_{F,\infty}} \mathfrak{gl}_n$ (resp.\ $\bigoplus_{v\in \Sigma_{F,\infty}} \mathfrak{u}(n)$).

\section{Main results}
\label{sec:main_result}

\subsection{General settings on cohomological automorphic representations} \label{subsec:settings}

As remarked in Section~\ref{subsec:notation}, the symbol $F$ always denotes a {\em totally imaginary} number field throughout the present article. For each open compact subgroup ${\mathcal K}_n$ of ${\rm GL}_n(F_{{\mathbf A}, {\rm fin}})$, define $Y^{(n)}_{\mathcal{K}}$ to be the double coset space
\begin{align*}
   Y^{(n)}_{\mathcal K} 
     = {\rm GL}_n(F)  \backslash  {\rm GL}_n(F_{\mathbf A})  / \widetilde{K}_n  {\mathcal K}_n.     
\end{align*}
Refer to Section~\ref{subsec:notation} for notation appearing in the definition of $Y^{(n)}_{\mathcal{K}}$ (in particular, recall that $\widetilde{K}_n$ is defined as $\prod_{v\in \Sigma_{F,\infty}} \mC^\times {\rm U}(n)$). Let $\Lambda_n$ be the set of dominant integral weights, defined as  
\begin{align} \label{eq:Lambda_n}
 \Lambda_n:=\{\lambda =(\lambda_1,\lambda_2,\dots ,\lambda_n)\in \mZ^n 
\mid \lambda_1\geq \lambda_2\geq \dots \geq \lambda_n\}. 
\end{align}
For each $\lambda\in \Lambda_n$, let $V_\lambda$ be an irreducible holomorphic finite dimensional representation of $\rGL_n(\mC)$ with highest weight $\lambda$. The dual weight $\lambda^\vee$ of $\lambda=(\lambda_1,\lambda_2,\dotsc,\lambda_n)\in \Lambda_n$ is defined as $\lambda^\vee=(-\lambda_n,-\lambda_{n-1},\dotsc,-\lambda_1)$, which is again an element of $\Lambda_n$. Recall that an irreducible cuspidal automorphic representation $\pi^{(n)}={\bigotimes}'_{v\in \Sigma_F} \pi^{(n)}_v$ of $\rGL_n(F_{\mA})$ is said to be {\em cohomological} (or {\em regular algebraic}) if there exists an open compact subgroup $\mathcal{K}_n$ of $\rGL_n(F_{\mA,\mathrm{fin}})$ and an element $\blambda=(\lambda_\sigma)_{\sigma\in I_F}\in \Lambda_n^{I_F}$ such that the $\pi^{(n)}_\mathrm{fin}$-isotypic component $H^*_{\rm cusp}(Y^{(n)}_{\mathcal{K}}, \widetilde{\mathcal{V}}(\blambda^\vee))[\pi^{(n)}_\mathrm{fin}]$ of the cuspidal cohomology group (as a module over the Hecke algebra $\mathcal{H}_{\cK_n}:=C_c^{\infty}(\rGL_n(F_{\mA,\mathrm{fin}})/\!/\cK_n)$) is not trivial for some degree $\ast$, where $\widetilde{\mathcal{V}}(\blambda^\vee)$ is the local system on $Y^{(n)}_{\mathcal{K}}$ defined by $\widetilde{V}(\blambda^\vee):=\bigotimes_{\sigma\in I_F} V_{\lambda_\sigma^\vee}$, regarded as a representation of $\rGL_n(F)$ (see Section~\ref{subsec:alg_rep} for details). The (dual) weight $\blambda$ of the local system is uniquely determined  by the Langlands parameters of the archimedean part $\pi^{(n)}_{\infty}$ of $\pi^{(n)}$, on which the following {\em purity condition} is imposed:
\begin{quotation}
 there exists an integer $\mathsf{w}\in \mZ$ (called the {\em purity weight} of $\blambda$) such that $\lambda_{\sigma}-\lambda_{\bar{\sigma}}^\vee =(\sw,\sw,\dots ,\sw )$ holds for every $\sigma\in I_F$, where $\bar{\sigma}$ denotes the complex conjugate of $\sigma\in I_F$.
\end{quotation}
In the present article, we call $\blambda$ the highest weight associated with $\pi^{(n)}$ for brevity. 
The symbols $\pi^{(n)}_\infty$ and $\pi^{(n)}_\rfin$ denote the archimedean part ${\bigotimes}_{v\in \Sigma_{F,\infty}} \pi^{(n)}_v$ 
and the non-archimedean part ${\bigotimes}'_{v\in \Sigma_{F,\rfin}} \pi^{(n)}_v$ of $\pi^{(n)}$, respectively. 

For $n>1$, we also consider a cohomological irreducible cuspidal automorphic representation $\pi^{(n-1)}$ of $\rGL_{n-1}(F_{\mA})$, and let $L(s,\pi^{(n)}\times \pi^{(n-1)})=\prod_{v\in \Sigma_F}L_v(s,\pi^{(n)}_v\times \pi^{(n-1)}_v)$ denote the (complete) Rankin--Selberg $L$-function of $\pi^{(n)}$ and $\pi^{(n-1)}$. We say that a half integer $\frac{1}{2}+m$ is a {\em critical point} of $L(s,\pi^{(n)}\times \pi^{(n-1)})$ if neither $L_\infty(s,\pi^{(n)}_\infty \times \pi^{(n-1)}_\infty )=\prod_{v\in \Sigma_{F,\infty}}L_v(s,\pi^{(n)}_v \times \pi^{(n-1)}_v )$ nor $L_\infty(1-s,\pi^{(n),\vee}_\infty \times \pi^{(n-1),\vee}_\infty )$ has a pole at $s=\frac{1}{2}+m$ (the superscript $\vee$ denotes the contragredient representation).  Let $\bmu\in \Lambda_{n-1}$ be the highest weight associated with $\pi^{(n-1)}$, whose purity weight is $\mathsf{w}'$. 
   Throughout of the present article, 
    assume that there exists $m_0\in \mZ$ such that $\blambda^\vee$ {\em interlaces} $\bmu+m_0:=(\mu_\sigma+m_0)_{\sigma\in I_F}$ 
(denoted as $\blambda^\vee \succeq \bmu+m_0$), that is, the inequalities 
\begin{align} \label{eq:interlace_condition} 
-\lambda_{\sigma,n}\geq \mu_{\sigma,1}+m_0\geq -\lambda_{\sigma,n-1}\geq \mu_{\sigma,2}+m_0\geq \cdots \geq -\lambda_{\sigma,2}\geq \mu_{\sigma,n-1}+m_0\geq -\lambda_{\sigma,1} 
\end{align}
hold for every $\sigma\in I_F$. 
Then it is well known that $\frac{1}{2}+m$ is a critical point of $L(s,\pi^{(n)}\times \pi^{(n-1)})$ if and only if $\blambda^\vee \succeq \bmu+m$; 
refer to \cite[Section~2.4.5]{rag16} and \cite[Lemma~6.3]{hn} for details.

\subsection{The cup product pairing and zeta integrals}  \label{subsec:cupproduct}

In this subsection we describe critical values of the Rankin--Selberg $L$-function as the cup product of certain specific cohomology classes. Such a cohomological presentation of critical values goes back to Kazhdan, Mazur and Schmidt's ``generalised modular symbol method'' introduced in \cite{kms00}, which is then further developed by many authors including Mahnkopf \cite{mah05}, Kasten and Schmidt \cite{ks13}, and Raghuram \cite{rag10,rag16}. 

\subsubsection{Cohomological cup product pairing}\label{sec:genmodsymb}

For each open compact subgroup ${\mathcal K}_n$ of ${\rm GL}_n(F_{{\mathbf A}, {\rm fin}})$, define $\mathcal{Y}^{(n)}_{\mathcal K}$ to be 
\begin{align*}
   \mathcal{Y}^{(n)}_{\mathcal K}
     = {\rm GL}_n(F)  \backslash  {\rm GL}_n(F_{\mA})  /    K_n \mathcal{K}_n 
\end{align*}
(recall that $K_n$ is defined as $\prod_{v\in \Sigma_{F,\infty}} {\rm U}(n)$). Then, by definition, there exists a canonical projection $\mathrm{p}_{n}\colon \mathcal{Y}^{(n)}_{\mathcal{K}} \to Y^{(n)}_{\mathcal{K}}$.

Assume that $n>1$ holds. Note that $\rGL_{n-1}$ is regarded as a subgroup of $\rGL_n$ via the diagonal map
\begin{align}
\label{eq:def_iotan}
&\iota_n\colon \mathrm{GL}_{n-1} \to \mathrm{GL}_{n}; g\mapsto 
\begin{pmatrix}
g&O_{n-1,1}\\ 
O_{1,n-1}&1
\end{pmatrix}.
\end{align}
Suppose that $\iota_n^{-1}(\mathcal{K}_n) = \mathcal{K}_{n-1}$ holds. Then the embedding $\iota_n$ also induces a morphism $\mathcal{Y}^{(n-1)}_{\mathcal{K}} \to \mathcal{Y}^{(n)}_{\mathcal{K}}$, for which we use the same symbol $\iota_n$ by abuse of notation.  
Then Borel and Prasad's lemma yields that the composition $j_n:= \mathrm{p}_n \circ \iota_n \colon \mathcal{Y}^{(n-1)}_{\mathcal{K}} \to Y^{(n)}_{\mathcal{K}}$
is a proper map (see \cite[Lemma 2.7]{ash80}); in particular the pullback with respect to $j_n$ preserves the cohomology groups with compact supports.

Now, as in Section~\ref{subsec:settings}, let $\pi^{(n)}$ and $\pi^{(n-1)}$ be cohomological irreducible cuspidal automorphic representations of $\rGL_n(F_{\mA})$ and $\rGL_{n-1}(F_{\mA})$, respectively. Let $\blambda$ and $\bmu$ respectively denote the highest weights associated with $\pi^{(n)}$ and $\pi^{(n-1)}$. 
Assume that $m\in \mZ$ satisfies $\blambda^\vee \succeq \bmu+m$ (in particular, $\frac{1}{2}+m$ is a critical point of $L(s,\pi^{(n)}\times \pi^{(n-1)})$). 
The branching rule for the pair  $(\rGL_n,\rGL_{n-1})$ then implies that $\mathrm{Hom}_{\rGL_{n-1}(F)}(\widetilde{V}(\blambda^\vee),\widetilde{V}(\bmu+m))\neq \{0\}$ holds, or in other words, there exists a nontrivial pairing $[\![\cdot,\cdot]\!]_m\in \mathrm{Hom}_{\rGL_{n-1}(F)}(\widetilde{V}(\blambda^\vee)\otimes_{\mC}\widetilde{V}(\bmu^\vee-m), \mC_{\rtriv} )$ (note that $(\bmu+m)^\vee$ equals $\bmu^\vee-m$), where $\mC_{\rtriv}:=\mC$ is the trivial $\rGL_{n-1}(F)$-module. Let $\mC(m)$ denote a free $\mC$-module of rank one on which $\rGL_{n-1}(F)$ acts via $g \mapsto \prod_{\sigma\in I_F} \det \sigma(g)^m$. Then, by composing with a natural $\rGL_{n-1}(F)$-equivariant isomorphism 
\begin{align*}
 \widetilde{V}(\bmu^\vee)\xrightarrow{\, \sim \,} \widetilde{V}(\bmu^\vee-m)\otimes_\mC \mC(m),
\end{align*} 
we obtain a $\rGL_{n-1}$-equivariant pairing
\begin{align}\label{eq:genpairing}
\widetilde{V}(\blambda^\vee)\otimes_\mC \widetilde{V}(\bmu^\vee) \xrightarrow{\, \sim\,} \widetilde{V}(\blambda^\vee) \otimes_\mC \widetilde{V}(\bmu^\vee-m)\otimes_{\mC}\mC(m) \xrightarrow{\, [\![\cdot,\cdot]\!]_m \otimes \mathrm{id}\,} \mC(m),
\end{align}    
for which we use the same symbol $[\![\cdot,\cdot]\!]_m$.
The cup product of cohomology groups then induces a pairing: 
\begin{align}\label{eq:pairingpre1c}
& H^{\rb_{n,F}}_\mathrm{c}(Y^{(n)}_\mathcal{K}, \widetilde{\mathcal{V}}(\boldsymbol{\lambda}^\vee)) 
    \times H^{\rb_{n-1,F}} (Y^{(n-1)}_\mathcal{K}, \widetilde{\mathcal{V}}(\boldsymbol{\mu}^\vee))   \\
&\quad \xrightarrow{\, (j_n^*,{\rm p}_{n-1}^*)\,}
     H^{\rb_{n,F}}_\mathrm{c}(\mathcal{Y}^{(n-1)}_\mathcal{K}, j^\ast_n \widetilde{\mathcal{V}}(\boldsymbol{\lambda}^\vee)) 
            \times H^{\rb_{n-1,F}} (\mathcal{Y}^{(n-1)}_\mathcal{K}, \mathrm{p}^\ast_{n-1} \widetilde{\mathcal{V}}(\boldsymbol{\mu}^\vee))   \xrightarrow{\, [\![\cdot,\cdot]\!]_m\circ \cup\,} H^{\rb_{n,F}+\mathrm{b}_{n-1,F}}_\mathrm{c}( \mathcal{Y}^{(n-1)}_\mathcal{K},  \widetilde{\mC}(m) ). \nonumber 
\end{align}
Here $\rb_{n,F}:=\frac{1}{4}[F:\mQ]n(n-1)$ (resp.\ $\rb_{n-1,F}$) is the bottom degree of the nontrivial cuspidal cohomology group $H^*_{\mathrm{cusp}}(Y^{(n)}_{\mathcal{K}},\widetilde{\mathcal{V}}(\blambda^\vee))$ (resp.\ $H^*_{\mathrm{cusp}}(Y^{(n-1)}_{\mathcal{K}},\widetilde{\mathcal{V}}(\bmu^\vee))$), and  $\widetilde{\mC}(m)$ denotes a local system on $\mathcal{Y}^{(n-1)}_\mathcal{K}$ determined by $\mC(m)$.
Write $\widetilde{\mC}_\mathrm{triv} = \widetilde{\mC}(0)$, and let $\mathrm{Tw}_m\colon \widetilde{\mC} (m) \longrightarrow \widetilde{\mC}_\mathrm{triv}$ be  a twisting morphism of local systems defined as (\ref{eq:tw_e}) in Section~\ref{sec:locsys}. Composing (\ref{eq:pairingpre1c}) with $\mathrm{Tw}_m$, we obtain 
\begin{align}\label{eq:pairingpre2c}
    H^{\rb_{n,F}}_\mathrm{c}(Y^{(n)}_\mathcal{K}, \widetilde{\mathcal{V}}(\boldsymbol{\lambda}^\vee)) 
    \times H^{\rb_{n-1,F}} (Y^{(n-1)}_\mathcal{K}, \widetilde{\mathcal{V}}(\boldsymbol{\mu}^\vee))   
  \longrightarrow 
  H^{\rb_{n,F}+\mathrm{b}_{n-1,F}}_\mathrm{c}( \mathcal{Y}^{(n-1)}_\mathcal{K},  \widetilde{\mC}_\mathrm{triv}).   
\end{align}
Next, by introducing an order on $\Sigma_{F,\infty}$ arbitrarily, we fix an orientation on 
$\mathrm{GL}_{n-1}(F_{\mathbf{A}, \infty})/ K_{n-1}$  
determined by 
  $\bigwedge_{v\in \Sigma_{F,\infty}}\mathbf{E}^\vee_{\gp_{n-1}} \in \bigwedge_{v\in\Sigma_{F,\infty}}\bigwedge^{(n-1)^2} \mathfrak{p}^\vee_{n-1\mathbf{C}}$  
  (see (\ref{eq:boldE}) in Section~\ref{subsec:diff_form} for the definition of $\mathbf{E}^\vee_{\gp_{n-1}}$). This induces an orientation on $\mathcal{Y}^{(n-1)}_{\mathcal{K}}$, and thus the corresponding fundamental class $[\mathcal{Y}^{(n-1)}_{\mathcal{K}}]$ is defined. 
Then, due to the numerical coincidence $\mathrm{b}_{n,F}+\mathrm{b}_{n-1,F} = \dim \mathcal{Y}^{(n-1)}_\mathcal{K}$,  Poincar\'e duality implies that the cap product with $[\mathcal{Y}^{(n-1)}_{\mathcal{K}}]$ induces an isomorphism $H^{\rb_{n,F}+\mathrm{b}_{n-1,F}}_\mathrm{c}( \mathcal{Y}^{(n-1)}_\mathcal{K},  \widetilde{\mC}_\mathrm{triv}  ) \xrightarrow{\, \sim \,} \mC$. Composing (\ref{eq:pairingpre2c}) with this isomorphism, we finally obtain a pairing
\begin{align}\label{eq:pairingpre3c}
 [\![\cdot,\cdot]\!]_m  \colon  H^{\rb_{n,F}}_\mathrm{c}(Y^{(n)}_\mathcal{K}, \widetilde{\mathcal{V}}(\boldsymbol{\lambda}^\vee)) 
    \times H^{\rb_{n-1,F}} (Y^{(n-1)}_\mathcal{K}, \widetilde{\mathcal{V}}(\boldsymbol{\mu}^\vee))   
    \longrightarrow \mC 
\end{align}
(the same notation is used in (\ref{eq:genpairing}), but no confusion likely occurs). 

Later we will define a local pairing $\langle \cdot,\cdot\rangle_{\lambda,\mu}^{(l)}$. By using it, we will normalise the pairing $[\![\cdot,\cdot]\!]_m$ introduced in  (\ref{eq:genpairing}) as
\begin{align} \label{eq:normalisation}
 [\cdot,\cdot]^{(m)}_{\blambda,\bmu}:=\prod_{v\in \Sigma_{F,\infty}}\langle \cdot,\cdot\rangle^{(m)}_{\lambda_v^\vee,\mu_v^\vee}\langle\cdot,\cdot\rangle^{(m)}_{\lambda_v-\mathsf{w},\mu_v-\mathsf{w'}}.
\end{align}
See (\ref{eq:pairingLambdaMu}) in Section~\ref{subsec:Pairings} and (\ref{eq:global_pairing1}) in Section~\ref{subsec:alg_rep} for precise definitions of $\langle \cdot,\cdot\rangle^{(l)}_{\lambda,\mu}$ and $[\cdot,\cdot]^{(m)}_{\blambda,\bmu}$.

\subsubsection{Relation with zeta integrals} \label{subsec:relationzeta} 

For each $\varepsilon\in \{\pm\}$, let $\cW(\pi^{(n)}, \psi_\varepsilon)$ denote the global Whittaker model of $\pi^{(n)}$, which is decomposed as the restricted tensor product ${\bigotimes}'_{v\in \Sigma_F} \cW(\pi^{(n)}_v,\psi_{\varepsilon,v})$ of the local Whittaker models so that the pure tensor $\bigotimes_{v\in \Sigma_F} w_v$ corresponds to the function $g=(g_v)_{v\in \Sigma_F}\mapsto \prod_{v\in \Sigma_F}w_v(g_v)$ for $g\in \rGL_n(F_{\mA})$. Then 
\begin{align*}
H^{\rb_{n,F}}(\mathfrak{g}_{n\mC},\widetilde{K}_n; \pi^{(n)}\otimes_{\mC} \widetilde{V}(\blambda^\vee))^{\mathcal{K}_n} \cong H^{\rb_{n,F}}(\mathfrak{g}_{n\mC},\widetilde{K}_n; \cW(\pi^{(n)}_\infty, \psi_{\varepsilon,\infty}) \otimes_{\mC} \widetilde{V}(\blambda^\vee) )\otimes \cW(\pi^{(n)}_\mathrm{fin},\psi_{\varepsilon,\mathrm{fin}})^{\mathcal{K}_n} 
\end{align*} 
appears as a unique isotypic component of the cuspidal cohomology group $H^{\rb_{n,F}}_{\mathrm{cusp}}(Y^{(n)}_{\cK},\widetilde{\cV}(\blambda^\vee))$ which is regarded as a  module over the Hecke algebra $\mathcal{H}_{\cK_n}:=C_c^{\infty}(\rGL_n(F_{\mA,\mathrm{fin}})/\!/\cK_n)$ associated to $\cK_n$; see, for example, \cite[p.p.\ 123--124]{clo90}. 
Since the $(\mathfrak{g}_{n\mC},\widetilde{K}_n)$-cohomology group $H^{\rb_{n,F}}(\mathfrak{g}_{n\mC},\widetilde{K}_n; \cW(\pi^{(n)}_\infty,\psi_{\varepsilon,\infty}) \otimes_{\mC} \widetilde{V}(\blambda^\vee) )$ is known to be of dimension one over $\mC$ (refer to \cite[Lemme 3.14]{clo90}), we here temporarily fix its generator $\{\pi^{(n)}_\infty\}_\varepsilon$, 
and consider a homomorphism
\begin{align} \label{eq:Fourier}
  \mathscr{F}_{\{\pi^{(n)}_\infty\}_\varepsilon} \colon \cW(\pi^{(n)}_{\rm fin}, \psi_{\varepsilon,\mathrm{fin}})^{\cK_n}  
    \longrightarrow   H^{\rb_{n,F}}(\mathfrak{g}_{n\mC}, \widetilde{K}_n \, ; \pi^{(n)}\otimes_{\mC} \widetilde{V}(\blambda^\vee))^{\mathcal{K}_n} \hookrightarrow 
         H^{\rb_{n,F}}_{\rm cusp}   (   Y^{(n)}_{\cK},    \widetilde{\cV}(\blambda^\vee)     )
\end{align}
induced by $w_{\mathrm{fin}} \mapsto \{\pi^{(n)}_\infty \}_\varepsilon \otimes w_{\mathrm{fin}}$.
Similarly we consider a map 
\begin{align*}
  \mathscr{F}_{\{\pi^{(n-1)}_\infty\}_{-\varepsilon}} \colon \cW(\pi^{(n-1)}_{\rm fin}, \psi_{-\varepsilon,\mathrm{fin}})^{\cK_{n-1}}  
    \longrightarrow   H^{\rb_{n-1,F}}(\mathfrak{g}_{n-1\mC},\widetilde{K}_{n-1}\,; \pi^{(n-1)}\otimes_{\mC} \widetilde{V}(\bmu^\vee))^{\mathcal{K}_{n-1}} \hookrightarrow 
         H^{\rb_{n-1,F}}_{\rm cusp}   (   Y^{(n-1)}_{\cK},    \widetilde{\cV}(\bmu^\vee)     )
\end{align*}
induced by $w'_{\mathrm{fin}} \mapsto \{\pi^{(n-1)}_\infty \}_{-\varepsilon} \otimes w'_{\mathrm{fin}}$ for a generator 
$\{\pi^{(n-1)}_\infty\}_{-\varepsilon}$ 
 of $H^{\rb_{n-1,F}}(\mathfrak{g}_{n-1\mC},\widetilde{K}_{n-1}; \cW(\pi^{(n-1)}_\infty, \psi_{-\varepsilon,\infty}) \otimes_{\mC} \widetilde{V}(\bmu^\vee) )$.

Since the cuspidal cohomology $H^{\rb_{n,F}}_\mathrm{cusp}(Y^{(n)}_\mathcal{K}, \widetilde{\mathcal{V}}(\boldsymbol{\lambda}^\vee))$ is naturally regarded as a submodule of $H^{\rb_{n,F}}_\mathrm{c}(Y^{(n)}_\mathcal{K}, \widetilde{\mathcal{V}}(\boldsymbol{\lambda}^\vee)) $ (see \cite[5.5~Corollary]{bor81},  \cite[page 123]{clo90} or \cite[Section 3.2.1]{mah05}), we can evaluate the global pairing $[\![\cdot,\cdot]\!]_m$ at the cuspidal cohomology classes $\mathscr{F}_{\{\pi^{(n)}_\infty\}_{\varepsilon}}(w_\mathrm{fin})$ and $\mathscr{F}_{\{\pi^{(n-1)}_\infty\}_{-\varepsilon}}(w'_\mathrm{fin})$, which is interpreted as the integration of the product of certain cuspidal automorphic forms belonging to $\pi^{(n)}$ and $\pi^{(n-1)}$ on $\mathcal{Y}^{(n-1)}_{\cK}$. 
For $v\in \Sigma_{F,\mathrm{fin}}$, 
 define $\mathrm{d}h_v$ to be the Haar measure on  
 $\rGL_{n-1}(F_v)$ normalised so that $\mathrm{vol}(\rGL_{n-1}(\gr_{F,v}),{\rm d}h_v)=1$ holds. 
 We also define the Haar measure    
 $\mathrm{d}x_v$ on $\mathrm{N}_{n-1}(F_v)$ by 
\begin{align*}
&\rd x_v :=\prod_{1\leq i<j\leq n-1}\rd x_{v,i,j}&\text{ with }\quad x_v=(x_{v,i,j})_{1\leq i,j\leq n-1}\in \rN_{n-1}(F_v), 
\end{align*}
where $\rd x_{v,i,j}$ is the self-dual measure on $F_v$ with respect to the fixed additive character $\psi_{-\varepsilon, v}$. 
We take the quotient measure $\mathrm{d}g_v$ on $\rN_{n-1}(F_v)\backslash \rGL_{n-1}(F_v)$ 
 which is characterised by the following identity for each measurable function $f$ on $\rGL_{n-1}(F_v)$: 
 \begin{align*}
 \int_{\rGL_{n-1}(F_v)}  f(h_v) \mathrm{d}h_v
 =\int_{\rN_{n-1}(F_v)\backslash \rGL_{n-1}(F_v)}  
    \left(\int_{\rN_{n-1}(F_v)} f(x_vg_v) \mathrm{d}x_v \right) \mathrm{d}g_v. 
 \end{align*}       
By standard unfolding calculations, it is decomposed into the product of local zeta integrals when both $w_\mathrm{fin}=\bigotimes_{v\in \Sigma_{F,\mathrm{fin}}}w_v$ and $w_{\mathrm{fin}}'=\bigotimes_{v\in \Sigma_{F,\mathrm{fin}}}w_v'$ are pure tensors: 
\begin{align} \label{eq:product_zeta}
\bigl[\!\bigl[ \mathscr{F}_{\{\pi^{(n)}_\infty\}_{\varepsilon}}(w_\mathrm{fin}), \mathscr{F}_{\{\pi^{(n-1)}_\infty\}_{-\varepsilon}}(w'_\mathrm{fin})\bigr]\!\bigr]_{m}
= 
   \widetilde{\mathcal{I}}_\infty^{(m)} (\{\pi^{(n)}_\infty\}_\varepsilon, \{\pi^{(n-1)}_\infty\}_{-\varepsilon})
\prod_{v\in \Sigma_{F, \mathrm{fin}}} 
   Z_v(s, w_v , w'_v  )\vert_{s=\frac{1}{2}+m}.
\end{align}
Here, for $v\in \Sigma_{F,\mathrm{fin}}$,  
\begin{align} \label{eq:local_zeta_int}
 Z_v(s,w_v,w_v')=\int_{\rN_{n-1}(F_v)\backslash \rGL_{n-1}(F_v)} w_v(\iota_n(g_v))w'_v(g_v) \lvert \det g_v\rvert_v^{s-\frac{1}{2}} {\rm d}g_v
\end{align}
is nothing but the local zeta integral introduced by Jacquet, Piatetski-Shapiro and Shalika \cite[p.~387 (2)]{jpss83}. 
Similarly 
\begin{align*}
\widetilde{\cI}_\infty^{(m)}(\{\pi^{(n)}_\infty\}_\varepsilon,\{\pi^{(n-1)}_\infty \}_{-\varepsilon})=\prod_{v\in \Sigma_{F,\infty}}\widetilde{\cI}_v^{(m)}(\{\pi^{(n)}_v\}_\varepsilon,\{\pi^{(n-1)}_v\}_{-\varepsilon}) 
\end{align*}
is deeply related with the archimedean local zeta integral, but is much more complicated because  $\{\pi^{(n)}_\infty\}_\varepsilon$ and $\{\pi^{(n-1)}_\infty\}_{-\varepsilon}$ have awful shapes in general. We shall propose a precise description of $\cI_v^{(m)}(\{\pi^{(n)}_v\}_\varepsilon,\{\pi^{(n-1)}_v\}_{-\varepsilon})$ for $v\in \Sigma_{F,\infty}$ in Section~\ref{subsec:explicit_arch_zeta}.  It is known to be nonzero by the work of B.~Sun \cite{sun17}, but quite uncomputable unless we specify several concerning data. Later we explicate the global pairing $[\cdot,\cdot]_{\blambda,\bmu}^{(m)}$ in Section~\ref{subsec:Pairings}, generators $[\pi^{(n)}_\infty]_\varepsilon$, $[\pi^{(n-1)}_\infty]_{-\varepsilon}$ of the relative Lie algebra cohomology in Section~\ref{subsec:U(n)_inv}, and non-archimedean Whittaker functions $w^\mathrm{ess}_\mathrm{fin}(\pi^{(n)})_{\varepsilon}$ and $w^\mathrm{ess}_\mathrm{fin}(\pi^{(n-1)})_{-\varepsilon}$ in Section~\ref{subsec:whittvec}. To construct $[\pi^{(\cdot)}_\infty]_{\pm \varepsilon}=\bigotimes_{v\in \Sigma_{F,\infty}}[\pi^{(\cdot)}_v]_{\pm \varepsilon}$, precise study of $(\mathfrak{gl}_{n\mC}, \mC^\times \rU(n))$-cohomology groups is indispensable. The next subsection is devoted to preparation for analysis of finite dimensional representations of the general linear groups. In particular, we introduce notion of {\em rational} Gel'fand--Tsetlin basis in Section~\ref{subsec:GT_basis}, which plays a crucial role in the explicit construction of generators $[\pi^{(n)}_\infty]_\varepsilon$ and  $[\pi^{(n-1)}_\infty]_{-\varepsilon}$  in Section~\ref{subsec:U(n)_inv}.

\subsection{Preliminaries on finite dimensional representations of $\rGL_n$} \label{subsec:finite_dimensional}

We prepare several facts from theory of finite dimensional representations of the general linear groups, which will be used in the arguments in Sections~\ref{subsec:explicit_generators} and \ref{subsec:uniformintegrality}. In Section~\ref{subsec:realisation}, we construct a certain realisation of highest weight representations over an integral domain of characteristic $0$. We then introduce the notion of Gel'fand--Tsetlin basis and its modification in Section~\ref{subsec:GT_basis}. Proposition~\ref{prop:xiM_integral} stated here will be used to equip $V_\lambda$ with an appropriate integral structure. The local pairing $\langle \cdot, \cdot\rangle_{\lambda,\mu}^{(l)}$ appearing in the preceding subsection will be defined in Section~\ref{subsec:Pairings}. Section~\ref{subsec:Cartan_comp} contains several notions on Cartan components of the tensor representation $V_\lambda(\mathcal{A})\otimes_\mathcal{A}V_{\lambda'}(\mathcal{A})$. 

\subsubsection{Determinantal realisation of rational representations of $\rGL_n$}
\label{subsec:realisation}

Let $\cA$ be a (commutative) integral domain of characteristic $0$. 
In the present article, we always identify 
the subring of $\cA$ generated by the multiplicative identity 
with $\mZ$. 
In this subsection, we define finite rank representations of 
$\rGL_n(\cA )$, which are 
irreducible holomorphic finite dimensional representations when $\cA =\mC$.

Put $\cI_{n}=\bigsqcup_{k=1}^n\cI_{n, k}$ with 
$\cI_{n, k}:=\{  I \subset \{1,2, \dots, n\} \mid \# I=k\}$, 
where $\# I$ denotes the cardinality of $I$. 
For $1\leq k\leq n,\ I, J \in {\mathcal I}_{n, k}$ and 
a square matrix $g=(g_{i,j})_{1\leq i,j\leq n}$ of size $n$, 
we define  ${\det}_{I, J}(g)$ to be
\begin{align*}
{\det}_{I,J}(g) := \det \left(\begin{array}{cccc}  
g_{i_1, j_1} & g_{i_1, j_2} & \cdots & g_{i_1, j_k} \\ 
g_{i_2, j_1} & g_{i_2, j_2} & \cdots & g_{i_2, j_k} \\ 
\vdots  &  \vdots & \ddots &  \vdots  \\
g_{i_k, j_1} & g_{i_k, j_2} & \cdots & g_{i_k, j_k} \\ 
\end{array}\right),
\end{align*}
where we write 
$I =\{ i_1, i_2, \dots, i_k  \}$ and $J =\{ j_1, j_2, \dots, j_k  \}$ 
with $i_1 < i_2 < \dots < i_k$ and $j_1 < j_2 < \dots < j_k$. 
In our notation, Cauchy--Binet's formula implies that, 
for $1\leq k\leq n,\ I, I'\in \cI_{n,k}$ and square matrices $g, g'$ of size $n$, the equality 
\begin{align}
\label{eq:cauchy_binet}
{\det}_{I,I'}(gg')&=\sum_{J\in \cI_{n,k}}
{\det}_{I,J}(g){\det}_{J,I'}(g')
\end{align}
holds. For $1\leq k\leq n$, $I \in {\mathcal I}_{n, k}$ 
and a square matrix $g$ of size $n$, 
we use the symbol ${\det}_{I}(g)$ to denote ${\det}_{ \{ 1, 2, \dots, k \}, I  } (g)$  for simplicity.

Let $z =(z_{i, j})_{1\leq i,j\leq n}$ be a square matrix in variables $z_{i, j}$ with $1\leq i,j \leq n$. 
Let $\cA [z,(\det z)^{-1} ]$ be the ring of polynomials 
in $z_{i, j}$ ($1\leq i,j \leq n$), $(\det z)^{-1}$ with 
coefficients in $\cA$. 
We regard $\cA [z,(\det z)^{-1} ]$ as a $\rGL_n(\cA )$-module 
via  
\begin{align*}
R(g)f(z) &=f(zg)&
&\text{for\;} g\in \rGL_n(\cA ) \text{\; and\;} f(z)\in \cA [z,(\det z)^{-1} ].  
\end{align*}
Recall that  $\Lambda_n$ denotes the set of dominant integral weights as in (\ref{eq:Lambda_n}).
For $\lambda =(\lambda_1,\lambda_2,\dots ,\lambda_n)\in \Lambda_n$,   
let $\cL (\lambda )$ be the set of 
$l =(l_I)_{I\in \cI_n}\in \mZ^{2^n-1}$ such that 
$l_{\{1,2,\dots ,n\}}=\lambda_n$, $l_I\in \mN_0$ for $I\in \bigsqcup_{k=1}^{n-1}\cI_{n,k}$ and 
\begin{align*}
 \sum_{I \in \cI_{n, k}  }l_I &= \lambda_k - \lambda_{k+1} & \text{for any\; } 1\leq k\leq n-1.
\end{align*} 
We then define $V_{\lambda}(\cA )$ as an $\cA$-submodule of 
$\cA [z,(\det z)^{-1} ]$ spanned by 
\begin{align}
\label{eq:Vlambda_generator}
f_l(z)
&:= \prod_{I \in \cI_{n}}
({\det}_I (z))^{l_I}&   
&\text{for all \;} l =(l_I)_{I\in \cI_n}\in \cL (\lambda ). 
\end{align}
Cauchy--Binet's formula (\ref{eq:cauchy_binet}) implies that 
$V_{\lambda}(\cA )$ is closed under the action $R$ of $\rGL_n(\cA )$, 
and we define $\tau_\lambda$ to be the induced action on $V_\lambda (\cA )$, 
that is, $\tau_\lambda (g)=R(g)|_{V_\lambda (\cA )}$ for $g\in \rGL_n(\cA )$.

Let $\cA'$ be an integral domain such that 
$\cA $ is a subring of $\cA'$. 
By definition, we have  
$V_{\lambda}(\cA )\subset V_{\lambda}(\cA')$ and 
there is an isomorphism 
\[
V_{\lambda}(\cA )\otimes_{\cA}\cA' \ni 
\rv \otimes c\mapsto c\rv  \in V_{\lambda}(\cA')
\]
of $\cA'$-modules. 
Let $\lambda'\in \Lambda_n$. We regard 
$V_{\lambda}(\cA )\otimes_\cA V_{\lambda'}(\cA )$ as 
an $\cA$-submodule of $V_{\lambda}(\cA')\otimes_{\cA'}V_{\lambda'}(\cA')$ 
via the natural embedding $\rv \otimes \rv'\mapsto \rv \otimes \rv'$.

For simplicity, denote the complex representation $V_{\lambda}(\mC )$ by $V_\lambda$. 
Then it is known that $(\tau_\lambda ,V_\lambda )$ is 
an irreducible holomorphic finite dimensional representation of 
$\rGL_n(\mC )$ with highest weight $\lambda$ 
(\textit{cf.} \cite[Section 16.1.5]{Vilenkin_Klimyk_001}). 
Let $h (\lambda)=(h(\lambda)_I)_{I\in \cI_n}$ be the element of $\cL (\lambda )$ 
determined by 
\begin{align}
\label{eq:def_hlambda_gen}
&h(\lambda )_I =\left\{\begin{array}{ll}
\lambda_k-\lambda_{k+1}&\text{if $I=\{1,2,\dots ,k\}$},\\
0&\text{otherwise}
\end{array}\right.&
&\text{for\;} I\in \cI_{n,k} \text{\; with \;} 1\leq k\leq n-1.
\end{align}
Then $f_{h (\lambda)}(z)$ is a highest weight vector in 
$V_\lambda $ of weight $\lambda$, that is, $f_{h(\lambda)}(z)$ satisfies 
\begin{align*}
\tau_{\lambda}(E_{i,i})f_{h (\lambda)}(z)&=\lambda_if_{h (\lambda)}(z)
\qquad \text{for\;} 1\leq i\leq n& \text{\; and\;}
&&\tau_{\lambda}(E_{j,k})f_{h (\lambda)}(z)&=0
\qquad \text{for\; }1\leq j<k\leq n. 
\end{align*}
Here we use the same symbol $\tau_\lambda$ for the associated $\mathfrak{gl}_n$-action by abuse of notation.
We fix a $\rU (n)$-invariant hermitian inner product 
$(\cdot ,\cdot )_\lambda$ on $V_\lambda$ so that 
$(f_{h (\lambda)}(z),f_{h (\lambda)}(z) )_\lambda =1$.

\subsubsection{Gel'fand--Tsetlin basis}
\label{subsec:GT_basis}

In this subsection, we introduce the Gel'fand--Tsetlin basis and prepare several notation. 
We call  
\begin{align*}
&M=({m}_{i,j})_{1\leq i\leq j\leq n}=\left(\begin{array}{c}
{m}_{1,n}\quad {m}_{2,n}\quad \quad  \dots \quad \quad {m}_{n,n}\\
{m}_{1,n-1}\ \ \dots \ \ {m}_{n-1,n-1}\\
\cdots \ \ \cdots \ \ \cdots \\
{m}_{1,2}\ \ {m}_{2,2}\\
{m}_{1,1}
\end{array}\right)&
&(m_{i,j}\in \mZ )
\end{align*}
an integral triangular array of size $n$, and call $m_{i,j}$ 
the $(i,j)$-th entry of $M$. 
For an integral triangular array 
$M=({m}_{i,j})_{1\leq i\leq j\leq n}$, 
we define the dual triangular array 
$M^\vee=({m}_{i,j}^\vee )_{1\leq i\leq j\leq n}$ 
of $M$ by setting ${m}_{i,j}^\vee :=-{m}_{j+1-i,j}$.

Let $\lambda =(\lambda_1,\lambda_2,\dots ,\lambda_n)\in \Lambda_n$. 
Recall that, when  $n>1$ and for $\mu =(\mu_1,\mu_2,\dots ,\mu_{n-1})\in \Lambda_{n-1}$, 
we use the symbol $\lambda \succeq \mu$ to indicate that $\lambda$ interlaces $\mu$, that is, the inequalities
\begin{align*} 
\lambda_1\geq  \mu_1\geq \lambda_2\geq \mu_2
\geq \dots \geq \lambda_{n-1} \geq \mu_{n-1} \geq \lambda_n
\end{align*}
hold. Let $\rG (\lambda)$ be the set of integral triangular arrays 
$M=({m}_{i,j})_{1\leq i\leq j\leq n}$ of size $n$
such that, for $1\leq j\leq n$, each row $m^{(j)}:=(m_{1,j},m_{2,j}, \dots ,m_{j,j})$ is an element of $\Lambda_j$ and satisfies
\begin{align}
\label{eq:GT_cdn}
m^{(n)}&=\lambda & \text{and} &&
m^{(k)}&\preceq m^{(k+1)}\quad \text{for\;} 1\leq k\leq n-1.
\end{align}
We call an element of $\rG (\lambda)$ a Gel'fand--Tsetlin pattern of 
type $\lambda$. 
For $M=({m}_{i,j})_{1\leq i\leq j\leq n}\in \rG (\lambda)$, 
we define the weight 
$\gamma^M=(\gamma^M_1,\gamma^M_2,\dots ,\gamma^M_n)$ of $M$ by 
\begin{align}
\label{eq:def_wt_M}
&\gamma_{j}^M:=\sum_{i=1}^j{m}_{i,j}-\sum_{i=1}^{j-1}{m}_{i,j-1}&
&(1\leq j\leq n).
\end{align}

Gel'fand and Tsetlin \cite{Gelfand_Tsetlin_001} construct 
an orthonormal basis $\{\zeta_M\}_{M\in \rG (\lambda)}$ 
of $V_\lambda$ which satisfies the following formulas on 
the $\ggl_n$-action (see Zhelobenko \cite{Zhelobenko_001} 
for a detailed proof): for $M=(m_{i,j})_{1\leq i\leq j\leq n}\in \rG (\lambda)$, we have
\begin{align}
\label{eq:GT_act_wt}
&\tau_\lambda (E_{k,k})\zeta_{M}
=\gamma_{k}^M \zeta_{M}&
&(1\leq k\leq n),\\
\label{eq:GT_act+}
&\tau_\lambda (E_{j,j+1})\zeta_{M}=
\underset{M+\Delta_{i,j}\in \rG (\lambda)}{\sum_{1\leq i\leq j}}
\ra_{i,j}(M)\zeta_{M+\Delta_{i,j}}&
&(1\leq j\leq n-1),\\
\label{eq:GT_act-}
&\tau_\lambda (E_{j+1,j})\zeta_{M}=
\underset{M+\Delta_{i,j}^\vee \in \rG (\lambda)}{\sum_{1\leq i\leq j}}
\ra_{i,j}(M^\vee)\zeta_{M+\Delta_{i,j}^\vee }&
&(1\leq j\leq n-1),
\end{align}
where 
$\Delta_{i,j}$ is the integral triangular array of size $n$ 
with $1$ at the $(i,j)$-th entry and $0$ at the other entries, and 
\begin{align*}
\ra_{i,j}(M)&:=\left|
\frac{\prod_{h=1}^{j+1}(m_{h,j+1}-m_{i,j}-h+i) \cdot 
\prod_{h=1}^{j-1}(m_{h,j-1}-m_{i,j}-h+i-1)}
{\prod_{1\leq h\leq j,\, h\neq i}
(m_{h,j}-m_{i,j}-h+i)(m_{h,j}-m_{i,j}-h+i-1)}
\right|^{\frac{1}{2}}\!. 
\end{align*}
The basis $\{\zeta_M\}_{M\in \rG (\lambda)}$ is called the 
{\em Gel'fand--Tsetlin basis} of $V_\lambda$. 
By (\ref{eq:GT_act_wt}), for $M\in \rG(\lambda)$, each $\zeta_{M}$ is a weight vector in $V_\lambda $ 
of weight $\gamma^M$. In particular, 
$\zeta_{H(\lambda )}$ is a highest weight vector in $V_\lambda $ of 
weight $\lambda$, where $H(\lambda )$ is a unique element of  
$\rG (\lambda)$ whose weight is  $\lambda $, 
that is, $H(\lambda )=(h_{i,j})_{1\leq i\leq j\leq n}$ with 
$h_{i,j}=\lambda_i$. 
In the present article, 
we fix the orthonormal basis $\{\zeta_M\}_{M\in \rG (\lambda)}$ of 
$V_{\lambda}$
so that $\zeta_{H(\lambda )}=f_{h (\lambda)}(z)$.

As in \cite[Section 2.5]{im}, we set 
\begin{align}
\label{eq:def_Qbasis}
\xi_M&:=\sqrt{\rr (M)}\zeta_M&
 \text{for \;} M=(m_{i,j})_{1\leq i\leq j\leq n}\in \rG (\lambda),
\end{align}
where $\rr (M)$ is the rational constant defined by 
\begin{align}
\label{eq:def_rM}
&\rr (M):=\prod_{1\leq i\leq j<k\leq n}
\frac{(m_{i,k}-m_{j,k-1}-i+j)!(m_{i,k-1}-m_{j+1,k}-i+j)!}
{(m_{i,k-1}-m_{j,k-1}-i+j)!(m_{i,k}-m_{j+1,k}-i+j)!}.
\end{align}
Then $\{\xi_M\}_{M\in \rG (\lambda)}$ is an orthogonal basis 
of $V_\lambda$ satisfying $(\xi_M,\xi_M)_\lambda =\rr (M)$ 
($M\in \rG (\lambda)$), and it is well-suited for the $\cA$-module structure for 
an integral domain $\cA$ of characteristic $0$ with a certain mild condition. 
In fact, for several Gel'fand--Tsetlin patterns $M$ of type $\lambda$, $\xi_{M}$ coincides with one of the generators $f_l(z)$ ($l\in \cL (\lambda)$) introduced in (\ref{eq:Vlambda_generator}); see Lemmas \ref{lem:extremal_vec_explicit} and \ref{lem:xi_Hmulambda_explicit} below for details. Consequently the following proposition holds.

\begin{prop}
\label{prop:xiM_integral}
Retain the notation. 
Let $\cA$ be an integral domain of characteristic $0$. 
Assume that  $(\lambda_1-\lambda_n + n-3)!$ is $($multiplicatively$)$ invertible in $\cA$ if $n\geq 3$. 
Then $\xi_M$ is contained in $V_\lambda (\cA \cap \mQ )$ for $M\in \rG (\lambda )$. 
Moreover, we have 
\[
V_\lambda (\cA )=\bigoplus_{M\in \rG (\lambda )}\cA \,\xi_M. 
\]
\end{prop}

A proof of Proposition~\ref{prop:xiM_integral} is given in 
Section \ref{subsec:Amod_str}. 
Here we remark that the assumption 
``$(\lambda_1-\lambda_n + n-3)!$ is invertible in $\cA$ if $n\geq 3$'' 
cannot be removed from Proposition~\ref{prop:xiM_integral} 
({\it cf.} Remark \ref{rem:xiM_gen_Vlambda}).

Let $\cE_\lambda$ be 
the set of extremal weights of $\tau_\lambda$, that is, 
$\cE_\lambda :=\{\sigma \lambda \mid \sigma \in \gS_n\}$ with 
the $\gS_n$-action 
(\ref{eq:def_actSn}). 
For $\gamma =(\gamma_1,\gamma_2,\dots,\gamma_n) 
\in \cE_{\lambda}$, let $H(\gamma )$ be 
the unique element of $\rG (\lambda)$ whose weight is  
$\gamma$, that is, 
$H(\gamma )=(h_{i,j})_{1\leq i\leq j\leq n}\in \rG (\lambda)$ such that, 
for any $1\leq j\leq n$, 
$(h_{1,j},h_{2,j},\dots ,h_{j,j})$ is the unique element of 
$\Lambda_j\cap 
\{\sigma (\gamma_1,\gamma_2,\dots ,\gamma_j)\mid \sigma \in \gS_j\}$. 
In Section \ref{subsec:C_def_ps}, 
extremal weight vectors $\xi_{H(\gamma ) }$ will be used 
for normalisations of the minimal $\rU (n)$-type vectors of 
a principal series representation of $\rGL_n(\mC )$. 
In Section \ref{subsec:Amod_str}, we prove the following lemma, which asserts that 
$\xi_{H(\gamma ) }$ coincides with one of the generators introduced in 
(\ref{eq:Vlambda_generator}). 
\begin{lem}
\label{lem:extremal_vec_explicit}
Let $\lambda =(\lambda_1,\lambda_2,\dots ,\lambda_n)\in \Lambda_n$. For $\gamma =\sigma \lambda \in \cE_\lambda $ 
with $\sigma \in \gS_n$, we have 
$\xi_{H(\gamma )}=\zeta_{H(\gamma )}=f_{h (\gamma)}(z)$
where $h (\gamma)=(h(\gamma)_I)_{I\in \cI_n}\in \cL (\lambda )$ is 
determined by 
\begin{align*}
&h(\gamma )_I =\left\{\begin{array}{ll}
\lambda_k-\lambda_{k+1}&
\text{if $I=\{\sigma (1),\sigma (2),\dots ,\sigma (k)\}$},\\
0&\text{otherwise}
\end{array}\right.&
&\text{for \;} I\in \cI_{n,k} \text{\; with\; } 1\leq k\leq n-1.
\end{align*}
\end{lem}

Assume that $n>1$. 
Let $\cA$ be an integral domain of characteristic $0$. 
Assume $\{(\lambda_1-\lambda_n + n-3)!\}^{-1}\in \cA$ if $n\geq 3$. 
We regard $\rGL_{n-1}(\cA )$ 
as a subgroup of $\rGL_n(\cA )$ 
via the embedding 
\begin{align}
\label{eq:embed_Gn-1_Gn}
&\iota_n\colon \rGL_{n-1}(\cA )\to \rGL_n(\cA );
g\mapsto \begin{pmatrix}
g&O_{n-1,1}\\ 
O_{1,n-1}&1
\end{pmatrix}.
\end{align}
We set $\widehat{M}:=(m_{i,j})_{1\leq i\leq j\leq n-1}$ 
for $M=(m_{i,j})_{1\leq i\leq j\leq n}\in \rG (\lambda)$. 
By Proposition~\ref{prop:xiM_integral}, we can take a decomposition 
\begin{align}
\label{eq:decomp_n_n-1}
&V_{\lambda}(\cA )=\bigoplus_{\mu \preceq \lambda }V_{\lambda,\mu}(\cA ),&
&V_{\lambda,\mu}(\cA ):=\bigoplus_{M\in \rG (\lambda ;\mu )}\cA\, \xi_M ,  
\end{align}
where $\rG (\lambda ;\mu )=\{M\in \rG (\lambda)\mid 
\widehat{M}\in \rG (\mu )\}$. 
Take $\mu =(\mu_1,\mu_2,\dots,\mu_{n-1})\in \Lambda_{n-1}$ such that $\mu \preceq \lambda $. 
For $M\in \rG (\mu ) $, let
$M[\lambda ]$ be the element of $\rG (\lambda ;\mu )$ characterised by 
$\widehat{M[\lambda ]}=M$, that is, 
\begin{align}
\label{eq:def_Mlambda}
&M[\lambda ]:=\begin{pmatrix}\lambda \\ M\end{pmatrix}
\in \rG (\lambda ;\mu ). 
\end{align}
By Proposition~\ref{prop:xiM_integral}, we define $\cA$-homomorphisms 
$\rI_\mu^\lambda \colon V_\mu (\cA )\to 
V_\lambda (\cA )$ and 
$\rR_\mu^\lambda \colon V_\lambda (\cA )\to V_\mu (\cA )$
respectively by 
\begin{align}
\label{eq:def_inj_restriction}
\rI_\mu^\lambda (\xi_M)&:= \xi_{M[\lambda ]} &  
\text{for each\; } M\in \rG (\mu ),
\end{align}
and 
\begin{align}
\label{eq:def_proj_restriction}
\rR^{\lambda}_\mu (\xi_M)&:=\left\{\begin{array}{ll}
\xi_{\widehat{M}}&\text{if}\ \widehat{M}\in \rG (\mu ),\\
0&\text{otherwise}
\end{array}\right. & \text{for each\;} M\in \rG (\lambda).
\end{align}
In Section \ref{subsec:Amod_str}, we prove the following lemma, 
which implies that 
(\ref{eq:decomp_n_n-1}) is a decomposition of 
$V_{\lambda}(\cA )$ into $\rGL_{n-1}(\cA )$-submodules 
and $V_{\lambda,\mu}(\cA )\simeq V_{\mu}(\cA )$. 

\begin{lem}
\label{lem:restriction_lambda_mu}
Retain the notation. 
Then $\rI_\mu^\lambda $ and $\rR^\lambda_\mu$ are 
$\rGL_{n-1}(\cA )$-equivariant homomorphisms satisfying 
$\rR^\lambda_\mu \circ \rI_\mu^\lambda =\id_{V_{\mu}(\cA )}$. 
\end{lem}

When $\cA =\mC$, Lemma~\ref{lem:restriction_lambda_mu} implies that 
$\xi_{H(\mu )[\lambda ]}$ is 
a highest weight vector in the $\rGL_{n-1} (\mC)$-module 
$V_{\lambda,\mu}(\mC )$. 
In Section \ref{subsec:prelim_fdrep}, we prove the following lemma, which asserts that 
$\xi_{H(\mu )[\lambda ]}$ coincides with one of the generators 
(\ref{eq:Vlambda_generator}).

\begin{lem}
\label{lem:xi_Hmulambda_explicit}
Retain the notation. Then we have 
$\xi_{H(\mu )[\lambda ]}=f_{h (\lambda,\mu)}(z)$
where $h (\lambda,\mu)
=(h(\lambda,\mu)_{I})_{I\in \cI_n}\in \cL (\lambda )$ is 
determined by 
\begin{align}
\label{eq:def_hlambdamu_gen}
&h(\lambda,\mu)_I:=\left\{\begin{array}{ll}
\lambda_k-\mu_k&
\text{if $I=\{1,2,\dots ,k-1,n\}$},\\
\mu_k-\lambda_{k+1}&
\text{if $I=\{1,2,\dots ,k\}$},\\
0&\text{otherwise}
\end{array}\right.&
&\text{for\;} I\in \cI_{n,k} \text{\; with\;} 1\leq k\leq n-1.
\end{align}
Here we understand $\{1,2, \ldots, k-1, n\} = \{ n \}$ if $k=1$. 
\end{lem}

When $\cA =\mC$, Lemmas \ref{lem:restriction_lambda_mu} and \ref{lem:xi_Hmulambda_explicit} imply that 
$\rI_\mu^\lambda $ and $\rR_\mu^\lambda $ can be characterised by the following equalities 
in terms of the generators (\ref{eq:Vlambda_generator}): 
\begin{align*}
&\rI_\mu^\lambda (f_{h (\mu)}(z))= f_{h (\lambda,\mu)}(z),&
&\rR^{\lambda}_\mu (f_{h (\lambda,\mu)}(z))=f_{h (\mu)}(z).
\end{align*}

\subsubsection{Complex conjugate representations and pairings}
\label{subsec:Pairings}

Let $\lambda =(\lambda_1,\lambda_2,\dots,\lambda_n)\in \Lambda_n$. 
For $M=(m_{i,j})_{1\leq i\leq j\leq n}\in \rG (\lambda)$, we define $\rrq (M)\in \mZ$ by 
\begin{align} \label{eq:def_qM}
&\rrq (M):=\sum_{j=1}^{n-1}(n-j)\gamma^M_j
=\sum_{1\leq i\leq j\leq n-1}  m_{i,j}.
\end{align}

\begin{lem}
\label{lem:Mdual_property}
Retain the notation. 
For $M\in \rG (\lambda )$, 
we have $\rr (M^\vee )=\rr (M)$, 
$\gamma^{M^\vee}=-\gamma^{M}$ and 
$\rrq (M^\vee )=-\rrq (M)$. 
\end{lem}
\begin{proof}
The identities about $\gamma^{M}$ and $\rrq (M)$ 
follow immediately from definition. 
Let $M=(m_{i,j})_{1\leq i\leq j\leq n}\in \rG (\lambda )$. 
By the definition of $\rr(M)$ (\ref{eq:def_rM}) and $M^\vee =(m_{i,j}^\vee)_{1\leq i\leq j\leq n}$ 
with $m_{i,j}^\vee =-m_{j+1-i,j}$, we have 
\begin{align*}
\rr (M^\vee )
&=\prod_{1\leq i\leq j<k\leq n}
\frac{(-m_{k+1-i,k}+m_{k-j,k-1}-i+j)!
(-m_{k-i,k-1}+m_{k-j,k}-i+j)!}
{(-m_{k-i,k-1}+m_{k-j,k-1}-i+j)!
(-m_{k+1-i,k}+m_{k-j,k}-i+j)!}=\rr (M).
\end{align*}
Here the second equality follows from the replacement 
$(i,j,k)\to (k-j,k-i,k)$. 
\end{proof}

For the representation $(\tau_\lambda ,V_\lambda )$ of $\rGL_n(\mC )$, 
we define the complex conjugate representation 
$(\tau_\lambda^\rconj ,V_\lambda^\rconj )$ of $\tau_\lambda$ by 
\begin{align} \label{eq:comp_rep}
&V_\lambda^\rconj :=V_\lambda ,&
&\tau_\lambda^\rconj (g):=\tau_\lambda (\overline{g})\qquad 
(g\in \rGL_n(\mC )).&
\end{align}
We define a $\mC$-linear map 
$\rI_\lambda^{\rconj}\colon V_{\lambda}\to 
V_{\lambda^{\!\vee}}^{\rconj}$ by 
\begin{align}\label{eq:conjI}
    \rI^\mathrm{conj}_{\lambda} \colon V_{\lambda} \longrightarrow V^\mathrm{conj}_{\lambda^\vee}\, ; \xi_M \mapsto (-1)^{\rrq(M)}\xi_{M^\vee}. 
\end{align}

\begin{lem}
\label{lem:hom_dual_conj}Retain the notation. 
Then $\rI_\lambda^{\rconj}$ is a $\rU (n)$-equivariant isomorphism. 
\end{lem}
\begin{proof}
For $1\leq i,j\leq n$, we have 
$\tau_{\lambda }(E_{i,j}^{\gu (n)})
=\tau_{\lambda }(E_{i,j})$ and $\tau_{\lambda}^{\rconj}(E_{i,j}^{\gu (n)})
=-\tau_{\lambda}(E_{j,i})$. 
Therefore, the assertion follows from 
(\ref{eq:GT_act_wt}), (\ref{eq:GT_act+}) and (\ref{eq:GT_act-}), 
since we have 
$\rI_\lambda^{\rconj}(\zeta_M)=(-1)^{\rrq (M)}\zeta_{M^\vee}$ 
$(M\in \rG (\lambda ))$ by Lemma~\ref{lem:Mdual_property}. 
\end{proof}

\begin{rem}
\label{rem:cpx_conj_rep}
In \cite[Section 2.6]{im}, we give another realisation 
$(\overline{\tau_\lambda} ,\overline{V_\lambda} )$ of 
the complex conjugate representation of $\tau_\lambda$. 
By the formulas (\ref{eq:GT_act_wt}), (\ref{eq:GT_act+}) and 
(\ref{eq:GT_act-}), we note that 
$V_\lambda^\rconj \simeq \overline{V_\lambda}$ as 
$\rGL_n(\mC)$-modules via the correspondence 
\begin{align}
&\sum_{M\in \rG (\lambda )}c_M\xi_M\leftrightarrow 
\sum_{M\in \rG (\lambda )}c_M\,\overline{\xi_M}&
&\left(\text{or equivalently, }\ 
\sum_{M\in \rG (\lambda )}c_M\zeta_M\leftrightarrow 
\sum_{M\in \rG (\lambda )}c_M\,\overline{\zeta_M}\quad 
\right)&
&\text{with $c_M\in \mC$}.
\end{align}
\end{rem}

Let $\cA$ be an integral domain of characteristic $0$. 
For $l\in \mZ$,    
the symbol $\lambda +l$ denotes $(\lambda_1+l,\lambda_2+l,\dots,\lambda_n+l)$, 
and $\lambda -l$ denotes $\lambda +(-l)$. 
We define an $\cA$-isomorphism $\rI^{\det}_{\lambda ,l}\colon V_\lambda (\cA )
\to V_{\lambda -l} (\cA )$ by $\rI^{\det}_{\lambda ,l}(\rv)=(\det z)^{-l}\rv$ for  $\rv\in V_\lambda (\cA )$. 
By definition, we have 
\begin{align}
\label{eq:detl_shift}
\rI^{\det}_{\lambda ,l}(\tau_{\lambda }(g)\rv) &=
(\det g)^l\tau_{\lambda -l}(g)\rI^{\det}_{\lambda ,l}(\rv)&
\text{for\; } g\in \rGL_n(\cA) \text{\; and\;} \rv\in V_{\lambda}(\cA ).
\end{align}
For $M=(m_{i,j})_{1\leq i\leq j\leq n}\in \rG (\lambda)$, 
the symbol $M+l$ denotes the integral triangular array $(m_{i,j}+l)_{1\leq i\leq j\leq n}$ 
in $\rG (\lambda +l)$, 
and $M-l$ denotes $M+(-l)$. 
In Section \ref{subsec:Amod_str}, we prove the following lemma. 

\begin{lem}
\label{lem:detl_shift}
Retain the notation. Assume $\{(\lambda_1-\lambda_n + n-3)!\}^{-1}\in \cA$ if $n\geq 3$. 
For $M\in \rG (\lambda )$, we have 
$\rI^{\det}_{\lambda ,l}(\xi_M)=\xi_{M-l}$. 
Moreover, if $\cA =\mC $, we have $\rI^{\det}_{\lambda ,l}(\zeta_M)=\zeta_{M-l}$ for $M\in \rG (\lambda )$, and 
$\rI^{\det}_{\lambda ,l}$ preserves the inner products{\rm ;} that is, $(\rI^{\det}_{\lambda,l}(\rv),\rI^{\det}_{\lambda,l}(\rv'))_{\lambda-l}=(\rv,\rv')_{\lambda}$ holds for every $\rv,\rv'\in V_{\lambda}$. 
\end{lem}

Assume $\{(\lambda_1-\lambda_n + n-2)!\}^{-1}\in \cA$ if $n\geq 2$. We set 
$\lambda^{\!\vee}=(-\lambda_n,-\lambda_{n-1},\dots,-\lambda_1)\in \Lambda_n$ as in Section \ref{subsec:settings}. 
Because of Proposition~\ref{prop:xiM_integral} and $\rr (M)\in \cA$ for $M\in \rG (\lambda )$, 
we can define an $\cA$-bilinear pairing 
$\langle \cdot, \cdot \rangle_{\lambda}
\colon V_{\lambda}(\cA )\otimes_\cA V_{\lambda^{\!\vee}}(\cA )\to \cA$ by 
\begin{align} \label{eq:pairing_naive}
&\langle \xi_M,\xi_N\rangle_{\lambda} :=\left\{\begin{array}{ll}
(-1)^{\rrq (M)}\rr (M)&\text{if $M=N^\vee $},\\
0&\text{otherwise}
\end{array}\right.&
&\text{for\;} M\in \rG (\lambda) \text{\; and\;} N\in \rG (\lambda^{\!\vee}).
\end{align}
In Section \ref{subsec:Amod_str}, we prove the following lemma. 
\begin{lem}
\label{lem:GL_inv_pairing}
Retain the notation. Then we have 
\begin{equation}
\label{eq:GL_inv_pairing}
\langle \cdot, \cdot \rangle_{\lambda}\in 
\Hom_{\rGL_n(\cA)}(V_{\lambda}(\cA )\otimes_\cA V_{\lambda^{\!\vee}}(\cA ),
\cA_{\rtriv}), 
\end{equation}
where $\cA_{\rtriv}=\cA $ is the trivial $\rGL_n(\cA)$-module. Moreover, we have 
\begin{align}
\label{eq:sym_pairing}
&\langle \rv_1, \rv_2\rangle_{\lambda}
=\langle \rv_2, \rv_1\rangle_{\lambda^\vee}&
&\text{ for }\rv_1\in V_{\lambda}(\cA ) \text{ and } \rv_2\in V_{\lambda^{\!\vee}}(\cA ).
\end{align}
\end{lem}

When $\cA =\mC$, since $\Hom_{\rGL_n(\mC )}(V_{\lambda}\otimes_\mC V_{\lambda^{\!\vee}},
\mC_{\rtriv})$ is one dimensional, 
Lemmas \ref{lem:extremal_vec_explicit} and \ref{lem:GL_inv_pairing} imply that 
$\rGL_{n}(\mC )$-invariant pairing  $\langle \cdot , \cdot \rangle_{\lambda}$ can be  
characterised by the following equality in terms of the generators introduced in (\ref{eq:Vlambda_generator}): 
\begin{align*}
&\langle f_{h (\lambda )}(z), f_{h (-\lambda )}(z)\rangle_{\lambda}
=(-1)^{\sum_{i=1}^n(n-i)\lambda_i}.
\end{align*}

Assume $n>1$, and let $\mu \in \Lambda_{n-1}$ 
such that $\mu -l\preceq \lambda^{\!\vee}$. 
We define an $\cA$-bilinear pairing 
$\langle \cdot, \cdot \rangle_{\lambda,\mu }^{(l)}
\colon V_{\lambda}(\cA )\otimes_\cA V_{\mu}(\cA )\to \cA$ by 
\begin{align}\label{eq:pairingLambdaMu}
&\langle \rv_1, \rv_2 \rangle_{\lambda,\mu }^{(l)}
:=\langle \rv_1,  \rI^{\lambda^\vee}_{\mu-l} (\rI^{\det}_{\mu ,l}(\rv_2))\rangle_{\lambda }&
&\text{for }\rv_1\in V_{\lambda}(\cA ) \text{ and } \rv_2\in V_{\mu}(\cA ). 
\end{align}
Then, by (\ref{eq:detl_shift}) and Lemma~\ref{lem:GL_inv_pairing}, we have 
\[
\langle \cdot, \cdot \rangle_{\lambda,\mu }^{(l)}\in 
\Hom_{\rGL_{n-1}(\cA)}(V_{\lambda}(\cA )\otimes_\cA V_{\mu}(\cA ),
\cA_{{\det}^l}), 
\]
where $\cA_{{\det}^l}:=\cA$ is a $\rGL_{n-1}(\cA )$-module on which   
$\rGL_{n-1}(\cA )$ acts by ${\det}^l$. 
Moreover, by the definition \eqref{eq:pairingLambdaMu} and Lemma~\ref{lem:detl_shift}, we have 
\begin{align}\label{eq:PairExplicit}
\langle \xi_M,\xi_N\rangle_{\lambda,\mu }^{(l)}
& =
\begin{cases}
    (-1)^{\rrq (M)}\rr (M)& \text{if $\widehat{M} =N^\vee +l$},\\
0&\text{otherwise}
   \end{cases}
& \text{for } M\in \rG (\lambda ) \text{ and } N\in \rG (\mu).
\end{align}

\subsubsection{The Cartan component}
\label{subsec:Cartan_comp}

Take $\lambda =(\lambda_1,\lambda_2,\dots,\lambda_n)$ and
$\lambda'=(\lambda_1',\lambda_2',\dots,\lambda_n')\in \Lambda_n$. 
In this subsection, we specify the Cartan component $V_{\lambda+\lambda'}$ 
of the tensor product $V_{\lambda}\otimes_{\mC} V_{\lambda'}$. 
We define a $\rU (n)$-invariant hermitian inner product on 
$V_{\lambda}\otimes_\mC V_{\lambda'}$ by 
\begin{align*}
(\rv_1\otimes \rv_1',\ \rv_2\otimes \rv_2')_{(\lambda,\lambda')}  
&:=(\rv_1,\rv_2 )_{\lambda}\, (\rv_1',\rv_2')_{\lambda'} &
\text{for }\rv_1,\rv_2\in V_{\lambda } \text{ and } \rv_1',\rv_2'\in V_{\lambda'}. 
\end{align*}
We note that $\xi_{H(\lambda)}\otimes \xi_{H(\lambda')}
=f_{h (\lambda )}(z)\otimes f_{h (\lambda')}(z)$ is 
a highest weight vector in $V_\lambda \otimes_{\mC }V_{\lambda'}$ 
of weight $\lambda +\lambda'$, 
which is unique up to scalar multiples. Hence, there is a unique $\rGL_n(\mC )$-equivariant injection
$\rI^{\lambda, \lambda'}_{\lambda +\lambda'}
\colon V_{\lambda +\lambda'}\to 
V_\lambda \otimes_{\mC }V_{\lambda'}$ such that 
\begin{align}
\label{eq:def_inj}
\rI^{\lambda, \lambda'}_{\lambda+\lambda'}
(\xi_{H(\lambda +\lambda')})
&=\xi_{H(\lambda)}\otimes \xi_{H(\lambda')}&
(\text{or equivalently, }\quad 
\rI^{\lambda, \lambda'}_{\lambda+\lambda'}(f_{h (\lambda +\lambda')}(z))
&=f_{h (\lambda )}(z)\otimes f_{h (\lambda')}(z)\,).
\end{align}
Moreover, we have 
\begin{align}
\label{eq:inj_generates}
&\Hom_{\rGL_n(\mC )}(V_{\lambda +\lambda'}, 
V_\lambda \otimes_{\mC }V_{\lambda'})
=\mC \,\rI^{\lambda, \lambda'}_{\lambda +\lambda'}.
\end{align}
Since $\bigl( \rI^{\lambda, \lambda'}_{\lambda+\lambda'}(\xi_{H(\lambda +\lambda')}),
\rI^{\lambda, \lambda'}_{\lambda+\lambda'}(\xi_{H(\lambda +\lambda')})\bigr)_{(\lambda,\lambda')}
=(\xi_{H(\lambda +\lambda')},\xi_{H(\lambda +\lambda')})_{\lambda+\lambda'}=1$, 
we note that $\rI^{\lambda, \lambda'}_{\lambda+\lambda'}$ preserves the inner products. 
We define constants $\rc^{M,M'}_{M''}$ ($M\in \rG (\lambda )$, $M'\in \rG (\lambda')$, 
$M''\in \rG (\lambda +\lambda')$) by the following equalities:
\begin{align}
\label{eq:CoefInj}
\rI^{\lambda, \lambda'}_{\lambda+\lambda'}(\xi_{M''})
&=\sum_{M\in \rG (\lambda)}\,\sum_{M'\in \rG (\lambda')}
\rc^{M,M'}_{M''}\xi_{M}\otimes \xi_{M'}
&\text{for\;} M''\in \rG (\lambda +\lambda'). 
\end{align}

Let $\cA$ be an integral domain of characteristic $0$. 
We define a $\rGL_n(\cA )$-equivariant surjection
$\rR^{\lambda, \lambda'}_{\lambda +\lambda'}\colon 
V_{\lambda}(\cA )\otimes_\cA V_{\lambda'}(\cA )\to 
V_{\lambda +\lambda'}(\cA )$ by 
\begin{align}
\label{eq:def_proj}
\rR^{\lambda, \lambda'}_{\lambda +\lambda'}(\rv\otimes \rv')&=\rv\rv'&
 \text{for }\rv\in V_{\lambda}(\cA ) \text{ and }\rv'\in V_{\lambda'}(\cA ),
\end{align}
where $\rv\rv'$ is the product of $\rv$ and $\rv'$ in $\cA [z,(\det z)^{-1} ]$. 
In Section \ref{subsec:proj_inj_tensor}, we prove the following proposition. 

\begin{prop}
\label{prop:injExp}
Retain the notation. Assume 
$\{(\lambda_1+\lambda_1'-\lambda_n-\lambda_n' + n-2)!\}^{-1}\in \cA$ if $n\geq 2$. 
\begin{enumerate}
\item \label{num:injExp_coeff} 
For $M\in \rG (\lambda )$, $M'\in \rG (\lambda')$ 
and $M''\in \rG (\lambda +\lambda')$, 
we have $\rc^{M, M'}_{M''}\in \cA \cap \mQ$ and 
\begin{align}
\label{eq:inj_coeff_explicit}
&\rc^{M, M'}_{M''} 
=\left\{\begin{array}{ll}
\displaystyle 
\frac{\rr (M+M')}{\rr (M)\rr (M')}&
\text{if $M''=M+M'$},\\[1mm]
0&\text{if $\gamma^{M''}\neq \gamma^{M}+\gamma^{M'}$}.
\end{array}\right.
\end{align}
\item \label{num:injExp_hom}
Using the statement \ref{num:injExp_coeff} 
and Proposition~$\ref{prop:xiM_integral}$, define an $\cA$-homomorphism $\rI^{\lambda, \lambda'}_{\lambda +\lambda'}
\colon V_{\lambda +\lambda'}(\cA )\to 
V_\lambda (\cA ) \otimes_{\cA}V_{\lambda'}(\cA )$ by $(\ref{eq:CoefInj})$. Then 
$\rI^{\lambda, \lambda'}_{\lambda +\lambda'}$ is a 
$\rGL_n(\cA )$-equivariant injection satisfying 
$\rR^{\lambda, \lambda'}_{\lambda +\lambda'}\circ 
\rI^{\lambda, \lambda'}_{\lambda +\lambda'}
=\id_{V_{\lambda +\lambda'}(\cA )}$. 
\end{enumerate}
\end{prop}

The following lemma will be used in the proof of Proposition~\ref{prop:twisting} in 
Section \ref{subsec:U(n)_inv}. 

\begin{lem}
\label{lem:sub1_twisting}Retain the notation. Let $l\in \mZ$. 
For $M\in \rG (\lambda )$, $M'\in \rG (\lambda')$  and 
$M''\in \rG (\lambda +\lambda')$, we have 
\begin{align}
\label{eq:sub1_twisting}
\rc^{M+l, M'}_{M''+l}&=\rc^{M, M'}_{M''}& \text{\; and\;} &
&\rc^{M, M'+l}_{M''+l}&=\rc^{M, M'}_{M''}. 
\end{align}
\end{lem}
\begin{proof}
We set $\rI_l:=(\rI^{\det}_{\lambda +l,l}\otimes \id_{V_{\lambda'}})\circ 
\rI^{\lambda +l, \lambda'}_{\lambda+\lambda'+l}\circ \rI^{\det}_{\lambda +\lambda',-l}
\in \Hom_{\rGL_n(\mC )}(V_{\lambda +\lambda'}, V_\lambda \otimes_{\mC }V_{\lambda'})$. Then we have 
\begin{align}
\label{eq:pf_sub1_twisting}
&\rI_l(\xi_{M''})
=\sum_{M\in \rG (\lambda)}\,\sum_{M'\in \rG (\lambda')}
\rc^{M+l,M'}_{M''+l}\xi_{M}\otimes \xi_{M'}&
&\text{for \;} M''\in \rG (\lambda +\lambda'). 
\end{align}
By (\ref{eq:def_inj}), (\ref{eq:inj_generates}) and $\rI_l(\xi_{H(\lambda +\lambda')})
=\xi_{H(\lambda)}\otimes \xi_{H(\lambda')}$, we have $\rI_l=\rI^{\lambda , \lambda'}_{\lambda+\lambda'}$. 
Therefore, comparing (\ref{eq:pf_sub1_twisting}) with 
(\ref{eq:CoefInj}), we obtain the first equality of (\ref{eq:sub1_twisting}). 
Similarly, we also have the second equality of (\ref{eq:sub1_twisting}). 
\end{proof}

\begin{rem}
Retain the notation. We define the Clebsch--Gordan coefficients $\rC^{M,M'}_{M''}$ ($M\in \rG (\lambda )$, $M'\in \rG (\lambda')$, 
$M''\in \rG (\lambda +\lambda')$) for the Cartan component $V_{\lambda+\lambda'}$ of $V_{\lambda}\otimes_{\mC} V_{\lambda'}$ by 
\begin{align*}
&\rI^{\lambda, \lambda'}_{\lambda+\lambda'}(\zeta_{M''})
=\sum_{M\in \rG (\lambda)}\,\sum_{M'\in \rG (\lambda')}
\rC^{M,M'}_{M''}\zeta_{M}\otimes \zeta_{M'}&
&\text{for \;} M''\in \rG (\lambda +\lambda'). 
\end{align*}
Alex, Kalus, Huckleberry and Delft \cite{akhd11} give  a numerical algorithm for the explicit calculation 
of the Clebsch--Gordan coefficients, and 
it provides a list of values of all the Clebsch--Gordan coefficients for each irreducible components of 
$V_{\lambda}\otimes_{\mC} V_{\lambda'}$ if we assign concrete dominant weights to $\lambda$ and $\lambda'$.     
However, it is still hard to obtain algebraic expressions of all the Clebsch--Gordan coefficients 
with abstract dominant weights $\lambda$ and $\lambda'$, even only for 
the Cartan component of $V_{\lambda}\otimes_{\mC} V_{\lambda'}$. 
Proposition~\ref{prop:injExp} \ref{num:injExp_coeff} implies that 
several particular Clebsch--Gordan coefficients have simple expressions
\begin{align*}
&\rC^{M, M'}_{M''}=
\sqrt{\frac{\rr (M)\rr (M')}{\rr (M'')}}\rc^{M, M'}_{M''} 
=\left\{\begin{array}{ll}
\displaystyle 
\sqrt{\frac{\rr (M+M')}{\rr (M)\rr (M')}}&
\text{if $M''=M+M'$},\\[1mm]
0&\text{if $\gamma^{M''}\neq \gamma^{M}+\gamma^{M'}$}
\end{array}\right.
\end{align*}
for $M\in \rG (\lambda )$, $M'\in \rG (\lambda')$ 
and $M''\in \rG (\lambda +\lambda')$. 
\end{rem}

\subsection{Explicit generator of the $(\mathfrak{g}_{n\mC}, \widetilde{K}_n)$-cohomology and evaluation of archimedean local zeta integral}
\label{subsec:explicit_generators}

In this subsection, we construct a generator $[\pi^{(n)}_\infty]_\varepsilon$ of
$H^{\rb_{n,F}}(\mathfrak{g}_{n\mC},\widetilde{K}_n\, ; \cW(\pi^{(n)}_\infty,\psi_{\varepsilon,\infty})\otimes_{\mC} \widetilde{V}(\blambda^\vee))$, and then state the explicit evaluation formula of the archimedean local zeta integral of the explicit generators $[\pi^{(n)}_\infty]_{\varepsilon}$ and $[\pi^{(n-1)}_\infty]_{-\varepsilon}$ (Theorem~\ref{thm:archzetaformula}, one of the main results of the present article). Recall that one may regard $\widetilde{V}(\blambda^\vee)$ as the exterior tensor representation $\bigboxtimes_{v\in \Sigma_{F,\infty}} (V_{\lambda_v^\vee}\otimes V_{\lambda_v-\mathsf{w}}^\mathrm{conj})$ (see Section~\ref{subsec:alg_rep} for details). The $(\mathfrak{g}_{n\mC},\widetilde{K}_n)$-cohomology is thus described as 
\begin{align*}
 H^{\rb_{n,F}}(\mathfrak{g}_{n\mC},\widetilde{K}_n\,; \cW(\pi^{(n)}_\infty,\psi_{\varepsilon,\infty})\otimes_{\mC} \widetilde{V}(\blambda^\vee)) & \cong \bigotimes_{v\in \Sigma_{F,\infty}} H^{\rb_n}(\mathfrak{gl}_{n\mC}, \mC^\times {\rm U}(n)\,; \cW(\pi^{(n)}_v, \psi_{\varepsilon,v})\otimes_{\mC} V_{\lambda_v^\vee}\otimes_{\mC} V^\mathrm{conj}_{\lambda_v-\mathsf{w}})
\end{align*}
for $\rb_n:=\frac{1}{2}n(n-1)$, due to the K\"unneth formula for relative Lie algebra cohomology (see \cite[Theorem~I.1.3]{bw80}). Therefore it suffices to construct an element $[\pi^{(n)}_v]_{\varepsilon}$ for each  $v\in \Sigma_{F,\infty}$ and put $[\pi^{(n)}_\infty]_\varepsilon :=\bigotimes_{v\in \Sigma_{F,\infty}}[\pi^{(n)}_v]_\varepsilon$. 

Take an archimedean place $v\in \Sigma_{F,\infty}$. 
Since the $(\mathfrak{gl}_{n\mC},\mC^\times {\rm U}(n))$-cohomology group under consideration is described as 
\begin{align}  
 H^{\rb_n}(\mathfrak{gl}_{n\mC}, &\mC^\times {\rm U}(n)\,; \cW(\pi^{(n)}_v,\psi_{\varepsilon,v})\otimes_{\mC} V_{\lambda_v^\vee}\otimes_{\mC}V^{\mathrm{conj}}_{\lambda_v-\mathsf{w}})   \label{eq:U(n)inv} \\
&\cong \left( \cW(\pi^{(n)}_v, \psi_{\varepsilon,v})\otimes_{\mC} {\bigwedge}^{\rb_n} \left(\mathfrak{gl}_{n\mC}/\mathrm{Lie}(\mC^\times {\rm U}(n))_{\mC}\right)^\vee \otimes_{\mC} V_{\lambda_v^\vee}\otimes_{\mC} V^{\mathrm{conj}}_{\lambda_v-\mathsf{w}}\right)^{{\rm U}(n)} \nonumber \\
&\cong \left( \cW(\pi^{(n)}_v, \psi_{\varepsilon,v})\otimes_{\mC} {\bigwedge}^{\rb_n} \gp^{0\vee}_{n\mC} \otimes_{\mC} V_{\lambda_v^\vee}\otimes_{\mC} V^\mathrm{conj}_{\lambda_v-\mathsf{w}}\right)^{{\rm U}(n)} \nonumber
\end{align}
due to \cite[Section~I.5]{bw80} (see also the footnote (b) of \cite[p.~209]{ks13}), it suffices to construct a ${\rm U}(n)$-invariant element of $\cW(\pi^{(n)}_v, \psi_{\varepsilon,v})\otimes_{\mC} {\bigwedge}^{\rb_n} \gp^{0\vee}_{n\mC} \otimes_{\mC} V_{\lambda_v^\vee}\otimes_{\mC} V^{\mathrm{conj}}_{\lambda_v-\mathsf{w}}$; the second isomorphy follows from the decompositions (\ref{eq:CartanDecomp}), (\ref{eq:decomp_p0}) and the fact that $\mC^\times {\rm U}(n)$ is decomposed as the direct product $(Z_n\cap A_n)\times \rU (n)$, which induces $\mathrm{Lie}(\mC^\times {\rm U}(n))=\mathfrak{z}_{\mathfrak{a}_n}\oplus \mathfrak{u}(n)$. 
According to the dictionary between the infinitesimal character of $\pi_\infty^{(n)}$ and the (dual of) highest weight $\blambda$ of $\widetilde{V}(\blambda^\vee)$ proposed in \cite[p.~113]{clo90}, we observe that $\pi^{(n)}_v$ is isomorphic to the subspace of ${\rm U}(n)$-finite vectors of the principal series representation 
$\pi_{{\rm B}_n,d_v,(\sw /2,\sw /2,\dots ,\sw /2)}$ defined in Section~\ref{subsec:C_def_ps} if we set $d_v:=2\lambda_v+2\rho_n-\mathsf{w}$ for $\rho_n=(\frac{n+1}{2}-i)_{1\leq i\leq n}$. 
After introducing archimedean local Whittaker functions $\mW_{d_v,(\sw /2,\sw /2,\dots ,\sw /2)}(\rv )$ in Section~\ref{subsec:C_def_ps}, we construct $[\pi^{(n)}_v]_\varepsilon =[\pi_{{\rm B}_n,d_v,(\sw /2,\sw /2,\dots ,\sw /2)}]_\varepsilon $ in Section~\ref{subsec:U(n)_inv}. 
In Section~\ref{subsec:explicit_arch_zeta}, the explicit formula of $\widetilde{\cI}^{(m)}_v([\pi^{(n)}_v]_\varepsilon,[\pi^{(n-1)}_v]_{-\varepsilon})$ is introduced for $m\in \mZ$ satisfying the interlacing conditions $\lambda_v^\vee \succeq \mu_v +m$ and $\lambda_v -\sw  \succeq \mu_v^\vee +\sw'+m$ (equivalently, $\lambda_{\bar{v}}^\vee \succeq \mu_{\bar{v}} +m$). We here only explain how to deduce the formula from the preceding work of Ishii and the second-named author \cite{im}, and postpone the hard calculation of the constant $c^{(m)}_{\lambda_v,\mathsf{w},\mu_v,\mathsf{w}'}$ to Section~\ref{sec:calcoeff}.

We always concentrate on an archimedean place $v \in \Sigma_{F, \infty}$ throughout this subsection, 
and the calculation made here is purely local. Thus we drop $v$ from subscript in many cases. 
In particular, for $\varepsilon \in \{\pm \}$, the standard additive character $\psi_{\varepsilon,v}$ of $F_v=\mC$ will be denoted simply by $\psi_{\varepsilon }$; that is, $\psi_{\varepsilon}(z)$ is defined as $\psi_{\varepsilon }(z):=\exp (\varepsilon 2\pi \sqrt{-1} (z +\overline{z}))$ for $z\in\mC$.

\subsubsection{Principal series representations of $\rGL_n(\mC )$ and explicit archimedean Whittaker functions}
\label{subsec:C_def_ps}

We here introduce principal series representations of $\rGL_n(\mC )$, and our normalisation of Whittaker functions for them. 
The normalisation influences the definition of the Whittaker period, and it is important in the present article. 

Let $d =(d_1,d_2,\dots, d_n)\in \mZ^n$ and 
$\nu =(\nu_1,\nu_2,\dots ,\nu_n)\in \mC^n$. 
We define a character $\chi_{d,\nu}$ of $\rT_n(\mC )$ by 
\begin{align}
\label{eq:def_chi_dnu}
\chi_{d,\nu}(a)
&:=\prod_{i=1}^n
\left(\frac{a_i}{|a_i|}\right)^{\!d_i}|a_i|^{2\nu_i}&
&\text{for \;} a=\diag (a_1,a_2,\dots ,a_n)\in \rT_n(\mC ).
\end{align}
Put $\rho_n=(\rho_{n,1},\rho_{n,2},\dots ,\rho_{n,n})\in \mQ^n$ with 
$\rho_{n,i}:=\tfrac{n+1}{2}-i$ for $1\leq i\leq n$. 
Let $C^\infty (\rGL_n(\mC))$ be the space of $C^\infty$-functions on $\rGL_n(\mC)$. 
We define a (smooth) principal series representation 
$(\pi_{\rB_n,d,\nu },I^\infty_{\rB_n} (d,\nu ))$ of $\rGL_n(\mC )$ by 
\begin{align}
\label{eq:def_ps_rep}
I^\infty_{\rB_n} (d,\nu ) :=\{f\in C^\infty (\rGL_n(\mC))
\mid f(xag)=\chi_{d, \nu +\rho_n}(a)f(g)\quad 
\text{for \;}x\in \rN_n(\mC ),\ a\in \rT_n(\mC) \text{\; and\; } g\in \rGL_n(\mC )\} 
\end{align}
and $(\pi_{\rB_n,d,\nu }(g)f)(h)=f(hg)$ for $g,h\in \rGL_n(\mC )$ and
$f\in I^\infty_{\rB_n} (d,\nu )$. 
We equip $I^\infty_{\rB_n} (d,\nu ) $ with the usual Fr\'{e}chet topology. 

Let $d^{\rdom} =(d^{\rdom}_1,d^{\rdom}_2,\dots ,d^{\rdom}_n)$ be 
the unique element of $\Lambda_n\cap \{\sigma d \mid \sigma \in \gS_n\}$. 
By (\ref{eq:GT_act_wt}) and 
the Frobenius reciprocity law \cite[Theorem 1.14]{Knapp_002}, 
we know that 
$\tau_{d^{\rdom}}|_{\rU (n)}$ is the minimal $\rU (n)$-type of 
$\pi_{\rB_n,d,\nu}$, and the space 
$\Hom_{\rU (n)}(V_{d^{\rdom}} ,I^\infty_{\rB_n} (d,\nu ))$ 
is one dimensional. 
In view of Lemma~\ref{lem:extremal_vec_explicit}, 
we fix a $\rU (n)$-embedding 
$\rf_{\rB_n,d,\nu}\colon V_{d^{\rdom}} \to I^\infty_{\rB_n} (d,\nu )$ 
so that $\rf_{\rB_n,d,\nu}(\xi_{H(d )})(1_n)=1$, that is, 
for $\rv\in V_{d^{\rdom}}$, the function $\rf_{\rB_n,d,\nu}(\rv)$ on 
$\rGL_n (\mC )$ is determined by 
\begin{align}
\label{eqn:def_minKtype1}
\rf_{\rB_n,d,\nu}(\rv)(xak)&=\chi^{A_n}_{\nu +\rho_n}(a)
(\tau_{d^{\rdom}} (k)\rv,\xi_{H(d )})_{d^{\rdom}}&
&\text{for \;}x\in \rN_n(\mC),\ a\in A_n \text{\; and \;} k\in \rU (n)
\end{align}
via the Iwasawa decomposition $\rGL_n(\mC)=\rN_n(\mC)A_n\rU (n)$, where $\chi^{A_n}_\nu$ is defined as
\begin{align}
\label{eq:def_chi_A}
\chi^{A_n}_{\nu}(a)
:=\prod_{i=1}^na_i^{2\nu_i}
\end{align}
for $a=\diag (a_1,a_2,\dots ,a_n)\in A_n$ (recall that $A_n$ is defined as the subgroup of diagonal matrices with positive real entries in Section~\ref{subsec:notation}).

For $\varepsilon \in \{\pm \}$, 
let $\psi_{\varepsilon ,\rN_n}$ be a character of $\rN_n(\mC)$ defined by 
\begin{align}
\label{eq:def_psiNn}
\psi_{\varepsilon ,\rN_n}(x)&:=\psi_\varepsilon (x_{1,2}+x_{2,3}+\dots  +x_{n-1,n})&
&(x=(x_{i,j})_{1\leq i,j\leq n}\in \rN_n(\mC)).
\end{align}
When $n=1$, we understand that $\psi_{\varepsilon,\rN_1}$ 
is the trivial character of $\rN_1(\mC)=\{1\}$. 
A $\psi_\varepsilon $-form on $I^\infty_{\rB_n} (d,\nu )$ is 
a continuous $\mC$-linear form 
$\cT \colon I^\infty_{\rB_n} (d,\nu ) \to \mC$ satisfying 
\begin{align*}
\cT (\pi_{\rB_n,d,\nu}(x)f)&=\psi_{\varepsilon ,\rN_n} (x)\cT (f)&
&\text{for\; }x\in \rN_n(\mC) \text{\; and\;} f\in I^\infty_{\rB_n} (d,\nu ).
\end{align*}
Kostant \cite{Kostant_001} shows that 
a $\psi_\varepsilon $-form on $I^\infty_{\rB_n} (d,\nu )$ is uniquely determined up to scalar multiples. 
If the inequalities $\rRe (\nu_1)>\rRe (\nu_2)>\dots >\rRe (\nu_n)$ hold, 
we define the Jacquet integral 
$\cJ_{\varepsilon}\colon I^\infty_{\rB_n} (d,\nu )\to \mC$ 
by the convergent integral 
\begin{align} \label{eq:Jacquet_int}
\cJ_{\varepsilon }(f)&:=
\int_{\rN_n(\mC )}f(w_nx)\psi_{-\varepsilon,\rN_n}(x)\,\rd x&
 \text{for }f\in I^\infty_{\rB_n} (d,\nu ),
\end{align}
where $w_n$ is the anti-diagonal matrix of size $n$ with $1$ at 
all anti-diagonal entries, and 
$\rd x =\prod_{1\leq i<j\leq n}\rd_{\mC}x_{i,j}$ is the Haar measure on 
$\rN_n(\mC )$ with 
$x=(x_{i,j})_{1\leq i,j\leq n}\in \rN_n(\mC )$. 
When $n=1$, we understand that $\cJ_{\varepsilon }(f)=f(1)$ 
for $f\in I^\infty_{\rB_n} (d,\nu )$. 
The Jacquet integral $\cJ_{\varepsilon}\colon 
I^\infty_{\rB_n} (d,\nu )\to \mC$ is extended to whole $\nu \in \mC^n$ as 
a $\psi_\varepsilon $-form on $I^\infty_{\rB_n} (d,\nu )$ 
by the holomorphic continuation with the standard sections 
(\textit{cf.} Section \ref{subsec:whittaker}). 
We set 
\begin{align} \label{eq:def_Whittaker_model}
\cW (\pi_{\rB_n,d,\nu},\psi_\varepsilon )
&:=\{\rW_{\varepsilon}(f)\mid f\in I^\infty_{\rB_n} (d,\nu )\}
\end{align}
with $\rW_{\varepsilon}(f)(g):=\cJ_{\varepsilon}(\pi_{\rB_n,d,\nu}(g)f)$ 
for $f\in I^\infty_{\rB_n} (d,\nu )$ and $g\in \rGL_n(\mC)$. 
When $\pi_{\rB_n,d,\nu}$ is irreducible, 
$\cW (\pi_{\rB_n,d,\nu},\psi_\varepsilon )$ is a Whittaker model of 
$\pi_{\rB_n,d,\nu}$, and for any $\sigma \in \gS_n$, we have 
\begin{align}
\label{eq:Sinv_whittaker}
&\cW (\pi_{\rB_n,d,\nu},\psi_\varepsilon )
=\cW (\pi_{\rB_n,\sigma d ,\sigma \nu },\psi_\varepsilon )
\end{align}
because there exists a $\rGL_n(\mC)$-isomorphism 
$I^\infty_{\rB_n} (d,\nu )\simeq 
I^\infty_{\rB_n} (\sigma d,\sigma \nu )$ 
(\textit{cf.} \cite[Corollary 2.8]{Speh_Vogan_001}). 
Hence, if $\pi_{\rB_n,d,\nu}$ is irreducible, for any $\sigma \in \gS_n$, 
there is a constant $c_\sigma$ 
such that $\rW_{\varepsilon} (\rf_{\rB_n,d,\nu}(\rv))=c_\sigma 
\rW_{\varepsilon} (\rf_{\rB_n,\sigma d ,\sigma \nu }(\rv))$ ($\rv \in V_{d^{\rdom}}$). 
Let us introduce a normalisation of the minimal $\rU (n)$-type Whittaker function, 
which is invariant under the action of $\gS_n$. 
For $\rv \in V_{d^{\rdom}}$, we define the normalised Whittaker function 
$\mW^{(\varepsilon)}_{d, \nu} (\rv) $ to be  
\begin{align}
\label{eq:def_mW}
\mW^{(\varepsilon)}_{d, \nu} (\rv)
&:=(-1)^{\sum^{n}_{i=1}(i-1)d^{\rdom}_i}
(\varepsilon \sqrt{-1})^{\sum^{n}_{i=1}(i-1)d_i}
\Gamma_n(\nu;d )
\rW_{\varepsilon} (\rf_{\rB_n,d,\nu}(\rv)),
\end{align}
where 
\begin{align} \label{eq:def_gamma}
\Gamma_n(\nu;d ):=\prod_{1\leq i<j\leq n}
\Gamma \bigl(\nu_i-\nu_j+1+\tfrac{|d_i-d_j|}{2}\bigr).
\end{align}
In Section \ref{sec:arch_whittaker}, we prove the following proposition. 
\begin{prop}
\label{prop:arch_Wh}
Retain the notation. Let $g\in \rGL_n(\mC)$ and $\rv\in V_{d^{\rdom}}$. 
Then $\mW^{(\varepsilon)}_{d, \nu} (\rv)(g)$ is 
an entire function in $\nu$ having the following symmetry$:$
\begin{align}
\label{eq:Sinv_Whittaker}
\mW^{(\varepsilon)}_{d, \nu} (\rv)(g)
&=\mW^{(\varepsilon)}_{\sigma d , \sigma \nu} (\rv)(g) &
\text{for each \;}\sigma \in \gS_n. 
\end{align}
Furthermore, for each $\nu_0\in \mC^n$ such that $\pi_{\rB_n,d,\nu_0}$ is irreducible, 
the meromorphic function $\Gamma_n(\nu;d )$ in $\nu$ 
is holomorphic at $\nu =\nu_0$. 
\end{prop}
For the spherical case $d =(0,0,\dots ,0)$, it is known that 
the symmetry (\ref{eq:Sinv_Whittaker}) follows from the original work 
of Jacquet \cite[Th\'eor\`eme (8.6)]{Jacquet_002}, and 
is found in the literature (see, for example, \cite[Section 2.1]{bhm20} or \cite[Section 8]{gmw}). 
In Section \ref{sec:arch_whittaker}, we prove the symmetry (\ref{eq:Sinv_Whittaker}) 
for general $d$ by another method based on the Godement sections, and 
rewrite the explicit result \cite[Theorem 2.7]{im} for the archimedean local zeta integrals 
in terms of the normalised Whittaker functions $\mW^{(\varepsilon)}_{d, \nu} (\rv) $ (see Proposition~\ref{prop:ArchInt}). 
Here we remark that our method can be applied to prove the similar result for the $\rGL_n(\mR )$-case 
although we only deal with the $\rGL_n (\mC)$-case in the present article. 
We also remark that the same kinds of symmetry at the newform $\rU (n)$-type is claimed 
in \cite[Remark 9.5]{hum} without proof. 

The following lemma will be used in the proof of Proposition~\ref{prop:twisting} in 
Section \ref{subsec:U(n)_inv}. 
\begin{lem}
\label{lem:sub2_twisting}Retain the notation. Let $l\in \mZ$, $s \in \mC$ and 
$\rv \in V_{d^\rdom}$. Set 
$\displaystyle \chi_{l,s}(a):=(a/|a|)^{l}|a|^{2s}$ for $a\in \rT_1(\mC)=\mC^\times$. 
Then we have 
\begin{align*}
\chi_{l,s}(\det g)\mW^{(\varepsilon)}_{d, \nu}(\rv)(g)&=(\varepsilon \sqrt{-1})^{-\rb_nl}
\mW^{(\varepsilon)}_{d+l, \nu +s }(\rI^{\det}_{d^\rdom ,-l}(\rv ))(g) &  \text{for \;} g\in \rGL_n(\mC).
\end{align*}
\end{lem}
\begin{proof}
Because of the uniqueness of the analytic continuation, we may assume $\rRe (\nu_1)>\rRe (\nu_2)>\dots >\rRe (\nu_n)$. 
By (\ref{eqn:def_minKtype1}) and Lemma~\ref{lem:detl_shift}, we have 
$\chi_{l,s}(\det g)\rf_{\rB_n,d,\nu}(\rv)(g)=\rf_{\rB_n,d+l,\nu +s}(\rI^{\det}_{d^\rdom ,-l}(\rv ))(g)$ 
for $g\in \rGL_n(\mC)$. 
Since $\det (w_nx)=(-1)^{\rb_n}$ holds for $x\in \rN_n(\mC)$, we have 
\begin{align*}
\chi_{l,s}(\det g)\rW_{\varepsilon}(\rf_{\rB_n,d,\nu}(\rv))(g)&=(-1)^{-\rb_nl}
\int_{\rN_n(\mC )}\rf_{\rB_n,d+l,\nu +s }(\rI^{\det}_{d^\rdom ,-l}(\rv ))(w_nxg)\psi_{-\varepsilon,\rN_n}(x)\,\rd x\\
&=(-1)^{-\rb_nl}\rW_{\varepsilon}(\rf_{\rB_n,d +l,\nu +s }(\rI^{\det}_{d^\rdom ,-l}(\rv )))(g)
\end{align*}
for $g\in \rGL_n(\mC)$. 
The assertion follows from this equality combined with the definition of $\mW^{(\varepsilon)}_{d,\nu}(\rv)$ (\ref{eq:def_mW}) and the trivial relation $\Gamma_n(\nu +s;d+l)=\Gamma_n(\nu ;d)$. 
\end{proof}

The following lemma implies that our functions $\mW^{(\varepsilon )}_{d, \nu} (\rv)$ 
satisfy the compatibility condition (\ref{eq:Wh_compatibility}) in Section \ref{subsec:Wh_compatibility}. 
\begin{lem}
\label{lem:arch_Wh_compatibility}Retain the notation. Let 
$\bepsilon_n:=\diag ((-1)^{n-1},(-1)^{n-2},\dots ,-1,1)\in \rGL_n(\mC )$. 
Then we have 
\begin{align*}
\mW^{(-\varepsilon )}_{d, \nu}(\rv)(g)&=\mW^{(\varepsilon )}_{d, \nu }(\rv )(\bepsilon_ng) &  \text{for $\rv \in V_{d^\rdom}$ and $g\in \rGL_n(\mC)$.} 
\end{align*}
\end{lem}
\begin{proof}
Because of the uniqueness of the analytic continuation, we may assume $\rRe (\nu_1)>\rRe (\nu_2)>\dots >\rRe (\nu_n)$. 
Let $f\in I^\infty_{\rB_n}(d,\nu)$. 
By \eqref{eq:def_psiNn}, we have $\psi_{\varepsilon,\rN_n}(\bepsilon_nx\bepsilon_n)=\psi_{-\varepsilon,\rN_n}(x)$ 
for $x\in \rN_n(\mC)$. 
By \eqref{eq:def_ps_rep}, we have 
\[
f(w_n\bepsilon_ng)=\chi_{d, \nu +\rho_n}(w_n\bepsilon_nw_n)f(w_ng)
=(-1)^{\sum^{n}_{i=1}(i-1)d_i}f(w_ng)
\]
for $g\in \rGL_n(\mC)$. Therefore, by the substitution $x\to \bepsilon_nx\bepsilon_n$, we have 
\begin{align*}
\rW_{-\varepsilon}(f)(g)&=\int_{\rN_n(\mC )}f(w_nxg)\psi_{\varepsilon,\rN_n}(x)\,\rd x
=\int_{\rN_n(\mC )}f(w_n\bepsilon_nx\bepsilon_ng)\psi_{\varepsilon,\rN_n}(\bepsilon_nx\bepsilon_n)\,\rd x\\
&=(-1)^{\sum^{n}_{i=1}(i-1)d_i}\int_{\rN_n(\mC )}f(w_nx\bepsilon_ng)\psi_{-\varepsilon,\rN_n}(x)\,\rd x
=(-1)^{\sum^{n}_{i=1}(i-1)d_i}\rW_{\varepsilon}(f)(\bepsilon_ng)
\end{align*}
for $g\in \rGL_n(\mC)$. 
The assertion follows from this equality combined with the definition of $\mW^{(\varepsilon)}_{d,\nu}(\rv)$ (\ref{eq:def_mW}). 
\end{proof}

In the following arguments, we mainly consider 
the case in which $\pi_{\rB_n,d,\nu}$ is irreducible and $d \in \Lambda_n$. In this case, we note that 
$d^{\rdom}=d $, and we obtain a ${\rm U}(n)$-equivariant homomorphism
\begin{align} \label{eq:Whittaker_hom}
\mW^{(\varepsilon)}_{d, \nu} \colon V_d\longrightarrow \cW(\pi_{\rB_n,d,\nu},\psi_{\varepsilon}) \, ;  \rv \longmapsto
(-\varepsilon \sqrt{-1})^{\sum^{n}_{i=1}(i-1)d_i}
\Gamma_n(\nu;d )
\rW_{\varepsilon} (\rf_{\rB_n,d,\nu}(\rv)).
\end{align}

\subsubsection{Explicit generator of the $(\mathfrak{gl}_{n\mC},\mC^\times \rU(n))$-cohomology}\label{subsec:U(n)_inv}
For $\lambda =(\lambda_1,\lambda_2,\dots ,\lambda_n)\in \Lambda_n$ and $\sw \in \mZ$, 
we set $d:=2\lambda +2\rho_n-\sw$. 
We often abbreviate $\pi_{\rB_n,d,(\sw /2,\sw /2,\dots ,\sw /2)}$ and $\mW^{(\varepsilon )}_{d,(\sw /2,\sw /2,\dots ,\sw /2)}$ to 
$\pi_{\rB_n,d,\mathsf{w}/2}$ and $\mW^{(\varepsilon )}_{d,\sw /2}$ respectively, if no confusion likely occurs.
We here construct $[\pi_{\rB_n,d,\sw /2}]_\varepsilon $ as a $\rU (n)$-invariant element of 
$\cW(\pi_{\rB_n,d,\sw /2}, \psi_{\varepsilon })\otimes_{\mC} {\bigwedge}^{\rb_n} \gp^{0\vee}_{n\mC} \otimes_{\mC} V_{\lambda^\vee}\otimes_{\mC} V^{\rconj}_{\lambda -\sw}$ with $\rb_n :=\frac{1}{2}n(n-1)$. 
For the construction, we firstly introduce several ${\rm U}(n)$-equivariant homomorphisms.  In Section~\ref{subsec:adjoint}, we see that 
\begin{align*}
 \bigwedge_{2\leq i\leq n}\bigwedge_{1\leq j\leq i-1} E^{\gp_n\vee}_{i,j}=E^{\gp_n\vee}_{2,1} \wedge 
       E^{\gp_n\vee}_{3,1} \wedge 
       E^{\gp_n\vee}_{3,2} \wedge \cdots \wedge 
       E^{\gp_n\vee}_{n,1} \wedge 
       E^{\gp_n\vee}_{n,2} \wedge \cdots \wedge
       E^{\gp_n\vee}_{n,n-1} \quad \in  {\bigwedge}^{\rb_n}\gp^{0\vee}_{n\mC}
\end{align*}
is a highest weight vector of weight $2\rho_n$ (see Lemma~\ref{lem:htwtE}). By using this,  we define a ${\rm U}(n)$-equivariant homomorphism
\begin{align}\label{eq:xiH2rho}
\rI^{\gp_n}_{2\rho_n} \colon V_{2\rho_n} \rightarrow {\bigwedge}^{\rb_n} \gp^{0\vee}_{n\mC} \, ;  \xi_{H(2\rho_n)} \longmapsto \bigwedge_{2\leq i\leq n}\bigwedge_{1\leq j\leq i-1} E^{\gp_n\vee}_{i,j}. 
\end{align}
Combining the maps $\rI^{\lambda,\lambda'}_{\lambda+\lambda'}$(\ref{eq:def_inj}), $\mW^{(\varepsilon)}_{d,\nu}$ (\ref{eq:Whittaker_hom}), $\rI^{\gp_n}_{2\rho_n}$ (\ref{eq:xiH2rho}) and $\rI^{\mathrm{conj}}_\lambda$ (\ref{eq:conjI}), we construct a ${\rm U}(n)$-equivariant map
\begin{align}
\label{eq:construct_U(n)inv}
\begin{array}{lll}
   V_d \otimes_{\mC} V_{d^\vee}  &
   \xrightarrow[\hphantom{---------------}]{\,{\rm id} \otimes \rI^{2\rho_n, 2\lambda^\vee + \mathsf{w}}_{d^\vee} \,} &
                 V_d \otimes_{\mC} V_{2\rho_n}  \otimes_{\mC} V_{2\lambda^\vee + \mathsf{w}}    \\
& \xrightarrow[\hphantom{---------------}]{\, {\rm id} \otimes {\rm id} \otimes \rI^{\lambda^\vee, \lambda^\vee+\mathsf{w}}_{2\lambda^\vee + \mathsf{w}} \, } & 
          V_d 
           \otimes_{\mC} V_{2\rho_n} 
            \otimes_{\mC} V_{\lambda^\vee}     
           \otimes_{\mC} V_{\lambda^\vee + \mathsf{w}}    \\
 &  \xrightarrow[\hphantom{---------------}]{\,  \mW^{(\varepsilon)}_{d, \mathsf{w}/2}  
                       \otimes \rI^{\gp_n}_{2\rho_n}  
                        \otimes {\rm id} \otimes \rI^{\rm conj}_{\lambda^\vee + \mathsf{w}}\, }     &
                   \cW(\pi_{ \mathrm{B}_n, d, \mathsf{w}/2  }, \psi_\varepsilon   ) 
                           \otimes_{\mC} \left(  {\bigwedge}^{\rb_n}  \gp^{0\vee}_{n\mC}   \right)  
            \otimes_{\mC} V_{\lambda^\vee}     
           \otimes_{\mC} V^{\rm conj}_{\lambda - \mathsf{w}}
\end{array}
\end{align}
(the first map is defined because $d^\vee=2\lambda^\vee+2\rho_n+\mathsf{w}$ holds by the definition of $d=2\lambda+2\rho_n-\mathsf{w}$).
Meanwhile existence of the ${\rm U}(n)$-invariant pairing $\langle \cdot,\cdot\rangle_d$ (verified in Lemma~\ref{lem:GL_inv_pairing}) and 
(\ref{eq:pairing_naive}) proposes a ${\rm U}(n)$-invariant element of $V_d\otimes_{\mC} V_{d^\vee}$
\begin{align} \label{eq:id_Vd}
[ \mathrm{id}_{V_d}] := \sum_{M\in {\rG(d)}} \dfrac{(-1)^{{\rm q}(M)}}{\rr(M)}\xi_M \otimes \xi_{M^\vee},
\end{align}
which corresponds to $\mathrm{id}_{V_d}\in \mathrm{Hom}_{{\rm U}(n)}(V_d,V_d)=(\mathrm{Hom}_{\mC}(V_d,V_d))^{{\rm U}(n)}$ via the isomorphism
\begin{align*}
 V_d\otimes_{\mC} V_{d^\vee} \longrightarrow  \mathrm{Hom}_{\mC}(V_d,V_d) \, ; \rv_1 \otimes \rv_2 \longmapsto [ \rv\mapsto  \langle \rv,\rv_2 \rangle_d \rv_1 ] 
\end{align*}
(verified in a way similar to \cite[Lemma~4.6 (2)]{im}). We define $[\pi_{\rB_n,d,\sw /2}]_\varepsilon$ as the image of $[\mathrm{id}_{V_d}]$ under the ${\rm U}(n)$-equivariant map (\ref{eq:construct_U(n)inv}). Using (\ref{eq:CoefInj}) and (\ref{eq:conjI}), we can represent $[\pi_{{\rm B}_n,d,\mathsf{w}/2}]_{\varepsilon}$ with respect to the Gel'fand--Tsetlin basis $\{\xi_M\}_M$ as 
\begin{multline}  
  [\pi_{  \mathrm{B}_n, d, \mathsf{w}/2}]_{\varepsilon}
  =  \sum_{M\in \rG (d)} 
             \sum_{N\in \rG (2\rho_n)} 
         \sum_{P\in \rG (\lambda^\vee)} 
         \sum_{Q \in \rG (\lambda -  \mathsf{w})  }    
      c(M, N, P, Q)
    \mW^{(\varepsilon)}_{d, \mathsf{w}/2}( \xi_M ) 
     \otimes  \rI^{\gp_n}_{2\rho_n}  ( \xi_{N} )  
      \otimes \xi_P \otimes \xi_{Q}, \label{eq:Pic}  \\
\text{where }c(M, N, P, Q)
= \frac{  (-1)^{\rrq (M) +\rrq (Q) } }{\rr (M)}
     \sum_{T\in {\rm G} (2\lambda^\vee + \mathsf{w})   }  
     {\rm c}^{N, T}_{M^\vee} {\rm c}^{P, Q^\vee}_{T}. 
\end{multline}
Also, we set $[\pi_{{\rm B}_n,d,\mathsf{w}/2}]_\varepsilon^\circ  :=(\varepsilon \sqrt{-1})^{\rb_n \sw}[\pi_{{\rm B}_n,d,\mathsf{w}/2}]_\varepsilon $, 
and let us discuss the compatibility of $[\pi_{{\rm B}_n,d,\mathsf{w}/2}]_\varepsilon^\circ $ with twisting.

\begin{prop} \label{prop:twisting}
 For $a,b\in \mZ$, let $\kappa_{a,b}\colon \mC^\times \rightarrow \mC^\times$ be an algebraic character defined as $\kappa_{a,b}(z)=z^a \bar{z}^b$ and consider the map
\begin{align} \label{eq:Wkappa}
 W_{\kappa_{a,b}} \colon \cW(\pi_{{\rm B}_n,d,\mathsf{w}/2},\psi_\varepsilon) \rightarrow \cW(\pi_{{\rm B}_n,d+a-b,(\mathsf{w}+a+b)/2},\psi_\varepsilon) \, ; w \mapsto [g\mapsto w(g)\kappa_{a,b}(\det g)].
\end{align} 
Then $\bigl(W_{\kappa_{a,b}}\otimes \mathrm{id}_{\bigwedge^{\rb_n}\gp^{0\vee}_{n\mC}}\otimes \rI^{\det}_{\lambda^\vee,a}\otimes \rI^{\det}_{\lambda-\mathsf{w},b}\bigr)([\pi_{{\rm B}_n,d,\mathsf{w}/2}]_\varepsilon^\circ )=[\pi_{{\rm B}_n,d+a-b,(\mathsf{w}+a+b)/2}]_\varepsilon^\circ $ holds.
\end{prop}  

\begin{proof}
 Via the bijection $\rG(\lambda)\xrightarrow{\, \sim \,}\rG(\lambda+l)\,; M\mapsto M+l$, the element $[\pi_{{\rm B}_n,d+a-b,(\mathsf{w}+a+b)/2}]_\varepsilon^\circ $ is described as 
\begin{multline*} 
  [\pi_{  \mathrm{B}_n, d+a-b, (\mathsf{w}+a+b)/2}]_{\varepsilon}^\circ 
  =  (\varepsilon \sqrt{-1})^{\rb_n (\mathsf{w}+a+b)} 
 \sum_{M\in \rG (d)} 
             \sum_{N\in \rG (2\rho_n)}
         \sum_{P\in \rG (\lambda^\vee)} 
         \sum_{Q \in \rG (\lambda -  \mathsf{w})  }c(M+a-b, N, P-a, Q-b)    \\
   \hspace*{15em}\times 
    \mW^{(\varepsilon)}_{d+a-b, (\mathsf{w}+a+b)/2}(\rI^{\det}_{d,-a+b}(\xi_{M})) 
     \otimes  \rI^{\gp_n}_{2\rho_n}  ( \xi_{N} )  
      \otimes \rI^{\det}_{\lambda^\vee,a}(\xi_P) \otimes \rI^{\det}_{\lambda-\mathsf{w},b}(\xi_{Q}),   \\
\text{where }c(M+a-b, N, P-a, Q-b)
= \frac{  (-1)^{\rrq (M+a-b) +\rrq (Q-b) } }{\rr (M+a-b)}
     \sum_{T\in {\rm G} (2\lambda^\vee + \mathsf{w})   }  
     {\rm c}^{N, T-a+b}_{M^\vee-a+b} {\rm c}^{P-a, Q^\vee+b}_{T-a+b}
\end{multline*}
by (\ref{eq:Pic}). One readily sees that $(-1)^{\rrq(M+a-b)+\rrq(Q-b)}=(-1)^{\rrq(M)+\rrq(Q)+\rb_n a}$ and $\rr(M+a-b)=\rr(M)$ hold by definition. 
By these equalities and Lemma~\ref{lem:sub1_twisting}, we deduce 
the equality $c(M+a-b,N,P-a,Q-b)=(-1)^{\rb_n a}c(M,N,P,Q)$. 
Meanwhile, Lemma~\ref{lem:sub2_twisting} implies that the equality 
\begin{align*}
 (\varepsilon \sqrt{-1})^{\rb_n\mathsf{w}}\kappa_{a,b}(\det g) \mW^{(\varepsilon)}_{d,\mathsf{w}/2}(\xi_{M})(g)=(\varepsilon \sqrt{-1})^{\rb_n(\mathsf{w}-a+b)} \mW^{(\varepsilon)}_{d+a-b,(\mathsf{w}+a+b)/2}(\rI^{\det}_{d,-a+b}(\xi_{M}))(g)
\end{align*}
holds, since $\kappa_{a,b}=\chi_{a-b,(a+b)/2}$. Therefore, we obtain the assertion. 
\end{proof}

\subsubsection{Explicit evaluation of archimedean local zeta integrals}\label{subsec:explicit_arch_zeta}
We retain the notation of the previous subsection and 
assume that  $n>1$. Take $\mu =(\mu_1,\mu_2,\dots ,\mu_{n-1})\in \Lambda_{n-1}$ and $\sw',m\in \mZ$, and 
assume  that both $\lambda^\vee \succeq \mu +m$ and $\lambda -\sw  \succeq \mu^\vee +\sw'+m$ hold. 
We set $d':=2\mu +2\rho_{n-1}-\mathsf{w}'$. 
For simplicity, let us respectively set
$\mathcal{W}_\varepsilon = \mathcal{W}(\pi_{\mathrm{B}_n, d,\mathsf{w}/2}, \psi_{\varepsilon})$ and  
$\mathcal{W}^\prime_{-\varepsilon} = \mathcal{W}(\pi_{\mathrm{B}_{n-1}, d',\mathsf{w}'/2}, \psi_{-\varepsilon}) $. Let us explain the detailed construction of the archimedean local zeta integral pairing 
\begin{align*}
 \widetilde{\cI}^{(m)}(\cdot,\cdot) \colon \left( \mathcal{W}_\varepsilon 
                   \otimes_{\mC} {\bigwedge}^{\rb_n} \gp^{0\vee}_{n\mC}   
                    \otimes_{\mC} V_{\lambda^\vee}\otimes_{\mC}V^\mathrm{conj}_{\lambda-\mathsf{w}}
                    \right) 
               \times 
               \left( \mathcal{W}^\prime_{-\varepsilon}
                   \otimes_{\mC} {\bigwedge}^{\rb_{n-1}}\gp^{0\vee}_{n-1\mC}    
                    \otimes_{\mC} V_{\mu^\vee}\otimes_{\mC} V^\mathrm{conj}_{\mu-\mathsf{w}'}
               \right) \to \mathbf{C}.
\end{align*}

We use the same symbol $\iota_n$ for the injection $\ggl_{n-1\mC} \to \ggl_{n\mC}$ induced from 
$\iota_n\colon \rGL_{n-1}(\mC)\to \rGL_n(\mC)$ defined by (\ref{eq:def_iotan}). 
Let $\iota^\vee_n \colon  {\bigwedge}^{\rb_{n}} \mathfrak{p}^\vee_{n\mathbf{C}} \to {\bigwedge}^{\rb_{n}} \mathfrak{p}^\vee_{n-1\mathbf{C}}$ 
denote the map induced from $\gp^\vee_{n\mathbf{C}}\to \gp^\vee_{n-1\mathbf{C}}\,; \omega \mapsto \omega\circ \iota_n$. 
Then define a pairing 
$\mathrm{s}_{n}(\cdot, \cdot)\colon {\bigwedge}^{\rb_{n}} \gp^{\vee}_{n\mathbf{C}} \otimes_\mathbf{C} {\bigwedge}^{\rb_{n-1}} \gp^{\vee}_{n-1\mathbf{C}} \to \mathbf{C}$ 
to be 
\begin{align}
\label{eq:pairing_Lie}
   \iota^\vee_{n}(\omega) \wedge \omega^\prime = \mathrm{s}_{n}(\omega, \omega^\prime) \mathbf{E}^\vee_{\gp_{n-1}}
        \quad \text{for }\omega\in {\bigwedge}^{\rb_{n}} \gp^{\vee}_{n\mathbf{C}} \text{ and } \omega^\prime \in {\bigwedge}^{\rb_{n-1}} \gp^{\vee}_{n-1\mathbf{C}}. 
\end{align}
Here $\mathbf{E}^\vee_{\gp_{n-1}}$ is a fixed basis of ${\bigwedge}^{(n-1)^2}\mathfrak{p}^\vee_{n-1\mC}$ defined as (\ref{eq:boldE}). Further for $W\in \cW_\varepsilon$, $W'\in \cW'_{-\varepsilon}$ and $s\in \mC$ with the real part
 sufficiently large, we define the archimedean local zeta integral $Z(s,W,W')$ as
\begin{align}
\label{eq:def_archzeta}
 Z(s,W,W')=\int_{\rN_{n-1}(\mC)\backslash \rGL_{n-1}(\mC)} W(\iota_n(g))W'(g)\lvert \det g\rvert^{2s-1} \rd g,
\end{align}
where the right invariant measure $\rd g$ on $\rN_{n-1}(\mC)\backslash \rGL_{n-1}(\mC)$ is normalised as (\ref{eqn:quot_GN_measure}). 
In \cite{Jacquet_001}, it is proved that $Z(s,W,W')$ is meromorphically continued to the whole $s$-plane. 
Then $\widetilde{\cI}^{(m)}(\cdot, \cdot)$ is defined as 
\begin{align*}
&\widetilde{\cI}^{(m)} (W\otimes \omega \otimes \rv_1 \otimes \rv_2 , W^\prime\otimes \omega^\prime \otimes \rv_1^\prime\otimes \rv_2^\prime)\\
&:=  2^{-n(n-1)}(\sqrt{-1})^{-\rb_{n-1}}\rs_{n}(\omega, \omega^\prime) 
      \langle \rv_1, \rv^\prime_1  \rangle^{(m)}_{\lambda^\vee,\mu^\vee}
      \langle \rv_2,\rv_2'\rangle^{(m)}_{\lambda -\sw ,\mu -\sw'}        
     Z(s, W, W^\prime)\vert_{s=\frac{1}{2}+m} 
\end{align*}
for $W\in \cW_\varepsilon$, $W'\in \cW'_{-\varepsilon}$, $\omega\in {\bigwedge}^{\rb_n}\gp^{0\vee}_{n\mC}$, $\omega'\in {\bigwedge}^{\rb_{n-1}}\gp^{0\vee}_{n-1\mC}$, $\rv_1\in V_{\lambda^\vee}$, $\rv_1'\in V_{\mu^\vee}$, $\rv_2\in V^\mathrm{conj}_{\lambda-\mathsf{w}}$ and $\rv_2'\in V^{\mathrm{conj}}_{\mu-\mathsf{w}'}$. Here we remark that the factor $2^{-n(n-1)}(\sqrt{-1})^{-\rb_{n-1}}$ 
corresponds to the choice of $\mE^\vee_{\gp_{n-1}}$ ({\it cf.} Remark \ref{rem:HaarEp}). 

Let $[\pi_{{\rm B}_n,d,\mathsf{w}/2}]_\varepsilon$ be the element constructed as in Section~\ref{subsec:U(n)_inv}. In the same manner, we construct
\begin{align*}
 [\pi_{{\rm B}_{n-1},d',\mathsf{w}'/2}]_{-\varepsilon} \in \left(\cW(\pi_{{\rm B}_{n-1},d',\mathsf{w'}/2},\psi_{-\varepsilon})\otimes_{\mC} {\bigwedge}^{\rb_{n-1}} \gp^{0\vee}_{n-1\mC} \otimes_{\mC} V_{\mu^\vee}\otimes_{\mC}V^{\mathrm{conj}}_{\mu-\mathsf{w}'}\right)^{{\rm U}(n-1)},
\end{align*}
as the element corresponding to $[\mathrm{id}_{V_{d'}}]$. The following theorem is one of the main results of the present article.

\begin{thm}[Main Result I] \label{thm:archzetaformula}
 Retain the notation. Then we have
\begin{multline*}
 \widetilde{\cI}^{(m)}([\pi_{{\rm B}_n,d,\mathsf{w}/2}]_{\varepsilon},[\pi_{{\rm B}_{n-1},d',\mathsf{w}'/2}]_{-\varepsilon}) \\
=2^{-n(n-1)}(\sqrt{-1})^{-\rb_{n-1}}  (-\varepsilon \sqrt{-1})^{(n-1)\mathsf{w}'}(-1)^{(m+1)\rb_n}
L(\tfrac{1}{2}+m,\pi_{{\rm B}_n,d,\mathsf{w}/2}\times \pi_{{\rm B}_{n-1},d',\mathsf{w}'/2}).
\end{multline*}
\end{thm}

\subsubsection{Sketch of the proof of Theorem~$\ref{thm:archzetaformula}$}

 Define $[\pi_{{\rm B}_n,d,\mathsf{w}/2}]^{\rm pre}$ as
\begin{multline*}
  [\pi_{ \mathrm{B}_n, d, \mathsf{w}/2}]^{\rm pre}
  =   \sum_{M\in \rG (d)} 
             \sum_{N\in \rG (2\rho_n)} 
         \sum_{P\in \rG (\lambda^\vee)} 
         \sum_{Q \in \rG (\lambda -  \mathsf{w})  }     
    c(M, N, P, Q)
     \xi_M  
     \otimes  \rI^{\gp_n}_{2\rho_n}  ( \xi_{N} )  
      \otimes \xi_P \otimes \xi_{Q}  \\
           \in \left(  V_{d} 
                           \otimes_{\mC} \left(  {\bigwedge}^{\rb_n}  \gp^{0\vee}_{n\mC}   \right)  
            \otimes_{\mC} V_{\lambda^\vee}     
           \otimes_{\mC} V^{\rm conj}_{\lambda - \mathsf{w}} \right)^{{\rm U}(n)},
\end{multline*}
where $c(M,N,P,Q)$ is defined as in (\ref{eq:Pic}). Then  $
\bigl(\mW^{(\varepsilon)}_{d,\mathsf{w}/2}\otimes \mathrm{id}\otimes \mathrm{id} \otimes \mathrm{id}\bigr)([\pi_{{\rm B}_n,d,\mathsf{w}/2}]^{\mathrm{pre}})=[\pi_{{\rm B}_n,d,\mathsf{w}/2}]_{\varepsilon}$ holds by definition. In the same manner, we define $[\pi_{\rB_{n-1},d',\sw'/2}]^{\rm pre}$ as 
\[
[\pi_{\rB_{n-1},d',\sw'/2}]^{\rm pre}
  =   \sum_{M'\in \rG (d')}\sum_{N'\in \rG (2\rho_{n-1})}\sum_{P'\in \rG (\mu^\vee)}\sum_{Q'\in \rG (\mu -\sw')}
    c(M', N', P', Q')\xi_{M'} \otimes  \rI^{\gp_{n-1}}_{2\rho_{n-1}}(\xi_{N'})\otimes \xi_{P'}\otimes \xi_{Q'}.
\]
Furthermore, if we define another pairing 
\begin{align*}
\Psi^{(m)}\colon \left(V_d 
                   \otimes_{\mC} {\bigwedge}^{\rb_{n}}\gp^{0\vee}_{n\mC}   
                    \otimes_{\mC} V_{\lambda^\vee}\otimes_{\mC} V^{\mathrm{conj}}_{\lambda-\mathsf{w}}
                    \right)              \otimes_{\mC}
               \left(V_{d^\prime}  
                   \otimes_{\mC} {\bigwedge}^{\rb_{n-1}}\gp^{0\vee}_{n-1\mC}   
                    \otimes_{\mC} V_{\mu^\vee}  \otimes_{\mC} V^{\mathrm{conj}}_{\mu-\mathsf{w}'}
               \right)
 \to V_d \otimes_{\mC} V_{d'} 
\end{align*}               
to be  
\begin{align*}
\Psi^{(m)}\bigl( (\mathrm{V} \otimes \omega \otimes \rv_1 \otimes \rv_2) \otimes (\mathrm{V}^\prime\otimes \omega^\prime \otimes \rv_1^\prime\otimes \rv_2^\prime)\bigr)    
   :=  {\rm s}_{n}(\omega, \omega^\prime) 
            \langle \rv_1, \rv^\prime_1  \rangle^{(m)}_{\lambda^\vee,\mu^\vee}
      \langle \rv_2,\rv_2'\rangle^{(m)}_{\lambda -\sw ,\mu -\sw'}        
 \mathrm{V}\otimes \mathrm{V}'
\end{align*}
for $\mathrm{V}\in V_d$, $\mathrm{V}'\in V_{d'}$, $\omega\in {\bigwedge}^{\rb_n}\gp^{0\vee}_{n\mC}$, $\omega'\in {\bigwedge}^{\rb_{n-1}}\gp^{0\vee}_{n-1\mC}$, $\rv_1\in V_{\lambda^\vee}$, $\rv_1'\in V_{\mu^\vee}$, $\rv_2\in V^\mathrm{conj}_{\lambda-\mathsf{w}}$ and $\rv_2'\in V^{\mathrm{conj}}_{\mu-\mathsf{w}'}$, we have
\begin{multline} \label{eq:I_Psi}
\widetilde{\cI}^{(m)}\bigl([\pi_{{\rm B}_n,d,\mathsf{w}/2}]_{\varepsilon},[\pi_{{\rm B}_{n-1},d',\mathsf{w}'/2}]_{-\varepsilon})\\
=2^{-n(n-1)} (\sqrt{-1})^{-\rb_{n-1}}
 \left.Z\Bigl(s,(\mW^{(\varepsilon)}_{d,\mathsf{w}/2}\otimes \mW^{(-\varepsilon)}_{d',\mathsf{w}'/2})\bigl(\Psi^{(m)}([\pi_{{\rm B}_n,d,\mathsf{w}/2}]^{\mathrm{pre}}\otimes [\pi_{{\rm B}_{n-1},d',\mathsf{w}'/2}]^{\mathrm{pre}})\bigr)\Bigr)\right|_{s=\frac{1}{2}+m}.
\end{multline}
Here we regard $Z(s,\cdot )$ as a $\mC$-linear form 
$\cW_{\varepsilon}\otimes_\mC 
\cW_{-\varepsilon}'\to \mC$ by setting 
$Z(s,W\otimes W'):=Z(s,W,W')$ ($W\in \cW_{\varepsilon}$, $W'\in \cW_{-\varepsilon}'$). 
Since \cite[Lemma 4.2]{im} combined with Lemma~\ref{lem:hom_dual_conj} implies that the ${\rm U}(n-1)$-invariant subspace of $V_d\otimes_{\mC} V_{d'}$ is of dimension one and spanned by 
\begin{align*}
 \sum_{M\in \rG (d^{\prime \vee})}  
        \frac{ (-1)^{{\rm q}(M)}  }{\rr (M)}   
           \xi_{M[d]} \otimes \xi_{M^\vee},
\end{align*}
there exists a nonzero complex number $c^{(m)}_{\lambda,\mathsf{w},\mu,\mathsf{w}'}$ satisfying
\begin{align} \label{eq:ratio}
 \Psi^{(m)}([\pi_{{\rm B}_n,d,\mathsf{w}/2}]^{\mathrm{pre}}\otimes [\pi_{{\rm B}_{n-1},d',\mathsf{w}'/2}]^{\mathrm{pre}})=c^{(m)}_{\lambda,\mathsf{w},\mu,\mathsf{w}'} \sum_{M\in \rG (d^{\prime \vee})}  
        \frac{ (-1)^{{\rm q}(M)}  }{\rr (M)}   
           \xi_{M[d]} \otimes \xi_{M^\vee}.
\end{align}
Then, combining (\ref{eq:I_Psi}) and (\ref{eq:ratio}), we have 
\begin{multline*}
 \widetilde{\cI}^{(m)}([\pi_{{\rm B}_n,d,\mathsf{w}/2}]_\varepsilon, [\pi_{{\rm B}_{n-1},d',\mathsf{w}'/2}]_{-\varepsilon}) \\
=2^{-n(n-1)} (\sqrt{-1})^{-\rb_{n-1}}c^{(m)}_{\lambda,\mathsf{w},\mu,\mathsf{w}'}\sum_{M\in \rG(d^{\prime \vee})} \dfrac{(-1)^{\rrq(M)}}{\rr(M)}Z(s,\mW^{(\varepsilon)}_{d,\mathsf{w}/2}(\xi_{M[d]}), \mW^{(-\varepsilon)}_{d',\mathsf{w}'/2}(\xi_{M^\vee}))\vert_{s=\frac{1}{2}+m}.
\end{multline*}
Then, by (\ref{eq:archint2}) in Proposition~\ref{prop:ArchInt}, we can deduce the following equality:
\begin{align*}
 \widetilde{\cI}^{(m)}&([\pi_{{\rm B}_n,d,\mathsf{w}/2}]_\varepsilon, [\pi_{{\rm B}_{n-1},d',\mathsf{w}'/2}]_{-\varepsilon})\\
&=2^{-n(n-1)} (\sqrt{-1})^{-\rb_{n-1}}c^{(m)}_{\lambda,\mathsf{w},\mu,\mathsf{w}'} (-\varepsilon\sqrt{-1})^{\sum_{i=1}^{n-1}d_i'} L(\tfrac{1}{2}+m,\pi_{{\rm B}_n,d,\mathsf{w}/2}\times \pi_{{\rm B}_{n-1},d',\mathsf{w}'/2}).
\end{align*}
Note that $\sum_{i=1}^{n-1}d_i'$ is calculated as $2\ell(\mu)-(n-1)\mathsf{w}'$ for $\ell(\mu)=\sum_{i=1}^{n-1} \mu_i$. Furthermore we can completely determine the constant $c^{(m)}_{\lambda,\mathsf{w},\mu,\mathsf{w}'}$:

\begin{prop}\label{prop:coefC}
Retain the notation. We have 
$c^{(m)}_{\lambda, \mathsf{w}, \mu, \mathsf{w}^\prime} 
= (-1)^{ (n-1) \mathsf{w}^\prime + (m+1) \mathrm{b}_n  + \ell(\mu) }$.
\end{prop} 

We verify Proposition~\ref{prop:coefC} in Section~\ref{sec:calcoeff}, which is a computational key of the proof. Using this, we can calculate as
\begin{align*}
c^{(m)}_{\lambda,\mathsf{w},\mu,\mathsf{w}'}(-\varepsilon \sqrt{-1})^{\sum_{i=1}^{n-1}d_i'}
&=(-1)^{(n-1)\mathsf{w}'+(m+1)\rb_n+\ell(\mu)} (-\varepsilon \sqrt{-1})^{2\ell (\mu)-(n-1)\mathsf{w}'}
=(-1)^{(m+1)\rb_{n}}(-\varepsilon \sqrt{-1})^{(n-1)\mathsf{w}'},
\end{align*}
which completes the proof of Theorem~\ref{thm:archzetaformula}.

\subsection{Uniform integrality of critical values}
\label{subsec:uniformintegrality}

In this subsection we discuss the uniform integrality of the critical values of $L(s,\pi^{(n)}\times \pi^{(n-1)})$ by a cohomological method, which is the main result of the present article. 
In Section \ref{subsec:Wh_compatibility}, we introduce a condition for Whittaker functions for $\psi_+$ and $\psi_-$ 
to correspond to a common cusp form.
After discussing how to choose local Whittaker functions at finite places in Section~\ref{subsec:whittvec}, we define the notion of {\em $p$-optimal Whittaker periods} in Section~\ref{subsec:WhittakerPeriods}. Here the explicit generator $[\pi^{(\cdot)}_\infty]_{\pm \varepsilon}$ constructed in Section~\ref{subsec:explicit_generators} and the integral structure of cohomology groups discussed in Appendix~\ref{subsec:ratint} play crucial roles. In Section~\ref{subsec:uniformintegrality}, we state and prove the uniform integrality result (Theorem~\ref{thm:UniformIntegrality}). 

We retain the notation and settings of the previous subsections; in particular let $\pi^{(n)}$ and $\pi^{(n-1)}$ be cohomological irreducible cuspidal automorphic representations of $\rGL_n(F_{\mA})$ and $\rGL_{n-1}(F_{\mA})$, respectively, and $\blambda$ and $\bmu$ the highest weight associated with $\pi^{(n)}$ and $\pi^{(n-1)}$, respectively. Assume the existence of $m_0\in \mZ$ satisfying $\blambda \succeq \bmu+m_0$. Hereafter let $p$ be a prime number satisfying
\begin{align} \label{eq:p>wt}
 p> \max \{ \lambda_{\sigma,1}-\lambda_{\sigma,n}+n-2 \mid \sigma\in I_F \}.
\end{align}
Note that the inequality (\ref{eq:p>wt}) and the interlacing condition $\blambda^\vee \succeq \bmu+m_0$ immediately imply 
\begin{align*}
 p>\max \{ \mu_{1,\sigma}-\mu_{n-1,\sigma}+(n-1)-2 \mid \sigma\in I_F\}.
\end{align*}

\subsubsection{Compatibility condition for Whittaker functions}
\label{subsec:Wh_compatibility}

Recall that  $\cW(\pi^{(n)},\psi_{\varepsilon})\cong {\bigotimes}_{v\in \Sigma_F}'\cW(\pi^{(n)}_v,\psi_{\varepsilon,v})$ denotes 
the Whittaker model of  $\pi^{(n)}$. 
Let $\varphi$ be a cusp form for $\pi^{(n)}$. 
For $\varepsilon \in \{\pm \}$, 
the Whittaker function $w_{\varphi,\psi_\varepsilon}\in 
\cW(\pi^{(n)}, \psi_\varepsilon)$ corresponding to $\varphi$ is defined by 
\begin{align*}
&w_{\varphi,\psi_\varepsilon}(g):=\int_{\rN_n(F)\backslash \rN_n(F_{\mA})}\varphi (xg)
\psi_{-\varepsilon ,\rN_n}(x)\,\rd x&
&(g\in \rGL_n(F_{\mA})),
\end{align*}
where $\psi_{\varepsilon ,\rN_n}(x):=\psi_\varepsilon (x_{1,2}+x_{2,3}+\dots  +x_{n-1,n})$ 
for $x=(x_{i,j})_{1\leq i,j\leq n}\in \rN_n(F_{\mA})$ 
and $\mathrm{d}x:=\prod_{v\in \Sigma_F}\rd x_v$ with the self-dual measure $\rd x_v$ on $\rN_n(F_v)\simeq F_v^{\frac{1}{2}n(n-1)}$ 
with respect to $\psi_{-\varepsilon,v}$ (as in Section \ref{subsec:relationzeta} and \ref{subsec:measure}).   
 Let 
\[
\bepsilon_n:=\diag ((-1)^{n-1},(-1)^{n-2},\dots ,-1,1)\in \rGL_n(F).
\]
By the substitution $x\to \bepsilon_nx\bepsilon_n$, we have  
\begin{align*}
w_{\varphi,\psi_{-\varepsilon}}(g)
&=\int_{\rN_n(F)\backslash \rN_n(F_{\mA})}\varphi (\bepsilon_nx\bepsilon_ng)
\psi_{\varepsilon,\rN_n}(\bepsilon_nx\bepsilon_n)\,\rd x
=\int_{\rN_n(F)\backslash \rN_n(F_{\mA})}\varphi (x\bepsilon_ng)
\psi_{-\varepsilon,\rN_n}(x)\,\rd x=w_{\varphi,\psi_{\varepsilon}}(\bepsilon_ng) 
\end{align*}
for $g\in \rGL_n(F_{\mA})$. Here the second equality follows from the left $\rGL_n(F)$-invariance of $\varphi$ and 
$\psi_{\varepsilon,\rN_n}(\bepsilon_nx\bepsilon_n)=\psi_{-\varepsilon ,\rN_n}(x)$ 
for $x\in \rN_n(F_{\mA})$. 
Conversely, by Shalika's Fourier--Whittaker expansion 
\cite[Theorem 5.9]{Sha74}, if we choose a ($K_n$-finite) Whittaker function $w_{\varepsilon}\in 
\cW(\pi^{(n)}, \psi_\varepsilon)$ for each $\varepsilon \in \{\pm \}$, 
so that 
\begin{align}
\label{eq:Wh_compatibility_global}
&w_{-\varepsilon}(g)=w_{\varepsilon}(\bepsilon_ng)&&\text{for $g\in \rGL_n(F_{\mA})$}, 
\end{align}
then $w_{+}$ and $w_{-}$ correspond to a common cusp form. 
For $w_{\varepsilon}=\prod_{v\in \Sigma_F}w_{\varepsilon ,v}\in \cW(\pi^{(n)},\psi_{\varepsilon})$ with 
$w_{\varepsilon,v}\in \cW(\pi^{(n)}_v, \psi_{\varepsilon,v})$, 
the global condition (\ref{eq:Wh_compatibility_global}) holds if the local condition 
\begin{align}
\label{eq:Wh_compatibility}
&w_{-\varepsilon ,v}(g_v)=w_{\varepsilon ,v}(\bepsilon_ng_v)&&\text{for $g_v\in \rGL_n(F_v)$}
\end{align}
holds at any place $v\in \Sigma_F$. Of course, if the condition (\ref{eq:Wh_compatibility}) holds for either $\varepsilon =+$ or $\varepsilon =-$, 
it holds for both. Lemma \ref{lem:arch_Wh_compatibility} implies that our archimedean Whittaker functions $\mW^{(\varepsilon)}_{d_v,\sw /2}(\rv )$  
satisfy the compatibility condition (\ref{eq:Wh_compatibility}). 
In Section \ref{subsec:whittvec}, we choose Whittaker functions satisfying 
the compatibility condition (\ref{eq:Wh_compatibility}) at finite places of $F$. 

The following lemma will be used in Section \ref{subsec:whittvec}.

\begin{lem} \label{lem:zeta_compatibility}
Retain the notation, and let $v\in \Sigma_F$.

\begin{enumerate}
\item \label{num:sub1_zeta_compatibility} 
Let $\varepsilon \in \{\pm \}$, $w\in \cW(\pi^{(n)}_v, \psi_{\varepsilon ,v})$, and set 
$w^{(-)}(g):=w(\bepsilon_ng)$ for $g\in \rGL_n(F_v)$. Then 
$w^{(-)}\in \cW(\pi^{(n)}_v, \psi_{-\varepsilon ,v})$. 

\item \label{num:sub2_zeta_compatibility}
Let $w_{\varepsilon ,v}\in \cW(\pi^{(n)}_v, \psi_{\varepsilon ,v})$ and 
$w_{\varepsilon ,v}'\in \cW(\pi^{(n-1)}_v, \psi_{\varepsilon ,v})$ for each $\varepsilon \in \{\pm\}$. Assume that 
$w_{-\varepsilon,v}(g)=w_{\varepsilon,v}(\bepsilon_ng)$ and 
$w_{-\varepsilon,v}'(h)=w_{\varepsilon,v}'(\bepsilon_{n-1}h)$ for $\varepsilon \in \{\pm\}$, $g\in \rGL_n(F_v)$ and $h\in \rGL_{n-1}(F_v)$. 
Then, for $\varepsilon \in \{\pm\}$, we have 
\begin{align*}
Z_v(s, w_{-\varepsilon,v}, w_{\varepsilon,v}')
=\omega_{\pi^{(n-1)}_v}(-1)Z_v(s, w_{\varepsilon,v}, w_{-\varepsilon,v}'),
\end{align*}
where $Z_v(s,-,-)$ is the local zeta integral defined as $\eqref{eq:local_zeta_int}$, 
and $\omega_{\pi^{(n-1)}_v}$ is the central character of $\pi^{(n-1)}_v$. 
\end{enumerate}
\end{lem}
\begin{proof}
The statement \ref{num:sub1_zeta_compatibility} follows from 
$\psi_{\varepsilon, \rN_n, v}(\bepsilon_n x \bepsilon_n)  
= \psi_{-\varepsilon, \rN_n, v}(x)$ for $x\in \rN_n(F_v)$, 
where $\psi_{\varepsilon, \rN_n, v}$ is the character of $\rN_n(F_v)$ corresponding to 
$\psi_{\varepsilon, v}$ (defined similarly to (\ref{eq:def_psiNn})).
The statement \ref{num:sub2_zeta_compatibility} immediately follows from the definition \eqref{eq:local_zeta_int} and the substitution 
$g_v\to -\bepsilon_{n-1}g_v$. 
\end{proof}

\subsubsection{Choice of sections} \label{subsec:whittvec}   

In this paragraph, we specify the choice of sections of 
\begin{align*}
 H^{\rb_{n,F}}(\mathfrak{g}_{n\mC},\widetilde{K}_n\,; \cW(\pi^{(n)},\psi_\varepsilon)\otimes_{\mC}\widetilde{V}(\blambda^\vee)) 
\cong  H^{\rb_{n,F}}(\mathfrak{g}_{n\mC},\widetilde{K}_n\,; \cW(\pi^{(n)}_\infty,\psi_{\varepsilon,\infty})\otimes_{\mC}\widetilde{V}(\blambda^\vee)) \otimes_{\mC} \cW(\pi^{(n)}_\mathrm{fin},\psi_{\varepsilon,\mathrm{fin}}).
\end{align*}

Let $\alpha \in \mathrm{Aut}(\mC)$ be an arbitrary automorphism of the complex field. Following \cite[Section~2.1.5]{rag16}, we define the $\alpha$-twist ${}^\alpha\pi^{(n)}_?$ of $\pi^{(n)}_?$ as ${}^\alpha \pi^{(n)}_?:=\pi^{(n)}_?\otimes_{\mC, \alpha}\mC$ for $?=\infty, \mathrm{fin}$ or $v\in \Sigma_F$ (see also \cite[Section I.1]{Waldspurger}).   
Recall from Section~\ref{subsec:C_def_ps} that the archimedean part $\pi^{(n)}_\infty$ of $\pi^{(n)}$ is described as $\bigotimes_{v\in \Sigma_{F,\infty}}\pi_{{\rm B}_n,d_v,\mathsf{w}/2}$ for $d_v=2\lambda_v+2\rho_n-\sw$. Then, if we set ${}^\alpha d_v=2{}^\alpha \lambda_v+2\rho_n-\mathsf{w}=2\lambda_{\alpha^{-1}\circ v}+2\rho_n-\mathsf{w}$, the $\alpha$-twist ${}^\alpha \pi^{(n)}_\infty$ of the archimedean part is known to be isomorphic to $\bigotimes_{v\in \Sigma_{F, \infty}} \pi_{{\rm B}_n,{}^\alpha d_v,\mathsf{w}/2}$ (see \cite[Lemma 7.1]{gr14} or \cite[Proposition I.3 (i)]{Waldspurger}). 
Then \cite[Th\'eor\`eme 3.13]{clo90} yields that there exists a unique 
cohomological irreducible cuspidal automorphic representation ${}^\alpha\pi^{(n)}$ of $\mathrm{GL}_n(F_\mathbf{A})$  
such that the finite (resp.\ archimedean) part of ${}^\alpha\pi^{(n)}$ is given by ${}^\alpha \pi^{(n)}_\mathrm{fin}$ (resp.\ ${}^\alpha \pi^{(n)}_\infty$).

At the archimedean part, set $[\pi^{(n)}_\infty]_\varepsilon:=\bigotimes_{v\in \Sigma_{F,\infty}} [\pi^{(n)}_v]_\varepsilon$ where each $[\pi^{(n)}_v]_\varepsilon=[\pi_{{\rm B}_n,d_v,\mathsf{w}/2}]_{\varepsilon}$ is the element constructed as in Section~\ref{subsec:U(n)_inv}. Then, by construction, $[\pi^{(n)}_\infty]_\varepsilon$ is a generator of the $(\mathfrak{g}_{n\mC},\widetilde{K}_n)$-cohomology group $H^{\rb_{n,F}}(\mathfrak{g}_{n\mC}, \widetilde{K}_n\,; \cW(\pi^{(n)}_\infty,\psi_{\varepsilon,\infty}) \otimes_{\mC} \widetilde{V}(\blambda^\vee))$. Similarly set $[{}^\alpha \pi^{(n)}_\infty ]_\varepsilon:=\bigotimes_{v\in \Sigma_{F,\infty}} [\pi_{{\rm B}_n,{}^\alpha d_v,\mathsf{w}/2}]_\varepsilon$ for each $\alpha\in \mathrm{Aut}(\mC)$.
Then  $[{}^\alpha \pi^{(n)}_\infty]_\varepsilon$ indeed generates $H^{\rb_{n,F}}(\mathfrak{g}_{n\mC}, \widetilde{K}_n\,; \cW({}^\alpha \pi^{(n)}_\infty,\psi_{\varepsilon,\infty}) \otimes_{\mC} \widetilde{V}({}^\alpha\blambda^\vee))$ because the highest weight associated with ${}^\alpha \pi^{(n)}$ equals ${}^\alpha \blambda$.

Next we specify the Whittaker functions at finite places. At first we briefly recall the theory of new vectors. 
Set $\mathfrak{N}:=\prod_{v\in \Sigma_{F,\mathrm{fin}}} \mathfrak{P}_v^{\mathfrak{f}(\pi_v^{(n)})}$ where $\mathfrak{P}_v$ is the prime ideal corresponding to $v$ and $\mathfrak{f}(\pi_v^{(n)}) \in \mN_0$ is the conductor of $\pi^{(n)}_v$, and let $\mathcal{K}_{n,1}(\mathfrak{N})$ be the mirahoric subgroup
\begin{align} \label{eq:mirahoric}
 \mathcal{K}_{n,1}(\mathfrak{N}):=
   \left\{ k=(k_{i,j})_{1\leq i,j\leq n} \in \rGL_n(\widehat{\mathfrak{r}}_F) 
       ~\middle|~ k_{n,j}\equiv 0 \pmod{\mathfrak{N}} \text{ for }1\leq j\leq n-1, \, k_{n,n}\equiv 1\pmod{\mathfrak{N}} \right\}
\end{align}
of level $\mathfrak{N}$, where $\widehat{\gr}_F:=\gr_F\otimes_{\mZ} \widehat{\mZ}$. 
Let $D_F$ be the discriminant of $F$
and assume that $D_F$ is prime to $\mathfrak{N}$.   
Let $\delta = \prod_{v\in \Sigma_{F, \mathrm{fin}}}  \delta_v \in F_{\mA,\mathrm{fin}}^\times$ be a finite id\`ele generating the different ideal of $F$. 
Note that $D_F\widehat{\mZ}=\mathrm{N}_{F/\mathbf{Q}}(\delta)\widehat{\mZ}$ holds and
  that the additive character $\psi_{\varepsilon}$ has the conductor $\delta^{-1} \widehat{\mathfrak{r}}_F$, 
 that is, $\delta^{-1} \widehat{\mathfrak{r}}_F$ is the maximal fractional ideal of $F$ among 
 the fractional ideals $\mathfrak{a}$ on which 
 $\psi_{\varepsilon}$ is trivial.  
For each $v\in \Sigma_{F,\mathrm{fin}}$, put 
\begin{align*}
  \boldsymbol{\delta}^{(n)}_v 
  = \mathrm{diag}(\delta^{n-1}_v, \delta^{n-2}_v,  \ldots, \delta_v, 1) \quad \in \rGL_n(F_v). 
\end{align*}
If $\pi^{(n)}$ is spherical at $v\in \Sigma_{F,\mathrm{fin}}$, we normalise a spherical vector $w^{\mathrm{sph}}_v(\pi^{(n)})_\varepsilon$ of $\mathcal{W}(\pi^{(n)}_v, \psi_{\varepsilon, v} )$  
so that 
\[
w^{\mathrm{sph}}_v(\pi^{(n)})_\varepsilon ((\boldsymbol{\delta}^{(n)}_v)^{-1}) = 1
\]
holds. 
By $w^{\mathrm{sph}}_v(\pi^{(n)})_\varepsilon (\bepsilon_n(\boldsymbol{\delta}^{(n)}_v)^{-1})
=w^{\mathrm{sph}}_v(\pi^{(n)})_\varepsilon ((\boldsymbol{\delta}^{(n)}_v)^{-1}) = 1$ and Lemma \ref{lem:zeta_compatibility} \ref{num:sub1_zeta_compatibility}, 
we know that the normalised spherical vectors $w^{\mathrm{sph}}_v(\pi^{(n)})_\varepsilon $ satisfy the compatibility condition 
(\ref{eq:Wh_compatibility}), that is,  
\begin{align}
\label{eq:sph_compatibility}
&w^{\mathrm{sph}}_v(\pi^{(n)})_{-\varepsilon}(g_v)
=w^{\mathrm{sph}}_v(\pi^{(n)})_\varepsilon (\bepsilon_ng_v)&&\text{for $g_v\in \rGL_n(F_v)$}. 
\end{align}

\begin{thm} \label{thm:essential}
The $\mathcal{K}_{n,1}(\mathfrak{N})$-fixed subspace 
$\mathcal{W}(\pi^{(n)}_\mathrm{fin}, \psi_{\varepsilon, \mathrm{fin}} )^{\mathcal{K}_{n,1}(\mathfrak{N})}$ 
of $\mathcal{W}(\pi^{(n)}_\mathrm{fin}, \psi_{\varepsilon, \mathrm{fin}} )$ is one dimensional. 
Moreover, for each finite place $v$ of $F$, 
there exists a unique vector $w^\mathrm{ess}_v(\pi^{(n)})_{\varepsilon} \in \mathcal{W}(\pi^{(n)}_v, \psi_{\varepsilon, v})^{\mathcal{K}_{n,1}(\mathfrak{N})_v}$ 
such that
\begin{align*}
Z_v(s, w^\mathrm{ess}_v(\pi^{(n)})_{\varepsilon}, w^{\mathrm{sph}}_v(\pi^{(n-1)})_{-\varepsilon})
=      \omega_{\pi^{(n-1)}_v}(\delta_v)^{-1}
        |\delta_v|^{   -\frac{1}{2}n(n-1)  (s-\frac{1}{2}) 
                            +\frac{1}{12}(n-1)(n-2)(2n-3)}_v
         L_v(s, \pi^{(n)}_v \times \pi^{(n-1)}_v)
\end{align*}
holds  if $\pi^{(n-1)}_v$ is spherical. Here $Z_v(s,-,-)$ is the local zeta integral defined as $\eqref{eq:local_zeta_int}$, $\omega_{\pi^{(n-1)}_v}$ is the central character of $\pi^{(n-1)}_v$ and 
$w^{\mathrm{sph}}_v(\pi^{(n-1)})_{-\varepsilon}$ is the normalised spherical vector of $\mathcal{W}(\pi^{(n-1)}_v, \psi_{-\varepsilon, v} )$. 
In particular, if $v$ does not divide $\mathfrak{N}$, the unique vector $w^\mathrm{ess}_v(\pi^{(n)})_{\varepsilon}$ above is given by $w^\mathrm{sph}_v(\pi^{(n)})_{\varepsilon}$.
\end{thm}
\begin{proof}
The fact that $\mathcal{W}(\pi^{(n)}_\mathrm{fin}, \psi_{\varepsilon, \mathrm{fin}} )^{\mathcal{K}_{n,1}(\mathfrak{N})}$ is one dimensional 
and the existence of $w^\mathrm{ess}_v(\pi^{(n)})_{\varepsilon}$ for each finite place $v\nmid D_F$ of $F$ follow from \cite[(5.1)~Th\'eor\`eme, (4.1)~Th\'eor\`eme (ii)]{jpss81}.   
If $v$ is prime to $D_F\mathfrak{N}$, then the equality $w^\mathrm{ess}_v(\pi^{(n)})_{\varepsilon} = w^\mathrm{sph}_v(\pi^{(n)})_{\varepsilon}$ immediately follows from \cite[Theorem]{shin76}.  
Suppose that $v$ divides $D_F$. Note that $v$ does not divide $\mathfrak{N}$, since $D_F$ is prime to $\mathfrak{N}$.   
Define the additive character $\psi^\circ_{\varepsilon, v}:F_v \to \mathbf{C}^\times$ of $F_v$ to be $\psi^\circ_{\varepsilon, v}(x) = \psi_{\varepsilon, v}(\delta^{-1}_v x)$ for $x\in F_v$, 
which has the conductor $\mathfrak{r}_{F, v}$. 
Then we find that 
$\psi^\circ_{\varepsilon, \mathrm{N}_n, v}(\boldsymbol{\delta}^{(n)}_v x (\boldsymbol{\delta}^{(n)}_v)^{-1})  
= \psi_{\varepsilon, \mathrm{N}_n, v}(x)$ for each $x\in \mathrm{N}_n(F_v)$, 
where $\psi^\circ_{\varepsilon, \rN_n, v}$ and $\psi_{\varepsilon, \rN_n, v}$ are characters of $\rN_n(F_v)$ corresponding to 
$\psi^\circ_{\varepsilon, v}$ and $\psi_{\varepsilon, v}$, respectively (defined similarly to (\ref{eq:def_psiNn})).
Hence there exists a spherical vector $w^{\mathrm{sph}, \circ}_{v}(\pi^{(n)})_{\varepsilon}\in \mathcal{W}(\pi^{(n)}_v, \psi^\circ_{\varepsilon, v})$ satisfying the following identity:
\begin{align*}
w^\mathrm{sph}_v(\pi^{(n)})_{\varepsilon}(g) = w^{\mathrm{sph}, \circ}_v (\pi^{(n)})_{\varepsilon}(\boldsymbol{\delta}^{(n)}_v g).   
\end{align*}
Since $w^\mathrm{sph}_v(\pi^{(n)})_{\varepsilon}$ is normalised so that $w^\mathrm{sph}_v(\pi^{(n)})_{\varepsilon}((\boldsymbol{\delta}^{(n)}_v)^{-1}) = 1$ holds, 
we have $w^{\mathrm{sph}, \circ}_v(\pi^{(n)})_{\varepsilon}(1_n)=1$.
Similarly there exists a spherical vector 
$w^{\mathrm{sph}, \circ}_v(\pi^{(n-1)})_{-\varepsilon}\in \mathcal{W}(\pi^{(n-1)}_v, \psi^\circ_{-\varepsilon, v})$
satisfying 
\[
w^{\mathrm{sph}}_v(\pi^{(n-1)})_{-\varepsilon}(g) = w^{\mathrm{sph}, \circ}_v (\pi^{(n-1)})_{-\varepsilon}(\boldsymbol{\delta}^{(n-1)}_v g).
\]
In Section \ref{subsec:relationzeta}, we fixed the Haar measures $\rd x$, $\rd h$ respectively on 
$\rN_{n-1}(F_v)$, $\rGL_{n-1}(F_v)$, and also took the quotient measure $\rd g$ on 
$\rN_{n-1}(F_v) \backslash \rGL_{n-1}(F_v)$ characterised by them. 
We define another measure $\mathrm{d}^\circ x$ on $\rN_{n-1}(F_v)\simeq F_v^{\frac{1}{2}(n-1)(n-2)}$ to be self-dual 
with respect to the additive character $\psi^\circ_{-\varepsilon, v}$, 
and take the quotient measure $\mathrm{d}^\circ g$ on 
$\mathrm{N}_{n-1}(F_v) \backslash \mathrm{GL}_{n-1}(F_v)$ characterised by $\mathrm{d}^\circ x$ and $\rd h$. 
We note that $\mathrm{d}^\circ x$ and $\mathrm{d}^\circ g$ are given by 
$|\delta_v|^{-\frac{1}{2} \cdot \frac{1}{2}(n-1)(n-2)}_v \mathrm{d} x$ and 
$|\delta_v|^{ \frac{1}{4}(n-1)(n-2)}_v \mathrm{d} g$, respectively. 
Then $Z_v(s, w^\mathrm{sph}_v(\pi^{(n)})_{\varepsilon}, w^{\mathrm{sph}}_v(\pi^{(n-1)})_{-\varepsilon})$ is calculated as follows (note that $\boldsymbol{\delta}^{(n)}_v\iota_n(g)=\iota_n(\delta_v \boldsymbol{\delta}^{(n-1)}_v g)$ holds):
\begin{align*}
 & Z_v(s, w^\mathrm{sph}_v(\pi^{(n)})_{\varepsilon}, w^{\mathrm{sph}}_v(\pi^{(n-1)})_{-\varepsilon}) 
      \times |\delta_v|^{-\frac{1}{4}(n-1)(n-2)}_v \\ 
&= \int_{\mathrm{N}_{n-1}(F_v) \backslash \mathrm{GL}_{n-1}(F_v)}  w^\mathrm{sph}_v (\pi^{(n)})_{\varepsilon}( \iota_n(g) )  
             w^\mathrm{sph}_v (\pi^{(n-1)})_{-\varepsilon}(g)
             \lvert \det(g) \rvert^{s-\frac{1}{2}}_v \mathrm{d}^\circ g  \\
&= \int_{\mathrm{N}_{n-1}(F_v) \backslash \mathrm{GL}_{n-1}(F_v)}  w^{\mathrm{sph}, \circ}_v (\pi^{(n)})_{\varepsilon}(\boldsymbol{\delta}^{(n)}_v \iota_n(g) )  
      w^{\mathrm{sph}, \circ}_v (\pi^{(n-1)})_{-\varepsilon}(\boldsymbol{\delta}^{(n-1)}_v g)
             \lvert \det(g)\rvert^{s-\frac{1}{2}}_v \mathrm{d}^\circ g   \\
&= \int_{\mathrm{N}_{n-1}(F_v) \backslash \mathrm{GL}_{n-1}(F_v)} 
        w^{\mathrm{sph}, \circ}_v (\pi^{(n)})_{\varepsilon}( \iota_n(g) ) 
           \omega_{\pi^{(n-1)}_v}(\delta^{-1}_v)    w^{\mathrm{sph}, \circ}_v (\pi^{(n-1)})_{-\varepsilon}( g)
             \lvert \det( (\delta_v  \boldsymbol{\delta}^{(n-1)}_v)^{-1} g)\rvert^{s-\frac{1}{2}}_v  \\
   & \quad \quad \quad \quad \quad \quad \quad \quad \quad \quad \quad \quad \quad \quad \quad 
   \quad \quad \quad \quad \quad \quad \quad \quad \quad \quad  
   \quad \quad \quad \quad \quad \quad        
       \times \Delta_{\mathrm{B}_{n-1},v}((\boldsymbol{\delta}^{(n-1)}_v)^{-1})^{-1}      \mathrm{d}^\circ g. 
\end{align*}
Here $\Delta_{\mathrm{B}_{n-1},v}: \mathrm{B}_{n-1}(F_v) \to \mathbf{Q}^\times$ is the modulus character, 
which is explicitly given by the formula: 
\begin{align*}
&\Delta_{\mathrm{B}_{n-1},v}(b) = \prod^{n-1}_{i=1} |b_{i,i}|^{n-2i}_v  &&
\text{for} \quad b=(b_{i,j})_{1\leq i,j\leq n-1} \in \mathrm{B}_{n-1}(F_v),  
\end{align*} 
and the third equality is obtained via the replacement $g\mapsto (\delta_v\boldsymbol{\delta}^{(n-1)}_v)^{-1}g$. 
By using \cite[Theorem]{shin76} again, we find that 
\begin{align*}    
&\int_{\mathrm{N}_{n-1}(F_v) \backslash \mathrm{GL}_{n-1}(F_v)} 
        w^{\mathrm{sph}, \circ}_v (\pi^{(n)})_{\varepsilon}( \iota_n(g) )   
                  w^{\mathrm{sph}, \circ}_v (\pi^{(n-1)})_{-\varepsilon}( g)
          \lvert \det(g)\rvert^{s-\frac{1}{2}}_v\mathrm{d}^\circ g        
=    L_v(s, \pi^{(n)}_v \times \pi^{(n-1)}_v).
\end{align*}
Since the direct calculation shows that  
\begin{align*}
  |\delta_v|^{-\frac{1}{4}(n-1)(n-2)}_v 
   |\det( (\delta_v  \boldsymbol{\delta}^{(n-1)}_v)^{-1}) |^{s-\frac{1}{2}}_v
   \Delta_{\mathrm{B}_{n-1},v}(\boldsymbol{\delta}^{(n-1)}_v)    
& = |\delta_v|^{-\frac{1}{4}(n-1)(n-2)
                       -\frac{1}{2}n(n-1)(s-\frac{1}{2})
                       +\sum^{n-1}_{i=1}(n-1-i)(n-2i)
                       }_v \\
&= |\delta_v|^{ -\frac{1}{2}n(n-1)(s-\frac{1}{2}) + \frac{1}{12}(n-1)(n-2)(2n-3) }_v, 
\end{align*}
the identity in the statement follows from the above argument. 
Namely, if $v$ divides the discriminant, 
this shows that $w^\mathrm{sph}_v(\pi^{(n)})_{\varepsilon}$ gives $w^\mathrm{ess}_v(\pi^{(n)})_{\varepsilon}$.
\end{proof}

The normalised function $w^\mathrm{ess}_v(\pi^{(n)})_{\varepsilon}$ in Theorem~\ref{thm:essential} is called the (normalised) {\em new vector} or the {\em essential vector}. 
Note that the normalisation of $w^\mathrm{ess}_v(\pi^{(n)})_{\varepsilon}$ is independent of the  auxiliary choice 
of $\pi^{(n-1)}_v$. For details refer to \cite[Corollary 3.3]{mat13}. 
Since $\omega_{\pi_v^{(n-1)}}(-1)=1$ if $\pi_v^{(n-1)}$ is spherical, by Lemma \ref{lem:zeta_compatibility}, 
(\ref{eq:sph_compatibility}) and Theorem~\ref{thm:essential}, 
we know that the new vectors $w^\mathrm{ess}_v(\pi^{(n)})_{\varepsilon}$ 
satisfy the compatibility condition 
(\ref{eq:Wh_compatibility}), that is,  
\begin{align}
\label{eq:ess_compatibility}
&w^{\mathrm{ess}}_v(\pi^{(n)})_{-\varepsilon}(g_v)
=w^{\mathrm{ess}}_v(\pi^{(n)})_\varepsilon (\bepsilon_ng_v)&&\text{for $g_v\in \rGL_n(F_v)$}. 
\end{align}
Set $w^\mathrm{ess}_\mathrm{fin} (\pi^{(n)})_{\varepsilon} = \bigotimes_{v\in \Sigma_{F, \mathrm{fin}}} w^\mathrm{ess}_v(\pi^{(n)})_{\varepsilon}$.

 Following \cite[Section 3.4.1]{mah05} or \cite[Section 2.2.3]{rag16}, at the finite part,  we distinguish a particular Whittaker function 
$w^\circ_\mathrm{fin}(\pi^{(n)})_{\varepsilon}=\bigotimes_{v\in \Sigma_{F, \mathrm{fin}}} w^\circ_v(\pi^{(n)})_{\varepsilon} \in \mathcal{W}(\pi^{(n)}_\mathrm{fin}, \psi_{\varepsilon,\mathrm{fin}})$ as follows. 
For each $v\in \Sigma_{F,\mathrm{fin}}$, choose an element $a(\pi^{(n)}_v) =\mathrm{diag}(a_1, \ldots, a_n)$ of $\mathrm{T}_n(F_v)$ so that 
         $a_i a^{-1}_{i+1} \in \gr_{F, v}$ holds for each $i=1, 2, \dotsc, n-1$ and 
$w^\mathrm{ess}_v(\pi^{(n)})_{\varepsilon}(a(\pi^{(n)}_v))=w^\mathrm{ess}_v(\pi^{(n)})_{+}(a(\pi^{(n)}_v))$ does not equal $0$; see  \cite[Section 1.3.2, Lemma]{mah05} for existence of such $a(\pi_v^{(n)})$. 
         Set $c(\pi^{(n)}_v) = \bigl(w^\mathrm{ess}_v(\pi^{(n)})_{\varepsilon}(a(\pi^{(n)}_v))\bigr)^{-1}\in \mC^\times$ and 
         define $w^\circ_v(\pi^{(n)})_{\varepsilon} = c(\pi^{(n)}_v) w^\mathrm{ess}_v(\pi^{(n)})_{\varepsilon}$. Note that, if $\pi^{(n)}$ is spherical at $v$, we can (and do) choose $(\boldsymbol{\delta}^{(n)}_v)^{-1}$ as $a(\pi^{(n)}_v)$. One thus observes that $c(\pi^{(n)}_v)=1$ holds and $w^\circ_v(\pi^{(n)})_{\varepsilon}=w^\mathrm{ess}_v(\pi^{(n)})_{\varepsilon}$ coincides with the normalised spherical Whittaker function $w^{\mathrm{sph}}_v(\pi^{(n)})_{\varepsilon}$ whenever $\pi^{(n)}$ is spherical at $v$. 

Now consider the twist with respect to $\alpha \in \mathrm{Aut}(\mC)$. Due to class field theory, there exists a unique finite id\`ele  $u_\alpha \in \widehat{\mathfrak{r}}^\times_F$ satisfying $\alpha(\psi_{\varepsilon,\rfin}(x)) = \psi_{\varepsilon,\rfin}( u_\alpha x )$ for each $x \in F_{\mathbf{A}, \mathrm{fin}}$ (see \cite[Section~3.2]{rs08} for details). 
Define $\bu_{\alpha}^{(n)}$ to be  $\mathrm{diag}(u^{n-1}_\alpha, u_\alpha^{n-2}, \dotsc, u_\alpha, 1) \in \mathrm{GL}_n(\widehat{\mathfrak{r}}_F)$ and set ${}^\alpha w(g)=\alpha(w((\bu_{\alpha}^{(n)})^{-1}g))$ for $w\in \mathcal{W}(\pi^{(n)}_\mathrm{fin},\psi_{\varepsilon,\mathrm{fin}})$ and $g\in \mathrm{GL}_n(F_{\mA,\mathrm{fin}})$. Note that $\boldsymbol{u}^{(n)}_\alpha$ is also an element of $\cK_{n,1}(\mathfrak{N})$. Then ${}^\alpha w$ is indeed a $\psi_{\varepsilon,\mathrm{fin}}$-Whittaker function of ${}^\alpha \pi^{(n)}_\mathrm{fin}$ and 
\begin{align*}
 \mathrm{tw}^{\cW}_\alpha \colon  \mathcal{W}(\pi^{(n)}_{\mathrm{fin}}, \psi_{\varepsilon,\mathrm{fin}})  \longrightarrow \mathcal{W}({}^\alpha\pi^{(n)}_{\mathrm{fin}}, \psi_{\varepsilon,\mathrm{fin}})\, ;  
  w \longmapsto   {}^\alpha w
\end{align*}
gives an $\alpha$-semilinear and $\mathrm{GL}_n(F_{\mathbf{A}, \mathrm{fin}})$-equivariant homomorphism. For each $v\in \Sigma_{F,\mathrm{fin}}$, it is easy to check that $w^\mathrm{ess}_v({}^\alpha \pi^{(n)})_{\varepsilon}(a(\pi^{(n)}_v))$ does not equal $0$, using right $\cK_{n,1}(\mathfrak{N})_v$-invariance of $w^\mathrm{ess}_v({}^\alpha \pi^{(n)})_{\varepsilon}$ and the fact that $w^\mathrm{ess}_v({}^\alpha \pi^{(n)})_{\varepsilon}$ is a nonzero multiple of the local component of ${}^\alpha w^\mathrm{ess}_{\rfin} (\pi^{(n)})_{\varepsilon}$ at $v$ (uniqueness of the new vector up to scalar multiples). In other words, we can choose $a(\pi^{(n)}_v)$ as  $a({}^\alpha\pi^{(n)}_v)$ for every $v\in \Sigma_{F,\mathrm{fin}}$ and $\alpha\in \mathrm{Aut}(\mC)$, and if we do choose so, we readily check that the distinguished Whittaker function $w^\circ_{\mathrm{fin}} (\pi^{(n)})_\varepsilon= \bigotimes_{v\in \Sigma_{F, \mathrm{fin}}} w^\circ_{v}(\pi^{(n)})_\varepsilon$ behaves compatibly under the $\alpha$-twist
\begin{align}\label{eq:whittcircequiv}
{}^\alpha w^\circ_{\mathrm{fin}}(\pi^{(n)})_\varepsilon = w^\circ_{\mathrm{fin}}({}^\alpha \pi^{(n)})_\varepsilon
\end{align}
for every $\alpha \in \mathrm{Aut}(\mC)$; see \cite[Section 3.4.1]{mah05} for details. Set $c(\pi^{(n)}) = \prod_{v\in \Sigma_{F, \mathrm{fin}}} c(\pi^{(n)}_v)$; as we have already observed, this is indeed a finite product. 
By \cite[Section 2.5.5.1]{rag16} (see also \cite[page 621, line 3]{mah05} or \cite[Proposition 2.3 (c)]{mah98}), 
 we find that 
   $\alpha(c(\pi^{(n)})) = c({}^\alpha\pi^{(n)})$ holds
for each $\alpha \in \mathrm{Aut}(\mathbf{C})$. 
Combining this with (\ref{eq:whittcircequiv}), we finally obtain 
\begin{align}\label{eq:whittessequiv}
{}^\alpha w^\mathrm{ess}_{\mathrm{fin}}(\pi^{(n)})_\varepsilon = w^\mathrm{ess}_{\mathrm{fin}}({}^\alpha \pi^{(n)})_\varepsilon. 
\end{align}

\subsubsection{$p$-optimal Whittaker periods} \label{subsec:WhittakerPeriods}

Using the explicit cohomology class $[{}^\alpha \pi^{(n)}_\infty]_{\varepsilon}$ introduced in Section~\ref{subsec:whittvec}, we can consider a map 
\begin{align*} 
  \mathscr{F}_{[{}^\alpha \pi^{(n)}_\infty]_\varepsilon} \colon \cW({}^\alpha \pi^{(n)}_{\rm fin}, \psi_{\varepsilon,\mathrm{fin}})^{\cK_{n,1}(\mathfrak{N})}  
    \longrightarrow   
         H^{\rb_{n,F}}  (\mathfrak{g}_{n\mC},\widetilde{K}_n\,; {}^\alpha \pi^{(n)}\otimes_{\mC}\widetilde{V}({}^\alpha \blambda^\vee))^{\mathcal{K}_{n,1}(\mathfrak{N})} \hookrightarrow H^{\rb_{n,F}}_\mathrm{cusp}(Y^{(n)}_\mathcal{K}, \widetilde{\cV}({}^\alpha \blambda^\vee)) 
\end{align*}
induced by $w_{\mathrm{fin}}\mapsto [{}^\alpha \pi^{(n)}_\infty]_\varepsilon\otimes w_{\mathrm{fin}}$ as in (\ref{eq:Fourier}), where $\mathcal{K}_{n,1}(\mathfrak{N})$ is the mirahoric group introduced in Section~\ref{subsec:whittvec}. Then, for each $\alpha\in \mathrm{Aut}(\mC)$, we define the cuspidal cohomology class $\delta({}^\alpha \pi^{(n)})$ as 
\begin{align}\label{eq:ESdelta}
 \delta({}^\alpha\pi^{(n)}) 
 := \mathscr{F}_{[{}^\alpha \pi^{(n)}_\infty]_\varepsilon} (w^{\mathrm{ess}}_\mathrm{fin}({}^\alpha \pi^{(n)})_\varepsilon )    
   \in H^{\rb_{n,F}}_{\rm cusp}   (   Y^{(n)}_{\mathcal K},    \widetilde{\mathcal V}({}^\alpha \blambda^\vee) ).
\end{align}
Here $w_\mathrm{fin}^{\mathrm{ess}}({}^\alpha \pi^{(n)})_\varepsilon$ is the new vector introduced in Section~\ref{subsec:whittvec}. 
We call $\delta({}^\alpha \pi^{(n)})$ the {\em Eichler--Shimura class} of $\pi^{(n)}$, which generates the $1$-dimensional complex vector space $H^{\rb_{n,F}}(\mathfrak{g}_{n\mC},\widetilde{K}_n\,; {}^\alpha \pi^{(n)}\otimes_{\mC} \widetilde{V}({}^\alpha \blambda^\vee))^{\mathcal{K}_{n,1}(\mathfrak{N})}$ by construction.

\begin{prop}
Retain the notation. Then $\delta({}^\alpha\pi^{(n)})$ is independent of the auxiliary choice of $\varepsilon \in \{\pm \}$.
\end{prop}
\begin{proof}
The assertion follows from Lemma \ref{lem:arch_Wh_compatibility}, (\ref{eq:ess_compatibility}) and the arguments in Section \ref{subsec:Wh_compatibility}. 
\end{proof}

Now for each $\alpha \in \mathrm{Aut}(\mC)$, set $\mQ({}^\alpha \pi^{(n)}):=\mQ(\pi^{(n)}_\mathrm{fin})\mQ({}^\alpha\blambda)$, which is a finite extension of $\mQ$, and let 
${}^\alpha \mathfrak{P}_n$ be a prime ideal of $\mQ({}^\alpha \pi^{(n)})$ induced by $\boldsymbol{i} \circ \alpha^{-1}$, where $\boldsymbol{i}\colon \mC \xrightarrow{\, \sim \,} \mC_p$ denotes the fixed field isomorphism. We write  $\mathfrak{r}({}^\alpha \pi^{(n)})_{({}^\alpha \mathfrak{P}_n)}$ for  the localisation of the ring of integers of $\mQ({}^\alpha \pi^{(n)})$ at ${}^\alpha \mathfrak{P}_n$. Then the cuspidal cohomology group $H_{\mathrm{cusp}}^{\rb_{n,F}}(Y^{(n)}_\mathcal{K}, \widetilde{\cV}({}^\alpha \blambda^\vee))$ admits an $\mathfrak{r}({}^\alpha \pi^{(n)})_{({}^\alpha \mathfrak{P}_n)}$-integral structure $H_{\mathrm{cusp}}^{\rb_{n,F}}(Y^{(n)}_\mathcal{K}, \widetilde{\cV}({}^\alpha \blambda^\vee)_{\mathfrak{r}({}^\alpha \pi^{(n)})_{({}^\alpha \mathfrak{P}_n)}})$ (see (\ref{eq:integral_str}) in Section~\ref{subsec:ratint} for details), which enables us to equip $H^{\rb_{n,F}}(\mathfrak{g}_{n\mC}, \widetilde{K}_n\,; \pi^{(n)}\otimes_{\mC}\widetilde{V}({}^\alpha\blambda^\vee))^{\mathcal{K}_{n,1}(\mathfrak{N})}$ with the $\mathfrak{r}({}^\alpha \pi^{(n)})_{({}^\alpha \mathfrak{P}_n)}$-integral structure by
\begin{equation} \label{eq:int_str_gK}
 \begin{aligned}
 &H^{\rb_{n,F}}(\mathfrak{g}_{n\mC}, \widetilde{K}_n\,; \pi^{(n)}\otimes_{\mC}\widetilde{V}({}^\alpha \blambda^\vee))^{\mathcal{K}_{n,1}(\mathfrak{N})}_{\mathfrak{r}({}^\alpha \pi^{(n)})_{({}^\alpha \mathfrak{P}_n)}}\\
&\qquad \quad := H^{\rb_{n,F}}(\mathfrak{g}_{n\mC}, \widetilde{K}_n\,; \pi^{(n)}\otimes_{\mC}\widetilde{V}({}^\alpha \blambda^\vee))^{\mathcal{K}_{n,1}(\mathfrak{N})}\cap H^{\rb_{n,F}}_\mathrm{cusp}(Y^{(n)}_\mathcal{K}, \widetilde{\cV}({}^\alpha \blambda^\vee)_{\mathfrak{r}({}^\alpha \pi^{(n)})_{({}^\alpha \mathfrak{P}_n)}}).
\end{aligned}
\end{equation}
Hereafter we always drop the superscript $\mathrm{id}$ from the notation when $\alpha$ is the identity map $\mathrm{id}\colon \mC\xrightarrow{\, =\,}\mC$.  
By construction, the module \eqref{eq:int_str_gK} is free of rank one over $\mathfrak{r}({}^\alpha \pi^{(n)})_{({}^\alpha \mathfrak{P}_n)}$. 
After appropriate renormalisation, if necessary, we can choose its generator $\eta({}^\alpha \pi^{(n)})$ so that,  for each $\alpha \in \mathrm{Aut}(\mC)$, the equality 
\begin{align}\label{eq:eta_equiv}
   \eta( {}^\alpha \pi^{(n)} )   = \mathrm{tw}_\alpha  (  \eta( \pi^{(n)} ) )     
\end{align}
holds with respect to the $\alpha$-semilinear isomorphism $\mathrm{tw}_\alpha \colon H^{\rb_{n,F}}_{\rm c}(Y^{(n)}_\mathcal{K}, \widetilde{\cV}(\blambda^\vee))\xrightarrow{\, \sim \,} H^{\rb_{n,F}}_{\rm c}(Y^{(n)}_\mathcal{K}, \widetilde{\cV}({}^\alpha\blambda^\vee))$ induced by (\ref{eq:tw_a}). 

\begin{dfn}\label{dfn:whittper}
For each $\alpha\in \mathrm{Aut}(\mC)$, let $p^{\rb}({}^\alpha \pi^{(n)})$ be the complex number satisfying
\begin{align*}
\delta({}^\alpha \pi^{(n)})=p^{\rb}({}^\alpha \pi^{(n)})\eta({}^\alpha \pi^{(n)}).
\end{align*}
We call $\boldsymbol{p}^{\rb}(\pi^{(n)}):=\{ p^{\rb}({}^\alpha \pi^{(n)}) \mid \alpha\in \mathrm{Aut}(\mC)\}$ the (set of) {\em $p$-optimal Whittaker periods} of $\pi^{(n)}$. 
\end{dfn}

\begin{rem}
 The set of $p$-optimal Whittaker periods is uniquely determined up to multiples of $\mathfrak{r}(\pi^{(n)})_{(\mathfrak{P}_n)}^\times$ in the following sense: if $\boldsymbol{p}^{\rb}(\pi^{(n)})$ and $\boldsymbol{p}^{\prime \rb}(\pi^{(n)})$ are $p$-optimal Whittaker periods of $\pi^{(n)}$, there exists $c\in \mathfrak{r}(\pi^{(n)})_{(\mathfrak{P}_n)}^\times$ satisfying 
\begin{align*}
 p^{\rb}({}^\alpha \pi^{(n)})=\alpha(c) p^{\prime {\rb}}({}^\alpha \pi^{(n)}) 
\end{align*}
for every $\alpha\in \mathrm{Aut}(\mC)$. Such ambiguity comes from the choice of integral basis $\{ \eta({}^\alpha \pi^{(n)}) \mid \alpha\in \mathrm{Aut}(\mC)\}$. Note that $\boldsymbol{p}^{\rb}(\pi^{(n)})$ is indeed a finite set because, by construction, we have $p^{\rb}({}^\alpha \pi^{(n)})=p^{\rb}(\pi^{(n)})$ if $\alpha$ fixes $\mQ(\pi^{(n)})$. Indeed it might be better to recognise $\boldsymbol{p}^{\rb}(\pi^{(n)})$ as an element of $(\mQ(\pi^{(n)})\otimes_{\mQ}\mC)^\times/\mathfrak{r}(\pi^{(n)})_{(\mathfrak{P}_n)}^\times$, as Deligne's conjecture suggests. 
\end{rem}

The $p$-optimal Whittaker periods are defined so that the normalised map $\mathscr{F}^\circ_{[{}^\alpha \pi^{(n)}_\infty]_\varepsilon}=p^{\rb}({}^\alpha\pi^{(n)})^{-1}\mathscr{F}_{[{}^\alpha \pi^{(n)}_\infty]_\varepsilon}$ commutes with the action of $\mathrm{Aut}(\mC)$; that is, 
\begin{align*}
 \mathscr{F}^\circ_{[{}^\alpha \pi^{(n)}_\infty]_\varepsilon} \circ \mathrm{tw}^{\cW}_\alpha =\mathrm{tw}_\alpha \circ \mathscr{F}^\circ_{[\pi^{(n)}_\infty]_\varepsilon}
\end{align*}
holds for any $\alpha \in \mathrm{Aut}(\mC)$. 

\begin{rem}
Let $\kappa\colon F^\times \backslash F_{\mA}^\times \rightarrow \mC^\times$ be an algebraic Hecke character 
of infinity type $(\boldsymbol{a},\boldsymbol{b})=(a_v,b_v)_{v\in \Sigma_{F,\infty}}$; 
that is, $\kappa$ satisfies $\kappa((z_v)_{v\in \Sigma_{F,\infty}})=\prod_{v\in \Sigma_{F,\infty}} z_v^{a_v}\bar{z}_v^{b_v}$. 
Then there is $c\in \mZ$ such that $a_v+b_v=c$ for any $v\in \Sigma_F$. 
Set $[\pi^{(n)}_\infty]_\varepsilon^\circ :=(\varepsilon \sqrt{-1})^{\rb_{n,F} \sw}[\pi^{(n)}_\infty]_\varepsilon =\bigotimes_{v\in \Sigma_{F,\infty}}[\pi_{{\rm B}_n,d_v,\mathsf{w}/2}]_\varepsilon^\circ $ and 
$\boldsymbol{W}(\kappa_\infty):=\bigotimes_{v\in \Sigma_{F,\infty}} W_{\kappa_{a_v,b_v}}\otimes \mathrm{id}\otimes 
\rI^{\det}_{\lambda_v^\vee,a_v}\otimes \rI^{\det}_{\lambda_v-\mathsf{w},b_v}$ for 
$[\pi_{{\rm B}_n,d_v,\mathsf{w}/2}]_\varepsilon^\circ :=(\varepsilon \sqrt{-1})^{\rb_n \mathsf{w}}[\pi_{{\rm B}_n,d_v,\mathsf{w}/2}]_\varepsilon$ and $W_{\kappa_{a,b}}$ defined as in Proposition~\ref{prop:twisting}. Then Proposition~\ref{prop:twisting} implies that
\begin{align} \label{eq:twisting}
 \boldsymbol{W}(\kappa_\infty)([\pi^{(n)}_\infty]_\varepsilon^\circ )=[(\pi^{(n)}\otimes \kappa)_\infty]_\varepsilon^\circ 
\end{align}
holds. Raghuram and Shahidi require to impose this constraint to the generator of the $(\mathfrak{g}_{n\mC},\widetilde{K}_n)$-cohomology; see \cite[p.~13]{rs08}. The normalisation factor $(\varepsilon \sqrt{-1})^{\rb_{n,F} \sw}$ in the definition of $[\pi_{\infty}^{(n)}]_\varepsilon^\circ$ is introduced to meet this requirement, and the condition (\ref{eq:twisting}) plays a crucial role in the proof of \cite[Proposition~4.6]{rs08}. 
Since we have  
$[(\pi^{(n)}\otimes \kappa)_\infty]_\varepsilon^\circ =(\varepsilon \sqrt{-1})^{\rb_{n,F} (\mathsf{w}+c)}[(\pi^{(n)}\otimes \kappa)_\infty]_\varepsilon $, 
we can verify that our $p$-optimal Whittaker periods satisfy, for example, the equality
\begin{align} \label{eq:period_twist}
 p^{\rb}(\pi^{(n)}\otimes \kappa) \sim_{\mQ(\pi^{(n)},\kappa)} 
(\varepsilon \sqrt{-1})^{\rb_{n,F}c}\gamma(\kappa_\mathrm{fin}, \psi_{\varepsilon,\mathrm{fin}})^{-\rb_n}p^{\rb}(\pi^{(n)}),
\end{align} 
where $\gamma(\kappa_\mathrm{fin}, \psi_{\varepsilon,\mathrm{fin}})$ denotes the Gauss sum of $\kappa_{\mathrm{fin}}$ and $\psi_{\varepsilon,\mathrm{fin}}$ (see \cite[Theorem~4.1 (2)]{rs08} for details). 
\end{rem}

\subsubsection{Uniform integrality of critical values}  \label{subsec:unformintegrality}

In the rest of this section, we assume that
\begin{itemize}[leftmargin=6em]
\item[(unr)$_{n-1}$] $\pi^{(n-1)}$ is spherical at every finite place $v\in \Sigma_{F,\mathrm{fin}}$,
\end{itemize}
or equivalently, we assume that $\pi^{(n-1)}$ has a nonzero $\rGL_{n-1}(\widehat{\mathfrak{r}}_F)$-fixed vector for 
$\widehat{\mathfrak{r}}_F:=\mathfrak{r}_F\otimes_{\mZ} \widehat{\mZ}$ and 
$w_\mathrm{fin}^{\mathrm{ess}}(\pi^{(n-1)})_{-\varepsilon}=\bigotimes_{v\in \Sigma_{F, \mathrm{fin}}} w^{\mathrm{sph}}_v(\pi^{(n-1)})_{-\varepsilon}$.
For each $\alpha\in \mathrm{Aut}(\mC)$, set $\mQ({}^\alpha \pi^{(n)},{}^\alpha\pi^{(n-1)})=\mQ({}^\alpha \pi^{(n)})\mQ({}^\alpha\pi^{(n-1)})$ and 
let ${}^\alpha \mathfrak{P}_0$ be a prime ideal of $\mQ({}^\alpha \pi^{(n)},{}^\alpha\pi^{(n-1)})$ induced by $\boldsymbol{i}\circ \alpha^{-1}$. We write $\mathfrak{r}({}^\alpha \pi^{(n)},{}^\alpha\pi^{(n-1)})_{({}^\alpha \mathfrak{P}_0)}$ for the localisation of the ring of integers of $\mQ({}^\alpha \pi^{(n)},{}^\alpha \pi^{(n-1)})$ at ${}^\alpha \mathfrak{P}_0$. As in Appendix~\ref{subsec:ratint}, take an auxiliary subfield $E^*$ of $\mC$ so that it contains both $\mQ(\pi^{(n)},\pi^{(n-1)})$ and the normal closure $F_{\mathrm{nc}}$ of $F$ in $\mC$. Let $\cE^*$ be the closure of the image of the composite map $E^* \subset \mC\xrightarrow[\, \boldsymbol{i}\,]{\,\sim\, } \mC_p$, and $\cO^*$ its ring of integers. Assume that $p$ satisfies the inequality (\ref{eq:p>wt}). Note that $\boldsymbol{i}\circ \alpha^{-1}$ induces a field embedding $i_p^{\alpha} \colon \mQ({}^\alpha \pi^{(n)},{}^\alpha \pi^{(n-1)})\hookrightarrow \cE^*$.

Recall that we always assume existence of an integer $m_0\in \mZ$ satisfying the interlacing condition $\blambda^\vee \succeq \bmu+m_0$ \eqref{eq:interlace_condition}, and thus the Rankin--Selberg $L$-function $L(s,\pi^{(n)}\times \pi^{(n-1)})$ admits a critical point. Suppose that a half integer $\frac{1}{2}+m$ is a critical point of $L(s,\pi^{(n)}\times \pi^{(n-1)})$; or equivalently, suppose that $\blambda^\vee \succeq \bmu+m$ holds. Then, by using $[\cdot,\cdot]^{(m)}_{{}^\alpha \blambda,{}^\alpha \bmu}$ defined as in (\ref{eq:global_pairing1})  and $\mathrm{Tw}_m$ defined as in (\ref{eq:tw_e}) for the chosen $m$, we obtain a pairing 
\begin{align} \label{eq:pairing_A}
 [\cdot,\cdot]^{(m)}_{{}^\alpha \blambda,{}^\alpha \bmu} \colon H_{\mathrm{c}}^{\rb_{n,F}}(Y^{(n)}_\mathcal{K},\widetilde{\cV}({}^\alpha \blambda^\vee)_A) \times H^{\rb_{n-1,F}}(Y^{(n-1)}_\mathcal{K},\widetilde{\cV}({}^\alpha \bmu^\vee)_A) \longrightarrow A
\end{align}
for $A=\mC$ or $\mQ({}^\alpha \pi^{(n)},{}^\alpha \pi^{(n-1)})$. When $\alpha$ is the identity map $\mathrm{id}$, we also consider the pairing \eqref{eq:pairing_A} for $A=\cE^*$. Here we note that any steps in construction of $[\cdot,\cdot]^{(m)}_{{}^\alpha \blambda,{}^\alpha \bmu}$ preserve the $A$-rational structure; indeed, pulling-backs along $\mathrm{p}_{n-1}$ and $j_n$ do not affect the rational structure of the local systems at all, and the pairing $[\cdot,\cdot]^{(m)}_{{}^\alpha \blambda,{}^\alpha \bmu}$ preserve the rational structure by Lemma~\ref{lem:pairing_rational}. By the same construction replacing the local systems by their $p$-adic avatars, and using $[\cdot,\cdot]_{\blambda,\bmu}^{(m),p}$ defined as (\ref{eq:global_pairing2}) and $\mathrm{Tw}_m^{(p)}$ defined as (\ref{eq:tw_a}), we also obtain a pairing
\begin{align} \label{eq:pairing_cA}
 [\cdot,\cdot]^{(m),p}_{\blambda,\bmu} \colon H_{\mathrm{c}}^{\rb_{n,F}}(Y^{(n)}_\mathcal{K},\widetilde{\cV}(\blambda^\vee)^{(p)}_{\cA}) \times H^{\rb_{n-1,F}}(Y^{(n-1)}_\mathcal{K},\widetilde{\cV}(\bmu^\vee)^{(p)}_{\cA}) \longrightarrow \cA
\end{align}
for $\cA=\cE^*$ or $\cO^*$. Then the diagram
\begin{align*}
\xymatrix @C=0.3pc {
H_{\mathrm{c}}^{\rb_{n,F}}(Y^{(n)}_\mathcal{K},\widetilde{\cV}({}^\alpha \blambda^\vee)_{\mQ({}^\alpha \pi^{(n)}, {}^\alpha \pi^{(n-1)})}) \ar@{^(->}[d]_-{\mathrm{tw}_{i_p^\alpha}} & \times & H^{\rb_{n-1,F}}(Y^{(n-1)}_\mathcal{K},\widetilde{\cV}({}^\alpha \bmu^\vee)_{\mQ({}^\alpha \pi^{(n)}, {}^\alpha \pi^{(n-1)})}) \ar[rrrrr]^-{[\cdot,\cdot]^{(m)}_{{}^\alpha \blambda,{}^\alpha \bmu}} \ar@{^(->}[d]^-{\mathrm{tw}_{i_p^\alpha}} &&&&& \mQ({}^\alpha \pi^{(n)},{}^\alpha \pi^{(n-1)}) \ar@{^(->}[d]^-{i_p^\alpha} \\
H_{\mathrm{c}}^{\rb_{n,F}}(Y^{(n)}_\mathcal{K},\widetilde{\cV}( \blambda^\vee)_{\cE^*}) \ar[d]_-{\rotatebox{90}{$\sim$}} & \times & H^{\rb_{n-1,F}}(Y^{(n-1)}_\mathcal{K},\widetilde{\cV}( \bmu^\vee)_{\cE^*}) \ar[rrrrr]^-{[\cdot,\cdot]^{(m)}_{ \blambda, \bmu}} \ar[d]^-{\rotatebox{90}{$\sim$}} &&&&& \cE^* \ar@{=}[d] \\
H_{\mathrm{c}}^{\rb_{n,F}}(Y^{(n)}_\mathcal{K},\widetilde{\cV}( \blambda^\vee)^{(p)}_{\cE^*}) & \times & H^{\rb_{n-1,F}}(Y^{(n-1)}_\mathcal{K},\widetilde{\cV}( \bmu^\vee)^{(p)}_{\cE^*}) \ar[rrrrr]^-{[\cdot,\cdot]^{(m),p}_{ \blambda, \bmu}} &&&&& \cE^* \\
H_{\mathrm{c}}^{\rb_{n,F}}(Y^{(n)}_\mathcal{K},\widetilde{\cV}( \blambda^\vee)^{(p)}_{\cO^*}) \ar@{->}[u] & \times & H^{\rb_{n-1,F}}(Y^{(n-1)}_\mathcal{K},\widetilde{\cV}( \bmu^\vee)^{(p)}_{\cO^*}) \ar[rrrrr]^-{[\cdot,\cdot]^{(m),p}_{ \blambda, \bmu}} \ar@{->}[u] &&&&& \cO^* \ar@{^(->}[u]
} 
\end{align*}
commutes (see Section~\ref{subsec:ratint} for the definition of $\mathrm{tw}_{i_p^\alpha}$); the commutativity of the upper and lower squares is due to the functoriality of the pairings (\ref{eq:pairing_A}) and (\ref{eq:pairing_cA}) under scalar extensions, and that of the middle square is due to the commutative diagram (\ref{eq:pairing_compatible}), under the identification of local systems proposed in Lemma~\ref{lem:p-compatible}. This implies that, on the $\mathfrak{r}({}^\alpha \pi^{(n)},{}^\alpha \pi^{(n-1)})_{({}^\alpha \mathfrak{P}_0)}$-integral structure of cohomology groups, defined as the ``intersection'' of $\mQ({}^\alpha \pi^{(n)},{}^\alpha\pi^{(n-1)})$-rational structure and $\cO^*$-integral structure, the pairing $[\cdot,\cdot]^{(m)}_{{}^\alpha \blambda,{}^\alpha \bmu}$ takes values in $ (i_p^{\alpha})^{-1}(\cO^*)=\mathfrak{r}({}^\alpha \pi^{(n)},{}^\alpha \pi^{(n-1)})_{({}^\alpha \mathfrak{P}_0)}$.

As a consequence, we obtain the following theorem, which is the main result of the present article. 
Define a constant $\mathcal{C}(m, {}^\alpha \pi^{(n)} \times {}^\alpha \pi^{(n-1)})$ 
for each critical point $\frac{1}{2}+m$ of $L(s, \pi^{(n)} \times \pi^{(n-1)})$ 
 and $\alpha \in \mathrm{Aut}(\mC)$ by
\begin{align} \label{eq:constant}
\begin{aligned}
\mathcal{C}(m,{}^\alpha \pi^{(n)} \times {}^\alpha \pi^{(n-1)}) 
=& \omega_{{}^\alpha \pi^{(n-1)}} (\delta)^{-1} D^{\frac{1}{2} n(n-1)m-\frac{1}{12}(n-1)(n-2)(2n-3)}_F  \\
  & \quad \times \prod_{v\in \Sigma_{F, \infty}}  
    \left(      2^{-n(n-1)} (\sqrt{-1})^{-\mathrm{b}_{n-1}} 
    (-\sqrt{-1})^{(n-1)\sw'} 
   (-1)^{  (m+1)\mathrm{b}_n   }
   \right).
\end{aligned}
\end{align}

Prior to introducing the statement of the main theorem on uniform integrality of critical values, let us summarise assumptions on $\pi^{(n)}$ and $\pi^{(n-1)}$ which we have imposed in the previous subsections.

\begin{itemize}
 \item[--] $\pi^{(n)}$ (resp.\ $\pi^{(n-1)}$) is a cohomological irreducible cuspidal automorphic representation of $\rGL_n(F_{\mA})$ (resp.\ $\rGL_{n-1}(F_{\mA})$) appearing in the cuspidal cohomology group $H^{\rb_{n,F}}_{\mathrm{cusp}}(Y^{(n)}_{\cK}, \widetilde{\cV}(\blambda^\vee))$ (resp.\ $H^{\rb_{n-1,F}}_{\mathrm{cusp}}(Y^{(n-1)}_{\cK}, \widetilde{\cV}(\bmu^\vee))$);
 \item[--]  the dominant integral weights $\blambda\in \Lambda_n^{I_F}$ and $\bmu\in \Lambda_{n-1}^{I_F}$ have purity weights $\mathsf{w}$ and $\mathsf{w}'$, respectively, and assume that there exists an integer $m_0\in \mZ$ satisfying the interlacing condition $\blambda^\vee \succeq \bmu+m_0$;
 \item[--] $\pi^{(n-1)}$ is spherical at every finite place $v\in \Sigma_{F,\mathrm{fin}}$; 
 \item[--] the conductor ideal $\mathfrak{N}$ of $\pi^{(n)}_\mathrm{fin}$ is relatively prime to the discriminant $D_F$ of $F$, where $\mathfrak{N}$ is an integral ideal of $\widehat{\mathfrak{r}}_F$ maximal among ideals $\mathfrak{A}$ such that $\pi^{(n)}_{\mathrm{fin}}$ has a nonzero $\cK_{n,1}(\mathfrak{A})$-fixed vector.

\end{itemize}

\begin{thm}[Main Result II]\label{thm:UniformIntegrality}
Assume that $n>1$. Let $F$ be a totally imaginary field and $\pi^{(n)}$ $($resp.\ $\pi^{(n-1)})$ be a cohomological irreducible cuspidal automorphic representation of $\rGL_n(F_{\mA})$ $($resp.\ $\rGL_{n-1}(F_{\mA}))$ as above.   
Let $p$ be a prime number which 
                                                       satisfies the inequality $\mathrm{(\ref{eq:p>wt})}$.
Then, for the Eichler--Shimura classes $\delta(\pi^{(n)})$ and $\delta(\pi^{(n-1)})$ defined in Section~$\ref{subsec:WhittakerPeriods}$, we find that   
\begin{align} \label{eq:formula_global}
\begin{aligned}
 \dfrac{1}{p^{\rb}(\pi^{(n)}) p^{\rb}(\pi^{(n-1)})}&[\delta(\pi^{(n)}), \delta(\pi^{(n-1)}) ]_{\blambda,\bmu }^{(m)}    \\
    &= \mathcal{C}(m,\pi^{(n)} \times \pi^{(n-1)})   \frac{ L(\frac{1}{2} +m, \pi^{(n)} \times \pi^{(n-1)}) }{  p^{\rb}(\pi^{(n)}) p^{\rb} ( \pi^{(n-1)})  } \qquad  \in \mathfrak{r}(\pi^{(n)},\pi^{(n-1)})_{(\mathfrak{P}_0)}
\end{aligned}
\end{align}
holds{\rm ;} that is, the critical values of $L(s,\pi^{(n)}\times \pi^{(n-1)})$ are {\em uniformly $\mathfrak{P}_0$-integral}. Furthermore, for each $\alpha \in \mathrm{Aut}(\mathbf{C})$, we have 
\begin{align*}
&  \alpha \left(  \mathcal{C}(m,\pi^{(n)} \times \pi^{(n-1)}) 
                       \frac{ L(\frac{1}{2} +m, \pi^{(n)} \times \pi^{(n-1)}) }{ p^{\rb}(\pi^{(n)}) p^{\rb}(\pi^{(n-1)}) } \right)  \\
& \qquad \quad   =  \mathcal{C}(m, {}^\alpha \pi^{(n)} \times {}^\alpha\pi^{(n-1)}) 
      \frac{ L(\frac{1}{2} +m, {}^\alpha\pi^{(n)} \times {}^\alpha\pi^{(n-1)}) }{  p^{\rb}({}^\alpha \pi^{(n)}) p^{\rb}({}^\alpha\pi^{(n-1)})  } \qquad  \in \mathfrak{r}({}^\alpha \pi^{(n)},{}^\alpha \pi^{(n-1)})_{({}^\alpha \mathfrak{P}_0)}.
\end{align*}
\end{thm}

\begin{proof}
 Since both $\eta(\pi^{(n)})=p^{\rb}(\pi^{(n)})^{-1}\delta(\pi^{(n)})$ and $\eta(\pi^{(n-1)})=p^{\rb}(\pi^{(n-1)})^{-1}\delta(\pi^{(n-1)})$ are, by definition, elements of the $\mathfrak{r}(\pi^{(n)},\pi^{(n-1)})_{(\mathfrak{P}_0)}$-integral structures, the argument above suggests that $[\eta(\pi^{(n)}),\eta(\pi^{(n-1)})]^{(m)}_{\blambda,\bmu}$ is indeed contained in $\mathfrak{r}(\pi^{(n)},\pi^{(n-1)})_{(\mathfrak{P}_0)}$. Hence the first half of the statement follows by the product formula (\ref{eq:product_zeta}), the choice of Whittaker functions $w^\mathrm{ess}_\mathrm{fin}(\pi^{(n)})_{\varepsilon}$ at finite places (see Section~\ref{subsec:whittvec}), and the explicit formula of archimedean local zeta integrals (Theorem~\ref{thm:archzetaformula}) with 
\[
\prod_{v\in \Sigma_{F, \infty}}\varepsilon^{(n-1)\sw'}=\prod_{v\in \Sigma_{F, \infty}}\omega_{\pi_v^{(n-1)}}(\varepsilon )
=\omega_{\pi^{(n-1)}}(\varepsilon )=1.
\]
The second half of the statement is a consequence of the $\mathrm{Aut}(\mC)$-equivariance of the pairing $[\cdot,\cdot]^{(m)}_{\blambda,\bmu}$ verified in Lemma~\ref{lem:pairing_equiv} (and the $\mathrm{Aut}(\mC)$-equivariance of the Poincar\'e duality); indeed, we have
\begin{align*}
 \dfrac{[\delta({}^\alpha \pi^{(n)}), \delta({}^\alpha \pi^{(n-1)})]^{(m)}_{{}^\alpha \blambda,{}^\alpha \bmu}}{p^{\rb}({}^\alpha \pi^{(n)})p^{\rb} ({}^\alpha\pi^{(n-1)})} 
 &= [\eta({}^\alpha \pi^{(n)}),\eta({}^\alpha \pi^{(n-1)})]^{(m)}_{{}^\alpha \blambda,{}^\alpha \bmu}\\
&= [\mathrm{tw}_\alpha(\eta(\pi^{(n)})),\mathrm{tw}_\alpha(\eta(\pi^{(n-1)}))]^{(m)}_{{}^\alpha \blambda,{}^\alpha \bmu} =\alpha\left(  [\eta(\pi^{(n)}),\eta(\pi^{(n-1)})]^{(m)}_{\blambda, \bmu}\right) \\
&=\alpha\left( \dfrac{[\delta(\pi^{(n)}), \delta(\pi^{(n-1)})]^{(m)}_{\blambda,\bmu}}{p^{\rb}(\pi^{(n)}) p^{\rb} (\pi^{(n-1)})}\right),
\end{align*}
and thus we should apply (\ref{eq:formula_global}) to the both sides. Here we use (\ref{eq:eta_equiv}) at the second equality.
\end{proof}

\section{Measures and Lie algebras} \label{sec:measures}

In this section,
we normalise measures and present basic results on Lie algebras, which are used throughout the article.

\subsection{Measures} \label{subsec:measure}

Assume that $n>1$. We here normalise the measures on $\rGL_{n-1}(\mC)$ 
and $\rN_{n-1}(\mC)\backslash \rGL_{n-1}(\mC)$. 

Let $\rN_{n-1}^{\rop}$ be the unipotent subgroup of $\rGL_{n-1/F}$ 
consisting of all lower triangular matrices with all diagonal entries equal to $1$. 
As in \cite{im}, 
we normalise the Haar measure $\rd h$ on $\rGL_{n-1}(\mC )$ by 
\begin{align}
\label{eq:GL_measure}
\int_{\rGL_{n-1}(\mC )}f(h)\,\rd h
&=\int_{A_{n-1}}\int_{\rN_{n-1}^\rop (\mC )}\int_{\rU (n-1)}f(kxa)\,\rd k\, \rd x\, \rd a
=\int_{\rU (n-1)}\int_{\rN_{n-1}^\rop (\mC )}\int_{A_{n-1}}f(axk)\, \rd a\, \rd x\,\rd k
\end{align}
for any integrable function $f$ on $\rGL_{n-1}(\mC )$, 
where $\rd k$, $\rd x$ and $\rd a$ are the Haar measures respectively 
on $\rU (n-1)$, $\rN_{n-1}^\rop (\mC )$ and $A_{n-1}$ such that 
\begin{align}
\label{eq:measure_KNopA}
&\int_{\rU (n-1)}\rd k=1,&
&\rd x =\prod_{1\leq j<i\leq n-1}\rd_\mC x_{i,j}
=\prod_{1\leq j<i\leq n-1}2\,\rd x^\rre_{i,j} \rd x^\rim_{i,j},&
&\rd a =\prod_{i=1}^{n-1}\frac{4\,\rd a_i}{a_i}
\end{align}
with $x=(x_{i,j})_{1\leq i,j\leq n-1}= (x^\rre_{i,j} + \sqrt{-1} x^\rim_{i,j})_{1\leq i,j\leq n-1}\in \rN_{n-1}^\rop (\mC )$ and 
$a=\diag (a_1,a_2,\dots ,a_{n-1})\in A_{n-1}$.

We normalise the right invariant measure $\rd g$ on 
$\rN_{n-1}(\mC )\backslash \rGL_{n-1}(\mC )$ by 
\begin{align}
\label{eqn:quot_GN_measure}
&\int_{\rGL_{n-1}(\mC )}f(h)\,\rd h
=\int_{\rN_{n-1}(\mC )\backslash \rGL_{n-1}(\mC )}
\left(\int_{\rN_{n-1}(\mC )}f(xg)\,\rd x\right) \rd g
\end{align}
for any integrable function $f$ on $\rGL_{n-1}(\mC )$, 
where $\rd x$ is the  Haar measure on $\rN_{n-1}(\mC )$ 
in Section \ref{subsec:C_def_ps}, that is, 
$\rd x =\prod_{1\leq i<j\leq n-1}\rd_{\mC}x_{i,j}$ with 
$x=(x_{i,j})_{1\leq i,j\leq n-1}\in \rN_{n-1}(\mC )$. 
The right invariant measure $\rd g$ on 
$\rN_{n-1}(\mC )\backslash \rGL_{n-1}(\mC )$ is used 
for the definition (\ref{eq:def_archzeta}) of 
the archimedean local zeta integral $Z(s,W,W')$ in 
Section \ref{subsec:explicit_arch_zeta}. 
Also note that our choices of the measures 
$\rd g$ and $\rd x$ are consistent with those at finite places in Section \ref{subsec:relationzeta}, 
since $\rd_{\mC}x_{i,j}$ is the self-dual 
measure on $\mC$ with respect to the fixed additive character $\psi_{-\varepsilon}$.

\subsection{Differential forms}\label{subsec:diff_form}

Assume that $n>1$. 
Define $\mE^\vee_{\gp_{n-1}}\in \bigwedge^{(n-1)^2} \mathfrak{p}^\vee_{n-1\mC}$ to be 
\begin{align}\label{eq:boldE}
\begin{aligned}
\mE^\vee_{  \gp_{n-1} }&= \left(\bigwedge_{1\leq i\leq n-1} \bigwedge_{1\leq j\leq i}E^{\gp_{n-1}\vee}_{i,j}\right) \wedge 
\left(\bigwedge_{2\leq j\leq n-1}\bigwedge_{1\leq i\leq j-1}E^{\gp_{n-1}\vee}_{i,j}\right)\\
 &=  E^{\gp_{n-1}\vee}_{1,1}  \wedge  
      E^{\gp_{n-1}\vee}_{2,1}  \wedge  
      E^{\gp_{n-1}\vee}_{2,2}  \wedge  \cdots \wedge
      E^{\gp_{n-1}\vee}_{n-1,1}  \wedge  
      E^{\gp_{n-1}\vee}_{n-1,2}  \wedge  \cdots \wedge
      E^{\gp_{n-1}\vee}_{n-1,n-1}  \\   
 &\qquad \qquad  \wedge  
      E^{\gp_{n-1}\vee}_{1,2}  \wedge  
            E^{\gp_{n-1}\vee}_{1,3}  \wedge  
            E^{\gp_{n-1}\vee}_{2,3}  \wedge  \cdots \wedge
            E^{\gp_{n-1}\vee}_{1,n-1}  \wedge  
            E^{\gp_{n-1}\vee}_{2,n-1}  \wedge  \cdots \wedge
            E^{\gp_{n-1}\vee}_{n-2,n-1}. 
\end{aligned}
\end{align}
Now let us write down the differential form on 
$A_{n-1}\rN^\rop_{n-1} (\mC) \cong \rGL_{n-1}(\mC)/ {\rm U}(n-1)$ 
 which is determined by $\mE^\vee_{\gp_{n-1}}$.     
Take real coordinates $a_1,a_2, \dots, a_{n-1}$, $x^\rre_{i,j},x^\rim_{i,j}$ (for $1\leq j<i\leq n-1$) of $A_{n-1}\rN^\rop_{n-1}(\mC) $ as 
\begin{align*}
a &= \diag (a_1,a_2, \dots, a_{n-1}) \in A_{n-1},&
x &= (x_{i,j})_{1\leq i,j\leq n-1}\in \rN^\rop_{n-1}(\mC) \\
& & &\qquad   \text{for } x_{i,j}=
\begin{cases}
x^\rre_{i,j} + \sqrt{-1} x^\rim_{i,j} & \text{if \;} 1\leq j<i\leq n-1 \text{\; holds,} \\
0 & \text{otherwise.} 					     
\end{cases}.
\end{align*}
Let $\ga_{n-1}$ and $\gn_{n-1}^\rop$ be the associated Lie algebras of $A_{n-1}$ and $\rN^\rop_{n-1} (\mC)$, respectively. 
Then, for $1\leq i\leq n-1$ and $1\leq k<j\leq n-1$, the identities
\begin{align} \label{eq:Liebasis}
  E_{i,i} &= \Bigl( \tfrac{\partial}{\partial a_i} \Bigr)_{1_{n-1}} \in \ga_{n-1}, &
  E_{j,k} &= \Bigl( \tfrac{\partial} {\partial x^{\rm re}_{j,k} } \Bigr)_{1_{n-1}}\in \gn^\rop_{n-1}, &  
  \sqrt{-1} E_{j,k} &= \Bigl( \tfrac{\partial}{ \partial x^{\rm im}_{j,k} } \Bigr)_{1_{n-1}} \in \gn^\rop_{n-1}
\end{align}
hold if we identify the Lie algebra $\ga_{n-1}\oplus \gn^\rop_{n-1}$ with the tangent space of $A_{n-1}\rN^\rop_{n-1}(\mC)$ at the identity matrix $1_{n-1}$. 
Since there are two decompositions of $\ggl_{n-1\mC}$
\[
\ggl_{n-1\mC}=\ga_{n-1\mC}\oplus \gn^\rop_{n-1\mC}\oplus \gu (n-1)_\mC 
=\gp_{n-1\mC}\oplus \gu (n-1)_\mC,
\]
we obtain an isomorphism $\rI^{ \gp_{n-1} }_{ \ga_{n-1} \oplus \gn^\rop_{n-1}} 
\colon  \ga_{n-1\mC}\oplus \gn^\rop_{n-1\mC} \to \gp_{n-1\mC}$ of $\mC$-vector spaces defined by the composition 
\begin{align*}
 \ga_{n-1\mC} \oplus \gn^\rop_{n-1\mC} \quad \hookrightarrow \quad \ggl_{n-1\mC} \quad \twoheadrightarrow \quad \gp_{n-1\mC},
\end{align*}
where the latter map is the first projection with respect to the decomposition $\ggl_{n-1 \mC}=\gp_{n-1\mC}\oplus \gu(n-1)_{\mC}$. 
Now recall from Section~\ref{subsec:notation} that
\begin{align*}
E_{i,i} \otimes 1
&= E^{\gp_{n-1}}_{i,i},     \\  
E_{j,k} \otimes 1  
&= \frac{1}{2} \Bigl\{\bigl(E^{\gp_{n-1}}_{j,k}+E^{\gp_{n-1}}_{k,j}\bigr)+\bigl(E^{\gu (n-1)}_{j,k}-E^{\gu (n-1)}_{k,j}\bigr)\Bigr\},  \\
(\sqrt{-1} E_{j,k} )  \otimes 1  
&= \frac{\sqrt{-1}}{2} \Bigl\{\bigl(E^{\gp_{n-1}}_{j,k}-E^{\gp_{n-1}}_{k,j}\bigr)+\bigl(E^{\gu (n-1)}_{j,k}+E^{\gu (n-1)}_{k,j}\bigr)\Bigr\} 
\end{align*}
hold for each $1\leq i \leq n-1$ and $ 1\leq k < j \leq n-1$.   
We thus obtain 
\begin{align*}
\rI^{ \gp_{n-1} }_{\ga_{n-1}\oplus \gn^\rop_{n-1} } (E_{i,i} \otimes 1) 
   &= E^{\gp_{n-1}}_{i,i},   \\
\rI^{ \gp_{n-1} }_{\ga_{n-1}\oplus \gn^\rop_{n-1} } (E_{j,k} \otimes 1) 
   &= \frac{1}{2} \bigl(E^{\gp_{n-1}}_{j,k}+E^{\gp_{n-1}}_{k,j}\bigr),   \\
\rI^{ \gp_{n-1} }_{\ga_{n-1}\oplus \gn^\rop_{n-1} } ((\sqrt{-1} E_{j,k}) \otimes 1) 
   &= \frac{\sqrt{-1}}{2} \bigl(E^{\gp_{n-1}}_{j,k}-E^{\gp_{n-1}}_{k,j}\bigr).
\end{align*}
Let $(\ga_{n-1\mC} \oplus \gn^\rop_{n-1\mC})^\vee$ denote the dual space of
 $\ga_{n-1\mC} \oplus \gn^\rop_{n-1\mC}$, and write the dual basis of 
$\{E_{i,j} \otimes 1\}_{1\leq j \leq i \leq n-1}  
   \cup \{\sqrt{-1} E_{i,j} \otimes 1\}_{1\leq j < i \leq n-1}$ 
as $\{ (E_{i,j} \otimes 1)^\vee \}_{1\leq j \leq i \leq n-1}  
   \cup \{   (\sqrt{-1} E_{i,j} \otimes 1)^\vee  \}_{1\leq j < i \leq n-1}$.        
Since (\ref{eq:Liebasis}) implies
\begin{align*}
&(E_{i,i} \otimes 1)^\vee = (\rd a_i)_{1_{n-1}},&
&(E_{j,k} \otimes 1)^\vee = (\rd x^\rre_{j,k})_{1_{n-1}},&
&(\sqrt{-1} E_{j,k} \otimes 1)^\vee = (\rd x^\rim_{j,k})_{1_{n-1}}, 
\end{align*}
the dual map $\rI^{ \gp_{n-1} \vee }_{\ga_{n-1}\oplus \gn^\rop_{n-1} }\colon 
\gp_{n-1\mC}^\vee
\to 
(\ga_{n-1\mC}\oplus \gn^\rop_{n-1\mC})^\vee\,;
\omega \mapsto \omega \circ \rI^{ \gp_{n-1} }_{\ga_{n-1}\oplus \gn^\rop_{n-1} }$ gives the identification
\begin{align*}
  \rI^{ \gp_{n-1} \vee }_{\ga_{n-1}\oplus \gn^\rop_{n-1} }
        (E^{\gp_{n-1} \vee}_{i,i})   
     &= (E_{i,i} \otimes 1)^\vee
     = (\rd a_i)_{1_{n-1}},   \\   
  \rI^{ \gp_{n-1} \vee }_{\ga_{n-1}\oplus \gn^\rop_{n-1} }
        (E^{\gp_{n-1} \vee}_{j,k})   
     &= \frac{1}{2}  
            \left\{    (E_{j,k} \otimes 1)^\vee  
                        + \sqrt{-1}  ((\sqrt{-1}E_{j,k} ) \otimes 1)^\vee      
                        \right\}
     =  \frac{1}{2} (\rd x^\rre_{j,k})_{1_{n-1}}  
         + \frac{\sqrt{-1}}{2} (\rd x^\rim_{j,k})_{1_{n-1}},   \\   
  \rI^{ \gp_{n-1} \vee }_{\ga_{n-1}\oplus \gn^\rop_{n-1} }
        (E^{\gp_{n-1} \vee}_{k,j})   
     &= \frac{1}{2}  
            \left\{    (E_{j,k} \otimes 1)^\vee  
                        - \sqrt{-1}  ((\sqrt{-1}E_{j,k} ) \otimes 1)^\vee      
                        \right\}
     =  \frac{1}{2} (\rd x^\rre_{j,k})_{1_{n-1}}  
         - \frac{\sqrt{-1}}{2} (\rd x^\rim_{j,k})_{1_{n-1}}
\end{align*}
for each $1\leq i \leq n-1$ and $1\leq k < j \leq n-1$.   
Let us use the same symbol $\rI^{ \gp_{n-1} \vee }_{\ga_{n-1}\oplus \gn^\rop_{n-1} }$ for the homomorphism  ${\bigwedge}^{(n-1)^2}\gp^\vee_{n-1\mC}  
\to  {\bigwedge}^{(n-1)^2} (\ga_{n-1\mC}\oplus \gn^\rop_{n-1\mC})^\vee$ induced from 
$\rI^{ \gp_{n-1} \vee }_{\ga_{n-1}\oplus \gn^\rop_{n-1} }$.  Then we obtain 
\begin{align*}
   \rI^{ \gp_{n-1} \vee }_{\ga_{n-1}\oplus \gn^\rop_{n-1} } 
       ( \mE^\vee_{\gp_{n-1}} )   
    &=  (2\sqrt{-1})^{-\mathrm{b}_{n-1}} \left( \bigwedge_{1\leq i\leq n-1} \left(\bigwedge_{1\leq j\leq i-1} ({\rm d}x^{\mathrm{re}}_{i,j})_{1_{n-1}}\right) \wedge (\rd a_i)_{1_{n-1}}\right) \wedge \left( \bigwedge_{2\leq i\leq n-1} \bigwedge_{1\leq j\leq i-1} ({\rm d}x^{\mathrm{im}}_{i,j})_{1_{n-1}} \right)  \\
 &=(2\sqrt{-1})^{-\mathrm{b}_{n-1}}
(\rd a_1)_{1_{n-1}} \wedge (\rd x^\rre_{2,1})_{1_{n-1}}\wedge (\rd a_2  )_{1_{n-1}}\wedge  \cdots \\
&\hspace*{10em} \cdots 
\wedge (\rd x^\rre_{n-1,1})_{1_{n-1}}\wedge (\rd x^\rre_{n-1,2})_{1_{n-1}}\wedge \cdots \wedge (\rd x^\rre_{n-1,n-2})_{1_{n-1}}
\wedge (\rd a_{n-1})_{1_{n-1}}   \\
& \qquad \wedge (\rd x^\rim_{2,1})_{1_{n-1}}\wedge (\rd x^\rim_{3,1})_{1_{n-1}}\wedge (\rd x^\rim_{3,2})_{1_{n-1}}   
\wedge \cdots \\
&\hspace*{15em} 
\cdots \wedge (\rd x^\rim_{n-1,1})_{1_{n-1}}\wedge (\rd x^\rim_{n-1,2})_{1_{n-1}} \wedge \cdots \wedge (\rd x^\rim_{n-1,n-2})_{1_{n-1}}. 
\end{align*}
Let  $L_{ (ax)^{-1}  }\colon A_{n-1}\rN^\rop_{n-1}(\mC) \to A_{n-1}\rN^\rop_{n-1}(\mC)\, ; g \mapsto (ax)^{-1} g$ be the left translation by $(ax)^{-1}$,  
and write the dual of its derivative as $L^\vee_{(ax)^{-1}}$.  
Then we obtain 
\begin{align}\label{eq:MeasEp}
\begin{aligned}
\rd_{ \gp_{n-1}  } (ax) 
 &:= L^\vee_{ (ax)^{-1}  } (  \rI^{ \gp_{n-1} \vee }_{\ga_{n-1}\oplus \gn^\rop_{n-1}}  
        (\mE^\vee_{\gp_{n-1}}) 
           ) \\
& =  (2\sqrt{-1})^{-\mathrm{b}_{n-1}}  \left( \bigwedge_{1\leq i\leq n-1} \left(\bigwedge_{1\leq j\leq i-1} {\rm d}x^{\mathrm{re}}_{i,j}\right) 
\wedge \frac{\rd a_i}{a_i}\right) \wedge \left( \bigwedge_{2\leq i\leq n-1} \bigwedge_{1\leq j\leq i-1} {\rm d}x^{\mathrm{im}}_{i,j} \right) \\
 &=         (2\sqrt{-1})^{-\mathrm{b}_{n-1}} 
      \frac{\rd a_1}{a_1}  
          \wedge \rd x^\rre_{2,1}
          \wedge \frac{\rd a_2}{a_2}  
          \wedge  \cdots 
          \wedge \rd x^\rre_{n-1,1}
          \wedge \rd x^\rre_{n-1,2}
          \wedge \cdots 
          \wedge \rd x^\rre_{n-1,n-2}
          \wedge \frac{\rd a_{n-1}}{a_{n-1}}\\
    & \hspace*{14em} \wedge \rd x^\rim_{2,1}  
        \wedge \rd x^\rim_{3,1}   
        \wedge \rd x^\rim_{3,2}   
        \wedge \cdots 
        \wedge \rd x^\rim_{n-1,1}   
        \wedge \rd x^\rim_{n-1,2} 
        \wedge \cdots 
        \wedge \rd x^\rim_{n-1,n-2}.
\end{aligned}
\end{align}

\begin{rem}\label{rem:HaarEp}
Let $\rd h$ and $\rd k$ respectively denote the Haar measures on $\rGL_{n-1}(\mC)$ and $\rU (n-1)$, which are 
normalised as (\ref{eq:GL_measure}) and (\ref{eq:measure_KNopA}). 
Then we have 
\begin{align*}
2^{-n(n-1)}(\sqrt{-1})^{-\rb_{n-1}}\int_{\rGL_{n-1}(\mC)} f(h) \rd h
= \int_{A_{n-1}\rN^\rop_{n-1}(\mC)} \left(   \int_{\rU (n-1)}   f(axk)  {\rm d}k   \right) \rd_{ \gp_{n-1}  } (ax)
\end{align*}
for a measurable function $f$ on $\rGL_{n-1}(\mC)$. The factor $2^{-n(n-1)}(\sqrt{-1})^{-\rb_{n-1}}$ is used for 
the definition of $\widetilde{\cI}^{(m)}(\cdot ,\cdot )$ in \S \ref{subsec:explicit_arch_zeta}.  
\end{rem}

\subsection{Adjoint representations} \label{subsec:adjoint}

Consider $\gp_{n\mC}$ as a ${\rm U}(n)$-module via the adjoint action $\rAd$.     
Write the contragredient representation of $(\rAd , \gp_{n\mC})$ as $(\rAd^\vee, \gp^\vee_{n\mC})$, and     
define $\langle \cdot , \cdot \rangle_{ \gp_n } \colon \mathfrak{p}^\vee_{n\mC} \otimes_{\mC} \gp_{n\mC} \to {\mathbf C}$
to be the canonical pairing. The derivatives of $\rAd$ and $\rAd^\vee$ are respectively denoted  by  $\rad$ and $\rad^\vee$. 
Let $w_n$ be the anti-diagonal matrix of size $n$ with $1$ at 
all anti-diagonal entries, that is, 
\begin{align*}
w_n = \begin{pmatrix} 0 & &  1 \\ &  \hspace{-5pt}\rotatebox{80}{$\ddots$} &  \\ 1 && 0 \end{pmatrix}\in \rU (n). 
\end{align*}
For later use, we prepare two lemmas. The first one describes the adjoint action of $E^{\gu(n)}_{a,b}$ and $w_n$ on $\gp_{n\mC}^\vee$.

\begin{lem}\label{lem:adact}
For $1\leq a,b,i,j\leq n$, we have the following identities$:$ 
\begin{enumerate}
\item \label{num:adact_alg}
$\rad^\vee(E^{\gu (n)}_{a, b})  E^{\gp_n\vee}_{i,j} 
    = \delta_{j, b} E^{\gp_n\vee}_{i,a} 
         - \delta_{i, a} E^{\gp_n\vee}_{b,j}$.   
\item \label{num:adact_wn}
$\rAd^\vee(w_n)  E^{\gp_n \vee}_{i,j} = E^{\gp_n \vee}_{j,i}$.
\end{enumerate}
Here $\delta_{*,*}$ denotes Kronecker's delta symbol.
\end{lem}

\begin{proof} 
By direct calculation, one readily obtains ${\rm ad}(E^{\gu (n)}_{a, b})  E^{\gp_n}_{i,j} 
    = \delta_{b, i} E^{\gp_n}_{a,j} 
         - \delta_{a, j} E^{\gp_n}_{i, b}$. Therefore, we have 
\begin{align*}
     \langle  \rad^\vee (E^{\gu (n)}_{a, b}) E^{\gp_n\vee}_{i, j}    ,   E^{\gp_n}_{k,l}     \rangle_{\gp_n} =  - \langle   E^{\gp_n\vee}_{i, j}    ,  {\rm ad} (E^{\gu (n)}_{a, b}) E^{\gp_n}_{k,l}     \rangle_{\gp_n}   &=  - \langle   E^{\gp_n\vee}_{i, j}    ,  \delta_{b, k} E^{\gp_n}_{a,l} 
                                                                              - \delta_{a, l} E^{\gp_n}_{k, b}     \rangle_{\gp_n}   \\
&= -\delta_{b,k} \delta_{i,a} \delta_{j,l} 
        + \delta_{a,l} \delta_{i,k} \delta_{j,b}
\end{align*}
for each $1\leq k,l\leq n$, which implies the identity \ref{num:adact_alg}. The identity \ref{num:adact_wn} immediately follows from the definition. 
\end{proof}

The second one is used in the construction of $\rI^{\gp_n}_{2\rho_n}$ in Section~\ref{subsec:U(n)_inv}. 

\begin{lem}\label{lem:htwtE}
As the finite dimensional complex representation $\bigl(\mathrm{Ad}^\vee, {\bigwedge}^{\rb_n} \gp^{0\vee}_{n\mC}\bigr)$ of $\rU(n)$, there exists a highest weight vector of weight $2\rho_n$ in $ {\bigwedge}^{\rb_n} \gp^{0\vee}_{n\mC}$, which is described as
\begin{align*}
\bigwedge_{2\leq i\leq n}\bigwedge_{1\leq j\leq i-1}E^{\gp_n\vee}_{i,j}= 
       E^{\gp_n\vee}_{2,1} \wedge 
       E^{\gp_n\vee}_{3,1} \wedge 
       E^{\gp_n\vee}_{3,2} \wedge \cdots \wedge
       E^{\gp_n\vee}_{n,1} \wedge 
       E^{\gp_n\vee}_{n,2} \wedge \cdots \wedge
       E^{\gp_n\vee}_{n,n-1}.
\end{align*}
\end{lem}
\begin{proof}
The assertion follows immediately from Lemma~\ref{lem:adact} \ref{num:adact_alg}. 
\end{proof}

\section{Highest weight representations of $\rGL_n$}
\label{sec:HWrepGLn}

In this section, we give the proofs of several results 
for finite dimensional representations of the general linear groups, 
which are introduced in Section \ref{subsec:finite_dimensional}.

\subsection{Preliminaries}
\label{subsec:prelim_fdrep}

In this subsection, we consider the case where $\cA =\mC$, 
and introduce some preliminary results for finite dimensional representations of $\rGL_n(\mC)$. 
We also give the  proof of Lemma~\ref{lem:xi_Hmulambda_explicit} in this subsection. 

As in Section $\ref{subsec:GT_basis}$, for $\lambda \in \Lambda_n$, 
let $\cE_\lambda :=\{\sigma \lambda \mid \sigma \in \gS_n\}$ be 
the set of extremal weights of $\tau_\lambda$, 
and $H(\gamma )$ the unique element of $\rG (\lambda)$ whose weight is  
$\gamma \in \cE_\lambda $. 
We first verify some properties of the rational constant $\rr (M)$ defined in (\ref{eq:def_rM}).

\begin{lem}
\label{lem:rM_property}
Let $\lambda \in \Lambda_n$. 
\begin{enumerate}
\item \label{num:Hgamma_rM_property}
For $\gamma \in \cE_\lambda $, 
we have $\rr (H(\gamma ))=1$.

\item \label{num:rest_rM_property}
Assume that $n>1$, and take $\mu \in \Lambda_{n-1}$ satisfying 
$\mu \preceq \lambda$. 
Then 
we have $\rr (M[\lambda ])=\rr (H(\mu )[\lambda ])\rr (M)$ for every $M\in \rG (\mu )$. 
\end{enumerate}
\end{lem}
\begin{proof}
The statement \ref{num:rest_rM_property} follows immediately from 
the definition (\ref{eq:def_rM}). 
Let us prove the statement \ref{num:Hgamma_rM_property}. 
Write $\gamma =(\gamma_1,\gamma_2,\dots ,\gamma_n)\in \cE_\lambda $. 
Then, by the definition of $H(\gamma)=(h_{i,j})_{1\leq i\leq j\leq n}$, each row  
$h^{(j)}:=(h_{1,j},h_{2,j},\dots ,h_{j,j})$ is a unique 
element of $\Lambda_j\cap 
\{\sigma (\gamma_1,\gamma_2,\dots ,\gamma_j)\mid \sigma \in \gS_j\}$ for any $1\leq j\leq n$. 
Hence, for $2\leq k \leq n$, we have 
\begin{align*}
h_{i,k-1}&=h_{i,k}\qquad  (1\leq i\leq i_k-1) & \text{ and} &&
h_{j,k-1} &= h_{j+1,k}\qquad (i_k\leq j\leq k-1)
\end{align*}
if we take $1\leq i_k \leq k$ so that $h_{i_k,k}=\gamma_k$ holds. These equalities imply 
\begin{align*}
&\frac{(h_{i,k}-h_{j,k-1}-i+j)!(h_{i,k-1}-h_{j+1,k}-i+j)!}
{(h_{i,k-1}-h_{j,k-1}-i+j)!(h_{i,k}-h_{j+1,k}-i+j)!}=1&
&\text{for each \;} 1\leq i\leq j\leq k-1. 
\end{align*}
Hence, we conclude $\rr (H(\gamma)) = 1$ from 
the definition (\ref{eq:def_rM}). 
\end{proof}

The following lemma is 
Lemma~\ref{lem:restriction_lambda_mu} for the case where $\cA= \mC$. 

\begin{lem}
\label{lem:C_restriction_lambda_mu}
Assume that $n>1$, and  
take $\lambda \in \Lambda_n$ and $\mu \in \Lambda_{n-1}$ so that
$\mu \preceq \lambda$ holds. 
Recall the $\mC$-linear maps 
$\rI_\mu^\lambda \colon V_\mu \to 
V_\lambda $ and 
$\rR_\mu^\lambda \colon V_\lambda \to V_\mu $ 
defined as $(\ref{eq:def_inj_restriction})$ and 
$(\ref{eq:def_proj_restriction})$, respectively. 
Then  both $\rI_\mu^\lambda $ and $\rR^\lambda_\mu$ are 
$\rGL_{n-1}(\mC )$-equivariant and satisfy 
$\rR^\lambda_\mu \circ \rI_\mu^\lambda =\id_{V_{\mu}}$. 
\end{lem}
\begin{proof}
By Lemma~\ref{lem:rM_property} \ref{num:rest_rM_property}, we have 
$\rI_\mu^\lambda (\zeta_M)= \rr (H(\mu )[\lambda ])^{\frac{1}{2}}
\zeta_{M[\lambda ]}$ 
 for $M\in \rG (\mu )$, and 
\begin{align*}
&\rR^{\lambda}_\mu (\zeta_M)=\left\{\begin{array}{ll}
\rr (H(\mu )[\lambda ])^{-\frac{1}{2}}
\zeta_{\widehat{M}}&\text{if}\ \widehat{M}\in \rG (\mu ),\\
0&\text{otherwise}
\end{array}\right.\qquad \text{for \;} M\in \rG (\lambda).
\end{align*}
Hence, the assertion follows from the formulas 
(\ref{eq:GT_act_wt}), (\ref{eq:GT_act+}) and (\ref{eq:GT_act-}). 
\end{proof}

Let $\lambda =(\lambda_1,\lambda_2,\dots ,\lambda_n)\in \Lambda_n$. 
For $l=(l_I)_{I\in \cI_n}\in \cL (\lambda )$, we define the weight 
$\gamma^{(l)}=(\gamma^{(l)}_1,\gamma^{(l)}_2,\dots ,
\gamma^{(l)}_n)\in \mZ^n$ of $l$ by 
\begin{align}
\label{eq:def_gammal_Llambda}
&\gamma_i^{(l)}:=\sum_{i\in J\in \cI_n}l_J\qquad (1\leq i\leq n), 
\end{align}
where $\cL (\lambda )$ is the set defined in Section \ref{subsec:realisation}. 
The following lemma describes specific $\ggl_n$-actions on 
the generators $f_l(z)$ ($l\in \cL (\lambda)$) of $V_\lambda$ defined as in (\ref{eq:Vlambda_generator}).

\begin{lem}
Retain the notation. 
For $1\leq i \leq n$, $1\leq j \leq n-1$ and 
$l =(l_I)_{I\in \cI_n}\in \cL (\lambda )$, we have 
\begin{align}
\label{eq:actwt_gen_Vlambda}
\tau_{\lambda}(E_{i,i})f_{l}(z)
&= \gamma_i^{(l)}f_{l}(z),\\
\label{eq:act+_gen_Vlambda}
\tau_{\lambda}(E_{j,j+1})f_{l}(z)
&= \sum_{j+1\in J\in \cI_n,\ j\not\in J}
l_Jf_{l -\be (J)+ 
\be ((J\smallsetminus \{j+1\})\cup \{j\})}(z),\\ 
\label{eq:act-_gen_Vlambda}
\tau_{\lambda}(E_{j+1,j})f_{l}(z)
&= \sum_{j\in J\in \cI_n,\ j+1\not\in J}
l_Jf_{l -\be (J)+ 
\be ((J\smallsetminus \{j\})\cup \{j+1\})}(z), 
\end{align}
where
$\be (J)=(e(J)_I)_{I\in \cI_n}$ is defined as $e(J)_J=1$ and 
$e(J)_I=0$ for $J \in \cI_n\smallsetminus \{I\}$. Here we understand that  $f_{l'}(z)=0$ if $l'$ is not contained in $\cL (\lambda )$. 
\end{lem}
\begin{proof}
By Cauchy--Binet's formula (\ref{eq:cauchy_binet}), 
for $1\leq i \leq n$, $1\leq j \leq n-1$ and $I \in \cI_{n}$, 
we have 
\begin{align*}
\left.\frac{d}{dt}\right|_{t=0}{\det}_{I}(z\exp (tE_{i,i}))
&=\left\{\begin{array}{ll}
{\det}_{I}(z)&\text{if $i\in I$},\\
0&\text{otherwise},
\end{array}\right.\\
\left.\frac{d}{dt}\right|_{t=0}{\det}_{I}(z\exp (tE_{j,j+1}))
&=\left\{\begin{array}{ll}
{\det}_{(I\smallsetminus \{j+1\})\cup \{j\}}(z)&
\text{if $j+1\in I$ and $j\not\in I$},\\
0&\text{otherwise},
\end{array}\right.\\
\left.\frac{d}{dt}\right|_{t=0}{\det}_{I}(z\exp (tE_{j+1,j}))
&=\left\{\begin{array}{ll}
{\det}_{(I\smallsetminus \{j\})\cup \{j+1\}}(z)&
\text{if $j\in I$ and $j+1\not\in I$},\\
0&\text{otherwise}.
\end{array}\right.
\end{align*}
By these equalities, we obtain (\ref{eq:actwt_gen_Vlambda}), 
(\ref{eq:act+_gen_Vlambda}) and (\ref{eq:act-_gen_Vlambda}). 
\end{proof}

For $\gamma =(\gamma_1,\gamma_2,\dots ,\gamma_n)\in \mZ^n$, 
we set $\ell (\gamma ):=\gamma_1+\gamma_2+\dots +\gamma_n$. 
Let us define
\begin{align}
\label{eqn:def_S0}
&\rS^\circ (\lambda' ,\lambda )=
\frac{\prod_{1\leq i\leq j\leq n}(\lambda_{i}'-\lambda_{j}-i+j)!}
{\prod_{1\leq i\leq j<n}(\lambda_{i}-\lambda_{j+1}'-i+j)!}
\end{align}
for $\lambda'=(\lambda_1',\lambda_2',\dots ,\lambda_n')
\in \Lambda_n$ such that $\lambda_1'\geq \lambda_1\geq \lambda_2'\geq \lambda_2
\geq \dots \geq \lambda_n'\geq \lambda_n$. 
When $n>1$, we also define
\begin{align}
\label{eqn:def_S+}
&\rS^+(\lambda ,\mu )=\prod_{1\leq i\leq j<n}
\frac{(\lambda_{i}-\mu_{j}-i+j)!}{(\mu_{i}-\lambda_{j+1}-i+j)!}
\end{align}
for $\mu =(\mu_1,\mu_2,\dots ,\mu_{n-1})\in \Lambda_{n-1}$ 
satisfying $\mu \preceq \lambda $. The following proposition describes the action of a power of $E_{j,j+1}$ on the Gel'fand--Tsetlin basis.

\begin{prop}
\label{prop:VK_formula_E+k}
Retain the notation. Then, for $M=({m}_{i,j})_{1\leq i\leq j\leq n}\in \rG (\lambda )$, 
$1\leq j\leq n-1$ and $k\geq 0$, we have 
\begin{align*}
&\tau_\lambda ((E_{j,j+1})^k)\zeta_{M}=
\underset{M+\Delta_j (\beta )\in \rG (\lambda )}
{\sum_{\beta \in \mN_0^j,\ \ell (\beta )=k}}
\ra_{j}(\beta ;M)\zeta_{M+\Delta_j (\beta )},
\end{align*}
where 
\begin{align*}
&\ra_{j}(\beta ;M):=\ell (\beta )!
\sqrt{\frac{\rS^\circ (m^{(j)}+\beta ,m^{(j)}+\beta )
\rS^\circ (m^{(j)},m^{(j)})}{\rS^\circ (m^{(j)}+\beta ,m^{(j)})^2}
\frac{\rS^+(m^{(j+1)},m^{(j)})\rS^+(m^{(j)}+\beta ,m^{(j-1)})}
{\rS^+(m^{(j+1)},m^{(j)}+\beta )\rS^+(m^{(j)},m^{(j-1)})}}
\end{align*}
and $\Delta_j (\beta ):=\sum_{i=1}^j\beta_i\Delta_{i,j}$ 
with $\beta =(\beta_1,\beta_2,\dots ,\beta_j)\in \mN_0^j$ and 
$m^{(j)}:=(m_{1,j},m_{2,j},\dots ,m_{j,j})$. 
Here we understand  that 
$\rS^+(m^{(j)}+\beta ,m^{(j-1)})=\rS^+(m^{(j)},m^{(j-1)})=1$
if $j=1$. 
\end{prop}

\begin{proof}
 See \cite[Section 18.2.5]{Vilenkin_Klimyk_001}.
\end{proof}

Assume that $n>1$, and take $\mu =(\mu_1,\mu_2,\dots ,\mu_{n-1})\in \Lambda_{n-1}$ so that 
$\mu \preceq \lambda$ holds. For $1\leq m\leq n$, 
we define $H(\lambda,\mu;m)\in \rG (\lambda )$ by 
\begin{align*}
&H(\lambda,\mu;m):=(h_{i,j})_{1\leq i\leq j\leq n}\qquad 
\text{with }\ 
h_{i,j}=\left\{\begin{array}{ll}
\lambda_i&\text{if $j\geq m$},\\
\mu_i&\text{if $j<m$}.
\end{array}\right.
\end{align*}
Here we note that both the equalities $H(\lambda,\mu;1)=H(\lambda )$ and 
$H(\lambda,\mu;n)=H(\mu )[\lambda]$ hold by definition. 

\begin{lem}
\label{lem:act_xiM_E+l}
Retain the notation. For $1\leq m\leq n-1$, we have 
\[
\tau_{\lambda}((E_{m,m+1})^{l_m})\xi_{H(\lambda,\mu;m+1)} 
=l_m!\xi_{H(\lambda,\mu;m)}
\]
with $l_m:=\sum_{i=1}^{m}(\lambda_i-\mu_i)$. 
Hence, we have 
\[
\tau_{\lambda}((E_{1,2})^{l_{1}}(E_{2,3})^{l_{2}}\dots 
(E_{n-1,n})^{l_{n-1}})\xi_{H(\mu )[\lambda]}
=l_1!\,l_2!\,\dots \,l_{n-1}!\,\xi_{H(\lambda )}.
\]
\end{lem}
\begin{proof}
By Proposition~\ref{prop:VK_formula_E+k}, we have 
$\tau_{\lambda}((E_{m,m+1})^{l_m})\zeta_{H(\lambda,\mu;m+1)} 
=\ra_{m}(\beta_m;H(\lambda,\mu;m+1))
\zeta_{H(\lambda,\mu;m)}$ with 
\[
\beta_m:=(\lambda_1-\mu_1,\lambda_2-\mu_2,\dots ,\lambda_{m}-\mu_{m}).
\]
For $H(\lambda,\mu;m+1)=(h_{i,j})_{1\leq i\leq j\leq n}$, we set 
\begin{align*}
h^{(j)}:=(h_{1,j},h_{2,j},\dots ,h_{j,j})
&=\left\{\begin{array}{ll}
(\lambda_1,\lambda_2,\dots ,\lambda_j)&\text{if $j\geq m+1$},\\
(\mu_1,\mu_2,\dots ,\mu_j)&\text{if $j<m+1$}\\
\end{array}\right.&
&\text{for \;} 1\leq j\leq n.
\end{align*}
Since $h^{(m)}+\beta_m=(\lambda_1,\lambda_2,\dots ,\lambda_m)$, we have the following equalities:
\begin{align*}
\frac{\rS^\circ(h^{(m)},h^{(m)})}{\rS^+(h^{(m)},h^{(m-1)}) }
&=  \prod_{1\leq i \leq m}(\mu_i-\mu_m-i+m)!,\\
\frac{\rS^\circ(h^{(m)}+\beta_m,h^{(m)}+\beta_m)}
{\rS^+(h^{(m+1)},h^{(m)}+\beta_m)}
&=\prod_{1\leq i \leq m}(\lambda_i-\lambda_{m+1}-i+m)!,\\
\frac{\rS^+(h^{(m+1)},h^{(m)})}
{\rS^\circ(h^{(m)}+\beta_m,h^{(m)}) }
&=  \prod_{1\leq i \leq m}
\frac{1}{(\mu_i-\lambda_{m+1}-i+m)!},  \\
\frac{\rS^+(h^{(m)}+\beta_m,h^{(m-1)})}
{\rS^\circ(h^{(m)}+\beta_m,h^{(m)})}
&=  \prod_{1\leq i \leq m}\frac{1}{(\lambda_i-\mu_m-i+m)!}.
\end{align*}
These equalities implies 
\begin{align*}
\ra_{m}(\beta_m;H(\lambda,\mu;m+1))
&=l_m!\prod_{1\leq i\leq m}
\sqrt{
\frac{(\lambda_{i}-\lambda_{m+1}-i+m)!(\mu_{i}-\mu_m-i+m)!}
{(\mu_{i}-\lambda_{m+1}-i+m)!(\lambda_{i}-\mu_m-i+m)!}}.
\end{align*}
Since
\begin{align*}
&\rr (H(\lambda,\mu;m')) 
=\prod_{1\leq i\leq  j\leq m'-1}
\frac{(\mu_{i}-\lambda_{j+1}-i+j)!(\lambda_i-\mu_j-i+j)!}
{(\lambda_{i}-\lambda_{j+1}-i+j)!(\mu_i-\mu_j-i+j)!}&
&\text{for\; }1\leq m'\leq n, 
\end{align*}
we have 
\[
\ra_{m}(\beta_m;H(\lambda,\mu;m+1))
=l_m!\sqrt{\dfrac{\rr (H(\lambda,\mu;m))}
{\rr (H(\lambda,\mu;m+1))}}.
\]
Since $\xi_M$ is defined as $\xi_M=\sqrt{\rr (M)}\zeta_M$ for $M\in \rG (\lambda )$, 
we obtain the assertion. 
\end{proof}

The following lemma will be used in the proof of Lemma~\ref{lem:Clem2} in 
Section \ref{sec:calcoeff}.

\begin{lem}
\label{lem:EactPair}
Retain the notation in Lemma~$\ref{lem:act_xiM_E+l}$. Then we have
\begin{align*} 
\bigl(\tau_{\lambda}((E_{n, n-1})^{l_{n-1}}(E_{n-1,n-2})^{l_{n-2}}\dots 
(E_{2, 1})^{l_1})\xi_{ H(\lambda )},\
\xi_{H(\mu )[\lambda ]}\bigr)_{\lambda}  
=l_{n-1}!\,l_{n-2}!\,\dots \,l_1!.
\end{align*}
\end{lem}
\begin{proof}
The $\rU (n)$-invariance of $( \cdot ,\cdot )_{\lambda}$ implies 
$(\tau_{\lambda}(E_{i,j})\rv,\rv')_{\lambda}
=(\rv,\tau_{\lambda}(E_{j,i})\rv')_{\lambda}$ for 
$1\leq i,j\leq n$ and $\rv,\rv'\in V_{\lambda}$. 
Hence, 
\begin{align*} 
\left( \tau_{\lambda}((E_{n, n-1})^{l_{n-1}}(E_{n-1,n-2})^{l_{n-2}}\dots 
        (E_{2, 1})^{l_1})\xi_{ H(\lambda )},
           \xi_{H(\mu )[\lambda ]}
   \right)_{\lambda}
=\left( \xi_{ H(\lambda )},\tau_{\lambda}((E_{1,2})^{l_{1}}(E_{2,3})^{l_{2}}\dots 
        (E_{n-1,n})^{l_{n-1}})\xi_{H(\mu )[\lambda ]}
    \right)_{\lambda} ,
\end{align*}
and the assertion follows from Lemma~\ref{lem:act_xiM_E+l} 
and $(\xi_{H(\lambda )},\xi_{H(\lambda )})_\lambda =\rr (H(\lambda ))=1$ 
(by Lemma~\ref{lem:rM_property} \ref{num:Hgamma_rM_property}). 
\end{proof}

For $1\leq m\leq n$, 
we define $h (\lambda,\mu;m)
=(h(\lambda,\mu;m)_{I})_{I\in \cI_n}\in \cL (\lambda )$ by 
\begin{align*}
&h(\lambda,\mu;m)_I:=\left\{\begin{array}{ll}
\lambda_k-\mu_k&
\text{if $k<m$ and $I=\{1,2,\dots ,k-1,m\}$},\\
\mu_k-\lambda_{k+1}&
\text{if $k<m$ and $I=\{1,2,\dots ,k\}$},\\
\lambda_k-\lambda_{k+1}&
\text{if $k\geq m$ and $I=\{1,2,\dots ,k\}$},\\
0&\text{otherwise}
\end{array}\right.&
&(I\in \cI_{n,k},\ 1\leq k\leq n-1).
\end{align*}
Here we understand $\{1,2, \ldots, k-1, m\} = \{ m \}$ if $k=1$. 
We note that $h (\lambda,\mu;1)=h (\lambda )$ 
and $h (\lambda,\mu;n)=h (\lambda ,\mu )$, where $h (\lambda )$ and $h (\lambda ,\mu )$ 
are defined by (\ref{eq:def_hlambda_gen}) and (\ref{eq:def_hlambdamu_gen}), respectively.

\begin{lem}
\label{lem:xi_Hmulambda_gen}
Retain the notation. We have 
$\xi_{H(\lambda,\mu;m)}=f_{h (\lambda,\mu;m)}(z)$ 
for $1\leq m\leq n$. 
\end{lem}
\begin{proof}
By (\ref{eq:actwt_gen_Vlambda}) and 
(\ref{eq:act+_gen_Vlambda}), we have 
\begin{align*}
\tau_{\lambda}(E_{i,i})f_{h (\lambda,\mu;n)}(z)
&=\mu_if_{h (\lambda,\mu;n)}(z)\qquad (1\leq i\leq n-1) & \text{ \; and \;} &
&\tau_{\lambda}(E_{j,j+1})f_{h (\lambda,\mu;n)}(z)&=0\qquad 
(1\leq j\leq n-2).
\end{align*}
Note that $f_{h (\lambda,\mu;n)}(z)$ is a highest weight vector in 
an irreducible $\rGL_{n-1}(\mC)$-submodule of $V_\lambda$ with highest weight $\mu$. 
By Lemma~\ref{lem:C_restriction_lambda_mu}, 
we have $V_\lambda \simeq \bigoplus_{\mu \preceq \lambda }V_{\mu}$ 
as $\rGL_{n-1}(\mC)$-modules, 
and there exists a constant $c\in \mC^\times $ such that 
$\xi_{H(\lambda,\mu;n)}=\xi_{H(\mu)[\lambda ]}
=cf_{h (\lambda,\mu;n)}(z)$. 
For $1\leq m\leq n-1$, set $l_m=\sum_{i=1}^{m}(\lambda_i-\mu_i)$. 
By a direct computation using (\ref{eq:act+_gen_Vlambda}), 
we have 
\begin{align*}
&\tau_{\lambda}((E_{m,m+1})^{l_m})f_{h (\lambda,\mu;m+1)}(z)
=l_m!f_{h (\lambda,\mu;m)}(z)&
\end{align*}
for $1\leq m\leq n-1$. 
By this equality, 
Lemma~\ref{lem:act_xiM_E+l} and the equality
$\xi_{H(\lambda,\mu;n)}=cf_{h (\lambda,\mu;n)}(z)$, we have 
\begin{align*}
&\xi_{H(\lambda,\mu;m)}=cf_{h (\lambda,\mu;m)}(z)&
&\text{for \;} 1\leq m\leq n. 
\end{align*}
This equality for $m=1$ implies 
$\xi_{H(\lambda )}=cf_{h (\lambda )}(z)$, 
since $H(\lambda )=H(\lambda,\mu;1)$ and 
$h (\lambda )=h (\lambda,\mu;1)$ hold. 
On the other hand, Lemma~\ref{lem:rM_property} \ref{num:Hgamma_rM_property} and our 
normalisation $\zeta_{H(\lambda )}=f_{h (\lambda )}(z)$ imply 
$\xi_{H(\lambda )}=\sqrt{\rr (H(\lambda ))}\zeta_{H(\lambda )}=f_{h (\lambda )}(z)$. 
Therefore, we have $c=1$ and obtain the assertion. 
\end{proof}

\begin{proof}[Proof of Lemma~$\ref{lem:xi_Hmulambda_explicit}$]
Since both $H(\lambda,\mu;n)=H(\mu )[\lambda]$ and $h (\lambda,\mu;n)=h (\lambda ,\mu )$ hold, the assertion follows immediately from 
Lemma~\ref{lem:xi_Hmulambda_gen}. 
\end{proof}

The following corollary will be used in the proof of Lemma~\ref{lem:std_schwartz} in Section \ref{subsec:def_schwartz}. 

\begin{cor}
\label{cor:matcoeff_n=2}
Assume that $n=2$. Let $\lambda =(\lambda_1,\lambda_2)\in \Lambda_2$, 
$M=\left(\begin{smallmatrix}\lambda_1\ \lambda_2\\ \mu_1\end{smallmatrix}\right)\in \rG (\lambda )$ and 
$g=(g_{i,j})_{1\leq i,j\leq 2}\in \rGL_2(\mC )$. 
Then we have 
\begin{equation}
\label{eq:xiM_n=2}
\xi_{M}=f_{h (\lambda,\mu_1)}(z)
=z_{1,1}^{\mu_1-\lambda_{2}}z_{1,2}^{\lambda_1-\mu_1}(\det z)^{\lambda_2}
\end{equation}
and 
\begin{align}
\label{eq:matcoeff_n=2}
(\tau_{\lambda}(g)\xi_{M},\xi_{H(\lambda )})_{\lambda}
&=(\det g)^{\lambda_2}g_{1,1}^{\mu_1-\lambda_{2}}g_{1,2}^{\lambda_1-\mu_1}.
\end{align}
\end{cor}

\begin{proof}
By the equality $M=H(\mu_1)[\lambda ]$ and Lemma~\ref{lem:xi_Hmulambda_explicit}, we obtain (\ref{eq:xiM_n=2}). 
Next, by direct calculations, we have 
\begin{align*}
\tau_{\lambda}(g)\xi_{M}
&=(z_{1,1}g_{1,1}+z_{1,2}g_{2,1})^{\mu_1-\lambda_{2}}
(z_{1,1}g_{1,2}+z_{1,2}g_{2,2})^{\lambda_1-\mu_1}(\det (zg))^{\lambda_2}\\
&=(\det g)^{\lambda_2}\sum_{i=0}^{\mu_1-\lambda_{2}}\sum_{j=0}^{\lambda_1-\mu_1}
\binom{\mu_1-\lambda_{2}}{i}\binom{\lambda_1-\mu_1}{j}
g_{1,1}^ig_{2,1}^{\mu_1-\lambda_{2}-i}g_{1,2}^jg_{2,2}^{\lambda_1-\mu_1-j}
\xi_{H(i+j+\lambda_2)[\lambda ]}.
\end{align*}
By this equality and 
\begin{align*}
&(\xi_{H(\mu_1')[\lambda ]},\xi_{H(\lambda )})_{\lambda}
=\left\{\begin{array}{ll}
1&\text{if $\mu_1'=\lambda_1$},\\
0&\text{if $\mu_1'\neq \lambda_1$}
\end{array}\right.&
&\text{for\; } \lambda_1\geq \mu_1'\geq \lambda_2,
\end{align*}
we obtain (\ref{eq:matcoeff_n=2}). 
\end{proof}

\subsection{The $\cA$-module structure}
\label{subsec:Amod_str}

In this subsection, we give the proofs of 
Lemmas \ref{lem:extremal_vec_explicit}, \ref{lem:restriction_lambda_mu}, 
\ref{lem:detl_shift}, \ref{lem:GL_inv_pairing} and Proposition~\ref{prop:xiM_integral}.

Let us recall some results on the construction of the Gel'fand--Tsetlin basis 
from \cite[Section 2]{Molev_001}. 
For $1\leq i<j\leq n$, 
we define the {\em raising operator} $D_{i,j}$ and the {\em lowering operator} $D_{j,i}$ 
in the universal enveloping algebra $\cU (\ggl_n)$ of $\ggl_n$ by 
\begin{align*}
&D_{i,j}:=(-1)^{i-1}
\sum_{h=0}^{i-1}\,\sum_{i=i_0>i_1>i_2>\dots >i_h\geq 1}
E_{i_0,i_1}E_{i_1,i_2}\dots E_{i_{h-1},i_h}E_{i_h,j}
\underset{k\not\in \{i_1,i_2,\dots ,i_h\}}{\prod_{k=1}^{i-1}}
(E_{i,i}-E_{k,k}-i+k),\\
&D_{j,i}:=\sum_{h=0}^{j-i-1}\,\sum_{i=i_0<i_1<i_2<\dots <i_h<j}
E_{i_1,i_0}E_{i_2,i_1}\dots E_{i_h,i_{h-1}}E_{j,i_h}
\underset{k\not\in \{i_1,i_2,\dots ,i_h\}}{\prod_{k=i+1}^{j-1}}
(E_{i,i}-E_{k,k}-i+k).
\end{align*}
Note that the elements of $\{E_{i,i}-E_{k,k}-i+k \mid 1\leq i,k\leq n\}$ are commutative. For our convenience, we here
slightly change the definition of the raising operator $D_{i,j}$; namely we 
multiply the original one by $(-1)^{i-1}$. 
Let $\lambda =(\lambda_1,\lambda_2,\dots,\lambda_n)\in \Lambda_n$. 
For $M=(m_{i,j})_{1\leq i\leq j\leq n}\in \rG (\lambda )$, 
we define elements $\cD^\pm_{M}$ of $\cU (\ggl_n)$ by 
\begin{align*}
&\cD^+_{M}:=
\cD^+_{M,n}\cD^+_{M,n-1}\dots \cD^+_{M,2}&
&\text{with }\ \cD^+_{M,k}:=
D_{1,k}^{m_{1,k}-m_{1,k-1}}
D_{2,k}^{m_{2,k}-m_{2,k-1}}
\dots 
D_{k-1,k}^{m_{k-1,k}-m_{k-1,k-1}},\\
&\cD^-_{M}:=
\cD^-_{M,2}\cD^-_{M,3}\dots \cD^-_{M,n}&
&\text{with }\ 
\cD^-_{M,k}:=D_{k,1}^{m_{1,k}-m_{1,k-1}}D_{k,2}^{m_{2,k}-m_{2,k-1}}\dots 
D_{k,k-1}^{m_{k-1,k}-m_{k-1,k-1}},
\end{align*}
and also define integers $\rr_1(M)$ and $\rr_2(M)$ by 
\begin{align*}
&\rr_1(M):=\prod_{1\leq i\leq j<k\leq n}
\frac{(m_{i,k}-m_{j+1,k}-i+j)!}{(m_{i,k-1}-m_{j+1,k}-i+j)!},&
&\rr_2(M):=\prod_{1\leq i\leq j<k\leq n}
\frac{(m_{i,k}-m_{j,k-1}-i+j)!}{(m_{i,k}-m_{j,k}-i+j)!}.
\end{align*}
For integral triangular arrays $M=(m_{i,j})_{1\leq i\leq j\leq n}$ and 
$M'=(m_{i,j}')_{1\leq i\leq j\leq n}$ of size $n$, 
we write $M>_{\rlex} M'$ if and only if there are positive integers
$i_0$ and $j_0$ with $1\leq i_0 \leq j_0\leq n$ satisfying
\begin{align}
\label{eq:def_<lex}
&m_{i,j}=m_{i,j}' \quad (1\leq i\leq j\leq j_0-1),&
& m_{i,j_0}=m_{i,j_0}'\quad (1\leq i\leq i_0-1) &
&\text{ \; and\;}&
&m_{i_0,j_0}>m_{i_0,j_0}'.
\end{align}
For integral triangular arrays $M$ and $M'$ of size $n$, 
we write $M\geq_{\rlex} M'$ if and only if 
either $M>_{\rlex} M'$ or $M'=M$ holds. 
Then $\geq_{\rlex}$ is a total order on the set of 
integral triangular arrays of size $n$, and 
we call it the lexicographical order. 
We know from \cite[Section 2]{Molev_001} that the following proposition holds. 
\begin{prop}
\label{prop:Molev_diff_op}
Retain the notation. Then, for $M\in \rG (\lambda )$, we have the following equalities:
\begin{align}
&\label{eq:Molev_D-}
\tau_\lambda (\cD^-_{M})\xi_{H(\lambda )}
=\rr_1(M)\xi_M,\\
&\label{eq:Molev_D+}
\tau_\lambda (\cD^+_{M})\xi_{M}=
\rr_2(M)\xi_{H(\lambda )},\\
&\label{eq:Molev_D+prime}
\tau_{\lambda'} (\cD^+_{M})\xi_{M'}=0\qquad 
\text{for\; } \lambda'\in \Lambda_n \text{ \; and \;} M'\in \rG (\lambda') \text{\; satisfying both \;} M>_{\rlex}M' \text{\; and \;} 
\gamma^M=\gamma^{M'}.
\end{align}
\end{prop}
\begin{proof}
By \cite[Theorems 2.3, 2.4 and 2.7]{Molev_001}, we have 
$\tau_\lambda (\cD^-_{M})\zeta_{H(\lambda )}=\rr_1(M)\sqrt{\rr (M)}\zeta_M$ for $M\in \rG (\lambda )$. 
Hence, by $\xi_M=\sqrt{\rr (M)}\zeta_M$ ($M\in \rG (\lambda )$) and Lemma~\ref{lem:rM_property} \ref{num:Hgamma_rM_property}, 
we obtain (\ref{eq:Molev_D-}). 
We set $\widetilde{\xi}_M:=\tau_\lambda (\cD^-_{M})\xi_{H(\lambda )}=\rr_1(M)\xi_M$ 
for $M\in \rG (\lambda )$. 
Now consider $\mu =(\mu_1,\mu_2,\dots ,\mu_{n-1})\in \Lambda_{n-1}$ satisfying
$\mu \preceq \lambda$. 
By \cite[Lemma 2.13]{Molev_001}, for $1\leq i\leq n-1$, we have 
\begin{align*}
\tau_\lambda (D_{i,n})\widetilde{\xi}_{H(\mu )[\lambda ]}
=\left\{\begin{array}{ll}
(-1)^i\left(\prod_{k=1}^{n}(\mu_i-i-\lambda_k+k)\right)
\widetilde{\xi}_{H(\mu +\delta_i)[\lambda ]}&
\text{if $\lambda_i>\mu_i$},\\
0&\text{if $\lambda_i=\mu_i$},
\end{array}\right.
\end{align*}
where $\delta_i$ is the element of $\mZ^n$ with $1$ at the $i$-th entry and 
$0$ at the other entries. 
By the repeated application of this equality, we have 
\begin{align*}
&\tau_\lambda (D_{1,n}^{\lambda_1-\mu_1}D_{2,n}^{\lambda_2-\mu_2}
\dots D_{n-1,n}^{\lambda_{n-1}-\mu_{n-1}})
\widetilde{\xi}_{H(\mu )[\lambda]}\\
&=(-1)^{\sum_{i=1}^{n-1}i(\lambda_i-\mu_i)}
\left(\prod_{i=1}^{n-1}\prod_{j=0}^{\lambda_i-\mu_i-1}
\prod_{k=1}^{n}(\mu_i+j-i-\lambda_k+k)
\right)
\widetilde{\xi}_{H(\lambda )}\\
&=\left\{\prod_{i=1}^{n-1}\prod_{j=0}^{\lambda_i-\mu_i-1}
\left(\prod_{k=1}^{i}(\lambda_k-\mu_i-j-k+i)\right)
\left(\prod_{k=i+1}^{n}(\mu_i+j-\lambda_k-i+k)\right)
\right\}\widetilde{\xi}_{H(\lambda )}\\
&=\left\{\prod_{i=1}^{n-1}
\left(\prod_{k=1}^{i}\frac{(\lambda_k-\mu_i-k+i)!}{(\lambda_k-\lambda_i-k+i)!}
\right)
\left(\prod_{k=i+1}^{n}
\frac{(\lambda_i-\lambda_k-i+k-1)!}{(\mu_i-\lambda_k-i+k-1)!}\right)
\right\}\widetilde{\xi}_{H(\lambda )}\\
&=
\left(\prod_{1\leq i\leq j\leq n-1}
\frac{(\lambda_i-\mu_j-i+j)!}{(\lambda_i-\lambda_j-i+j)!}\right)
\rr_1(H(\mu )[\lambda])\widetilde{\xi}_{H(\lambda )}
\end{align*}
and for $\lambda'=(\lambda_1',\lambda_2',\dots,\lambda_n')\in \Lambda_n$ 
such that $\mu \preceq \lambda'$, we have 
\begin{align*}
&\tau_{\lambda'} (D_{1,n}^{\lambda_1-\mu_1}D_{2,n}^{\lambda_2-\mu_2}
\dots D_{n-1,n}^{\lambda_{n-1}-\mu_{n-1}})
\widetilde{\xi}_{H(\mu)[\lambda']}=0&
&\text{unless $\lambda_i'\geq \lambda_i$ for any $1\leq i\leq n-1$}. 
\end{align*}
By these equalities, 
we have 
\begin{align}
\label{eq:pf_Molev001}
&\tau_\lambda (D_{1,n}^{\lambda_1-\mu_1}D_{2,n}^{\lambda_2-\mu_2}
\dots D_{n-1,n}^{\lambda_{n-1}-\mu_{n-1}})
\xi_{H(\mu )[\lambda]}
=\left(\prod_{1\leq i\leq j\leq n-1}
\frac{(\lambda_i-\mu_j-i+j)!}{(\lambda_i-\lambda_j-i+j)!}\right)
\xi_{H(\lambda )}
\end{align}
and for $\lambda'=(\lambda_1',\lambda_2',\dots,\lambda_n')\in \Lambda_n$ 
such that $\mu \preceq \lambda'$, we have 
\begin{align}
\label{eq:pf_Molev002}
&\tau_{\lambda'}(D_{1,n}^{\lambda_1-\mu_1}D_{2,n}^{\lambda_2-\mu_2}
\dots D_{n-1,n}^{\lambda_{n-1}-\mu_{n-1}})
\xi_{H(\mu)[\lambda']}=0&
&\text{unless $\lambda_i'\geq \lambda_i$ for any $1\leq i\leq n-1$}. 
\end{align}
Hence we obtain (\ref{eq:Molev_D+}) by Lemma~\ref{lem:C_restriction_lambda_mu} and recursive application of (\ref{eq:pf_Molev001}). 
Let $\lambda'\in \Lambda_n$, $M'=(m_{i,j}')_{1\leq i\leq j\leq n}
\in \rG (\lambda')$ and 
$M=(m_{i,j})_{1\leq i\leq j\leq n}\in \rG (\lambda )$ such that $M>_{\rlex} M'$ and 
$\gamma^M=\gamma^{M'}$. 
Then $M>_{\rlex} M'$ implies that there are positive integers
$i_0$ and $j_0$ with $1\leq i_0\leq j_0\leq n$ satisfying (\ref{eq:def_<lex}). 
Moreover, the equality $\gamma^M=\gamma^{M'}$ implies  
$2\leq j_0$ and $i_0<j_0$. 
Hence we obtain (\ref{eq:Molev_D+prime}) by Lemma~\ref{lem:C_restriction_lambda_mu} and recursive application of (\ref{eq:pf_Molev001}) 
and (\ref{eq:pf_Molev002}). 
\end{proof}

Let $\cA$ be an integral domain of characteristic $0$. 
Let $\rT_n(\cA )$ be the subgroup of $\rGL_n(\cA )$ 
consisting of all diagonal matrices. 
For $\sigma \in \gS_n$, let $u_\sigma $ be the permutation matrix 
corresponding to $\sigma$, that is, 
$u_\sigma $ is the $n\times n$-matrix with $1$ at the 
$(\sigma (i),i)$-th entry for $1\leq i\leq n$ and 
$0$ at the other entries. 
For $\sigma \in \gS_n$ and $I \in \cI_{n}$, we set 
\begin{align*}
m(I,\sigma )&:=\# \{(i,j)\in I^2\mid i<j \text{\; and \;}
\sigma (i)>\sigma (j)\},
&\sigma (I) &:=\{ \sigma (i)\mid i\in I\}.
\end{align*}
For $\sigma \in \gS_n$ and 
$l=(l_I)_{I\in \cI_n}\in \cL (\lambda )$, we set 
$m(l ,\sigma ):=\sum_{I \in \cI_{n}}l_Im(I,\sigma )$ and 
$\sigma (l):=(l_I')_{I\in \cI_n}$ with $l_I'=l_{\sigma^{-1}(I)}$.

\begin{lem}
Retain the notation. Then, for $l \in \cL (\lambda )$, we have 
\begin{align}
\label{eq:actT_gen_Vlambda}
&\tau_{\lambda}(a)f_{l}(z)
=\left(\prod_{i=1}^na_i^{\gamma_i^{(l)}}\right)f_{l}(z)&
&(a=\diag (a_1,a_2,\dots ,a_n)\in \rT_n(\cA )),\\
\label{eq:actS_gen_Vlambda}
&\tau_\lambda (u_\sigma )f_{l}(z)=(-1)^{m(l ,\sigma )}
f_{\sigma (l)}(z)&
&(\sigma \in \gS_n),
\end{align}
where $\gamma^{(l)}=(\gamma^{(l)}_1,\gamma^{(l)}_2,\dots ,
\gamma^{(l)}_n)$ is the weight of $l$ defined by $(\ref{eq:def_gammal_Llambda})$. 
\end{lem}
\begin{proof}
The equality (\ref{eq:actT_gen_Vlambda}) immediately follows from 
${\det}_{I}(za)=(\prod_{i\in I}a_i){\det}_{I}(z)$ 
for $a=\diag (a_1,a_2,\dots ,a_n)\in \rT_n(\cA )$ and $I \in \cI_{n}$. 
For $\sigma \in \gS_n$ and $I \in \cI_{n}$, 
we have ${\det}_{I}(zu_\sigma)=(-1)^{m(I,\sigma )}{\det}_{\sigma (I)}(z)$ 
because of the alternating property of determinants and  
$zu_\sigma =(z_{i,j}')_{1\leq i,j\leq n}$ with $z_{i,j}'=z_{i,\sigma (j)}$. 
Hence, we obtain (\ref{eq:actS_gen_Vlambda}). 
\end{proof}

\begin{lem}
\label{lem:VlambdaA_closed_Eij}
Let $\cA$ be an integral domain of characteristic $0$ 
such that $\cA$ is a subring of $\mC$. Let $\lambda \in \Lambda_n$. 
Then we have 
$\tau_{\lambda }(E_{i,j})\rv \in V_\lambda (\cA )$
for every $\rv \in V_\lambda (\cA )$ and $1\leq i,j\leq n$.  
\end{lem}

\begin{proof}
The assertion follows immediately from (\ref{eq:actwt_gen_Vlambda}), 
(\ref{eq:act+_gen_Vlambda}), (\ref{eq:act-_gen_Vlambda}) and the relations 
\begin{align*}
&\tau_{\lambda}(E_{i,j+1})=
\tau_{\lambda}(E_{i,j})\circ \tau_{\lambda}(E_{j,j+1})
-\tau_{\lambda}(E_{j,j+1})\circ \tau_{\lambda}(E_{i,j}),\\
&\tau_{\lambda}(E_{j+1,i})=
\tau_{\lambda}(E_{j+1,j})\circ \tau_{\lambda}(E_{j,i})
-\tau_{\lambda}(E_{j,i})\circ \tau_{\lambda}(E_{j+1,j})
\end{align*}
for $1\leq i\leq j\leq n-1$. 
\end{proof}

\begin{lem}
\label{lem:xiM_gen_Vlambda}
Let $\cA$ be an integral domain of characteristic $0$. Let $\lambda =(\lambda_1,\lambda_2,\dots,\lambda_n)\in \Lambda_n$. 
\begin{enumerate}
\item \label{num:xiM_in_VlambdaA}
Let $M\in \rG (\lambda)$. If $n\geq 3$, assume that $\{(\mu_1-\mu_{n-1} + n-3)!\}^{-1}\in \cA$ 
for $\mu =(\mu_1,\mu_2,\dots ,\mu_{n-1})\in \Lambda_{n-1}$ satisfying $\widehat{M}\in \rG (\mu )$. 
Then we have $\xi_M\in V_\lambda (\cA \cap \mQ )$ and 
\begin{align}
\label{eq:actT_xiM}
&\tau_{\lambda}(a)\xi_M
=\left(\prod_{i=1}^na_i^{\gamma_i^{M}}\right)\xi_M&
&\text{for \;} a=\diag (a_1,a_2,\dots ,a_n)\in \rT_n(\cA )
\end{align}
where $\gamma^M=(\gamma^M_1,\gamma^M_2,\dots ,\gamma^M_n)$ is the weight of $M$ defined by $(\ref{eq:def_wt_M})$.

\item \label{num:fl_to_xiM}
Assume that $\{(\lambda_1-\lambda_n + n-3)!\}^{-1}\in \cA$ if $n\geq 3$. 
Then, for $l \in \cL (\lambda)$, we have 
\begin{align*}
f_{l}(z)&=\sum_{M\in \rG (\lambda ),\, \gamma^M=\gamma^{(l)}}
c_{l,M}\xi_M 
&\text{with some $c_{l,M}\in \cA \cap \mQ $}. 
\end{align*}
\end{enumerate}
\end{lem}
\begin{proof}
When $n=1$, the assertion follows immediately from the equalities $\rG (\lambda_1)=\{\lambda_1\}$,
$\xi_{\lambda_1}=f_{h (\lambda_1)}(z)=z_{1,1}^{\lambda_1}$ and (\ref{eq:actT_gen_Vlambda}). When $n=2$, the assertion follows immediately from (\ref{eq:xiM_n=2}) and (\ref{eq:actT_gen_Vlambda}). 

Next, let us prove the statement \ref{num:xiM_in_VlambdaA} when $n\geq 3$. 
Since $\{(\mu_1-\mu_{n-1} + n-3)!\}^{-1}\in \cA$, we have 
$\rr_1(\widehat{M})\in \cA^\times \cap \mQ^\times $. 
Hence, by Lemmas \ref{lem:xi_Hmulambda_explicit}, \ref{lem:C_restriction_lambda_mu}, \ref{lem:VlambdaA_closed_Eij} and (\ref{eq:Molev_D-}), 
we have 
\[
\xi_M=\rr_1(\widehat{M})^{-1}
\tau_\lambda (\iota_n(\cD^-_{\widehat{M}}))f_{h (\lambda,\mu)}(z)
\in V_\lambda (\cA \cap \mQ ). 
\]
By (\ref{eq:GT_act_wt}) and (\ref{eq:actwt_gen_Vlambda}), we know that 
$\xi_M$ can be expressed as a linear combination of 
generators $f_{l}(z)$ for $l \in \cL (\lambda )$ satisfying
$\gamma^{M}=\gamma^{(l)}$ with coefficients in $\cA \cap \mQ$. 
Hence we obtain (\ref{eq:actT_xiM}) by (\ref{eq:actT_gen_Vlambda}). 

Finally, let us prove the statement \ref{num:fl_to_xiM} when $n\geq 3$. 
For $l \in \cL (\lambda )$, set 
\begin{align*}
\{M\in \rG (\lambda )\mid \gamma^M=\gamma^{(l)}\}
&=\{M_1,M_2,\dots, M_k\}&
& \text{with \;} \widehat{M_1}>_\rlex \widehat{M_2}>_\rlex \dots >_\rlex \widehat{M_k}.
\end{align*}
Since $\{\xi_M\}_{M\in \rG (\lambda )}$ is a basis of $V_\lambda$, 
(\ref{eq:GT_act_wt}) and (\ref{eq:actwt_gen_Vlambda}) imply that 
$f_{l}(z)=\sum_{i=1}^kc_{l,M_i}\xi_{M_i}$ with some $c_{l,M_i}\in \mC$. 
Our task is to show that $c_{l,M_i}\in \cA \cap \mQ$ for any 
$1\leq i\leq k$. 
For each $1\leq i\leq k$, we take $\mu^{(i)}\in \Lambda_{n-1}$ so that 
$\widehat{M_i}\in \rG (\mu^{(i)})$. 
Since $\{(\lambda_1-\lambda_n + n-3)!\}^{-1}\in \cA$, we have 
$\rr_2(\widehat{M_i})\in \cA^\times \cap \mQ^\times $ for $1\leq i\leq k$. 
Let $1\leq j\leq k$. By Lemmas \ref{lem:xi_Hmulambda_explicit}, \ref{lem:C_restriction_lambda_mu}, 
(\ref{eq:Molev_D+}) and (\ref{eq:Molev_D+prime}), we have 
\[
c_{l,M_{j}}\rr_2(\widehat{M_{j}})
f_{h (\lambda,\mu^{(j)})}(z)=
\sum_{i=j}^{k}c_{l,M_i}\tau_\lambda (\iota_n(\cD^+_{\widehat{M_{j}}}))\xi_{M_i}
=
\tau_\lambda (\iota_n(\cD^+_{\widehat{M_{j}}}))
\left(f_{l}(z)-\sum_{i=1}^{j-1}c_{l,M_i}\xi_{M_i}
\right).
\]
If $c_{l,M_i}\in \cA \cap \mQ$ for any 
$1\leq i\leq j-1$, we have $c_{l,M_j}\in \cA \cap \mQ$, since the right hand side of this equality is in $V_{\lambda }(\cA \cap \mQ)$ by 
Lemma \ref{lem:VlambdaA_closed_Eij} and the statement \ref{num:xiM_in_VlambdaA}. Hence, we inductively obtain  
$c_{l,M_j}\in \cA \cap \mQ$ for each $1\leq j\leq k$. 
\end{proof}

\begin{proof}[{Proof of Proposition~$\ref{prop:xiM_integral}$}]
The assertion follows immediately from Lemma~\ref{lem:xiM_gen_Vlambda} and 
the definition of $V_\lambda (\cA)$. 
\end{proof}

\begin{rem}
\label{rem:xiM_gen_Vlambda}
In Lemma~\ref{lem:xiM_gen_Vlambda} \ref{num:xiM_in_VlambdaA}, 
we cannot remove the assumption 
$\{(\mu_1-\mu_{n-1} + n-3)!\}^{-1}\in \cA$ when $n\geq 3$. 
In fact, when $n=3$, 
for $M=\left(\begin{smallmatrix}\lambda_1\ \lambda_2\ \lambda_3\\ 
\mu_1\ \mu_2\\ \alpha_1 
\end{smallmatrix}\right)\in \rG (\lambda )$, 
we have the following expression of $\xi_M$ in terms of the generators 
$\{f_l(z)\}_{l \in \cL (\lambda)}$ by a direct computation using 
Lemma~\ref{lem:xi_Hmulambda_explicit}, (\ref{eq:GT_act-}) and (\ref{eq:act-_gen_Vlambda}): 
\begin{align*}
\xi_{M}&=
\frac{1}{\displaystyle \binom{\mu_1-\mu_2}{\mu_1-\lambda_{2}}}
\sum_{\beta =\max \{\alpha_1,\lambda_2\}}^{
\min \{\mu_1 ,\lambda_2+\alpha_1-\mu_2\}}
\binom{\mu_1-\alpha_1}{\beta -\alpha_1}
\binom{\alpha_1-\mu_2}{\beta -\lambda_{2}}
f_{l_{(M,\beta )}}(z).
\end{align*}
Here $l_{(M,\beta )}=(l_{(M,\beta ),I})_{I\in \cI_n}\in \cL (\lambda)$ is defined by
\begin{align*}
l_{(M,\beta ),\{1\}}&=\beta -\lambda_{2}, &
l_{(M,\beta ),\{2\}}&=\mu_1-\beta ,&
l_{(M,\beta ),\{3\}}&=\lambda_1-\mu_1, & & \\
l_{(M,\beta ),\{1,2\}}&=\mu_2-\lambda_3, &
l_{(M,\beta ),\{1,3\}}&=\lambda_2+\alpha_1-\mu_2-\beta,&
l_{(M,\beta ),\{2,3\}}&=\beta -\alpha_1,&
l_{(M,\beta ),\{1,2,3\}}&=\lambda_3.
\end{align*}
\end{rem}

\begin{proof}[{Proof of Lemma~$\ref{lem:extremal_vec_explicit}$}]
Let $\gamma =\sigma \lambda \in \cE_\lambda $ 
with $\sigma \in \gS_n$. 
By Lemma~\ref{lem:rM_property} \ref{num:Hgamma_rM_property}, 
we have $\xi_{H(\gamma )}=\zeta_{H(\gamma )}$. 
Let us prove $\zeta_{H(\gamma )}=f_{h (\gamma)}(z)$. 
We take a Lie subalgebra 
$\gn_{n}^{\rop}:=\bigoplus_{1\leq j<i\leq n}\mC E_{i,j}$ of 
$\ggl_n$, and let $\cU (\gn_{n}^{\rop})$ denote its universal enveloping algebra. It is then easy to verify the following assertions:
\begin{enumerate}
\item 
We have $V_\lambda 
=\tau_\lambda (\cU (\ggl_n))f_{h (\lambda )}(z)
=\tau_\lambda (\cU (\gn_{n}^{\rop}))f_{h (\lambda )}(z)$. 

\item The Lie algebra $\gn_{n}^{\rop}$ is generated by 
$\{E_{j+1,j}\mid 1\leq j\leq n-1\}$. 

\item For $1\leq j\leq n-1$ and a weight vector $\rv $ in $V_\lambda$, 
$\tau_\lambda (E_{j+1,j})\rv $ is either a weight vector or $0$. 
\end{enumerate}
Hence, there is an element of $\cU(\gn^{\rop}_n)$ of the form
\begin{align*}
E(\gamma )&=E_{j_1+1,j_1}E_{j_2+1,j_2}\dots E_{j_k+1,j_k}&
&\text{with \;} j_1,j_2,\dots ,j_k\in \{1,2,\dots, n-1\}
\end{align*}
such that $\tau_\lambda (E(\gamma ))f_{h (\lambda )}(z)$ is a 
(nonzero) weight vector in $V_\lambda$ of weight $\gamma$.  
Note that, for $M\in \rG (\lambda)$ and $l\in \cL (\lambda )$, 
both $\zeta_{M}$ and $f_{l}(z)$ are 
weight vectors in $V_\lambda$ respectively of weight $\gamma^M$ and $\gamma^{(l)}$ by (\ref{eq:GT_act_wt}) and (\ref{eq:actwt_gen_Vlambda}). 
Since $H(\gamma )$ and $h (\gamma)$ are unique elements respectively in $\rG (\lambda)$ and $\cL (\lambda)$ 
whose weights are $\gamma$, we have 
\begin{align*}
\tau_\lambda (E(\gamma ))f_{h (\lambda )}(z)
=c_1\zeta_{H(\gamma )}=c_2f_{h (\gamma)}(z)
\end{align*}
with some nonzero constants $c_1$ and $c_2$.  
By (\ref{eq:GT_act-}) and (\ref{eq:act-_gen_Vlambda}), 
we know that $c_1$ and $c_2$ are both positive real numbers. 
Furthermore, since $\{\zeta_M\}_{M\in \rG(\lambda )}$ is an orthonormal basis, 
 we have $(\zeta_{H(\gamma )},\zeta_{H(\gamma )})_{\lambda}=1$. 
The formula (\ref{eq:actS_gen_Vlambda}), combined with the relations $\sigma (h (\lambda ))=h (\gamma )$ and $f_{h (\lambda )}(z)=\zeta_{H(\lambda )}$, then implies the following equalities:
\[
(f_{h (\gamma)}(z),\ f_{h (\gamma)}(z))_{\lambda}
=((-1)^{m(h(\lambda),\sigma)}\tau_{\lambda}(u_\sigma )f_{h (\lambda )}(z),\
(-1)^{m(h(\lambda),\sigma)}\tau_{\lambda}(u_\sigma )f_{h (\lambda )}(z))_{\lambda}
=(f_{h (\lambda )}(z),\ f_{h (\lambda )}(z))_{\lambda}=1.
\]
The positive real numbers $c_1$ and $c_2$ thus satisfy  
\[
c_1^2=(c_1\zeta_{H(\gamma )},c_1\zeta_{H(\gamma )})_{\lambda}
=(c_2f_{h (\gamma)}(z),c_2f_{h (\gamma)}(z))_{\lambda}
=c_2^2.
\]
Hence, we have $c_1=c_2$, 
and this implies $\zeta_{H(\gamma )}=f_{h (\gamma)}(z)$. 
\end{proof}

Let $w_n$ be the anti-diagonal matrix of size $n$ with $1$ at 
all anti-diagonal entries. 
For $\gamma = (\gamma_1,\gamma_2, \dots, \gamma_n)\in \cE_{\lambda}$, 
we set $\gamma^*:=
(\gamma_n,\gamma_{n-1}, \dots, \gamma_1)\in \cE_{\lambda}$.
The following lemma will be used in the proofs of 
Lemma~\ref{lem:ps_rel_op_dual} in Section \ref{sec:arch_whittaker} and 
Lemma~\ref{lem:Clem3} in Section \ref{sec:calcoeff}.

\begin{lem}
\label{lem:GTwn}
Let $\lambda =(\lambda_1,\lambda_2,\dots ,\lambda_n)\in \Lambda_n$ and $\gamma \in \cE_\lambda$. 
Then we have 
$\tau_\lambda (w_n) \xi_{H(\gamma)} 
= (-1)^{  \sum^n_{i=1}  (i-1) \lambda_i  } \xi_{ H(\gamma^*) }$. 
\end{lem}

\begin{proof}
By Lemma~\ref{lem:extremal_vec_explicit}, the formula (\ref{eq:actS_gen_Vlambda}) and the relation
$\sigma (h (\gamma ))=h (\sigma \gamma )$ for $\sigma \in \gS_n$, we have 
\begin{align}
\label{eq:pf_GTwn_001}
\tau_\lambda (u_\sigma )\xi_{H(\gamma)}&=(-1)^{m(h (\gamma ) ,\sigma )}
\xi_{H(\sigma \gamma)}& \text{for\; } \sigma \in \gS_n.
\end{align}
Let $\sigma_{\mathrm{long}}\in \gS_n$ be a permutation determined by 
$\sigma_{\mathrm{long}}(i)=n+1-i$ for $1\leq i\leq n$. 
Then we have $\gamma^*=\sigma_{\mathrm{long}} \gamma $ and 
$w_n=u_{\sigma_{\mathrm{long}}}$. Moreover, since 
$m(I,\sigma_{\mathrm{long}})=\dfrac{k(k-1)}{2}$ holds
 for every $I\in \cI_{n,k}$ with $1\leq k\leq n$, we have 
\begin{align*}
m(h (\gamma ) ,\sigma_{\mathrm{long}})
=\lambda_n \frac{n(n-1)}{2}
+\sum^{n-1}_{k=1}(\lambda_k-\lambda_{k+1})\frac{k(k-1)}{2}
=\sum^n_{i=1}(i-1) \lambda_i.
\end{align*}
Therefore the assertion follows from (\ref{eq:pf_GTwn_001}) for 
$\sigma =\sigma_{\mathrm{long}}$. 
\end{proof}

\begin{lem}
\label{lem:GLnA_generator}
Let $\cA$ be a field of characteristic $0$. 
Then the  group $\rGL_n(\cA)$ is generated by 
$\rT_n(\cA)\cup \rGL_n(\mZ)$. 
\end{lem}
\begin{proof}
For $c\in \cA^\times $ and $1\leq i\leq n$, let $u_{i,i}(c)\in \rT_n(\cA)$ 
be the diagonal matrix of size $n$ with $c$ at the $(i,i)$-th entry and 
$1$ at the other diagonal entries. 
For $c\in \cA^\times $ and $1\leq i\neq j\leq n$, 
let $u_{i,j}(c)\in \rGL_n(\cA)$ 
be the square matrix of size $n$ with $c$ at the $(i,j)$-th entry, 
$1$ at all diagonal entries and 
$0$ at the other entries. 
By Gaussian elimination, one observes that $\rGL_n(\cA)$ is generated by 
$\{u_{i,j}(c)\mid 1\leq i,j\leq n,\ c\in \cA^\times \}\cup 
\{u_\sigma \mid \sigma \in \gS_n\}$. 
Since $u_{i,j}(1)\in \rGL_n(\mZ)$ for $1\leq i\neq j\leq n$ and 
$u_{\sigma}\in \rGL_n(\mZ)$ for $\sigma \in \gS_n$, the assertion follows from 
this fact and the famous relation $u_{i,j}(c)=u_{i,i}(c)u_{i,j}(1)u_{i,i}(c^{-1})$  for $1\leq i\neq j\leq n$ and $c\in \cA^\times$. 
\end{proof}

\begin{proof}[{Proof of Lemma~$\ref{lem:restriction_lambda_mu}$}]
Taking Proposition~\ref{prop:xiM_integral} into accounts, 
we have $\rR^\lambda_\mu \circ \rI_\mu^\lambda =\id_{V_{\mu}(\cA )}$ by definition. 
Hence, it suffices to show that 
\begin{align*}
\rR^\lambda_\mu (\tau_\lambda (\iota_n(g))\rv)
&=\tau_\mu (g)\rR^\lambda_\mu (\rv)\qquad 
(\rv\in V_{\lambda}(\cA )) & \text{\; and\;} &&
\rI^\lambda_\mu (\tau_\mu (g)\rv')
&=\tau_\lambda (\iota_n(g))\rI^\lambda_\mu (\rv')\qquad 
(\rv'\in V_{\mu}(\cA ))
\end{align*}
hold for each $g\in \rGL_{n-1}(\cA )$. 
Replacing $\cA$ with its quotient field if necessary, 
we may assume that $\cA$ is a field of characteristic $0$. 
Because of Proposition~\ref{prop:xiM_integral} and 
Lemma~\ref{lem:GLnA_generator}, 
it suffices to show that
\begin{align*}
\rR^\lambda_\mu (\tau_\lambda (\iota_n(g))\xi_M)
& =\tau_\mu (g)\rR^\lambda_\mu (\xi_M)\qquad 
(M\in \rG (\lambda ))& \text{\; and \;}
&&\rI^\lambda_\mu (\tau_\mu (g)\xi_N)
&=\tau_\lambda (\iota_n(g))\rI^\lambda_\mu (\xi_N)\qquad 
(N\in \rG (\mu ))
\end{align*}
hold for each $g\in \rT_{n-1}(\cA )\cup \rGL_{n-1}(\mZ)$, but these immediately  follow from (\ref{eq:actT_xiM}) and 
Lemma~\ref{lem:C_restriction_lambda_mu}. 
\end{proof}

\begin{proof}[Proof of Lemma~$\ref{lem:detl_shift}$]
Because of Proposition~\ref{prop:xiM_integral}, it suffices to consider 
the case where $\cA =\mC$. Let $M\in \rG (\lambda )$. 
By definition, we have $\cD^-_{M}=\cD^-_{M-l}$ and $\rr_1(M)=\rr_1(M-l)$. 
Since $\cD^-_{M}$ is a sum of products of 
$E_{j,i}$ and $E_{i,i}-E_{j,j}-i+j$ ($1\leq i<j\leq n$), 
the equality (\ref{eq:detl_shift}) implies 
\begin{align*}
&\rI^{\det}_{\lambda ,l}(\tau_{\lambda }(\cD^-_{M})\rv )=
\tau_{\lambda -l}(\cD^-_{M-l})\rI^{\det}_{\lambda ,l}(\rv )&
&\text{for \;} \rv \in V_{\lambda}.
\end{align*}
Therefore, by (\ref{eq:Molev_D-}) and the equalities $\rI^{\det}_{\lambda ,l}(\xi_{H(\lambda )})
=(\det z)^{-l}f_{h (\lambda)}(z)
=f_{h (\lambda -l)}(z)=\xi_{H(\lambda -l)}$, 
we have 
\begin{align*}
\rI^{\det}_{\lambda ,l}(\xi_{M})
&=\rr_1(M)^{-1}
\rI^{\det}_{\lambda ,l}(\tau_{\lambda }(\cD^-_{M})\xi_{H(\lambda )})
=\rr_1(M-l)^{-1}\tau_{\lambda -l}(\cD^-_{M-l})\xi_{H(\lambda -l)}
=\xi_{M-l}.
\end{align*}
We also have
$\rI^{\det}_{\lambda ,l}(\zeta_{M})=\zeta_{M-l}$ due to the equality $\rr (M)=\rr (M-l)$. 
Finally, since $\{\zeta_{M}\}_{M\in \rG (\lambda )}$ and $\{\zeta_{N}\}_{N\in \rG (\lambda -l)}$ are 
orthonormal basis, we readily check that $\rI^{\det}_{\lambda ,l}$ preserves the inner products. 
\end{proof}

\begin{lem}
\label{lem:pair_conj}
Retain the notation. We define a $\mC$-bilinear pairing 
$\langle \cdot, \cdot \rangle_{\lambda,\rconj}
\colon V_{\lambda}\otimes_\mC V_{\lambda}^{\rconj}\to \mC$ by 
\begin{align*}
&\langle \rv,\rv'\rangle_{\lambda,\rconj}
:=\sum_{M\in \rG (\lambda)}c_M c_M' (\xi_{M}, \xi_{M})_{\lambda}
=\sum_{M\in \rG (\lambda)}c_M c_M' \rr (M)
\end{align*}
for $\rv=\sum_{M\in \rG (\lambda)}c_M \xi_{M}\in V_{\lambda}$ and 
$\rv'=\sum_{M\in \rG (\lambda)}c_M'\xi_{M}\in V_{\lambda}^{\rconj}$ 
with $c_M,c_M'\in \mC$. Then we have 
\[
\langle \cdot, \cdot \rangle_{\lambda,\rconj} \in 
\Hom_{\rU (n)}(V_{\lambda}\otimes_\mC V_{\lambda}^{\rconj},
\mC_{\rtriv}),
\]
where $\mC_{\rtriv}=\mC$ is the trivial $\rU (n)$-module.
\end{lem}

\begin{proof}
The assertion follows from the properties \cite[Section 2.6]{im} 
of complex conjugate representations, combined with Remark \ref{rem:cpx_conj_rep}. 
\end{proof}

\begin{proof}[Proof of Lemma~$\ref{lem:GL_inv_pairing}$]
The equality (\ref{eq:sym_pairing}) follows immediately from the definition and Lemma~\ref{lem:Mdual_property}. 
Hence, the main concern is to prove (\ref{eq:GL_inv_pairing}). 
First, we consider the case where $\cA =\mC$. 
Since 
$\langle \rv ,\rv' \rangle_\lambda
=\langle \rv , \rI^{\rconj}_{\lambda^{\!\vee}}(\rv') \rangle_{\lambda,\rconj}$ 
for $\rv \in V_\lambda$ and $\rv' \in V_{\lambda^\vee}$, we have 
$\langle \cdot, \cdot \rangle_{\lambda}\in 
\Hom_{\rU (n)}(V_{\lambda}\otimes_\mC V_{\lambda^{\!\vee}},
\mC_{\rtriv})$ 
by Lemmas \ref{lem:pair_conj} and \ref{lem:hom_dual_conj}. 
Hence, by Weyl's unitary trick \cite[Proposition 5.7]{Knapp_002}, 
we obtain (\ref{eq:GL_inv_pairing}) when $\cA =\mC$. 

Next, let us prove (\ref{eq:GL_inv_pairing}) for the general case. 
Replacing $\cA$ with its quotient field if necessary, 
we may assume that $\cA$ is a field of characteristic $0$. 
Because of Proposition~\ref{prop:xiM_integral} and 
Lemma~\ref{lem:GLnA_generator}, 
it suffices to show  that 
\begin{align*}
\langle \tau_\lambda (g) \xi_M, 
\tau_{\lambda^\vee} (g)\xi_{M'} \rangle_\lambda 
&=\langle \xi_M, \xi_{M'} \rangle_\lambda 
\end{align*}
holds for $M\in \rG (\lambda )$, $M'\in \rG (\lambda^\vee )$ and $g\in \rT_{n}(\cA )\cup \rGL_{n}(\mZ)$, but it follows from (\ref{eq:actT_xiM}), 
Lemma~\ref{lem:Mdual_property} and (\ref{eq:GL_inv_pairing}) for the case 
where $\cA=\mC$. 
\end{proof}

\subsection{Projectors and injectors}
\label{subsec:proj_inj_tensor}

In this subsection, we give the proof of Proposition~\ref{prop:injExp}. 
For $\lambda =(\lambda_1,\lambda_2,\dots,\lambda_n)$ and
$\lambda'=(\lambda_1',\lambda_2',\dots,\lambda_n')\in \Lambda_n$, 
let $\rR^{\lambda, \lambda'}_{\lambda +\lambda'}\colon 
V_{\lambda}\otimes_\mC V_{\lambda'}\to V_{\lambda +\lambda'}$ 
be the  $\rGL_n(\mC )$-equivariant surjection defined by (\ref{eq:def_proj}).

\begin{lem}
\label{lem:proj_image}
Retain the notation. 
\begin{enumerate}
\item \label{num:proj_xiHgamma}
For $\sigma \in \gS_n$, set $\gamma =\sigma \lambda $ and 
$\gamma' =\sigma \lambda' $. 
Then we have 
$\rR^{\lambda, \lambda'}_{\lambda +\lambda'}
(\xi_{H(\gamma )}\otimes \xi_{H(\gamma')})=\xi_{H(\gamma +\gamma')}$. 

\item \label{num:proj_xiHmulambda}
Assume that $n>1$, and take $\mu ,\mu'\in \Lambda_{n-1}$ satisfying 
$\mu \preceq \lambda $ and $\mu'\preceq \lambda'$. 
Then, for $1\leq m\leq n$, we have 
\begin{align*}
&\rR^{\lambda, \lambda'}_{\lambda +\lambda'}
(\xi_{H(\lambda,\mu;m)}\otimes \xi_{H(\lambda',\mu';m)})
=\xi_{H(\lambda +\lambda',\mu +\mu';m)}. 
\end{align*}
In particular, we have 
$\rR^{\lambda, \lambda'}_{\lambda +\lambda'}
(\xi_{H(\mu )[\lambda ]}\otimes \xi_{H(\mu')[\lambda']})
=\xi_{H(\mu +\mu')[\lambda +\lambda']}$.
\end{enumerate}
\end{lem}

\begin{proof}
The statement \ref{num:proj_xiHgamma} follows from 
Lemma~\ref{lem:extremal_vec_explicit} and 
$h(\gamma )_I+h(\gamma')_I=h(\gamma +\gamma')_I$ 
for $I\in \cI_n$. 
The statement \ref{num:proj_xiHmulambda} follows from 
Lemma~\ref{lem:xi_Hmulambda_gen} and 
$h(\lambda,\mu;m)_I+h(\lambda',\mu';m)_I
=h(\lambda +\lambda',\mu +\mu';m)_I$ 
for $I\in \cI_n$ and $1\leq m\leq n$. 
\end{proof}

Let $\rI^{\lambda, \lambda'}_{\lambda +\lambda'}
\colon V_{\lambda +\lambda'}\to 
V_\lambda \otimes_{\mC }V_{\lambda'}$ be 
 the $\rGL_n(\mC )$-equivariant injection normalised as 
(\ref{eq:def_inj}), and let 
$\rc^{M,M'}_{M''}$ be its coefficient at $M\in \rG (\lambda )$, $M'\in \rG (\lambda')$ and $M''\in \rG (\lambda +\lambda')$
defined as (\ref{eq:CoefInj}). 

\begin{lem}
\label{lem:injExpH(mu)}
Retain the notation. 
Assume that $n>1$, and take $\mu ,\mu'\in \Lambda_{n-1}$ satisfying 
$\mu \preceq \lambda $ and $\mu'\preceq \lambda'$. 
Then we have 
\begin{align}
\label{eq:injExpH(mu)}
&\rc^{H(\mu )[\lambda ], H(\mu')[\lambda']}_{
H(\mu +\mu')[\lambda +\lambda']}
=\frac{\rr (H(\mu +\mu')[\lambda +\lambda'])}
{\rr (H(\mu )[\lambda ])\rr (H(\mu')[\lambda'])}.
\end{align}
Moreover, we have $\rc^{H(\mu )[\lambda ], H(\mu')[\lambda']}_{M''}=0$ 
for $M''\in \rG (\lambda +\lambda')$ such that $M''\neq H(\mu +\mu')[\lambda +\lambda']$. 
\end{lem}
\begin{proof}
We define a $\mC$-linear map 
$\widetilde{\rR}^{\lambda, \lambda'}_{\lambda +\lambda'}\colon 
V_{\lambda}\otimes_\mC V_{\lambda'}\to V_{\lambda +\lambda'}$ by 
\begin{align*}
\widetilde{\rR}^{\lambda, \lambda'}_{\lambda +\lambda'}(\rv )
&:=\sum_{M\in \rG (\lambda +\lambda')}
(\rv ,\rI^{\lambda, \lambda'}_{\lambda +\lambda'}(\zeta_M))_{(\lambda,\lambda')}
\zeta_M
=\sum_{M\in \rG (\lambda )}\rr (M)^{-1}
(\rv ,\rI^{\lambda, \lambda'}_{\lambda +\lambda'}(\xi_M))_{(\lambda,\lambda')}
\xi_M
&\text{for \;} \rv \in V_{\lambda}\otimes_\mC V_{\lambda'}.
\end{align*}
Then $\widetilde{\rR}^{\lambda, \lambda'}_{\lambda +\lambda'}$ 
is $\rU (n)$-equivariant because $(\cdot ,\cdot )_{(\lambda,\lambda')}$ 
is $\rU (n)$-invariant and 
$\{\zeta_M\}_{M\in \rG (\lambda +\lambda')}$ is an orthonormal basis 
of $V_{\lambda +\lambda'}$. Since any highest weight vector 
in $V_{\lambda}\otimes_\mC V_{\lambda'}$ of weight $\lambda +\lambda'$ 
is a constant multiple of $\xi_{H(\lambda )}\otimes \xi_{H(\lambda')}$, 
$\Hom_{\rU (n)}(V_{\lambda}\otimes_\mC V_{\lambda'},
V_{\lambda +\lambda'})$ is one dimensional. 
By Lemma 
\ref{lem:proj_image} \ref{num:proj_xiHgamma} and (\ref{eq:def_inj}), we have 
\begin{align*}
\rR^{\lambda, \lambda'}_{\lambda +\lambda'}
(\xi_{H(\lambda )}\otimes \xi_{H(\lambda')})&=\xi_{H(\lambda +\lambda')},
&\widetilde{\rR}^{\lambda, \lambda'}_{\lambda +\lambda'}
(\xi_{H(\lambda )}\otimes \xi_{H(\lambda')})&=\xi_{H(\lambda +\lambda')}.
\end{align*}
These equalities imply
$\rR^{\lambda, \lambda'}_{\lambda +\lambda'}
=\widetilde{\rR}^{\lambda, \lambda'}_{\lambda +\lambda'}$, 
since both $\rR^{\lambda, \lambda'}_{\lambda +\lambda'}$ and
$\widetilde{\rR}^{\lambda, \lambda'}_{\lambda +\lambda'}$ are generators of 
$\Hom_{\rU (n)}(V_{\lambda}\otimes_\mC V_{\lambda'},
V_{\lambda +\lambda'})$. 
Therefore, we obtain (\ref{eq:injExpH(mu)}) by the following calculation: 
\begin{align*}
\rr (H(\mu )[\lambda ])\rr (H(\mu')[\lambda'])
\rc^{H(\mu )[\lambda ], H(\mu')[\lambda']}_{
H(\mu +\mu')[\lambda +\lambda']}
&=\left(\rI^{\lambda, \lambda'}_{\lambda +\lambda'}
       (\xi_{H(\mu +\mu')[\lambda +\lambda']}), 
          \xi_{H(\mu )[\lambda ]}\otimes \xi_{H(\mu')[\lambda']}
     \right)_{(\lambda,\lambda')}\\
&=\left(\xi_{H(\mu +\mu')[\lambda +\lambda']},
            \widetilde{\rR}^{\lambda, \lambda'}_{\lambda +\lambda'}
              (\xi_{H(\mu )[\lambda ]}\otimes \xi_{H(\mu')[\lambda']})
     \right)_{(\lambda,\lambda')}\\
&=\left( \xi_{H(\mu +\mu')[\lambda +\lambda']},
            \xi_{H(\mu +\mu')[\lambda +\lambda']}
      \right)_{(\lambda,\lambda')}
=\rr (H(\mu +\mu')[\lambda +\lambda']).
\end{align*}
Here the third equality follows from 
$\rR^{\lambda, \lambda'}_{\lambda +\lambda'}
=\widetilde{\rR}^{\lambda, \lambda'}_{\lambda +\lambda'}$ 
and Lemma~\ref{lem:proj_image} \ref{num:proj_xiHmulambda}. 

Next, let us prove the latter assertion. 
Let $M''\in \rG (\lambda +\lambda')$, and take $\mu''\in \Lambda_{n-1}$ 
so that $\widehat{M''}\in \rG (\mu'')$. 
By (\ref{eq:GT_act_wt}), (\ref{eq:decomp_n_n-1}) and Lemma~\ref{lem:restriction_lambda_mu} with $\cA =\mC$, 
we know that $\xi_{M''}$ is a weight vector of weight $\gamma^{\widehat{M''}}$ in 
the irreducible $\rGL_{n-1}(\mC)$-submodule 
$V_{\lambda +\lambda',\mu''}(\mC )\simeq V_{\mu''}$ of $V_{\lambda +\lambda'}$. 
Moreover, we note that $\xi_{H(\mu )[\lambda ]}\otimes \xi_{H(\mu')[\lambda']}$ is a highest weight vector in 
an irreducible $\rGL_{n-1}(\mC)$-submodule of $V_{\lambda}\otimes_\mC V_{\lambda'}$ with highest weight $\mu +\mu'$. 
Therefore, if $M''\neq H(\mu +\mu')[\lambda +\lambda']$, 
we have  
\[
\rr (H(\mu )[\lambda ])\rr (H(\mu')[\lambda'])
\rc^{H(\mu )[\lambda ], H(\mu')[\lambda']}_{M''}=
\left( \rI^{\lambda, \lambda'}_{\lambda +\lambda'}(\xi_{M''}), 
\xi_{H(\mu )[\lambda ]}\otimes \xi_{H(\mu')[\lambda']}
\right)_{(\lambda,\lambda')}=0, 
\]
since either $\gamma^{\widehat{M''}}\neq \mu +\mu'$ or $\mu''\neq \mu +\mu'$ holds. 
Hence, we obtain the latter assertion. 
\end{proof}

\begin{lem}\label{lem:cinduct}
Retain the notation in Lemma~$\ref{lem:injExpH(mu)}$. 
Then, for $M\in \rG (\mu )$, $M' \in \rG (\mu')$ and 
$M'' \in \rG (\mu +\mu')$, 
we have  
\begin{align*}
\rc^{M[\lambda],M'[\lambda']}_{M''[\lambda +\lambda']}
=\frac{\rr (H(\mu +\mu')[\lambda +\lambda'])}
{\rr (H(\mu )[\lambda ])\rr (H(\mu')[\lambda'])}
\rc^{M, M'}_{M''}.
\end{align*}
\end{lem}

\begin{proof}
For $M'' \in \rG (\mu +\mu')$, we have 
\begin{align}
\label{eq:pf_cinduct_001}
((\rR^\lambda_\mu \otimes \rR^{\lambda'}_{\mu'})
\circ \rI^{\lambda, \lambda'}_{\lambda +\lambda'}
\circ \rI^{\lambda +\lambda'}_{\mu + \mu'})(\xi_{M''})
=\sum_{M\in \rG (\mu )}\sum_{M'\in \rG (\mu')}
\rc^{M [\lambda], M'[\lambda']}_{M''[\lambda +\lambda']}
\xi_M \otimes \xi_{M'},
\end{align}
where $\rR^\lambda_\mu $, $\rR^{\lambda'}_{\mu'}$ and 
$\rI^{\lambda +\lambda'}_{\mu + \mu'}$ are the maps
defined by (\ref{eq:def_proj_restriction}) and 
(\ref{eq:def_inj_restriction}) which are $\rGL_{n-1}(\mC )$-equivariant 
by Lemma~\ref{lem:restriction_lambda_mu}. 
By (\ref{eq:inj_generates}), we have 
\[
(\rR^\lambda_\mu \otimes \rR^{\lambda'}_{\mu'})
\circ \rI^{\lambda, \lambda'}_{\lambda +\lambda'}
\circ \rI^{\lambda +\lambda'}_{\mu + \mu'}
\in \Hom_{\rGL_{n-1}(\mC )}(V_{\mu + \mu'},V_{\mu }\otimes_\mC V_{\mu'})
=\mC \,\rI^{\mu, \mu'}_{\mu +\mu'}.
\]
Hence, 
comparing (\ref{eq:pf_cinduct_001}) for $M''=H(\mu +\mu')$ with 
$\rI^{\mu , \mu '}_{\mu + \mu'}(\xi_{H(\mu +\mu')})
=\xi_{H(\mu )}\otimes \xi_{H(\mu')}$, we have 
\begin{align}
\label{eq:pf_cinduct_002}
(\rR^\lambda_\mu \otimes \rR^{\lambda'}_{\mu'})
\circ \rI^{\lambda, \lambda'}_{\lambda +\lambda'}
\circ \rI^{\lambda +\lambda'}_{\mu + \mu'}
=\rc^{H(\mu )[\lambda ], H(\mu')[\lambda']}_{
H(\mu +\mu')[\lambda +\lambda']}\,\rI^{\mu, \mu'}_{\mu +\mu'}.
\end{align}
By Lemma~\ref{lem:injExpH(mu)}, for $M'' \in \rG (\mu +\mu')$, we have 
\begin{align}
\label{eq:pf_cinduct_003}
&\rc^{H(\mu )[\lambda ], H(\mu')[\lambda']}_{
H(\mu +\mu')[\lambda +\lambda']}\,
\rI^{\mu, \mu'}_{\mu +\mu'}(\xi_{M''})
=\sum_{M\in \rG (\mu )}\,\sum_{M'\in \rG (\mu')}
\frac{\rr (H(\mu +\mu')[\lambda +\lambda'])}
{\rr (H(\mu )[\lambda ])\rr (H(\mu')[\lambda'])}
\rc^{M,M'}_{M''}\xi_{M}\otimes \xi_{M'}. 
\end{align}
Comparing the right hand side of (\ref{eq:pf_cinduct_001}) 
with the right hand side of (\ref{eq:pf_cinduct_003}) via (\ref{eq:pf_cinduct_002}), 
we obtain the assertion. 
\end{proof}

\begin{proof}[Proof of Proposition~$\ref{prop:injExp}$]
First, let us prove (\ref{eq:inj_coeff_explicit}) by induction 
on $n$. It is obvious when $n=1$, since 
$\rG (\lambda_1)=\{\lambda_1\}$, $\rG (\lambda_1')=\{\lambda_1'\}$, 
$\rG (\lambda_1+\lambda_1')=\{\lambda_1+\lambda_1'\}$ and 
$\rc^{\lambda_1,\lambda_1'}_{\lambda_1+\lambda_1'}=1$ 
(by (\ref{eq:def_inj})). 
Now let us assume that $n>1$, and take 
$M\in \rG (\lambda )$, $M'\in \rG (\lambda')$ 
and $M''\in \rG (\lambda +\lambda')$. 
By Lemmas \ref{lem:rM_property} \ref{num:rest_rM_property}, 
\ref{lem:cinduct} and the induction hypothesis, we have 
\begin{align*}
&\rc^{M, M'}_{M+M'} 
=\frac{\rr (M+M')}{\rr (M)\rr (M')}.
\end{align*}
On (\ref{eq:CoefInj}), 
$\rI^{\lambda, \lambda'}_{\lambda+\lambda'}(\xi_{M''})$ and 
$\xi_M\otimes \xi_{M'}$ are weight vectors in 
$V_\lambda \otimes_{\mC }V_{\lambda'}$ of weight $\gamma^{M''}$ and 
$\gamma^{M}+\gamma^{M'}$, respectively. 
Hence, we have $\rc^{M, M'}_{M''}=0$ if 
$\gamma^{M''}\neq \gamma^{M}+\gamma^{M'}$, 
which completes the proof of (\ref{eq:inj_coeff_explicit}).

Let us prove the remaining part of the statement 
\ref{num:injExp_coeff}.  
By (\ref{eq:def_inj}), (\ref{eq:Molev_D-}) and 
Lemma~\ref{lem:extremal_vec_explicit}, 
we have 
\begin{align*}
\rI^{\lambda, \lambda'}_{\lambda+\lambda'}(\xi_{M''})
&=\rr_1(M'')^{-1}
\rI^{\lambda, \lambda'}_{\lambda+\lambda'}
(\tau_{\lambda +\lambda'} (\cD^-_{M''})\xi_{H(\lambda +\lambda')})\\
&=\rr_1(M'')^{-1}(\tau_{\lambda }\otimes \tau_{\lambda'})(\cD^-_{M''})
f_{h (\lambda )}(z)\otimes f_{h (\lambda' )}(z)  \qquad \text{for \; } M''\in \rG (\lambda +\lambda').
\end{align*}
Since $\{(\lambda_1+\lambda_1'-\lambda_n-\lambda_n' + n-2)!\}^{-1}\in \cA$ if $n\geq 2$, 
we have $\rr_1(M'')\in \cA^\times \cap \mQ^\times $. 
Therefore, by Lemma~\ref{lem:VlambdaA_closed_Eij}, we have 
\[
\rI^{\lambda, \lambda'}_{\lambda+\lambda'}(\xi_{M''})\in 
V_\lambda (\cA \cap \mQ ) \otimes_{\cA \cap \mQ }V_{\lambda'}(\cA \cap \mQ ). 
\]
Hence, by Proposition~\ref{prop:xiM_integral}, 
we have $\rc^{M, M'}_{M''}\in \cA \cap \mQ$ 
for $M\in \rG (\lambda )$ and $M'\in \rG (\lambda')$, and 
obtain the statement \ref{num:injExp_coeff}.

Next, let us prove the statement \ref{num:injExp_hom} when 
$\cA =\mC$. By definition, $\rI^{\lambda, \lambda'}_{\lambda +\lambda'}
\colon V_{\lambda +\lambda'}\to 
V_\lambda  \otimes_{\mC}V_{\lambda'}$ is  
$\rGL_n(\mC )$-equivariant. 
By the normalisation (\ref{eq:def_inj}) and 
Lemma~\ref{lem:proj_image} \ref{num:proj_xiHgamma}, we have 
$\rR^{\lambda, \lambda'}_{\lambda +\lambda'}(
\rI^{\lambda, \lambda'}_{\lambda +\lambda'}(\xi_{H(\lambda +\lambda')}))
=\xi_{H(\lambda +\lambda')}$. Hence, 
by Schur's lemma \cite[Proposition 1.5]{Knapp_002}, we have 
$\rR^{\lambda, \lambda'}_{\lambda +\lambda'}\circ 
\rI^{\lambda, \lambda'}_{\lambda +\lambda'}
=\id_{V_{\lambda +\lambda'}}$. In particular, we have 
\begin{align}
\label{eq:pf_injExp_001}
\rR^{\lambda, \lambda'}_{\lambda +\lambda'}(
\rI^{\lambda, \lambda'}_{\lambda +\lambda'}(\xi_{M}))
&=\xi_{M}& \text{for \;}M\in \rG (\lambda +\lambda'). 
\end{align}

Finally, let us prove the statement \ref{num:injExp_hom} 
for the general case. 
Since we have $\rR^{\lambda, \lambda'}_{\lambda +\lambda'}\circ 
\rI^{\lambda, \lambda'}_{\lambda +\lambda'}
=\id_{V_{\lambda +\lambda'}(\cA )}$ 
by (\ref{eq:pf_injExp_001}) and Proposition~\ref{prop:xiM_integral}, 
our task is to show the
$\rGL_n(\cA )$-equivariance of $\rI^{\lambda, \lambda'}_{\lambda +\lambda'}$. 
Replacing $\cA$ with its quotient field if necessary, 
we may assume that $\cA$ is a field of characteristic $0$. 
Because of Proposition~\ref{prop:xiM_integral} and 
Lemma~\ref{lem:GLnA_generator}, 
it suffices to show that 
\begin{align*}
&\rI^{\lambda, \lambda'}_{\lambda +\lambda'}
(\tau_{\lambda +\lambda'}(g)\xi_{M})=
(\tau_{\lambda }\otimes \tau_{\lambda'})(g)
\rI^{\lambda, \lambda'}_{\lambda +\lambda'}
(\xi_{M})
\end{align*}
holds for $M\in \rG (\lambda +\lambda')$ and $g\in \rT_{n}(\cA )\cup \rGL_{n}(\mZ)$, but 
it immediately follows from (\ref{eq:actT_xiM}) and 
the assertion for the case where $\cA=\mC$. 
\end{proof}

The following lemma will be used in the proofs of 
Lemma~\ref{lem:std_schwartz} in Section~\ref{sec:arch_whittaker}, and 
Lemma~\ref{lem:Clem1} in Section \ref{sec:calcoeff}. 
\begin{lem}
\label{lem:extremal_injector}
Let $\lambda =(\lambda_1,\lambda_2,\dots,\lambda_n)$ and
$\lambda'=(\lambda_1',\lambda_2',\dots,\lambda_n')\in \Lambda_n$. 
For $\sigma \in \gS_n$, set $\gamma =\sigma \lambda $ and 
$\gamma' =\sigma \lambda' $. 
Then, for $M\in \rG (\lambda)$ and $M'\in \rG (\lambda')$, we have 
\begin{align}
\label{eq:cHeq}
\rc^{M, M'}_{H(\gamma + \gamma')} 
=\left\{\begin{array}{ll}
1&\text{if $M=H(\gamma )$ and $M'=H(\gamma')$},\\
0&\text{otherwise}.
\end{array}\right.    
\end{align}
In other words, we have 
$\rI^{\lambda, \lambda'}_{\lambda +\lambda'}(\xi_{H(\gamma + \gamma')})
=\xi_{H(\gamma )}\otimes \xi_{H(\gamma')}$. 
\end{lem}
\begin{proof}
The vector $\xi_{H(\gamma )}\otimes \xi_{H(\gamma')}$ is 
a weight vector in $V_{\lambda}\otimes_\mC V_{\lambda'}$ of 
weight $\gamma +\gamma'$, which is unique up to scalar multiples. 
Hence, the assertion follows from (\ref{eq:inj_coeff_explicit}), 
Lemma~\ref{lem:rM_property} \ref{num:Hgamma_rM_property} and the equality
$H(\gamma + \gamma')=H(\gamma )+H(\gamma')$. 
\end{proof}

\section{Archimedean Whittaker functions}
\label{sec:arch_whittaker}

The purpose of this section is to verify Proposition~\ref{prop:arch_Wh}, 
and to rewrite the explicit result \cite[Theorem 2.7]{im} of Ishii and the second-named author for the archimedean local zeta integral 
for $\rGL_n(\mC)\times \rGL_{n-1}(\mC)$ into our notation (Proposition \ref{prop:ArchInt}). 
In this section, we use the notation in Section \ref{subsec:C_def_ps}, and 
drop $v$ from subscript in many cases since we always concentrate on an archimedean place $v \in \Sigma_{F, \infty}$. 
In Section \ref{sec:main_result}, 
we used the realisation $(\pi_{\rB_n,d,\nu },I^\infty_{\rB_n}(d,\nu ))$ of 
a principal series representation of $\rGL_n(\mC)$ 
induced from a character of the {\em upper} triangular Borel subgroup $\rB_n(\mC)$ 
to state our main theorems, 
because it is most commonly used in the theory of automorphic forms. 
However, throughout this section, 
we mainly use another realisation $(\pi_{\rB_n^{\rop},d,\nu },I^\infty_{\rB_n^{\rop}}(d,\nu ))$ 
induced from a character of the {\em lower} triangular Borel subgroup $\rB_n^{\rop}(\mC)$, 
because the results in \cite{Jacquet_001} and \cite{im} play important roles here and they are given 
with the latter realisation.

\subsection{Another realisation of a principal series representation}
\label{subsec:another_def_ps}

Let $\rB_n^{\rop}$ be the Borel subgroup of $\rGL_{n/F}$ consisting of all 
lower triangular matrices. 
We note that $\rB_n$ and $\rB_n^{\rop}$ are opposite Borel subgroups. 
As in Section~\ref{subsec:measure}, 
let $\rN_n^{\rop}$ be the unipotent subgroup of $\rGL_{n/F}$ 
consisting of all lower triangular matrices with all diagonal entries equal to $1$, which is the unipotent radical of $\rB_n^{\rop}$. 
In this subsection, we introduce another realisation of a
principal series representation of $\rGL_n(\mC)$, which is 
induced from a character of the lower triangular Borel subgroup $\rB_n^{\rop}(\mC)$.

Let $d =(d_1,d_2,\dots, d_n)\in \mZ^n$ and 
$\nu =(\nu_1,\nu_2,\dots ,\nu_n)\in \mC^n$. Let $\chi_{d,\nu}$ be a character of $\rT_n(\mC )$ defined by (\ref{eq:def_chi_dnu}). 
We define a representation 
$(\pi_{\rB_n^{\rop},d,\nu },I^\infty_{\rB_n^{\rop}}(d,\nu ))$ of $\rGL_n(\mC )$ by 
\begin{align}
I^\infty_{\rB_n^{\rop}}(d,\nu ):=\{f\in C^\infty (\rGL_n(\mC))
\mid f(xag)=\chi_{d,\nu -\rho_n}(a)f(g)\quad
\text{ for \;} x\in \rN_n^{\rop}(\mC ),\ a\in \rT_n(\mC) \text{ and } g\in \rGL_n(\mC )\} 
\end{align}
and $(\pi_{\rB_n^{\rop},d,\nu }(g)f)(h)=f(hg)$ for $g,h\in \rGL_n(\mC )$ and $f\in I^\infty_{\rB_n^{\rop}}(d,\nu )$. 
We equip $I^\infty_{\rB_n^{\rop}}(d,\nu )$ 
with the usual Fr\'{e}chet topology. 
Let $I^\infty_{\rB_n^{\rop}}(d,\nu )_{\rU (n)}$ denote
the subspace of $I^\infty_{\rB_n^{\rop}}(d,\nu )$ 
consisting of all $\rU (n)$-finite vectors. Define 
\begin{align*}
&I^\infty (d ):=\{f\in C^\infty (\rU (n))
\mid f(ak)=a_1^{d_1}a_2^{d_2}\dots a_n^{d_n}f(k)\quad 
 \text{for \;} a=\diag (a_1,a_2,\dots ,a_n)\in \rT_n(\mC)\cap \rU (n) \text{\; and \;} k\in \rU (n)\}, 
\end{align*}
and we equip this space with the usual Fr\'{e}chet topology. 
Using the Iwasawa decomposition $\rGL_n(\mC)=\rN_n^{\rop}(\mC )A_n\rU (n)$ and the definition of 
$I^\infty_{\rB_n^{\rop}}(d,\nu )$, 
we obtain a bijection $I^\infty (d )\ni f\mapsto f_\nu 
\in I^\infty_{\rB_n^{\rop}}(d,\nu )$ determined by 
\begin{align}
\label{eq:def_st_section}
f_\nu (xak)&=\chi^{A_n}_{\nu -\rho_n}(a)f(k)&
\text{for \;} x\in \rN_n^{\rop}(\mC ),\ a\in A_n \text{\; and \;} k\in \rU (n),
\end{align}
where $\chi^{A_n}_{\nu}$ is a character of $A_n$ defined by (\ref{eq:def_chi_A}). 
We regard $I^\infty (d )$ as a $\rGL_n(\mC)$-module 
via this identification, and 
let $\pi_{\rB_n^{\rop},\nu}$ denote the action of $\rGL_n(\mC)$ on $I^\infty (d )$ corresponding to 
$\pi_{\rB_n^{\rop},d,\nu }$, that is,  
\begin{align*}
(\pi_{\rB_n^{\rop},\nu}(g)f)(k)&=f_\nu (kg)&
\text{for \;} g\in \rGL_n(\mC),\ k\in \rU (n) \text{\; and \;} f\in I^\infty (d ).
\end{align*}
Here we note that $\pi_{\rB_n^{\rop},\nu}|_{\rU (n)}$ is 
the right translation and does not depend on $\nu$. 
Let $I^\infty (d )_{\rU (n)}$\label{page:Id1} denote
the subspace of $I^\infty (d )$ 
consisting of all $\rU (n)$-finite vectors. 
For $f\in I^\infty (d )$, we call the map 
$\mC^n\ni \nu \mapsto f_\nu \in C^\infty (\rGL_n(\mC))$ 
defined by (\ref{eq:def_st_section}) 
the {\em standard section} corresponding to $f$.

Let $d^{\rdom} =(d^{\rdom}_1,d^{\rdom}_2,\dots ,d^{\rdom}_n)$ be 
the unique element of $\Lambda_n\cap \{\sigma d \mid \sigma \in \gS_n\}$. 
By (\ref{eq:GT_act_wt}) and 
the Frobenius reciprocity law \cite[Theorem 1.14]{Knapp_002}, 
we know that 
$\tau_{d^{\rdom}}|_{\rU (n)}$ is the minimal $\rU (n)$-type of 
$\pi_{\rB_n^{\rop},d,\nu}$, and the space 
$\Hom_{\rU (n)}(V_{d^{\rdom}} ,I^\infty_{\rB_n^{\rop}}(d,\nu ))$ 
is one dimensional. 
We fix a $\rU (n)$-embedding 
$\rf_{\rB_n^{\rop},d,\nu}\colon 
V_{d^{\rdom}} \to I^\infty_{\rB_n^{\rop}}(d,\nu )$ 
so that $\rf_{\rB_n^{\rop},d,\nu}(\xi_{H(d )})(1_n)=1$, that is, 
for $\rv\in V_{d^{\rdom}}$, 
the function $\rf_{\rB_n^{\rop},d,\nu}(\rv)$ on 
$\rGL_n (\mC )$ is determined by 
\begin{align}
\label{eqn:def_minKtype2}
\rf_{\rB_n^{\rop},d,\nu}(\rv)(xak)&=\chi^{A_n}_{\nu -\rho_n}(a)
(\tau_{d^{\rdom}} (k)\rv,\xi_{H(d )})_{d^{\rdom}}
&(x\in \rN_n^{\rop}(\mC),\ a\in A_n,\ k\in \rU (n)).
\end{align}
Here we note that, 
for each $\rv\in V_{d^{\rdom}}$, 
the map $\mC^n\ni \nu \mapsto \rf_{\rB_n^{\rop},d,\nu}(\rv)
\in C^\infty (\rGL_n(\mC))$ is the standard section corresponding to 
$(\tau_{d^{\rdom}} (\cdot )\rv,\xi_{H(d )})_{d^{\rdom}} 
\in I^\infty (d )$.

By Lemma~\ref{lem:hom_dual_conj}, we have 
$V_{d^{\rdom}}\simeq V_{d^{\rdom \vee }}^\rconj $ as $\rU (n)$-modules 
with $d^{\rdom \vee }=
(-d^{\rdom}_n,-d^{\rdom}_{n-1},\dots ,-d^{\rdom}_1)\in \Lambda_n$. 
We also fix a $\rU (n)$-embedding 
$\rf_{\rB_n^{\rop},d,\nu}^\rconj \colon 
V_{d^{\rdom \vee }}^\rconj  \to I^\infty_{\rB_n^{\rop}}(d,\nu )$ 
so that $\rf_{\rB_n^{\rop},d,\nu}^\rconj (\xi_{H(-d )})(1_n)=1$, that is, 
for $\rv\in V_{d^{\rdom \vee }}^\rconj $, 
the function $\rf_{\rB_n^{\rop},d,\nu}^\rconj (\rv)$ on $\rGL_n (\mC )$ 
is determined by 
\begin{align}
\label{eqn:def_minKtype3}
\rf_{\rB_n^{\rop},d,\nu}^\rconj (\rv)(xak)&=\chi^{A_n}_{\nu -\rho_n}(a)
(\tau_{d^{\rdom \vee}}^\rconj (k)\rv,\xi_{H(-d )})_{d^{\rdom \vee}}&
(x\in \rN_n^{\rop}(\mC),\ a\in A_n,\ k\in \rU (n)).
\end{align}
In \cite{im}, because of a technical reason, 
the archimedean local zeta integral is calculated 
in terms of the two embeddings $\rf_{\rB_n^{\rop},d,\nu}$ and $\rf_{\rB_n^{\rop},d,\nu}^\rconj$ above. 
The following lemma gives the explicit relation between the two embeddings.

\begin{lem}
\label{lem:ps_rel_conj}
Retain the notation. Then we have 
\begin{align*}
\rf_{\rB_n^{\rop},d ,\nu}^\rconj (\rI_{d^\rdom}^\rconj (\rv ))
&=(-1)^{\sum_{i=1}^n(n-i)d_i}
\rf_{\rB_n^{\rop},d,\nu}(\rv )&
\text{for \;} \rv \in V_{d^{\rdom}}.
\end{align*}
\end{lem}
\begin{proof}
By $H(d )^\vee =H(-d )$, 
$\rrq (H(d ))=\sum_{i=1}^n(n-i)d_i$ and 
Lemma~\ref{lem:hom_dual_conj}, we have 
\begin{align*}
\rf_{\rB_n^{\rop},d ,\nu}^\rconj 
(\rI_{d^\rdom}^\rconj (\xi_{H(d )}))(1_n)
=(-1)^{\sum_{i=1}^n(n-i)d_i}
\rf_{\rB_n^{\rop},d ,\nu}^\rconj (\xi_{H(-d )})(1_n)
=(-1)^{\sum_{i=1}^n(n-i)d_i}.
\end{align*}
Hence, since $\rf_{\rB_n^{\rop},d ,\nu}^\rconj 
\circ \rI_{d^\rdom}^\rconj\in \Hom_{\rU (n)}
(V_{d^{\rdom}},I^\infty_{\rB_n^{\rop}}(d ,\nu ))=\mC \,
\rf_{\rB_n^{\rop},d ,\nu }$ and 
$\rf_{\rB_n^{\rop},d ,\nu }(\xi_{H(d )})(1_n)=1$, 
we obtain the assertion. 
\end{proof}

We define endomorphisms $\vartheta_{n}$ and $\vartheta_{n}'$ on $C^\infty (\rGL_n(\mC))$ by 
\begin{align*}
\vartheta_n(f)(g)&=f(w_ng),&
\vartheta_n'(f)(g)&=f(w_n{}^t\!g^{-1})&
\text{for \;} f\in C^\infty (\rGL_n(\mC)) \text{\; and \;}
g\in \rGL_n(\mC), 
\end{align*}
where $w_n$ is the anti-diagonal matrix of size $n$ 
with $1$ at all anti-diagonal entries. 
Note that both $\vartheta_{n}$ and $\vartheta_{n}'$ are involutions 
on $C^\infty (\rGL_n(\mC))$ since $w_n^2=1_n$. 
For $z = (z_1,z_2, \dots, z_n)\in \mC^n$, 
we set $z^*:=(z_n,z_{n-1}, \dots, z_1)$.

\begin{lem}
\label{lem:ps_rel_op_dual}
Retain the notation. 
\begin{enumerate}
\item \label{num:ps_rel_op}
The restriction 
$\vartheta_{n}|_{I^\infty_{\rB_n} (d^*,\nu^*)}$ is a $\rGL_n(\mC )$-isomorphism from 
$I^\infty_{\rB_n} (d^*,\nu^*)$ to $I^\infty_{\rB_n^{\rop}}(d ,\nu )$, 
and  its inverse map is given by $\vartheta_{n}|_{I^\infty_{\rB_n^{\rop}}(d ,\nu )}$. 
Moreover, we have 
\begin{align*}
\vartheta_n(\rf_{\rB_n,d^*,\nu^*}(\rv ))
&=(-1)^{\sum_{i=1}^n(i-1)d^{\rdom}_i}\rf_{\rB_n^{\rop},d ,\nu }(\rv )&
\text{for \;} \rv \in V_{d^{\rdom}}.
\end{align*}
Here $\rf_{\rB_n,d,\nu}\colon V_{d^{\rdom}} \to I^\infty_{\rB_n} (d,\nu )$ is the $\rU (n)$-embedding 
defined as $(\ref{eqn:def_minKtype1})$. 

\item \label{num:ps_rel_dual}
The restriction 
$\vartheta_{n}'|_{I^\infty_{\rB_n^{\rop}}(d ,\nu )}$ is 
a topological isomorphism from 
$I^\infty_{\rB_n^{\rop}}(d ,\nu )$ to $I^\infty_{\rB_n^{\rop}}(-d^*,-\nu^*)$ 
satisfying 
\begin{align}
\label{eq:ps_dual1}
\vartheta_n'(\pi_{\rB_n^{\rop},d ,\nu }(g)f)
&=\pi_{\rB_n^{\rop},-d^*,-\nu^*}({}^{t\!}g^{-1})\vartheta_n'(f)&
\text{for \;} g\in \rGL_n(\mC) \text{\; and \;} f\in I^\infty_{\rB_n^{\rop}}(d ,\nu ),
\end{align}
and its inverse map is given by $\vartheta_{n}'|_{I^\infty_{\rB_n^{\rop}}(-d^*,-\nu^*)}$. 
Moreover, we have 
\begin{align*}
\vartheta_n'(\rf_{\rB_n^{\rop},d ,\nu }(\rv ))
&=(-1)^{\sum_{i=1}^n(i-1)d^{\rdom}_i}
\rf_{\rB_n^{\rop},-d^*,-\nu^*}^\rconj (\rv )&
\text{for \;} \rv \in V_{d^{\rdom }}=V_{d^{\rdom }}^\rconj.
\end{align*}
Here we regard $\rv \in V_{d^{\rdom }}$ in the left hand side, and 
$\rv \in V_{d^{\rdom }}^\rconj $ in the right hand side. 
\end{enumerate}
\end{lem}
\begin{proof}
First, let us prove the statement \ref{num:ps_rel_op}. 
The former part of the statement \ref{num:ps_rel_op} is deduced from the following three facts: 
$w_n^2=1_n$, $w_nxw_n\in \rN_n(\mC)$ for $x\in \rN_n^{\rop}(\mC)$, and 
\begin{align*}
w_naw_n&=\diag (a_n,a_{n-1},\dots ,a_1)&
\text{for \;} a=\diag (a_1,a_2,\dots ,a_n)\in \rT_n(\mC).
\end{align*}
By Lemmas \ref{lem:extremal_vec_explicit} and \ref{lem:GTwn}, we have 
\begin{align*}
&\vartheta_n(\rf_{\rB_n,d^*,\nu^*}(\xi_{H(d )}))(1_n)
=( \tau_{d^{\rdom}} (w_n)\xi_{H(d )},\xi_{H(d^*)})_{d^{\rdom}} 
=(-1)^{\sum^n_{i=1}(i-1)d_i^{\rdom}}
(\xi_{H(d^*)},\xi_{H(d^*)})_{d^{\rdom}} 
=(-1)^{\sum^n_{i=1}(i-1)d_i^{\rdom}}.
\end{align*}
Hence, since $\vartheta_n\circ \rf_{\rB_n,d^*,\nu^*}\in \Hom_{\rU (n)}
(V_{d^{\rdom}},I^\infty_{\rB_n^{\rop}}(d ,\nu ))=\mC \,
\rf_{\rB_n^{\rop},d ,\nu }$ and 
$\rf_{\rB_n^{\rop},d ,\nu }(\xi_{H(d )})(1_n)=1$, 
we obtain the latter part of the statement \ref{num:ps_rel_op}. 

Next, let us prove the statement \ref{num:ps_rel_dual}. 
The former part of the statement \ref{num:ps_rel_dual} is deduced from the following three facts: 
$w_n^2=1_n$, $w_n{}^{t\!}x^{-1}w_n\in \rN_n^\rop (\mC)$ for $x\in \rN_n^{\rop}(\mC)$, and 
\begin{align*}
w_n{}^{\!t}a^{-1}w_n&=\diag (a_n^{-1},a_{n-1}^{-1},\dots ,a_1^{-1})&
\text{for \;} a=\diag (a_1,a_2,\dots ,a_n)\in \rT_n(\mC).
\end{align*}
Because of (\ref{eq:ps_dual1}) and the relation ${}^{t\!}k^{-1}=\overline{k}$ for $k \in \rU (n)$, 
we can regard $\vartheta_n'\circ \rf_{\rB_n,d ,\nu }$ as 
a $\rU (n)$-equivariant homomorphism from $V_{d^{\rdom}}^\rconj $ 
to $I^\infty_{\rB_n^{\rop}}(-d^*,-\nu^*)$. 
By Lemmas \ref{lem:extremal_vec_explicit} and \ref{lem:GTwn}, we have 
\begin{align*}
&\vartheta_n'(\rf_{\rB_n^{\rop},d ,\nu }(\xi_{H(d^* )}))(1_n)
=(\tau_{d^{\rdom}} (w_n)\xi_{H(d^*)},\xi_{H(d )})_{d^{\rdom}} 
=(-1)^{\sum^n_{i=1}(i-1)d_{i}^{\rdom}}
(\xi_{H(d )},\xi_{H(d )})_{d^{\rdom}}
=(-1)^{\sum^n_{i=1}(i-1)d_{i}^{\rdom}}.
\end{align*}
Therefore, since $\vartheta_n'\circ \rf_{\rB_n,d ,\nu }\in 
\Hom_{\rU (n)}(V_{d^{\rdom }}^\rconj ,
I^\infty_{\rB_n^{\rop}}(-d^*,-\nu^*))=\mC \,
\rf_{\rB_n^{\rop},-d^*,-\nu^*}^\rconj $ and 
$\rf_{\rB_n^{\rop},-d^*,-\nu^*}^\rconj (\xi_{H(d^* )})(1_n)=1$, 
we obtain the latter part of the statement \ref{num:ps_rel_dual}. 
\end{proof}

\subsection{Whittaker functions} 
\label{subsec:whittaker}

In this subsection, we recall basic facts on the Jacquet integrals, 
and rewrite Proposition~\ref{prop:arch_Wh} into the form using the realisation of a principal series representation introduced in the previous subsection. 

For $\varepsilon \in \{\pm \}$, take additive characters $\psi_{\varepsilon}$ and  
$\psi_{\varepsilon ,\rN_n}$ as in Section \ref{subsec:C_def_ps}. 
Let $d=(d_1,d_2,\dots,d_n)\in \mZ^n$. 
For 
$\nu =(\nu_1,\nu_2,\dots ,\nu_n)\in \mC^n$ satisfying
$\rRe (\nu_1)<\rRe (\nu_2)<\dots <\rRe (\nu_n)$, 
we define the {\em Jacquet integral}
$\cJ_{\varepsilon}^{\rop}\colon I^\infty_{\rB_n^{\rop}}(d,\nu )\to 
\mC$ by the convergent integral 
\begin{align*}
\cJ_{\varepsilon}^{\rop}(f)&=
\int_{\rN_n(\mC )}f(x)\psi_{-\varepsilon,\rN_n}(x)\,\rd x&
(f\in I^\infty_{\rB_n^{\rop}}(d,\nu ))
\end{align*}
with respect to the  Haar measure $\rd x$ on $\rN_{n}(\mC )$ normalised as 
in Section \ref{subsec:C_def_ps}, 
and set 
$\cJ_{\varepsilon ,d ,\nu }^{\rop}(f)
=\cJ_{\varepsilon }^{\rop}(f_{\nu} )$ ($f\in I^\infty (d )$), 
where $f_{\nu}$ is the standard section corresponding to $f$ (see also (\ref{eq:def_st_section})). 
By \cite[Theorem 15.4.1]{Wallach_003}, we know that 
$\cJ_{\varepsilon ,d ,\nu }^{\rop}(f)$ has the holomorphic 
continuation to whole $\nu \in \mC^n$ for every $f\in I^\infty (d )$, 
and the map $\mC^n\times I^\infty (d )\ni (\nu ,f)\mapsto 
\cJ_{\varepsilon ,d ,\nu }^{\rop}(f)\in \mC$ is continuous. 
Furthermore, this extends $\cJ_{\varepsilon ,d ,\nu }^{\rop}$ 
to all $\nu \in \mC^n$ as a nonzero continuous $\mC$-linear form 
on $I^\infty (d )$ satisfying 
\begin{align*}
\cJ_{\varepsilon ,d ,\nu }^{\rop}(\pi_{\rB_n^{\rop},\nu}(x)f)
&=\psi_{\varepsilon ,\rN_n} (x)\cJ_{\varepsilon ,d ,\nu }^{\rop}(f)&
\text{for \;} x\in \rN_n(\mC ),\ f\in I^\infty (d ).
\end{align*} 
We extends the Jacquet integral 
$\cJ_{\varepsilon}^\rop \colon 
I^\infty_{\rB_n^{\rop}}(d,\nu )\to \mC$ 
to whole $\nu \in \mC^n$ by 
\begin{align}
\label{eq:ext_Jac_op}
\cJ_{\varepsilon }^\rop (f)
&=\cJ_{\varepsilon, d ,\nu }^\rop (f|_{\rU (n)}) &
\text{for \;} f\in I^\infty_{\rB_n^{\rop}}(d,\nu ), 
\end{align}
and set 
\begin{align}
\label{eq:def_JacWhittaker}
\rW_{\varepsilon}^\rop (f)(g)
&=\cJ_{\varepsilon}^\rop (\pi_{\rB_n^{\rop},d,\nu }(g)f)&
\text{for \;} f\in I^\infty_{\rB_n^{\rop}}(d,\nu ) \text{\; and \;} g\in \rGL_n(\mC). 
\end{align}
We note that 
$\rW_{\varepsilon}^\rop (f_\nu )(g)
=\cJ_{\varepsilon ,d ,\nu }^\rop (\pi_{\rB_n^{\rop},\nu}(g)f)$ 
is an entire function in $\nu$ for each $g\in \rGL_n(\mC)$ and 
$f\in I^\infty (d )_{\rU (n)}$. 
For $\rv \in V_{d^{\rdom}}$, we define 
$\mW^{(\varepsilon),\rop}_{d, \nu} (\rv) $ by  
\begin{align}
\label{eq:def_normWhittaker_op}
\mW^{(\varepsilon),\rop}_{d, \nu} (\rv)
&:=(\varepsilon \sqrt{-1})^{\sum^{n}_{i=1}(n-i)d_i}
\Gamma_n^{\rop} (\nu;d )
\rW_{\varepsilon}^\rop (\rf_{\rB_n^\rop ,d,\nu}(\rv ))
\end{align}
where $\Gamma_n^{\rop} (\nu;d )$ is defined as
\begin{align*}
\Gamma_n^{\rop} (\nu;d ):=\prod_{1\leq i<j\leq n}
\Gamma \bigl(\nu_j-\nu_i+1+\tfrac{|d_j-d_i|}{2}\bigr).
\end{align*}

For $\nu =(\nu_1,\nu_2,\dots ,\nu_n)\in \mC^n$ satisfying  
$\rRe (\nu_1)<\rRe (\nu_2)<\dots <\rRe (\nu_n)$, we have 
\begin{align*}
\cJ_{\varepsilon }(f)
&=\cJ_{\varepsilon}^{\rop}(\vartheta_n(f))&
(f\in I^\infty_{\rB_n} (d^*,\nu^*)),
\end{align*}
where $\cJ_{\varepsilon}$ is defined by (\ref{eq:Jacquet_int}). 
This equality and (\ref{eq:ext_Jac_op}) extend the Jacquet integral 
$\cJ_{\varepsilon} \colon 
I^\infty_{\rB_n} (d^*,\nu^*)\to \mC$ 
to whole $\nu \in \mC^n$. By Lemma~\ref{lem:ps_rel_op_dual} \ref{num:ps_rel_op}, we see that $\rW_\varepsilon(f)$ defined in Section~\ref{subsec:C_def_ps} coincides with
$\rW_{\varepsilon}^\rop (\vartheta_n(f))$ 
 for $f\in I^\infty_{\rB_n} (d^*,\nu^*)$, and thus the space $\cW(\pi_{\rB_n,d^*,\nu^*},\psi_\varepsilon)$ defined as \eqref{eq:def_Whittaker_model} is described as
\begin{align*}
\cW (\pi_{\rB_n,d^*,\nu^*},\psi_\varepsilon )
&=\{\rW_{\varepsilon}^\rop (f)\mid f\in I^\infty_{\rB_n^{\rop}}(d,\nu )\}.
\end{align*}
By Lemma~\ref{lem:ps_rel_op_dual} \ref{num:ps_rel_op} and the obvious relation
$\Gamma_n(\nu^*;d^*)=\Gamma_n^{\rop} (\nu;d )$, we have 
\begin{align}
\label{eq:norWh_rel_op}
\mW^{(\varepsilon)}_{d^*, \nu^*}(\rv)
&=\mW^{(\varepsilon),\rop}_{d, \nu}(\rv)
&\text{for }\rv\in V_{d^{\rdom}}.
\end{align}
Hence Proposition~\ref{prop:arch_Wh} 
is equivalent to the following proposition.

\begin{prop}
\label{prop:arch_Wh_op}
Retain the notation. Let $g\in \rGL_n(\mC)$ and $\rv \in V_{d^{\rdom}}$. 
Then $\mW^{(\varepsilon),\rop}_{d, \nu} (\rv )(g)$ is 
an entire function in $\nu$ having the following symmetry$:$
\begin{align}
\label{eq:Sinv_Wh_op}
\mW^{(\varepsilon),\rop}_{d, \nu}(\rv )(g)
&=\mW^{(\varepsilon),\rop}_{\sigma d , \sigma \nu} (\rv )(g)&
\text{for \;} \sigma \in \gS_n. 
\end{align}
Furthermore, for each $\nu_0\in \mC^n$ such that $\pi_{\rB_n^\rop ,d,\nu_0}$ is irreducible, 
the meromorphic function $\Gamma_n^{\rop} (\nu;d )$ in $\nu$ 
is holomorphic at $\nu =\nu_0$. 
\end{prop}

We will prove Proposition~\ref{prop:arch_Wh_op} in Section \ref{subsec:god_Jacquet}, 
using Godement sections introduced in \cite{Jacquet_001}. 
For the proof of Proposition~\ref{prop:arch_Wh_op}, we prepare the following lemma.

\begin{lem}
\label{lem:dual_Wh}
Retain the notation. Then, for $g\in \rGL_n(\mC)$ and 
$\rv \in V_{d^{\rdom }}=V_{d^{\rdom }}^\rconj $, we have 
\begin{align*}
&\vartheta_n'(\mW^{(\varepsilon),\rop}_{d, \nu}(\rv ))(g)
=(-1)^{\sum_{i=1}^n(i-1)d^{\rdom}_i}
(\varepsilon \sqrt{-1})^{(n-1)\sum^{n}_{i=1}d_i}
\mW^{(-\varepsilon ),\rop}_{-d^*,-\nu^*}((\rI_{d^{\rdom \vee}}^\rconj )^{-1}(\rv ))(g).
\end{align*}
Here we regard $\rv \in V_{d^{\rdom }}$ in the left hand side, and 
$\rv \in V_{d^{\rdom }}^\rconj $ in the right hand side. 
\end{lem}
\begin{proof}
Let $g\in \rGL_n(\mC)$ and 
$\rv \in V_{d^{\rdom }}=V_{d^{\rdom }}^\rconj $. 
Because of the uniqueness of analytic continuation, 
it suffices to verify the statement when $\nu =(\nu_1,\nu_2,\dots ,\nu_n)\in \mC^n$ satisfies   
$\rRe (\nu_1)<\rRe (\nu_2)<\dots <\rRe (\nu_n)$. 
By direct computation, we have 
\begin{align*}
\vartheta_n'(\rW_{\varepsilon}^\rop (\rf_{\rB_n^\rop ,d,\nu}(\rv )))(g)
&=
\int_{\rN_n(\mC )}\rf_{\rB_n^\rop ,d,\nu}(\rv )(xw_n{}^{t\!}g^{-1})\psi_{-\varepsilon,\rN_n}(x)\,\rd x\\
&=
\int_{\rN_n(\mC )}\rf_{\rB_n^\rop ,d,\nu}(\rv )(w_n{}^{t\!}(xg)^{-1})
\psi_{\varepsilon,\rN_n}(x)\,\rd x\\
&=\rW_{-\varepsilon}^\rop (\vartheta_n'(\rf_{\rB_n^\rop ,d,\nu}(\rv )))(g).
\end{align*}
Here the second equality follows from the substitution 
$x\mapsto w_n{}^{t\!}x^{-1}w_n$ and the equality
$\psi_{-\varepsilon,\rN_n}(w_n{}^{t\!}x^{-1}w_n)=\psi_{\varepsilon,\rN_n}(x)$. 
By Lemmas \ref{lem:ps_rel_op_dual} \ref{num:ps_rel_dual} and \ref{lem:ps_rel_conj}, we have 
\begin{align*}
\vartheta_n'(\rf_{\rB_n^\rop ,d,\nu}(\rv ))
=(-1)^{\sum_{i=1}^n(i-1)d^{\rdom}_i}\rf_{\rB_n^{\rop},-d^*,-\nu^*}^\rconj (\rv )
=(-1)^{\sum_{i=1}^n(i-1)(d^{\rdom}_i-d_i)}
\rf_{\rB_n^{\rop},-d^*,-\nu^*}((\rI_{d^{\rdom \vee}}^\rconj )^{-1}(\rv )).
\end{align*}
These equalities, combined with the relations
$\Gamma_n^{\rop} (\nu;d )
=\Gamma_n^{\rop} (-\nu^*;-d^* )$ and 
\[
(\varepsilon \sqrt{-1})^{\sum^{n}_{i=1}(n-i)d_i}
=
(-1)^{\sum^{n}_{i=1}(i-1)d_i}
(\varepsilon \sqrt{-1})^{(n-1)\sum^{n}_{i=1}d_i}
(-\varepsilon \sqrt{-1})^{-\sum^{n}_{i=1}(n-i)d_{n+1-i}}, 
\]
imply the assertion. 
\end{proof}

\subsection{Standard Schwartz functions}
\label{subsec:def_schwartz}

Assume that $n>1$. In this subsection, we introduce some Schwartz functions on $\rM_{n-1,n}(\mC )$, 
which play important roles on our proof of Proposition~\ref{prop:arch_Wh_op}.

Let $C(\rM_{n-1,n}(\mC ))$ be the space of continuous 
functions on $\rM_{n-1,n}(\mC )$. 
We define actions of $\rGL_{n-1}(\mC )$ and $\rGL_{n}(\mC )$ on 
$C(\rM_{n-1,n}(\mC ))$ by 
\begin{align*}
(L(g)f)(z)&=f(g^{-1}z) \quad \text{\; and\;} \quad (R(h)f)(z)=f(zh) \\
&\qquad \text{for \;} g\in \rGL_{n-1}(\mC ),\ h\in \rGL_{n}(\mC ),\ f\in C(\rM_{n-1,n}(\mC )) \text{\; and \;}
z\in \rM_{n-1,n}(\mC ).
\end{align*}
Let $\cS (\rM_{n-1,n}(\mC ))$ be the space of Schwartz functions on 
$\rM_{n-1,n}(\mC )$, and $\cS_0(\rM_{n-1,n}(\mC ))$ the subspace of 
$\cS (\rM_{n-1,n}(\mC ))$ consisting of all 
functions $\phi$ of the form 
$\phi (z)=P(z,\overline{z})\me_{(n-1,n)}(z)$, where $\me_{(n-1,n)}(z)$ is defined as
\begin{align*}
\me_{(n-1,n)} (z)&:=\exp (-2\pi \mathrm{Tr}({}^t\overline{z}z))&
\text{for \;} z\in \rM_{n-1,n}(\mC )
\end{align*}
and $P$ is a polynomial function on $\rM_{n-1,n}(\mC )\times \rM_{n-1,n}(\mC )$. 
We call elements of $\cS_0(\rM_{n-1,n}(\mC ))$ 
{\em standard Schwartz functions} on $\rM_{n-1,n}(\mC )$. 
Since $\me_{(n-1,n)}$ is $\rU (n-1)\times \rU (n)$-invariant, 
the space $\cS_0 (\rM_{n-1,n}(\mC ))$ is closed under 
the action $L\boxtimes R$ of $\rU (n-1)\times \rU (n)$, and all elements 
of $\cS_0 (\rM_{n-1,n}(\mC ))$ are $\rU (n-1)\times \rU (n)$-finite.

For $\gamma =(\gamma_1,\gamma_2,\dots ,\gamma_n)\in \mZ^n$, we define 
$\gamma^\rpos =(\gamma_1^\rpos ,\gamma_2^\rpos ,\dots ,\gamma_n^\rpos )$ and 
$\gamma^\rneg =(\gamma_1^\rneg ,\gamma_2^\rneg ,\dots ,\gamma_n^\rneg )$ by 
\begin{align*}
&\gamma_i^\rpos =\left\{\begin{array}{ll}
\gamma_i&\text{if }\ \gamma_i\geq 0,\\
0&\text{otherwise},
\end{array}\right.&
&\gamma_i^\rneg =\left\{\begin{array}{ll}
\gamma_i&\text{if }\ \gamma_i\leq 0,\\
0&\text{otherwise}.
\end{array}\right.&
\end{align*}
Then 
$\gamma =\gamma^\rpos +\gamma^\rneg $ holds
for $\gamma \in \mZ^n$. As in Section $\ref{subsec:GT_basis}$, for $\lambda \in \Lambda_n$, 
let $\cE_\lambda :=\{\sigma \lambda \mid \sigma \in \gS_n\}$ be 
the set of extremal weights of $\tau_\lambda$, 
and $H(\gamma )$ the unique element of $\rG (\lambda)$ whose weight is  
$\gamma \in \cE_\lambda $.

\begin{lem}
\label{lem:std_schwartz}
For  $\mu =(\mu_1,\mu_2,\dots ,\mu_{n-1})\in \Lambda_{n-1}$, set 
$\lambda =(\lambda_1,\lambda_2,\dots ,\lambda_{n})=(\mu^\rpos ,0)+(0,\mu^\rneg )$. 
Then there is a $\rU (n-1)\times \rU (n)$-equivariant homomorphism 
$\Phi_{\mu}\colon V_{\mu}^\rconj \boxtimes_\mC V_{\lambda}
\to \cS_0(\rM_{n-1,n}(\mC ))$ 
such that, 
for $g\in \rGL_n(\mC)$ and $\gamma \in \cE_{\mu}$,  we have
\begin{equation}
\label{eq:std_schwartz}
\begin{aligned}
\Phi_{\mu }&(\zeta_{H(\gamma )}\boxtimes 
\zeta_{H((\gamma,0))})((1_{n-1},O_{n-1,1})g)\\
& =(\tau_{\lambda^\rpos }(g)\zeta_{H((\gamma^\rpos ,0))},
\zeta_{H((\gamma^\rpos ,0))})_{\lambda^\rpos } 
(\tau_{\lambda^{\rneg \vee}}(\overline{g})\zeta_{H((-\gamma^\rneg ,0))},
\zeta_{H((-\gamma^\rneg ,0))})_{\lambda^{\rneg \vee}}
\me_{(n-1,n)}((1_{n-1},O_{n-1,1})g).
\end{aligned}
\end{equation}
When $n=2$, let $(z_1,z_2)\in \rM_{1,2}(\mC)$ and 
$M=\left(\begin{smallmatrix}\lambda_1\ \lambda_2\\ m\end{smallmatrix}\right)\in \rG (\lambda )$. Then we have 
\begin{align}
\label{eq:std_schwartz_n=2}
&\Phi_{\mu_1}(\xi_{\mu_1}
\boxtimes \xi_{M})((z_1,z_2))
=\left\{\begin{array}{ll}
z_1^{m}z_2^{\mu_1-m}\exp (-2\pi (|z_1|^2+|z_2|^2))&\text{if $\mu_1\geq 0$},\\[1mm]
(-1)^{m-\mu_1}\overline{z_1}^{-m}
\overline{z_2}^{m-\mu_1}\exp (-2\pi (|z_1|^2+|z_2|^2))&\text{if $\mu_1\leq 0$}.
\end{array}\right.
\end{align}
\end{lem}

\begin{proof}
By \cite[Lemma 4.9]{im} combined with Remark \ref{rem:cpx_conj_rep}, 
for $\mu\in \Lambda_{n-1}\cap \mN_0^{n-1}$, 
there exists a $\rU (n-1)\times \rGL_n(\mC )$-equivariant homomorphism 
$\rP_{\mu} \colon V_{\mu}^\rconj \boxtimes_\mC V_{(\mu,0)}
\to C(\rM_{n-1,n}(\mC ))$ such that, for each $M\in \rG(\mu )$ and $N\in \rG ((\mu ,0))$, 
the image $\rP_{\mu } (\zeta_{M}\boxtimes \zeta_N)$ is a polynomial function  
characterised by  
\begin{align*}
\rP_{\mu } (\zeta_{M}\boxtimes \zeta_N)
((1_{n-1},O_{n-1,1})g)
&=(\tau_{(\mu,0)}(g)\zeta_N,\zeta_{M[(\mu,0)]})_{(\mu,0)}&
(g\in \rGL_n(\mC)).
\end{align*}
Because of the definition of the complex conjugate representations, 
we can define 
a $\rU (n-1)\times \rGL_n(\mC )$-equivariant homomorphism
$\rP_{\mu}^\rconj  \colon V_{\mu} \boxtimes_\mC V_{(\mu,0)}^\rconj 
\to C(\rM_{n-1,n}(\mC ))$ for $\mu \in \Lambda_{n-1}\cap \mN_0^{n-1}$ by 
\begin{align*}
\rP_{\mu}^\rconj (\rv_1\boxtimes \rv_2)(z)&=
\rP_{\mu} (\rv_1\boxtimes \rv_2)(\overline{z})&
(\rv_1\in V_{\mu}=V_{\mu}^\rconj ,\ \rv_2\in V_{(\mu,0)}^\rconj =V_{(\mu,0)}).
\end{align*}
We can also define 
a $\rGL_n(\mC )$-equivariant homomorphism
$(\rI^{\lambda, \lambda'}_{\lambda +\lambda'})^\rconj 
\colon V_{\lambda +\lambda'}^\rconj \to 
V_\lambda^\rconj  \otimes_{\mC }V_{\lambda'}^\rconj $ for $\lambda,\lambda'\in \Lambda_n$  by 
\begin{align*}
(\rI^{\lambda, \lambda'}_{\lambda +\lambda'})^\rconj (\rv )
&=\rI^{\lambda, \lambda'}_{\lambda +\lambda'}(\rv )&
( \rv \in V_{\lambda +\lambda'}^\rconj =V_{\lambda +\lambda'}).
\end{align*}
When $n=2$, take $\mu_1\in \mN_0$ and set $N=\left(\begin{smallmatrix}\mu_1\ 0\\ m\end{smallmatrix}\right)\in \rG ((\mu_1,0))$. Then we can calculate $\rP_{\mu_1}(\zeta_{\mu_1}\boxtimes \zeta_N)$ explicitly by using Corollary~\ref{cor:matcoeff_n=2} as
\begin{align}
\label{eq:pf_std_schwartz_n=2}
\rP_{\mu_1}(\zeta_{\mu_1}\boxtimes \zeta_{N})
((z_1,z_2))&=z_{1}^{m}z_{2}^{\mu_1-m}&
\text{for \;} (z_1,z_2)\in \rM_{1,2}(\mC). 
\end{align}

For general $\mu \in \Lambda_{n-1}$, set 
$\lambda =(\mu^\rpos ,0)+(0,\mu^\rneg )$. Note that 
$\lambda^\rpos =(\mu^\rpos ,0)$ and $\lambda^\rneg =(0,\mu^\rneg )$ hold by definition. 
Let 
$\Phi_{\mu}\colon V_{\mu}^\rconj \boxtimes_\mC V_{\lambda}
\to \cS_0(\rM_{n-1,n}(\mC ))$ be the composition of the 
$\rU (n-1)\times \rU (n)$-equivariant homomorphisms
\begin{align*}
V_\mu^\rconj \boxtimes_\mC V_\lambda&
\xrightarrow[\hphantom{--------------------}]{\ \bigl(\rI_{\mu }^{\mu^\rpos ,\mu^\rneg}\bigr)^\rconj \boxtimes \rI_{\lambda }^{\lambda^\rpos ,\lambda^\rneg }\ }
(V_{\mu^\rpos}^\rconj \otimes_\mC V_{\mu^{\rneg} }^\rconj) \boxtimes_\mC 
(V_{\lambda^\rpos} \otimes_\mC V_{\lambda^{\rneg} })\\
& \xrightarrow[\hphantom{--------------------}]{\ \bigl(\id \otimes (\rI^\rconj_{\mu^{\rneg \vee }})^{-1}\bigr)
\boxtimes (\id \otimes \rI^{\rconj}_{\lambda^\rneg})\ }
(V_{\mu^\rpos}^\rconj \otimes_\mC V_{\mu^{\rneg \vee} })\boxtimes_\mC 
(V_{\lambda^\rpos} \otimes_\mC V_{\lambda^{\rneg \vee} }^\rconj )
\end{align*}
with
\begin{align*}
&(V_{\mu^\rpos}^\rconj \otimes_\mC 
V_{\mu^{\rneg \vee} })\boxtimes_\mC (V_{\lambda^\rpos} 
\otimes_\mC V_{\lambda^{\rneg \vee} }^\rconj )
\ni (\rv_1\otimes \rv_2)\boxtimes (\rv_3\otimes \rv_4)\\
&\hspace{50mm}
\mapsto (-1)^{\ell (\mu^{\rneg})}
\rP_{\mu^\rpos}(\rv_1\boxtimes \rv_3)(z)
\rP_{\mu^{\rneg \vee }}^\rconj (\rv_2\boxtimes \rv_4)(z)
\me_{(n-1,n)} (z)\in \cS_0(\rM_{n-1,n}(\mC )).
\end{align*}
The assertion follows from the definition of $\Phi_{\mu }$, combined with
Lemmas~\ref{lem:extremal_vec_explicit}, \ref{lem:extremal_injector} and (\ref{eq:pf_std_schwartz_n=2}). 
\end{proof}

\subsection{Godement sections}
\label{subsec:god_Jacquet}

In this subsection, we recall the notion of Godement sections and give the  proof of Proposition~\ref{prop:arch_Wh_op} by using them. {\em Godement sections} are meromorphic sections for principal series representations constructed in a specific way, 
which play important roles in \cite{Jacquet_001} and \cite{im}.

Assume that $n>1$, and let us 
take $d =(d_1,d_2,\dots ,d_n)\in \mZ^n$ and 
$\nu =(\nu_1,\nu_2,\dots ,\nu_n)\in \mC^n$. 
We then set $\widehat{d} =(d_1,d_2,\dots ,d_{n-1})\in \mZ^{n-1}$ 
and $\widehat{\nu} =(\nu_1,\nu_2,\dots ,\nu_{n-1})\in \mC^{n-1}$. 
For $f\in I^\infty (\widehat{d})_{\rU (n-1)}$, let $f_{\widehat{\nu}}$ denote the standard section corresponding to $f$. 
In addition we choose a standard Schwartz function $\phi \in \cS_0 (\rM_{n-1,n}(\mC ))$. 
Then, if $\nu$ satisfies 
$\mathrm{Re}(\nu_n-\nu_i)>-1$ for each $1\leq i\leq n-1$, 
we define the Godement section 
$\rg^+_{d_n,\nu_n}(f_{\widehat{\nu}},\phi )$ 
by the convergent integral 
\begin{equation*}
\rg^+_{d_n,\nu_n}(f_{\widehat{\nu}},\phi )(g)=
\left(\frac{\det g}{\lvert \det g\rvert}\right)^{\!d_n}\lvert\det g\rvert^{2\nu_n+n-1}
\int_{\rGL_{n-1}(\mC )}
\phi((h,O_{n-1,1})g)f_{\widehat{\nu}}(h^{-1})
\left(\frac{\det h}{\lvert \det h\rvert}\right)^{\!d_n}\lvert \det h\rvert^{2\nu_n+n}\,\rd h
\end{equation*}
for $g\in \rGL_n(\mC)$. Here $\rd h$ is the Haar measure on $\rGL_{n-1}(\mC )$ 
normalised as (\ref{eq:GL_measure}) in Section \ref{subsec:measure}. 
Jacquet shows in \cite[Proposition 7.1]{Jacquet_001} that $\rg^+_{d_n,\nu_n}(f_{\widehat{\nu}},\phi )(g)$ 
extends to a meromorphic function in $\nu_n$ defined on $\mC$, and 
$\rg^+_{d_n,\nu_n}(f_{\widehat{\nu}},\phi )$ is an element of 
$I^\infty_{\rB_n^{\rop}}(d,\nu )_{\rU (n)}$ if it is defined. 
For the embedding $\rf_{\rB_n^{\rop},d,\nu}$ defined by (\ref{eqn:def_minKtype2}), 
we first prove the following lemma, which is a generalisation of \cite[Lemma 5.2 (1)]{im}.

\begin{lem}
\label{lem:god_explicit}
Retain the notation. Them, for $M\in \rG (d^\rdom )$ and $N\in \rG (\widehat{d}^\rdom )$, we have 
\begin{equation*}
\rg^+_{d_n,\nu_n}\bigl(\,\rf_{\rB_{n-1}^{\rop},\widehat{d},\widehat{\nu}}(\zeta_{N}),
\Phi_{\widehat{d}^\rdom -d_n}(\zeta_{N-d_n}\boxtimes \zeta_{M-d_n})\,\bigr)
=\frac{1}{\dim_{\mC} V_{\widehat{d}^\rdom }}
\left(\prod_{i=1}^{n-1}\Gamma_{\mC}\bigl(
\nu_n-\nu_i+1+\tfrac{\lvert d_n-d_i\rvert}{2}\bigr)\right)
\rf_{\rB_n^{\rop},d,\nu}(\zeta_M).
\end{equation*}
Here we note that the right hand side does not depend on the choice of $N$. 
\end{lem}

\begin{proof}
Take $N\in \rG (\widehat{d}^\rdom )$ and define a 
$\mC$-linear map $\rg_N \colon V_{d^\rdom}
\to I^\infty_{\rB_n^{\rop}}(d,\nu )$ by 
\begin{align*}
\rg_N (\zeta_{M})&=
\rg^+_{d_n,\nu_n}\bigl(\,\rf_{\rB_{n-1}^{\rop},\widehat{d},\widehat{\nu}}(\zeta_{N}),
\Phi_{\widehat{d}^\rdom -d_n}(\zeta_{N-d_n}\boxtimes \zeta_{M-d_n})\,\bigr)&
\text{for \;}M\in \rG (d^\rdom ).
\end{align*}
The map $\rg_N$ is then  $\rU (n)$-equivariant by (\ref{eq:detl_shift}), Lemma~\ref{lem:detl_shift} and \cite[Lemma 3.1]{im}. 
Since $\Hom_{\rU (n)}(V_{d^\rdom}, I^\infty_{\rB_n^{\rop}}(d,\nu ))$ is one dimensional, 
there is a nonzero constant $c_N$ such that 
$\rg_N =c_N\rf_{\rB_n^{\rop},d,\nu}$. 
Let us calculate $c_N$. It is described as
\begin{align*}
c_N&=c_N\rf_{\rB_n^{\rop},d,\nu}(\zeta_{H(d )})(1_n)=\rg_N(\zeta_{H(d )})(1_n)\\
&=\int_{\rGL_{n-1}(\mC )}
\Phi_{\widehat{d}^\rdom -d_n}(\zeta_{N-d_n}\boxtimes \zeta_{H(d -d_n)})((h,O_{n-1,1}))
\rf_{\rB_{n-1}^{\rop},\widehat{d},\widehat{\nu}}(\zeta_{N})(h^{-1})
\left(\frac{\det h}{\lvert\det h\rvert}\right)^{\!d_n}\lvert\det h\rvert^{2\nu_n+n}\,\rd h
\end{align*}
by definition. We calculate the integrand as follows. Write $h=kxa\in \rGL_{n-1}(\mC)$ with $k\in \rU (n-1)$, 
$x\in \rN_{n-1}^\rop (\mC)$ and 
$a\in A_{n-1}$ via the Iwasawa decomposition. Then we have 
\begin{align*}
\left(\frac{\det h}{\lvert\det h\rvert}\right)^{\!d_n}
\rf_{\rB_{n-1}^{\rop},\widehat{d},\widehat{\nu}}(\zeta_{N})(h^{-1})
=\chi^{A_{n-1}}_{\widehat{\nu}-\rho_{n-1}}(a^{-1})
\overline{\bigl( \tau_{\widehat{d}^\rdom -d_n}(k)\zeta_{H(\widehat{d}-d_n)},\,
\zeta_{N-d_n}\bigr)_{\widehat{d}^\rdom -d_n} }
\end{align*}
by Lemma~\ref{lem:detl_shift}, and the Schwartz function is calculated as
\begin{align*}
\Phi_{\widehat{d}^\rdom -d_n}&(\zeta_{N-d_n}\boxtimes \zeta_{H(d -d_n)})((h,O_{n-1,1}))
=\Phi_{\widehat{d}^\rdom -d_n}(\tau_{\widehat{d}^\rdom -d_n}^\rconj (k^{-1})\zeta_{N-d_n}
\boxtimes \zeta_{H(d -d_n)})((xa,O_{n-1,1}))\\
& \qquad =\sum_{M\in \rG (\widehat{d}^\rdom -d_n)}
( \tau_{\widehat{d}^\rdom -d_n}^\rconj (k^{-1})\zeta_{N-d_n},
\zeta_{M})_{\widehat{d}^\rdom -d_n} 
\Phi_{\widehat{d}^\rdom -d_n}(\zeta_{M}\boxtimes \zeta_{H(d -d_n)})
((xa,O_{n-1,1})).
\end{align*}
Moreover, for $M\in \rG (\widehat{d}^\rdom -d_n)$ and $k\in \rU (n-1)$, we have 
\begin{align*}
(\tau_{\widehat{d}^\rdom -d_n}^\rconj (k^{-1})\zeta_{N-d_n},\, \zeta_{M})_{\widehat{d}^\rdom -d_n}
=\overline{(\tau_{\widehat{d}^\rdom -d_n}^\rconj (k)\zeta_{M},\, \zeta_{N-d_n})_{\widehat{d}^\rdom -d_n}}
=(\tau_{\widehat{d}^\rdom -d_n}(k)\zeta_{M},\,\zeta_{N-d_n})_{\widehat{d}^\rdom -d_n}
\end{align*}
by Lemma~\ref{lem:xiM_gen_Vlambda} and the definitions of 
$\tau_\lambda$ and $\tau_\lambda^\rconj$ for $\lambda \in \Lambda_n$. 
By these equalities and 
Schur's orthogonality \cite[Corollary 1.10]{Knapp_002}, combined with the equality 
$\dim_{\mC} V_{\widehat{d}^\rdom -d_n}=\dim_{\mC} V_{\widehat{d}^\rdom }$, we have 
\begin{align*}
c_N&=\sum_{M\in \rG (\widehat{d}^\rdom -d_n)}
\int_{A_{n-1}}\int_{\rN_{n-1}^\rop (\mC)}
\Phi_{\widehat{d}^\rdom -d_n}(\zeta_{M}\boxtimes \zeta_{H(d -d_n)})((xa,O_{n-1,1}))\\
&\hspace{10mm}\times \left(
\int_{\rU (n-1)} \!\!
(\tau_{\widehat{d}^\rdom -d_n}(k)\zeta_{M},\zeta_{N-d_n})_{\widehat{d}^\rdom -d_n}
\overline{\bigl( \tau_{\widehat{d}^\rdom -d_n}(k)\zeta_{H(\widehat{d}-d_n)},\, 
\zeta_{N-d_n}\bigr)_{\widehat{d}^\rdom -d_n} }\,dk\right) \\
&\hspace{100mm}\times \chi^{A_{n-1}}_{\widehat{\nu}-\rho_{n-1}}(a^{-1})
\lvert\det a\rvert^{2\nu_n+n}\,\rd x\,\rd a\\
&=\frac{1}{\dim_{\mC} V_{\widehat{d}^\rdom }}
\int_{A_{n-1}}\int_{\rN_{n-1}^\rop (\mC)}
\Phi_{\widehat{d}^\rdom -d_n}(\overline{\zeta_{H(\widehat{d}-d_n)}}\boxtimes 
\zeta_{H(d -d_n)})((xa,O_{n-1,1}))
\chi^{A_{n-1}}_{\widehat{\nu}-\rho_{n-1}}(a^{-1})\lvert \det a\rvert^{2\nu_n+n}
\,\rd x\,\rd a.
\end{align*}
By \cite[Lemmas 2.3 and 5.1]{im} and (\ref{eq:std_schwartz}), we have 
\begin{align*}
\Phi_{\widehat{d}^\rdom -d_n}& \left( \zeta_{H(\widehat{d}-d_n)}\boxtimes \zeta_{H(d -d_n)}\right)((xa,O_{n-1,1}) )\\
&=\left(  \tau_{(\widehat{d}^\rdom -d_n)^\rpos }(xa)
            \zeta_{H((\widehat{d}-d_n)^\rpos )},\, \zeta_{H((\widehat{d}-d_n)^\rpos )} 
       \right)_{(\widehat{d}^\rdom -d_n)^\rpos}\\
&\hspace{5mm}\times 
   \left(  \tau_{(\widehat{d}^\rdom -d_n)^{\rneg \vee}}(xa)\zeta_{H(-(\widehat{d}-d_n)^\rneg )},\,
          \zeta_{H(-(\widehat{d}-d_n)^\rneg )}
      \right)_{(\widehat{d}^\rdom -d_n)^{\rneg \vee}} 
\me_{(n-1)}(xa)\\
&=\me_{(n-1)}(xa)\prod_{i=1}^{n-1}a_i^{\lvert d_n-d_i\rvert} \qquad \text{for \; } a=\diag (a_1,a_2,\dots ,a_{n-1})\in A_{n-1},
\end{align*}
where $\me_{(n-1)}(z)$ is defined as $\me_{(n-1)} (z)=\exp (-2\pi \mathrm{Tr}({}^t\overline{z}z))$ for $z\in \rM_{n-1}(\mC )$. 
Therefore we have 
\begin{align*}
c_N&=\frac{1}{\dim_\mC V_{\widehat{d}^\rdom }}
\int_{A_{n-1}}\int_{U_{n-1}}
\me_{(n-1)}(xa)\prod_{i=1}^na_i^{\lvert d_n-d_i\rvert}
\chi^{A_{n-1}}_{\widehat{\nu}-\rho_{n-1}}(a^{-1})\lvert \det a\rvert^{2\nu_n+n}
\,\rd x\,\rd a\\
&=\frac{1}{\dim_\mC V_{\widehat{d}^\rdom }}
\int_{A_{n-1}}\me_{(n-1)}(a)
\prod_{i=1}^na_i^{\lvert d_n-d_i\rvert}
\chi^{A_{n-1}}_{\widehat{\nu}}(a^{-1})\lvert\det a\rvert^{2\nu_n+2}
\,\rd a\\
&=\frac{1}{\dim_\mC V_{\widehat{d}^\rdom }}\prod_{i=1}^{n-1}\int_0^\infty 
\exp (-2\pi a_i^2)a_i^{2(\nu_n-\nu_i+1)+\lvert d_n-d_i\rvert}\,\frac{4\,\rd a_i}{a_i}\\
&=\frac{1}{\dim_\mC V_{\widehat{d}^\rdom }}\prod_{i=1}^{n-1}
\Gamma_\mC \bigl(\nu_n-\nu_i+1+\tfrac{\lvert d_n-d_i\rvert}{2}\bigr)
\end{align*}
 by \cite[Lemma 5.1]{im}, which completes the proof. 
\end{proof}

Fix a signature $\varepsilon \in \{\pm \}$. 
In \cite[Section 7.2]{Jacquet_001}, 
Jacquet gives convenient integral representations of 
Whittaker functions. If $\nu$ satisfies $\rRe (\nu_1)<\rRe (\nu_2)<\dots <\rRe (\nu_n)$, we have 
\begin{equation}\label{eqn:W_god+}
\begin{aligned}
\rW_{\varepsilon}^{\rop}
\bigl(\rg^+_{d_n,\nu_n}(f_{\widehat{\nu}},\phi )\bigr)(g)=
&\,\left(\frac{\det g}{\lvert \det g\rvert}\right)^{\!d_n}\lvert \det g\rvert^{2\nu_n+n-1}
\int_{\rGL_{n-1}(\mC )}\left(\int_{\rM_{n-1,1}(\mC )}
\phi\left(\left(h,hz\right)g\right)
\psi_{-\varepsilon }(e_{n-1}z)\,\rd_\mC z\right)\\
&\hspace{60mm} \times 
\rW_{\varepsilon}^{\rop}(f_{\widehat{\nu}})(h^{-1})
\left(\frac{\det h}{\lvert \det h\rvert}\right)^{\!d_n}\lvert \det h\rvert^{2\nu_n+n}
\,\rd h
\end{aligned}
\end{equation}
for each $g\in \rGL_n(\mC)$, where $e_{n-1}=(O_{1,n-2},1)\in \rM_{1,n-1}(\mC )$ and $\rd_\mC z$ is 
the measure on $\rM_{n-1,1}(\mC )$ defined by 
\begin{align*}
\rd_\mC z&=\prod_{i=1}^{n-1}\rd_\mC z_{i,1}&
\text{for \; } z=(z_{i,1})_{1\leq i\leq n-1}\in \rM_{n-1,1}(\mC ). 
\end{align*}
Jacquet also shows that 
the right hand side of (\ref{eqn:W_god+}) 
converges absolutely for all $\nu \in \mC^n$, 
and thus defines an entire function in $\nu$ 
(see \cite[Proposition 7.2]{Jacquet_001}). 
Hence the equality holds for all $\nu $. 
Consequently, we obtain the following integral representation
of the normalised Whittaker function 
$\mW^{(\varepsilon),\rop}_{d, \nu} (\rv)$ defined by (\ref{eq:def_normWhittaker_op}). 

\begin{cor}
\label{cor:god_explicit_1}
Retain the notation in Lemma~$\ref{lem:god_explicit}$. 
Then we have  
\begin{equation}
\label{eq:Wh_recursive}
\begin{aligned}
\mW^{(\varepsilon),\rop}_{d, \nu}(\zeta_M)(g)
& =(\varepsilon \sqrt{-1})^{\sum^{n-1}_{i=1}d_i}\bigl(\dim V_{\widehat{d}^\rdom}\,\bigr)
\left(\frac{\det g}{\lvert \det g\rvert}\right)^{\!d_n}\lvert \det g\rvert^{2\nu_n+n-1}\\
&\qquad \times 
\int_{\rGL_{n-1}(\mC )}\left(\int_{\rM_{n-1,1}(\mC )}
\Phi_{\widehat{d}^\rdom -d_n}(\zeta_{N-d_n}\boxtimes \zeta_{M-d_n})\left(\left(h,hz\right)g\right)
\psi_{-\varepsilon }(e_{n-1}z)\,\rd_\mC z\right)\\
&\hspace*{60mm} \times 
\mW^{(\varepsilon),\rop}_{\widehat{d}, \widehat{\nu}}(\zeta_{N})
(h^{-1})\left(\frac{\det h}{\lvert \det h\rvert}\right)^{\!d_n}\lvert \det h\rvert^{2\nu_n+n}
\,\rd h
\end{aligned}
\end{equation}
for $M\in \rG (d^\rdom )$, $N\in \rG (\widehat{d}^\rdom )$ and $g\in \rGL_n(\mC)$. 
\end{cor}

\begin{proof}
The assertion follows immediately from (\ref{eqn:W_god+}) and 
Lemma~\ref{lem:god_explicit}. 
\end{proof}

\begin{lem}
\label{lem:Wh_explicit_n=2}
 Set $n=2$, and let 
$d =(d_1,d_2)\in \mZ^2$, $\nu =(\nu_1,\nu_2)\in \mC^2$ 
and $a=\diag (a_1,a_2)\in A_2$. 
Furthermore take $M=\left( \begin{smallmatrix} d^{\rdom}_1\ \,d^{\rdom}_2 \\ m \end{smallmatrix}\right)
\in \rG (d^{\rdom})$ for a unique element $d^{\rdom} =(d^{\rdom}_1,d^{\rdom}_2)$ of $\Lambda_2\cap \{(d_1,d_2),(d_2,d_1)\}$. Then $\mW^{(\varepsilon),\rop}_{d, \nu}(\xi_M)(a)$ is calculated as 
\begin{equation}
\label{eq:Wh_explicit1_n=2}
\begin{aligned}
\mW^{(\varepsilon),\rop}_{d, \nu}(\xi_M)(a)
&=(\varepsilon \sqrt{-1})^{m}a_1^{2\nu_2+1+m-d_2}a_2^{2\nu_2-1+m-d_1}\\
&\qquad \times \int_{0}^\infty
\exp (-2\pi ((ra_1)^2+(ra_2)^{-2}))
r^{2\nu_2-2\nu_1+2m-d_1-d_2}
\,\frac{4\,\rd r}{r}
\end{aligned}
\end{equation}
if $d_1\geq d_2$, and
\begin{equation}
\label{eq:Wh_explicit2_n=2}
\begin{aligned}
\mW^{(\varepsilon),\rop}_{d, \nu}(\xi_M)(a)
&=(\varepsilon \sqrt{-1})^{m}a_1^{2\nu_1+1+m-d_1}a_2^{2\nu_1-1+m-d_2}\\
&\qquad \times 
\int_{0}^\infty \exp (-2\pi ((ra_1)^2+(ra_2)^{-2}))
r^{2\nu_1-2\nu_2+2m-d_1-d_2}\,\frac{4\,\rd r}{r}
\end{aligned}
\end{equation}
if $d_1\leq d_2$.
\end{lem}

\begin{proof}
First we have 
\begin{align} \label{eq:GL1_Wh1}
\begin{aligned}
\mW^{(\varepsilon),\rop}_{d, \nu}(\xi_M)(a)
&=(\varepsilon \sqrt{-1})^{d_1}(a_1a_2)^{2\nu_2+1}\\
&\quad \times 
\int_{\rGL_{1}(\mC )}\left(\int_{\mC }
\Phi_{d_1-d_2}(\xi_{d_1-d_2}\boxtimes \xi_{M-d_2})((ha_1,ha_2z))
\psi_{-\varepsilon }(z)\,\rd_\mC z\right)
\left(\frac{h}{\lvert h\rvert}\right)^{\! d_2-d_1}\lvert h\rvert^{2\nu_2-2\nu_1+2}
\,\rd h
\end{aligned}
 \end{align}
by the formula (\ref{eq:Wh_recursive}) combined with 
$\dim_{\mC} V_{d_1}=1$, $\rr (M)=\rr (M-d_2)$, $\xi_{d_1-d_2}=\zeta_{d_1-d_2}$ and 
\begin{align}
\label{eq:GL1_Wh}
\mW^{(\varepsilon),\rop}_{d, \nu} (\zeta_{d_1})(g)
&=\left(\frac{g}{|g|}\right)^{\!d_1}|g|^{2\nu_1}&
\text{for \;} g\in \rGL_1(\mC).
\end{align}
To compute the right hand side of \eqref{eq:GL1_Wh1}, we decompose $h\in \rGL_1(\mC)$ as $h=ru$ with $r>0$ and $u\in \rU(1)$. Since
\[
\Phi_{d_1-d_2}(\xi_{d_1-d_2}\boxtimes \xi_{M-d_2})((ha_1,ha_2z))
=\overline{u}^{d_2-d_1}\Phi_{d_1-d_2}(\xi_{d_1-d_2}\boxtimes \xi_{M-d_2})((ra_1,ra_2z))
\] 
holds, we obtain
\begin{align*}
\mW^{(\varepsilon),\rop}_{d, \nu}(\xi_M)(a)
&=(\varepsilon \sqrt{-1})^{d_1}(a_1a_2)^{2\nu_2+1}
\!\! \int_{0}^\infty\!\! \left(\int_{\mC }
\Phi_{d_1-d_2}(\xi_{d_1-d_2}\boxtimes \xi_{M-d_2})((ra_1,ra_2z))
\psi_{-\varepsilon }(z)\,\rd_\mC z\right)
r^{2\nu_2-2\nu_1+2}
\,\frac{4\,\rd r}{r}\\
&=(\varepsilon \sqrt{-1})^{d_1}a_1^{2\nu_2+1}a_2^{2\nu_2-1}
\!\! \int_{0}^\infty\!\!\left(\int_{\mC }
\Phi_{d_1-d_2}(\xi_{d_1-d_2}\boxtimes \xi_{M-d_2})((ra_1,z))
\psi_{-\varepsilon }((ra_2)^{-1}z)\,\rd_\mC z\right)
r^{2\nu_2-2\nu_1}
\,\frac{4\,\rd r}{r}
\end{align*}
by integrating with respect to $u$. Here the second equality follows from the substitution $z\mapsto (ra_2)^{-1}z$. Furthermore, we have 
\begin{align*}
&\Phi_{d_1 -d_2}(\xi_{d_1-d_2}
\boxtimes \xi_{M-d_2})((z_1,z_2))
=\left\{\begin{array}{ll}
z_1^{m-d_2}z_2^{d_1-m}\exp (-2\pi (|z_1|^2+|z_2|^2))&\text{if $d_1\geq d_2$},\\[1mm]
(-1)^{m-d_1}\overline{z_1}^{d_2-m}
\overline{z_2}^{m-d_1}\exp (-2\pi (|z_1|^2+|z_2|^2))&\text{if $d_1\leq d_2$}
\end{array}\right.
\end{align*}
for $z_1,z_2\in \mC$ by (\ref{eq:std_schwartz_n=2}). 
Therefore, if $d_1\geq d_2$, we have 
\begin{align*}
&\int_{\mC }
\Phi_{d_1-d_2}(\xi_{d_1-d_2}\boxtimes \xi_{M-d_2})((ra_1,z))
\psi_{-\varepsilon }((ra_2)^{-1}z)\,\rd_\mC z\\
&\qquad =(ra_1)^{m-d_2}
\exp (-2\pi (ra_1)^2)
\int_{\mC }
z^{d_1-m}
\exp (-2\pi |z|^2
-2\pi \varepsilon \sqrt{-1}(ra_2)^{-1}(z+\overline{z}))
\,\rd_\mC z\\
&\qquad =(\varepsilon \sqrt{-1})^{m-d_1}
r^{2m-d_1-d_2}
a_1^{m-d_2}a_2^{m-d_1}
\exp (-2\pi ((ra_1)^2+(ra_2)^{-2}))
\end{align*}
by \cite[(4.19)]{im}, and obtain (\ref{eq:Wh_explicit1_n=2}). 
Meanwhile, if $d_1\leq d_2$, we have 
\begin{align*}
&\int_{\mC }
\Phi_{d_1-d_2}(\xi_{d_1-d_2}\boxtimes \xi_{M-d_2})((ra_1,z))
\psi_{-\varepsilon }((ra_2)^{-1}z)\,\rd_\mC z\\
&\qquad =(-1)^{m-d_1}(ra_1)^{d_2-m}
\exp (-2\pi (ra_1)^2)
\int_{\mC}
\overline{z}^{m-d_1}
\exp (-2\pi |z|^2
-2\pi \varepsilon \sqrt{-1}(ra_2)^{-1}(z+\overline{z}))
\,\rd_\mC z\\
&\qquad =(\varepsilon \sqrt{-1})^{m-d_1}
r^{d_1+d_2-2m}
a_1^{d_2-m}a_2^{d_1-m}\exp (-2\pi ((ra_1)^2+(ra_2)^{-2}))
\end{align*}
by \cite[(4.19)]{im} again, and obtain (\ref{eq:Wh_explicit2_n=2}) as follows:
\begin{align*}
\mW^{(\varepsilon),\rop}_{d, \nu}(\xi_M)(a)
&=(\varepsilon \sqrt{-1})^{m}a_1^{2\nu_2+1+d_2-m}a_2^{2\nu_2-1+d_1-m}
\int_{0}^\infty \exp (-2\pi ((ra_1)^2+(ra_2)^{-2}))
r^{2\nu_2-2\nu_1+d_1+d_2-2m}\,\frac{4\,\rd r}{r}\\
&=(\varepsilon \sqrt{-1})^{m}a_1^{2\nu_1+1+m-d_1}a_2^{2\nu_1-1+m-d_2}
\int_{0}^\infty \exp (-2\pi ((ra_1)^2+(ra_2)^{-2}))
r^{2\nu_1-2\nu_2+2m-d_1-d_2}\,\frac{4\,\rd r}{r}. 
\end{align*}
Here the last equality follows from the substitution $r\mapsto (a_1a_2r)^{-1}$
\end{proof}

\begin{proof}[Proof of Proposition~$\ref{prop:arch_Wh_op}$]
First, let us prove the entireness of 
$\mW^{(\varepsilon),\rop}_{d, \nu} (\rv )(g)$ in $\nu$ by induction on $n$. 
It suffices to consider the case where $\rv =\zeta_M$ with $M\in \rG (d^\rdom )$. 
When $n=1$, the assertion obviously follows from (\ref{eq:GL1_Wh}). 
Next assume that $n\geq 2$.   
Then $\mW^{(\varepsilon),\rop}_{\widehat{d},\widehat{\nu}} (\zeta_{N})(h^{-1})$ 
is an entire function in $\widehat{\nu}$ for any 
$N\in \rG (\widehat{d}^\rdom )$ and $h\in \rGL_{n-1}(\mC )$ 
by the induction hypothesis. 
Applying \cite[Lemma 5.4]{im} to 
$\beta (z)=\Gamma_{\mC}\bigl(z+1+\tfrac{|d_i-d_j|}{2}\bigr)$ ($1\leq i<j\leq n-1$), 
we see that the majorisation  
\cite[Proposition 3.3 with $X=1$]{Jacquet_001} 
for $\rW_{\varepsilon}^{\rop}(\rf_{\rB_{n-1}^{\rop},\widehat{d},\widehat{\nu}}
(\zeta_{N}))$ is also valid for 
$\mW^{(\varepsilon),\rop}_{\widehat{d},\widehat{\nu}} (\zeta_{N})$. 
Hence, similarly to the proof of 
\cite[Proposition 7.2]{Jacquet_001}, 
we see that the right hand side of 
(\ref{eq:Wh_recursive}) converges absolutely and 
is an entire function in $\nu$. 

Next, let us prove the assertion on $\Gamma_n^{\rop} (\nu;d )$. 
Take $\nu_0 \in \mC^n$ such that $\pi_{\rB_n^{\rop},d,\nu_0}$ is irreducible. 
Then there exist $g\in \rGL_n(\mC)$ and $\rv \in V_{d^{\rdom}}$ such that 
$\rW_{\varepsilon}^{\rop}(\rf_{\rB_n^{\rop},d,\nu_0}(\rv ))(g)=
\cJ_{\varepsilon}^{\rop}(\pi_{\rB_n^{\rop},d,\nu_0}(g)\rf_{\rB_n^{\rop},d,\nu_0}(\rv ))$ is nonzero, 
since the Jacquet integral $\cJ_{\varepsilon}^{\rop}$ is 
a nonzero continuous $\mC$-linear form on $I^\infty_{\rB_n^{\rop}}(d,\nu_0)$. 
Hence, the entireness of $\mW^{(\varepsilon),\rop}_{d, \nu} (\rv )(g)$ implies 
that the meromorphic function $\Gamma_n^{\rop} (\nu;d )$ in $\nu$ 
is holomorphic at $\nu =\nu_0$. 

Finally, let us prove the symmetry (\ref{eq:Sinv_Wh_op}) 
by induction on $n$. 
When $n=1$, (\ref{eq:Sinv_Wh_op}) is obvious. 
When $n=2$, (\ref{eq:Sinv_Wh_op}) follows from Lemma~\ref{lem:Wh_explicit_n=2} and the equality
\begin{align*}
\mW^{(\varepsilon),\rop}_{d, \nu}(\rv )(xak)&=
\psi_{\varepsilon ,\rN_2}(x)\mW^{(\varepsilon),\rop}_{d, \nu}(\tau_{d^\rdom}(k)\rv )(a)&
&\text{for \;} x\in \rN_2(\mC),\ a\in A_2,\ k\in \rU (2) \text{\; and \;} \rv \in V_{d^\rdom}
\end{align*}
with the Iwasawa decomposition $\rGL_2(\mC )=\rN_2(\mC)A_2\rU (2)$. 
Now assume that $n\geq 3$. By the expression (\ref{eq:Wh_recursive}) and 
the induction hypothesis, (\ref{eq:Sinv_Wh_op}) holds for 
$\sigma \in \gS_n$ satisfying $\sigma (n)=n$. Moreover, by Lemma~\ref{lem:dual_Wh}, 
we know that (\ref{eq:Sinv_Wh_op}) also holds for 
$\sigma \in \gS_n$ satisfying $\sigma (1)=1$. Since $\gS_n$ is generated by 
\[
\{\sigma \in \gS_n\mid \sigma (1)=1\}\cup \{\sigma \in \gS_n\mid \sigma (n)=n\}
\]
for $n\geq 3$, we know that (\ref{eq:Sinv_Wh_op}) holds for 
any $\sigma \in \gS_n$, which completes the proof. 
\end{proof}

\begin{proof}[Proof of Proposition~$\ref{prop:arch_Wh}$]
As explained in Section \ref{subsec:whittaker}, 
Proposition \ref{prop:arch_Wh} is equivalent to Proposition \ref{prop:arch_Wh_op}. 
Hence, the assertion immediately follows from Proposition~\ref{prop:arch_Wh_op}. 
\end{proof}

\subsection{Explicit formulas for  archimedean local zeta integrals}
\label{subsec:ExplicitArchZeta}

In this subsection, we assume that $n>1$. Let 
\begin{equation*}
\begin{aligned}
&d=(d_1,d_2,\dots ,d_n)\in \mZ^n, &
&\nu =(\nu_1,\nu_2,\dots ,\nu_n)\in \mC^n, \\
&d'=(d_1',d_2',\dots ,d_{n-1}')\in \mZ^{n-1},& 
&\nu'=(\nu_1',\nu_2',\dots ,\nu_{n-1}')\in \mC^{n-1}.
\end{aligned}
\end{equation*}
As before, let $d^{\rdom} =(d^{\rdom}_1,d^{\rdom}_2,\dots ,d^{\rdom}_n)$ and 
$d^{\prime \rdom} =(d^{\prime \rdom}_1,d^{\prime \rdom}_2,\dots ,d^{\prime \rdom}_{n-1})$ respectively denote unique elements of $\Lambda_n\cap \{\sigma d \mid \sigma \in \gS_n\}$ and $\Lambda_{n-1}\cap \{\sigma'd' \mid \sigma'\in \gS_{n-1}\}$. The archimedean local $L$-factor $L(s,\pi_{\rB_n,d,\nu}\times \pi_{\rB_{n-1},d',\nu'})$ is defined by 
\[
L(s,\pi_{\rB_n,d,\nu}\times \pi_{\rB_{n-1},d',\nu'})
=\prod_{i=1}^n\prod_{j=1}^{n-1}
\Gamma \bigl(s+\nu_i+\nu_j'+\tfrac{|d_i+d_j'|}{2}\bigr).
\]

Take $\varepsilon \in \{\pm \}$, 
$W\in \cW (\pi_{\rB_n,d,\nu},\psi_\varepsilon )$, 
$W'\in \cW (\pi_{\rB_{n-1},d',\nu'},\psi_{-\varepsilon})$ 
and $s\in \mC$ with the real part sufficiently large. 
We define the archimedean local zeta integral 
$Z(s,W,W')$ by 
\begin{align} 
Z(s,W,W')=\int_{\rN_{n-1}(\mC )\backslash \rGL_{n-1}(\mC )}
W(\iota_n(g))W'(g)\lvert \det g\rvert^{2s-1}\,\rd g, 
\end{align}
where $\rd g$ is the right invariant measure on 
$\rN_{n-1}(\mC )\backslash \rGL_{n-1}(\mC )$ normalised as (\ref{eqn:quot_GN_measure}). 
We rewrite the explicit result \cite[Theorem 2.7]{im} for the archimedean local zeta integrals 
in terms of our normalised Whittaker functions as follows.

\begin{prop}
\label{prop:ArchInt}
Retain the notation. Then 
for $\rv \in V_{d^\rdom}$ and $\rv'\in V_{d^{\prime \rdom}}$, we have 
\begin{equation}
\label{eq:archint1} 
\begin{aligned}
&Z(s,\mW_{d,\nu}^{(\varepsilon )}(\rv ),\mW_{d',\nu'}^{(-\varepsilon )}(\rv'))\\
&\qquad =\left\{\begin{array}{ll}
\dfrac{(-\varepsilon \sqrt{-1})^{\sum_{i=1}^{n-1}d_i'}}
{\dim V_{d^{\prime \rdom \vee}}}
L(s,\pi_{\rB_n,d,\nu}\times \pi_{\rB_{n-1},d',\nu'})
\langle \rR_{d^{\prime \rdom \vee}}^{d^\rdom }(\rv ), \rv'\rangle_{d^{\prime \rdom \vee}}&
\text{if $d^{\prime \rdom \vee }\preceq d^{\rdom}$},\\[1mm]
0&\text{otherwise}.
\end{array}\right. 
\end{aligned}
\end{equation}
In particular, if $d^{\prime \rdom \vee }\preceq d^{\rdom}$, we have 
\begin{align}
\label{eq:archint2}
&\sum_{M\in \rG (d^{\prime \rdom \vee })}\frac{(-1)^{\rrq (M)}}{\rr (M)}
Z(s,\mW_{d,\nu}^{(\varepsilon )}(\xi_{M[d^\mathrm{dom} ]}),
\mW_{d',\nu'}^{(-\varepsilon )}(\xi_{M^\vee }))
=(-\varepsilon \sqrt{-1})^{\sum_{i=1}^{n-1}d_{i}'}
L(s,\pi_{\rB_n ,d,\nu }\times \pi_{\rB_{n-1},d',\nu'}).
\end{align}
\end{prop}
\begin{proof}
Take $\rv \in V_{d^\rdom}$ and $\rv'\in V_{d^{\prime \rdom}}$. 
By Lemma~\ref{lem:ps_rel_conj} and the definitions of $\mW^{(\varepsilon),\rop}_{d,\nu}$ and $\mW^{(-\varepsilon),\rop}_{d', \nu'}$, we have 
\begin{align*}
\mW^{(\varepsilon),\rop}_{d, \nu} (\rv )
&=(\varepsilon \sqrt{-1})^{\sum^{n}_{i=1}(n-i)d_i}\Gamma_n^{\rop} (\nu;d )
\rW_{\varepsilon}^\rop (\rf_{\rB_n^\rop ,d,\nu}(\rv )),\\
\mW^{(-\varepsilon),\rop}_{d', \nu'} (\rv')
&=(\varepsilon \sqrt{-1})^{\sum^{n-1}_{i=1}(n-1-i)d_i'}\Gamma_{n-1}^{\rop} (\nu';d')
\rW_{-\varepsilon}^\rop (\rf_{\rB_{n-1}^{\rop},d',\nu'}^\rconj (\rI_{d^{\prime \rdom}}^\rconj (\rv'))).
\end{align*}
Hence, if we assume that $d\in \Lambda_n$ and $-d'\in \Lambda_{n-1}$, 
we have 
\begin{align*}
&Z(s,\mW_{d,\nu}^{(\varepsilon ),\rop }(\rv ),\mW_{d',\nu'}^{(-\varepsilon ),\rop}(\rv'))\\
&=(-\varepsilon \sqrt{-1})^{\sum_{i=1}^{n-1}d_i'}
(\varepsilon \sqrt{-1})^{\sum^{n-1}_{i=1}(n-i)(d_i+d_i')}
\Gamma_n^{\rop} (\nu ;d )\Gamma_{n-1}^{\rop} (\nu';d')
Z(s,\rW_{\varepsilon}^\rop (\rf_{\rB_n^\rop ,d,\nu}(\rv )),
\rW_{-\varepsilon}^\rop (\rf_{\rB_{n-1}^{\rop},d',\nu'}^\rconj (\rI_{d^{\prime *}}^\rconj (\rv'))))\\
&=\left\{\begin{array}{ll}
\dfrac{(-\varepsilon \sqrt{-1})^{\sum_{i=1}^{n-1}d_i'}}
{\dim V_{-d'}}
L(s,\pi_{\rB_n,d,\nu}\times \pi_{\rB_{n-1},d',\nu'})
\langle \rR_{-d'}^{d }(\rv ), \rv'\rangle_{-d'}&
\text{if $-d'\preceq d$},\\[1mm]
0&\text{otherwise}
\end{array}\right. 
\end{align*}
by \cite[Theorem 2.7]{im} combined with Remark \ref{rem:cpx_conj_rep}. 
Even if we do not assume that $d\in \Lambda_n$ and $-d'\in \Lambda_{n-1}$, 
the equation above implies
\begin{align*}
&Z(s,\mW_{d,\nu}^{(\varepsilon ),\rop }(\rv ),\mW_{d',\nu'}^{(-\varepsilon ),\rop}(\rv'))\\
 &\qquad =\left\{\begin{array}{ll}
\dfrac{(-\varepsilon \sqrt{-1})^{\sum_{i=1}^{n-1}d_i'}}
{\dim V_{d^{\prime \rdom \vee}}}
L(s,\pi_{\rB_n,d,\nu}\times \pi_{\rB_{n-1},d',\nu'})
\langle \rR_{d^{\prime \rdom \vee}}^{d^\rdom }(\rv ), \rv'\rangle_{d^{\prime \rdom \vee}}&
\text{if $d^{\prime \rdom \vee }\preceq d^{\rdom}$},\\[1mm]
0&\text{otherwise}
\end{array}\right. 
\end{align*}
due to Proposition~\ref{prop:arch_Wh_op} and the fact that 
$L(s,\pi_{\rB_n ,\sigma d,\sigma \nu }\times \pi_{\rB_{n-1},\sigma'd',\sigma'\nu'})
=L(s,\pi_{\rB_n ,d,\nu }\times \pi_{\rB_{n-1},d',\nu'})$ holds 
for every  $\sigma \in \gS_n$ and $\sigma'\in \gS_{n-1}$. 
We thus obtain (\ref{eq:archint1}) by the last equation and (\ref{eq:norWh_rel_op}). The equality (\ref{eq:archint2}) follows immediately from (\ref{eq:archint1}). 
\end{proof}

\section{Calculation on coefficients of the injectors}
\label{sec:calcoeff}

In this section, we use the notation in Section \ref{subsec:explicit_generators}, and 
drop $v$ from subscript in many cases since we always concentrate on an archimedean place $v \in \Sigma_{F, \infty}$. 
The purpose of this section is verification of Proposition~\ref{prop:coefC}, or more concretely, 
to calculate $c^{(m)}_{\lambda, \mathsf{w}, \mu, \mathsf{w}^\prime}$ appearing in (\ref{eq:ratio}). 
We divide the calculation of $c^{(m)}_{\lambda,\mathsf{w},\mu,\mathsf{w}'}$ into several steps.

\begin{lem}\label{lem:Clem1}
Retain the notation in Section $\ref{subsec:explicit_arch_zeta}$. Then we have 
\begin{align*}
\displaystyle 
      c^{(m)}_{\lambda, \mathsf{w}, \mu, \mathsf{w}^\prime} 
     =  (-1)^{ (n-1) \mathsf{w}^\prime + m \mathrm{b}_n  + \ell(\mu) + \frac{1}{6} n(n-1)(n-2)  } 
           \frac{   \rs_n(\rI^{\gp_n}_{2\rho_n}(\xi_{H(2\rho_{n-1})[2\rho_n]}), 
                                   \rI^{\gp_{n-1}}_{2\rho_{n-1}}(\xi_{H(-2\rho_{n-1})  })  )   }{\rr (H( 2\rho_{n-1} )[2\rho_n])}, 
\end{align*}
where $\rs_n(\cdot,\cdot)$ is the pairing defined as $(\ref{eq:pairing_Lie})$.
\end{lem}
\begin{proof}
We compare the coefficients of $\xi_{H(-d^\prime)[d]} \otimes \xi_{H(d^\prime)}$ in the both sides of (\ref{eq:ratio}), which we regard as linear combinations of $\xi_{M}\otimes \xi_{M'}$ for $M\in \rG (d)$ and $M'\in \rG (d')$. 
Since $\rr (H(-d'))=1$ by Lemma~\ref{lem:rM_property} \ref{num:Hgamma_rM_property}, 
the coefficient of $\xi_{H(-d^\prime)[d]} \otimes \xi_{H(d^\prime)}$ in the right hand side of (\ref{eq:ratio}) 
equals $(-1)^{\rrq (H(-d'))}c^{(m)}_{\lambda,\sw ,\mu,\sw'}$. Let us consider 
the coefficient of $\xi_{H(-d^\prime)[d]} \otimes \xi_{H(d^\prime)}$ in 
$\Psi^{(m)}  (  [\pi_{\rB_n,d,\sw /2}]^{\rpre} \otimes [\pi_{\rB_{n-1},d',\sw'/2}]^{\rm pre})$, which is the left hand side of (\ref{eq:ratio}). 

Since $d^\prime = 2\mu + 2\rho_{n-1} -\mathsf{w}^\prime$, 
the identity (\ref{eq:cHeq}) implies that
\begin{align*}
{\rm c}^{N', T'}_{H(-d^\prime)}
&= \begin{cases} 1&\text{if }N'= H(-2\rho_{n-1})\text{ and }T'=H(-2\mu+\mathsf{w}^\prime), \\
                          0&\text{otherwise},   \end{cases}    \\   
{\rm c}^{P^\prime, Q^{\prime\vee}}_{H(-2\mu+\mathsf{w}^\prime)}
&= \begin{cases} 1&\text{if }P^\prime = H(- \mu)\text{ and }Q^\prime=H( \mu-\mathsf{w}^\prime), \\
                          0 &\text{otherwise}  \end{cases}
\end{align*}
for $N' \in \rG (2\rho_{n-1})$, $T'\in \rG (2\mu^\vee + \mathsf{w}^\prime)$, 
$P'\in \rG (\mu^\vee)$ and $Q' \in \rG (\mu-\sw')$. 
Since $\rr (H(d^\prime))=1$ by Lemma~\ref{lem:rM_property} \ref{num:Hgamma_rM_property}, 
the coefficient $c(H(d^\prime), N^\prime, P^\prime, Q^\prime)$
in  $[\pi_{\rB_{n-1},d',\sw'/2}]^{\rpre}$ is given by 
\begin{align*}
c(H(d^\prime), N^\prime, P^\prime, Q^\prime)  
&=\frac{  (-1)^{\rrq (H(d^\prime)) +\rrq (Q^\prime) } }{\rr (H(d^\prime))}
     \sum_{T'\in {\rm G} (2\mu^\vee+\mathsf{w}^\prime)   } {\rm c}^{N^\prime, T'}_{H(-d^\prime)} {\rm c}^{P^\prime, Q^{\prime\vee}}_{T'}\\
&=
\begin{cases} 
(-1)^{\rrq (H(d^\prime)) +\rrq (  H(\mu-\mathsf{w}^\prime)  ) }
&\text{if }N^\prime = H(-2\rho_{n-1}),\ P^\prime = H(- \mu)\text{ and }Q^\prime=H(\mu-\sw'), \\
0&\text{otherwise}  \end{cases}
\end{align*}
for $N^\prime \in \rG (2\rho_{n-1})$, 
$P^\prime\in \rG (\mu^\vee)$ and $Q^\prime \in \rG (\mu-\sw')$. 
This implies that 
\begin{equation}
\label{eq:pf_Clem1_001}
\begin{aligned}
&[\pi_{\rB_{n-1},d',\sw'/2}]^{\rm pre}
=(-1)^{\rrq (H(d^\prime)) +\rrq (  H(\mu-\mathsf{w}^\prime)  ) }
\xi_{H(d')}\otimes  \rI^{\gp_{n-1}}_{2\rho_{n-1}}(\xi_{H(-2\rho_{n-1})})\otimes \xi_{H(-\mu )}\otimes \xi_{H(\mu-\sw')}\\
&\hspace{10mm}
+\sum_{H(d')\neq M'\in \rG (d')}\sum_{N'\in \rG (2\rho_{n-1})}\sum_{P'\in \rG (\mu^\vee)}\sum_{Q'\in \rG (\mu -\sw')}
    c(M', N', P', Q')\xi_{M'} \otimes  \rI^{\gp_{n-1}}_{2\rho_{n-1}}(\xi_{N'})\otimes \xi_{P'}\otimes \xi_{Q'},
\end{aligned}
\end{equation}
and we note that only the first term in the right hand side contributes to the coefficient of $\xi_{H(-d^\prime)[d]} \otimes \xi_{H(d^\prime)}$ in 
$\Psi^{(m)}  (  [\pi_{\rB_n,d,\sw /2}]^{\rpre} \otimes [\pi_{\rB_{n-1},d',\sw'/2}]^{\rm pre})$.

By \cite[Lemma 4.3]{im}, Lemma \ref{lem:hom_dual_conj} and Remark \ref{rem:cpx_conj_rep}, 
one readily sees that $\Hom_{\rU (n-1)}(V_{2\rho_n}\otimes_\mC V_{2\rho_{n-1}},\mC_\rtriv )$ is one dimensional, and 
this implies that 
$\rs_n(\rI^{\gp_n}_{2\rho_n}(\cdot ), \rI^{\gp_{n-1}}_{2\rho_{n-1}}(\cdot ))$ and 
$\langle \cdot , \cdot \rangle^{(0)}_{2\rho_n, 2\rho_{n-1}}$ coincide up to a constant multiple. 
The explicit descriptions of the pairings (\ref{eq:PairExplicit}) yield that 
\begin{align}
\label{eq:pf_Clem1_004}
 &\rs_n(\rI^{\gp_n}_{2\rho_n}(\xi_N), \rI^{\gp_{n-1}}_{2\rho_{n-1}}(\xi_{H(-2\rho_{n-1})}))=0
\quad \text{if }N\neq H(2\rho_{n-1})[2\rho_n],\\
&\langle \xi_P, \xi_{H(-\mu)} \rangle^{(m)}_{\lambda^\vee, \mu^\vee}
= \begin{cases}(-1)^{\rrq (P)}\rr (P)&\text{if }P=H(\mu+m)[\lambda^\vee],  \\
                                 0&\text{otherwise},   \end{cases} \nonumber\\
&\langle \xi_Q, \xi_{H(\mu-\mathsf{w}^\prime)}  \rangle^{(m)}_{\lambda-\mathsf{w}, \mu-\mathsf{w}^\prime}
= \begin{cases}  (-1)^{\rrq (Q)}\rr (Q)   &\text{if }Q=H(-\mu+\sw'+m)[\lambda-\sw ],  \\
                                 0&\text{otherwise} \end{cases}  \nonumber
\end{align}
for $N\in \rG (2\rho_n)$, $P\in \rG (\lambda^\vee)$ and $Q \in \rG (\lambda-\sw )$. 
By these equalities, we have 
\begin{equation}
\label{eq:pf_Clem1_002}
\begin{aligned}
\Psi^{(m)}  &\bigl( [\pi_{\rB_n,d,\sw /2}]^{\rpre} \otimes 
(\xi_{H(d')}\otimes  \rI^{\gp_{n-1}}_{2\rho_{n-1}}(\xi_{H(-2\rho_{n-1})})\otimes \xi_{H(-\mu )}\otimes \xi_{H(\mu-\sw')})\bigr)\\
&=(-1)^{\rrq (H(\mu+m)[\lambda^\vee])+\rrq (H(-\mu+\sw'+m)[\lambda-\sw ])}
\rr (H(\mu+m)[\lambda^\vee])
\rr (H(-\mu+\sw'+m)[\lambda-\sw ])\\
& \qquad \times \rs_{n}(\rI^{\gp_n}_{2\rho_n}(\xi_{H(2\rho_{n-1})[2\rho_n]}), \rI^{\gp_{n-1}}_{2\rho_{n-1}}(\xi_{H(-2\rho_{n-1})}))\\
& \qquad \times \sum_{M\in \rG (d)}c(M, H(2\rho_{n-1})[2\rho_n], H(\mu+m)[\lambda^\vee], H(-\mu+\sw'+m)[\lambda-\sw ])
\xi_M\otimes \xi_{H(d')}.
\end{aligned}
\end{equation}

By using (\ref{eq:pf_Clem1_001}) and (\ref{eq:pf_Clem1_002}), 
we compare the coefficients of  
$\xi_{H(-d^\prime)[d]} \otimes \xi_{H(d^\prime)}$ in the both sides of (\ref{eq:ratio}), and obtain the equality
\begin{equation}
\label{eq:pf_Clem1_003}
\begin{aligned}
&(-1)^{\rrq (H(-d'))} c^{(m)}_{\lambda,\sw ,\mu,\sw'} \\
&\hspace{2mm} =(-1)^{\rrq (H(d')) + \rrq (H(\mu-\mathsf{w}^\prime))+\rrq (H(\mu+m)[\lambda^\vee])+\rrq (H(-\mu+\sw'+m)[\lambda-\sw ])} \rr (H(\mu+m)[\lambda^\vee])
\rr (H(-\mu+\sw'+m)[\lambda-\sw ])\\
&\hspace{4mm} \times 
\rs_{n}(\rI^{\gp_n}_{2\rho_n}(\xi_{H(2\rho_{n-1})[2\rho_n]}), \rI^{\gp_{n-1}}_{2\rho_{n-1}}(\xi_{H(-2\rho_{n-1})})) \, c(H(-d^\prime)[d], H(2\rho_{n-1})[2\rho_n], H(\mu+m)[\lambda^\vee], H(-\mu+\sw'+m)[\lambda-\sw ]). \\
\end{aligned}
\end{equation}
Furthermore, by the definition (\ref{eq:Pic}) and the latter assertion of Lemma~\ref{lem:injExpH(mu)}, we have 
\begin{align*}
 c&( H(-d^\prime)[d],  H(2\rho_{n-1})[2\rho_n],   
       H(\mu+m)[\lambda^\vee],  H(-\mu+\mathsf{w}^\prime+m)[\lambda-\mathsf{w}]  ) \\
&\quad =   \frac{  (-1)^{\rrq (H(-d^\prime)[d]) 
       +\rrq (H(-\mu+\mathsf{w}^\prime+m)[\lambda-\mathsf{w}]) } }{\rr (H(-d^\prime)[d])}
     \sum_{T\in {\rm G} (2\lambda^\vee+\mathsf{w})   } 
     {\rm c}^{H(2\rho_{n-1})[2\rho_n], T}_{ H(d^\prime)[d^\vee]  } 
     {\rm c}^{H(\mu+m)[\lambda^\vee], H(\mu-\mathsf{w}^\prime-m)[\lambda^\vee+m] }_{T}\\
&\quad = 
\frac{  (-1)^{\rrq (H(-d^\prime)[d])  
      +\rrq (H(-\mu+\mathsf{w}^\prime+m)[\lambda-\mathsf{w}]) } }{\rr (H(-d^\prime)[d])} 
      {\rm c}^{H(2\rho_{n-1})[2\rho_n], 
                      H(2\mu-\mathsf{w}^\prime)[2\lambda^\vee+\mathsf{w}]}_{ H(d^\prime)[d^\vee]  } 
      {\rm c}^{H(\mu+m)[\lambda^\vee], 
                     H(\mu-\mathsf{w}^\prime-m)[\lambda^\vee+\mathsf{w}] }_{H(2\mu-\mathsf{w}^\prime)[2\lambda^\vee+\mathsf{w}]  }.  
\end{align*}
Using (\ref{eq:inj_coeff_explicit}) in Proposition~\ref{prop:injExp} and Lemma \ref{lem:Mdual_property}, 
we find that the equality
\begin{align*}
c&( H(-d^\prime)[d],  H(2\rho_{n-1})[2\rho_n],   
       H(\mu+m)[\lambda^\vee],  H(-\mu+\mathsf{w}^\prime+m)[\lambda-\sw ]  )\\
&\hspace{15mm}= \frac{  (-1)^{\rrq (H(-d^\prime)[d])+\rrq (H(-\mu+\mathsf{w}^\prime+m)[\lambda-\mathsf{w}]) } }
{\rr (H( 2\rho_{n-1} )[2\rho_n])\rr ( H(\mu+m)[\lambda^\vee] )\rr ( H(-\mu+\mathsf{w}^\prime+m)[\lambda-\mathsf{w}] )}
\end{align*}
holds. By this equality and (\ref{eq:pf_Clem1_003}), we can calculate the coefficient $c^{(m)}_{\lambda,\sw ,\mu,\sw'}$ under consideration as
\begin{align*}
&c^{(m)}_{\lambda,\sw ,\mu,\sw'}
=(-1)^{\rrq (H(d'))-\rrq (H(-d')) + \rrq (H(\mu-\mathsf{w}^\prime))+\rrq (H(\mu+m)[\lambda^\vee])+\rrq (H(-d^\prime)[d]) } \\
&\hspace{45mm} \times \frac{\rs_{n}(\rI^{\gp_n}_{2\rho_n}(\xi_{H(2\rho_{n-1})[2\rho_n]}), \rI^{\gp_{n-1}}_{2\rho_{n-1}}(\xi_{H(-2\rho_{n-1})}))}
{\rr (H( 2\rho_{n-1} )[2\rho_n])}.
\end{align*}
Furthermore the signature of $c^{(m)}_{\lambda,\sw ,\mu,\sw'}$ is described as
\begin{align*}
&(-1)^{\rrq (H(d'))-\rrq (H(-d'))}=1,&
&(-1)^{ \rrq (  H(\mu-\mathsf{w}^\prime)  )+\rrq ( H(\mu+m)[\lambda^\vee] ) }  
= (-1)^{  \rb_{n-1}\mathsf{w}^\prime +m\rb_n + \ell (\mu )},  \\
&(-1)^{\rrq (H(-d^\prime)[d])}=  (-1)^{ \frac{1}{6} n(n-1)(n-2)+\rb_n\sw'}.
\end{align*}
due to the definitions of $\mathrm{q}(M)$ (\ref{eq:def_qM}), $d^\prime = 2\mu + 2 \rho_{n-1} - \mathsf{w}^\prime$ and $\ell(\mu)=\sum^{n-1}_{i=1} \mu_i$. Combining these equalities with $(-1)^{  \rb_{n-1}\sw'+\rb_n\sw'}=(-1)^{(n-1)^2\sw'}=(-1)^{(n-1)\sw'}$, 
we obtain the desired assertion. 
\end{proof}

To calculate the right-hand side of the identity in Lemma~\ref{lem:Clem1}, we prepare three lemmas.

\begin{lem}\label{lem:Clem2}
We have 
\begin{align*}
   \rr & (H(2\rho_{n-1})[2\rho_n]  )^{-1}  
         \rs_n( \rI^{\gp_n}_{2\rho_n}( \xi_{H(2\rho_{n-1})[2\rho_n]}  ), 
                          \rI^{\gp_{n-1}}_{2\rho_{n-1}}(  \xi_{H(-2\rho_{n-1})}  )  )     \\
 & \quad = \left(  (n-1)! (n-2)! \cdots 2 ! 1!  \right)^{-1} \rs_n\bigl(  \rad^\vee \bigl( (E^{\gu (n)}_{n,n-1})^{n-1}(E^{\gu (n)}_{n-1,n-2})^{n-2} \cdots E^{\gu (n)}_{2,1}\bigr)  
\rI^{\gp_n}_{2\rho_n}(  \xi_{H(2\rho_n)} ),\rI^{\gp_{n-1}}_{2\rho_{n-1}}(  \xi_{H(-2\rho_{n-1})}  )\bigr).    
\end{align*}
\end{lem}
\begin{proof}
Applying Lemma~\ref{lem:EactPair} to the case where $\lambda = 2 \rho_n$ and $\mu = 2 \rho_{n-1}$, we have 
\begin{align}
\label{eq:pf001_Clem2} 
&\left( \tau_{2\rho_n}(
(E_{n, n-1})^{n-1}(E_{n-1, n-2})^{n-2}\dots 
E_{2, 1})\xi_{ H(2\rho_n )},
\xi_{H(2\rho_{n-1})[2\rho_n]} \right)_{2\rho_n}  
=(n-1)!(n-2)!\dots 2!1!.
\end{align}
Since $\tau_{2\rho_n}(E^{\gu (n)}_{i,j})=\tau_{2\rho_n}(E_{i,j})$ for $1\leq i,j\leq n$, 
we have 
\begin{align*}
\rs_n&\bigl(  \rad^\vee \bigl( (E^{\gu (n)}_{n,n-1})^{n-1}(E^{\gu (n)}_{n-1,n-2})^{n-2} \cdots E^{\gu (n)}_{2,1}\bigr)  
\rI^{\gp_n}_{2\rho_n}(  \xi_{H(2\rho_n)} ),\, \rI^{\gp_{n-1}}_{2\rho_{n-1}}(  \xi_{H(-2\rho_{n-1})}  )\bigr)\\
&=\rs_n\bigl( \rI^{\gp_n}_{2\rho_n}(  
\tau_{2\rho_n} ( (E_{n,n-1})^{n-1}(E_{n-1,n-2})^{n-2} \cdots E_{2,1})\xi_{H(2\rho_n)} ),\, \rI^{\gp_{n-1}}_{2\rho_{n-1}}(  \xi_{H(-2\rho_{n-1})}  )\bigr)\\
&=\sum_{M\in \rG (2\rho_n)}\rr (M)^{-1}  
\bigl(\tau_{2\rho_n} ( (E_{n,n-1})^{n-1}(E_{n-1,n-2})^{n-2} \cdots E_{2,1})\xi_{H(2\rho_n)},\, \xi_M\bigr)_{2\rho_n}
\rs_n\bigl( \rI^{\gp_n}_{2\rho_n}(  \xi_M ),\rI^{\gp_{n-1}}_{2\rho_{n-1}}(  \xi_{H(-2\rho_{n-1})}  )\bigr).
\end{align*}
Applying (\ref{eq:pf_Clem1_004}) and (\ref{eq:pf001_Clem2}) to this equalities, 
we obtain the assertion. 
\end{proof}

\begin{lem}\label{lem:Clem3}
We have   \begin{align*}
\rI^{\gp_{n-1}}_{2\rho_{n-1}}(  \xi_{H(-2\rho_{n-1})}  ) 
              &=(-1)^{-\frac{1}{6}n(n-1)(n-2)} 
                 E^{\gp_{n-1}\vee}_{1,2} \wedge   E^{\gp_{n-1}\vee}_{1,3} 
                 \wedge E^{\gp_{n-1}\vee}_{2,3}   \\
                & \quad    \wedge E^{\gp_{n-1}\vee}_{1,4} 
                       \wedge E^{\gp_{n-1}\vee}_{2,4}  
                       \wedge E^{\gp_{n-1}\vee}_{3,4}   \wedge \cdots   
                   \wedge E^{\gp_{n-1}\vee}_{1, n-1} \wedge E^{\gp_{n-1}\vee}_{2, n-1}
                         \wedge \cdots \wedge   E^{\gp_{n-1}\vee}_{n-2, n-1}.
            \end{align*}
\end{lem}
\begin{proof}
Lemma~\ref{lem:GTwn} yields that 
\begin{align*}
\rI^{\gp_{n-1}}_{2\rho_{n-1}}(\xi_{H(-2\rho_{n-1})}) 
 = \rI^{\gp_{n-1}}_{2\rho_{n-1}}(\xi_{H\left( (2\rho_{n-1})^\ast \right)}) 
 &= (-1)^{\sum^{n-1}_{i=1}  (i-1) \rho_{n-1, i}  } \rad^\vee(w_{n-1}) 
\rI^{\gp_{n-1}}_{2\rho_{n-1}}(\xi_{ H(2\rho_{n-1})}) \\
 &= (-1)^{-\frac{1}{6}n(n-1)(n-2)} \rad^\vee(w_{n-1}) \rI^{\gp_{n-1}}_{2\rho_{n-1}}(\xi_{ H(2\rho_{n-1})  }). 
\end{align*}
Hence Lemma~\ref{lem:adact} \ref{num:adact_wn} and 
the definition (\ref{eq:xiH2rho}) of $\rI^{\gp_{n-1}}_{2\rho_{n-1}}$ imply the statement. 
\end{proof}

For $\mX ,\mY \in {\bigwedge}^{\rb_n} \gp^{\vee}_{n\mC}$ and $1\leq l\leq n-1$, we write $\mX \equiv_{(l)}\mY$ if and only if 
\[
\rs_n\bigl(\rad^\vee \bigl((E^{\gu (n)}_{n,n-1})^{n-1}(E^{\gu (n)}_{n-1,n-2})^{n-2} \cdots (E^{\gu (n)}_{l+1,l})^l\bigr)(\mX -\mY ),\,
\rI^{\gp_{n-1}}_{2\rho_{n-1}}(\xi_{H(-2\rho_{n-1})})\bigr)=0.
\]
For $\mX ,\mY \in {\bigwedge}^{\rb_n} \gp^{\vee}_{n\mC}$, we write $\mX \equiv_{(n)}\mY$ if and only if 
$\rs_n\bigl(\mX -\mY ,\,
\rI^{\gp_{n-1}}_{2\rho_{n-1}}(\xi_{H(-2\rho_{n-1})})\bigr)$ equals $0$, or equivalently, if and only if the equality $\rs_n\bigl(\mX ,\,\rI^{\gp_{n-1}}_{2\rho_{n-1}}(\xi_{H(-2\rho_{n-1})})\bigr)
=\rs_n\bigl(\mY ,\,\rI^{\gp_{n-1}}_{2\rho_{n-1}}(\xi_{H(-2\rho_{n-1})})\bigr)$ holds.

\begin{lem}\label{lem:Clem4}
Retain the notation. 
\begin{enumerate}
\item \label{num:Clem4_basic}
For each $1\leq l\leq n$, the binary relation $\equiv_{(l)}$ is an equivalence relation on ${\bigwedge}^{\rb_n} \gp^{\vee}_{n\mC}$ with the following property$:$ if elements $\mX ,\mY ,\mX'$ and $\mY'$ of ${\bigwedge}^{\rb_n} \gp^{\vee}_{n\mC}$ satisfy $\mX \equiv_{(l)}\mY$ and $\mX'\equiv_{(l)}\mY'$, we have
\begin{align*}
\mX +\mX' &\equiv_{(l)}\mY +\mY' \qquad \text{\; and \;} \qquad 
c\mX \equiv_{(l)}c\mY \qquad \text{for \;}c\in \mC.
\end{align*}

\item \label{num:Clem4_trivial}
For $1\leq l\leq n-1$,  take $\mX ,\mY \in {\bigwedge}^{\rb_n} \gp^{\vee}_{n\mC}$ satisfying $\mX \equiv_{(l)}\mY$. 
Then $\rad^\vee \bigl((E^{\gu (n)}_{l+1,l})^l\bigr)\mX \equiv_{(l+1)}\rad^\vee \bigl((E^{\gu (n)}_{l+1,l})^l\bigr)\mY$ holds.

\item \label{num:Clem4_vanish}
For $1\leq l\leq n-1$, take
$X_1,X_2,\dots ,X_{\rb_n}\in \{E^{\gp_n\vee}_{i,j}\mid 1\leq i,j\leq n\}$. Then 
$X_1\wedge X_2\wedge \dots \wedge X_{\rb_n}\equiv_{(l+1)} 0$ holds if 
$\{E^{\gp_n\vee}_{i,j}\mid 1\leq j\leq i\leq l\}$ is not contained in $\{X_1,X_2,\dots ,X_{\rb_n}\}$. 
\end{enumerate}
\end{lem}

\begin{proof}
The statements \ref{num:Clem4_basic} and \ref{num:Clem4_trivial} follow immediately from the definition. 
The statement \ref{num:Clem4_vanish} for $l=n-1$ follows from the definition of $\rs_n(\cdot,\cdot)$ (\ref{eq:pairing_Lie}) and Lemma \ref{lem:Clem3}. 
The statement \ref{num:Clem4_vanish} for $1\leq l\leq n-2$ follows from 
Lemma~\ref{lem:adact} \ref{num:adact_alg} and the statement for $l=n-1$. 
\end{proof}

\noindent
\begin{proof}[{Proof of Proposition~$\ref{prop:coefC}$}]
By Lemmas~\ref{lem:Clem1} and \ref{lem:Clem2}, it suffices to verify the equality
\begin{equation}
\label{eq:pf001_coefC}
\begin{aligned} 
\rs_n &\bigl(  \rad^\vee \bigl((E^{\gu (n)}_{n,n-1})^{n-1}(E^{\gu (n)}_{n-1,n-2})^{n-2} \cdots E^{\gu (n)}_{2,1}\bigr)  
\rI^{\gp_n}_{2\rho_n}(  \xi_{H(2\rho_n)} ),\rI^{\gp_{n-1}}_{2\rho_{n-1}}(  \xi_{H(-2\rho_{n-1})}  )\bigr)\\
&\hspace{75mm} =(-1)^{\rb_n -\frac{1}{6} n(n-1)(n-2) }(n-1)! (n-2)! \cdots 2 ! 1!.   
\end{aligned}
\end{equation}
For $1\leq l \leq n-1$, we set 
\begin{align*}
\mE^{\gp_n, >}_{l+1}&:=E^{\gp_n\vee}_{l+1,1} \wedge E^{\gp_n\vee}_{l+1,2} \wedge \cdots 
\wedge E^{\gp_n\vee}_{l+1, l}&  \text{ \; and \;} &&
\mE^{\gp_n, \geq }_l&:=E^{\gp_n\vee}_{l,1} \wedge E^{\gp_n\vee}_{l,2} \wedge \cdots 
\wedge E^{\gp_n\vee}_{l,l}.
\end{align*}
By Lemmas~\ref{lem:adact} \ref{num:adact_alg} and \ref{lem:Clem4} \ref{num:Clem4_basic}, \ref{num:Clem4_vanish}, 
 we have  
\begin{align*}
\rad^\vee& ((E^{\gu (n)}_{l+1,l})^l)
\mE^{\gp_n, \geq }_1\wedge \mE^{\gp_n, \geq }_2
\wedge \dots \wedge \mE^{\gp_n, \geq }_{l-1}\wedge 
\mE^{\gp_n, >}_{l+1}\wedge \mE^{\gp_n, >}_{l+2}\wedge \dots \wedge \mE^{\gp_n, >}_{n}\\
&\qquad \equiv_{(l+1)} \mE^{\gp_n, \geq }_1\wedge \mE^{\gp_n, \geq }_2
\wedge \dots \wedge \mE^{\gp_n, \geq }_{l-1}\wedge 
\bigl((-1)^ll!\,\mE^{\gp_n, \geq }_{l}\bigr)\wedge \mE^{\gp_n, >}_{l+2}\wedge \dots \wedge \mE^{\gp_n, >}_{n}
\end{align*}
for $1\leq l\leq n-1$. By these equalities combined with Lemma~\ref{lem:Clem4} \ref{num:Clem4_basic} and \ref{num:Clem4_trivial}, we have 
\begin{equation}
\label{eq:pf002_coefC}
\begin{aligned}
\rad^\vee&  \bigl( (E^{\gu (n)}_{n,n-1})^{n-1}(E^{\gu (n)}_{n-1,n-2})^{n-2} \cdots E^{\gu (n)}_{2,1}\bigr) 
\rI^{\gp_n}_{2\rho_n}(  \xi_{H(2\rho_n)} )\\
&\qquad \equiv_{(n)}(-1)^{\rb_n}(n-1)! (n-2)! \cdots 2 ! 1!\,
\mE^{\gp_n, \geq }_1\wedge \mE^{\gp_n, \geq }_2
\wedge \dots \wedge \mE^{\gp_n, \geq }_{n-1}.
\end{aligned}
\end{equation}
Note that $\rb_n=\frac{1}{2}n(n-1)=\sum_{l=1}^{n-1}l$ and $\rI^{\gp_n}_{2\rho_n}(\xi_{H(2\rho_n)})$ is described as $\mE^{\gp_n, >}_2\wedge \mE^{\gp_n, >}_3
\wedge \dots \wedge \mE^{\gp_n, >}_{n}$ by the definition of $\rI^{\gp_n }_{2\rho_n}$ (\ref{eq:xiH2rho}) . 
Furthermore, Lemma~\ref{lem:Clem3} combined with (\ref{eq:pairing_Lie}), (\ref{eq:boldE}) implies the equality
\begin{align}
\label{eq:pf003_coefC}
&\rs_n\bigl(\mE^{\gp_n, \geq }_1\wedge \mE^{\gp_n, \geq }_2\wedge \dots \wedge \mE^{\gp_n, \geq }_{n-1},
\rI^{\gp_{n-1}}_{2\rho_{n-1}}(  \xi_{H(-2\rho_{n-1})}  )\bigr)
=(-1)^{-\frac{1}{6}n(n-1)(n-2)}.
\end{align}
By (\ref{eq:pf002_coefC}) and (\ref{eq:pf003_coefC}), 
we obtain the equality (\ref{eq:pf001_coefC}), which completes the proof of Proposition~\ref{prop:coefC}. 
\end{proof}

\appendix

\section{Local systems and rational/integral structures} \label{sec:localsystem}

In this appendix, we summarise several facts on local systems defined on the symmetric space $Y^{(n)}_{\mathcal{K}}$, and discuss rational and integral structures of various cohomology groups. Let $F$ be a totally imaginary number field as before, whose ring of integers is denoted as $\mathfrak{r}_F$. For any prime number $p$, let $\mC_p$ denote the completion of an algebraic closure of the $p$-adic number field $\mQ_p$, and fix an isomorphism $\bi\colon \mC\xrightarrow{\, \sim \,}\mC_p$ of fields. 

\subsection{Rational representation of $\rGL_n$}  \label{subsec:alg_rep} 

Let us review several facts on rational representations of $\rGL_{n/F}$. Refer also to \cite[Section~2.2]{hr20}. For $\blambda=(\lambda_\sigma)_{\sigma\in I_F}\in \Lambda_n^{I_F}$, we first consider the tensor product  $\widetilde{V}(\blambda):=\bigotimes_{\sigma \in I_F} V_{\lambda_\sigma}$, where $(\tau_{\lambda_\sigma}, V_{\lambda_\sigma})$ is a complex irreducible representation of $\rGL_n(\mC)$ with highest weight $\lambda_\sigma$. We define the action $\tau_\blambda$ of $\rGL_n(F)$ on $\widetilde{V}(\blambda)$ as 
\begin{align*}
 \tau_{\blambda}(g) \left(\bigotimes_{\sigma\in I_F} \rv_\sigma\right)&:=\bigotimes_{\sigma\in I_F} \bigl( \tau_{\lambda_\sigma}(\sigma (g)) \rv_\sigma\bigr) &\text{for }g\in \rGL_n(F).
\end{align*}
This action  is naturally extended to that of $\rGL_n(F_{\mA,\infty})$ by 
\begin{align*}
 \tau_{\blambda}\bigl((g_v)_{v\in \Sigma_{F,\infty}}\bigr) \left(\bigotimes_{v\in \Sigma_{F,\infty}} (\rv_v \otimes \rv_{\bar{v}})\right)&=\bigotimes_{v\in \Sigma_{F,\infty}} \bigl( \tau_{\lambda_v}(g_v) \rv_v \otimes \tau_{\lambda_{\bar{v}}}(\bar{g}_v)\rv_{\bar{v}}\bigr) &\text{for }(g_v)_{v\in \Sigma_{F,\infty}}\in \rGL_n(F_{\mA,\infty}).
\end{align*}
If $V^{\mathrm{conj}}_{\lambda_{\bar{v}}}$ denotes the complex conjugate representation defined as (\ref{eq:comp_rep}), we can regard $\widetilde{V}(\blambda)$ as the exterior tensor representation $\bigboxtimes_{v\in \Sigma_{F,\infty}} (V_{\lambda_v}\otimes_{\mC} V^{\mathrm{conj}}_{\lambda_{\bar{v}}})$ of $\rGL_n(F_{\mA,\infty})$. For each $\boldsymbol{M}=(M_\sigma)_{\sigma\in I_F}\in \rG(\blambda)=\prod_{\sigma\in I_F}\rG(\lambda_\sigma)$, set $\xi_{\boldsymbol{M}}:=\bigotimes_{\sigma\in I_F}\xi_{M_\sigma}$. Then $\xi_{H(\blambda)}$ for $H(\blambda)=(H(\lambda_\sigma))_{\sigma \in I_F}$ is a highest weight vector of $\widetilde{V}(\blambda)$. 

For any $\alpha\in \mathrm{Aut}(\mC)$, the $\alpha$-twist ${}^\alpha \widetilde{V}(\blambda):=\widetilde{V}(\blambda)\otimes_{\mC,\alpha}\mC$ of $\widetilde{V}(\blambda)$, on which $\rGL_n(F)$ acts via $\tau_\blambda \otimes \mathrm{id}_{\mC}$, is known to be isomorphic to $\widetilde{V}({}^\alpha \blambda)$ for ${}^\alpha \blambda=({}^\alpha \lambda_\sigma)_{\sigma \in I_F}$ (where ${}^\sigma \lambda_\sigma=\lambda_{\alpha^{-1}\circ \sigma}$ for $\sigma\in I_F$) as a complex $\rGL_n(F)$-representation; refer also to \cite[Lemma 7.1]{gr14}. We can explicate this isomorphism as follows. By using the determinantal model of the highest weight representation $V_{\lambda}$ of $\rGL_n(\mC)$ of weight $\lambda\in \Lambda_n$ (see Section~\ref{subsec:realisation} for details), let us put 
\begin{align*}
 \alpha(\rv)&= \sum_{l\in \cL(\lambda)} \alpha(c_l)f_l(z) & \text{for \;} \alpha \in \mathrm{Aut}(\mC) \text{\; and \;} \rv=\sum_{l\in \cL(\lambda)}c_lf_l(z) \in V_\lambda.
\end{align*}
Then the map $\Upsilon_\alpha \colon {}^\alpha \widetilde{V}(\blambda)\rightarrow \widetilde{V}({}^\alpha \blambda)$ defined by $\bigl(\bigotimes_{\sigma \in I_F} \rv_\sigma\bigr)  \otimes c  \mapsto c\bigl(\bigotimes_{\sigma \in I_F} \alpha (\rv_{\alpha^{-1}\circ \sigma})\bigr)$ indeed gives an isomorphism ${}^\alpha \widetilde{V}(\blambda)\cong \widetilde{V}({}^\alpha \blambda)$. We can consider an $\alpha$-semilinear isomorphism (``$\alpha$-twisting isomorphism'')
\begin{align} \label{eq:twist_a}
 \mathrm{tw}_\alpha \colon \widetilde{V}(\blambda)\rightarrow {}^\alpha \widetilde{V}(\blambda)\cong \widetilde{V}({}^\alpha \blambda) \, ; \, \mv \mapsto \Upsilon_{\alpha}(\mv \otimes 1).
\end{align}
Then, if we define ${}^\alpha \boldsymbol{M}=({}^\alpha M_\sigma)_{\sigma \in I_F}$ for $\boldsymbol{M}\in \rG(\blambda)$ and $\alpha\in \mathrm{Aut}(\mC)$ by setting ${}^\alpha M_\sigma=M_{\alpha^{-1}\circ \sigma}$,  we have 
\begin{align*}
 \mathrm{tw}_\alpha(\xi_{\boldsymbol{M}})&=\Upsilon_\alpha \bigl( \xi_{\boldsymbol{M}}\otimes 1\bigr) =\bigotimes_{\sigma \in I_F} \alpha(\xi_{M_{\alpha^{-1}\circ \sigma}})=\bigotimes_{\sigma \in I_F} \xi_{{}^\sigma M_\sigma}= \xi_{{}^\alpha \boldsymbol{M}}
\end{align*}
(the third equality follows from Proposition~\ref{prop:xiM_integral}). Note that, if we regard $\blambda$ as a character on the diagonal torus $\rT_n(F)\rightarrow \mC^\times \, ; (t_i)_{1\leq i\leq n}\mapsto \prod_{\sigma \in I_F} \prod_{1\leq i\leq n} (\sigma(t_i))^{\lambda_{\sigma,i}}$, the eigenspace  of $\widetilde{V}(\blambda)$ on which $\rT_n(F)$ acts via $\blambda$ is of dimension one, due to highest weight theory. This implies that $\widetilde{V}(\blambda)$ admits a $\mQ(\blambda)$-rational structure
\begin{align} \label{eq:rationalstr}
 \widetilde{V}(\blambda)_{\mQ(\blambda)}=\langle \tau_\blambda(g) \xi_{H(\blambda)} \mid g\in \rGL_n(F)\rangle_{\mQ(\blambda)}
\end{align}
for the subfield $\mQ(\blambda)$ of $\mC$ fixed by $\{\alpha \in \mathrm{Aut}(\mC) \mid {}^\alpha\blambda=\blambda\}$, due to \cite[Lemme~I.1]{Waldspurger}. For any field  $A$ equipped with an injection $\mQ(\blambda)\hookrightarrow A$, set $\widetilde{V}(\blambda)_A:=\widetilde{V}(\blambda)_{\mQ(\blambda)}\otimes_{\mQ(\blambda)}A$.

Finally, for $\blambda\in \Lambda_n^{I_F}$,  $\bmu\in\Lambda_{n-1}^{I_F}$ and $m\in \mZ$ satisfying $\blambda^\vee \succeq \bmu+m$, let us define global pairings
\begin{align} \label{eq:global_pairing1}
 [\cdot,\cdot ]_{\blambda,\bmu}^{(m)}:=\prod_{\sigma\in I_F} \langle \cdot, \cdot\rangle_{\lambda_\sigma^\vee, \mu_\sigma^\vee}^{(m)}  \in  \mathrm{Hom}_{\rGL_{n-1}(F)}( \widetilde{V}(\blambda^\vee) \otimes_{\mC} \widetilde{V}(\bmu^\vee), \mC_{\lVert \det(\cdot)\rVert_{F_{\mA,\infty}}^m}) 
\end{align}
where $\langle \cdot,\cdot\rangle_{\lambda,\mu}^{(l)}$ denotes the local pairing defined as (\ref{eq:pairingLambdaMu}).

\begin{lem} \label{lem:pairing_equiv}
The pairing $[\cdot,\cdot]_{\blambda,\bmu}^{(m)}$ is $\mathrm{Aut}(\mC)$-equivariant{\rm ;} that is, 
\begin{align*}
[\mathrm{tw}_\alpha (\mv_1), \mathrm{tw}_\alpha(\mv_2)]_{{}^\alpha \blambda,{}^\alpha \bmu}^{(m)}=\alpha([\mv_1,\mv_2]_{\blambda,\bmu}^{(m)}) 
\end{align*}
holds for $\mv_1\in \widetilde{V}(\blambda^\vee)$, $\mv_2\in \widetilde{V}(\bmu^\vee)$ and $\alpha\in \mathrm{Aut}(\mC)$. 

\end{lem}

\begin{proof}
Due to $\mC$-bilinearity of $[\cdot,\cdot]_{\blambda,\bmu}^{(m)}$, it suffices to verify the statement for $\mv_1=\xi_{\boldsymbol{M}}$ and $\mv_2=\xi_{\boldsymbol{N}}$ with $\boldsymbol{M}\in \rG(\blambda^\vee)$ and $\boldsymbol{N}\in \rG(\bmu^\vee)$. But, due to (\ref{eq:PairExplicit}), we have 
\begin{align*}
 \alpha([\xi_{\boldsymbol{M}},\xi_{\boldsymbol{N}}]_{\blambda,\bmu}^{(m)})&=
\begin{cases}
 \prod_{\sigma\in I_F} (-1)^{\rrq(M_\sigma)}\rr(M_\sigma) & \text{if } \widehat{\boldsymbol{M}}=\boldsymbol{N}^\vee+m, \\
0 & \text{otherwise},
\end{cases} \\
[\mathrm{tw}_\alpha(\xi_{\boldsymbol{M}}), \mathrm{tw}_\alpha(\xi_{\boldsymbol{N}})]_{{}^\alpha \blambda,{}^\alpha \bmu}^{(m)} &=[
\xi_{{}^\alpha \boldsymbol{M}}, \xi_{{}^\alpha \boldsymbol{N}}]_{{}^\alpha \blambda,{}^\alpha\bmu}^{(m)}=
\begin{cases}
 \prod_{\sigma\in I_F} (-1)^{\rrq({}^\alpha M_\sigma)}\rr({}^\alpha M_\sigma) & \text{if } \widehat{{}^\alpha\boldsymbol{M}}={}^\alpha\boldsymbol{N}^\vee+m, \\
0 & \text{otherwise}.
\end{cases}
\end{align*} 
Since $\prod_{\sigma \in I_F} (-1)^{\rrq(M_\sigma)}\rr(M_\sigma)$ equals $\prod_{\sigma\in I_F}(-1)^{\rrq({}^\alpha M_\sigma)}\rr({}^\alpha M_\sigma)$ by definition, and $\widehat{{}^\alpha \boldsymbol{M}}={}^\alpha \boldsymbol{N}^\vee+m$ holds if and only if $\widehat{\boldsymbol{M}}=\boldsymbol{N}^\vee+m$ holds, we readily obtain the desired equation $[\mathrm{tw}_\alpha(\xi_{\boldsymbol{M}}), \mathrm{tw}_\alpha(\xi_{\boldsymbol{N}})]_{{}^\alpha \blambda,{}^\alpha \bmu}^{(m)}=\alpha([\xi_{\boldsymbol{M}},\xi_{\boldsymbol{N}}]_{\blambda,\bmu}^{(m)})$.
\end{proof}

\begin{lem} \label{lem:pairing_rational}
The global pairing $[\cdot,\cdot]_{\blambda,\bmu}^{(m)}$ 
is rational 
in the following sense$;$ for any field $A$ containing both $\mQ(\blambda)$ and $\mQ(\bmu)$, 
we have
\begin{align*}
  [\cdot,\cdot ]_{\blambda,\bmu}^{(m)} \in  \mathrm{Hom}_{\rGL_{n-1}(F)}( \widetilde{V}(\blambda^\vee)_{A} \otimes_{A} \widetilde{V}(\bmu^\vee)_{A}, A_{\lVert\det(\cdot)\rVert_{F_{\mA,\infty}}^m}) 
\end{align*}
\end{lem}

Note that it suffices to prove Lemma~\ref{lem:pairing_rational} when $A$ is a subfield of $\mC$. We first verify the following fact. Set 
\begin{align*}
 \langle \cdot, \cdot\rangle_\blambda:=\prod_{\sigma\in I_F}\langle \cdot,\cdot\rangle_{\lambda_\sigma}\colon \widetilde{V}(\blambda)\otimes_{\mC}\widetilde{V}(\blambda^\vee)\rightarrow \mC, 
\end{align*}
where $\langle \cdot,\cdot\rangle_{\lambda_\sigma}$ is the pairing defined as (\ref{eq:pairing_naive}); here we consider the case where $\cA=\mC$. 

\begin{lem} \label{lem:pairing_inv}
Take $\boldsymbol{M}=(M_\sigma)_{\sigma\in I_F} \in \rG(\blambda)$ and $\boldsymbol{N}=(N_\sigma)_{\sigma\in I_F} \in \rG(\blambda^\vee)$. Suppose that $\alpha\in \mathrm{Aut}(\mC)$ satisfies ${}^\alpha \blambda=\blambda$, ${}^{\alpha}\boldsymbol{M}=\boldsymbol{M}$ and ${}^{\alpha}\boldsymbol{N}=\boldsymbol{N}$. Then, 
 for every $g\in \rGL_n(F)$, 
the value $\langle \tau_\blambda(g)\xi_{\boldsymbol{M}}, \xi_{\boldsymbol{N}}\rangle_\blambda$ 
is fixed by $\alpha$.
\end{lem}

\begin{proof}
 We can formally verify the statement as 
\begin{align*}
\alpha( \langle \tau_\blambda(g)\xi_{\boldsymbol{M}},\xi_{\boldsymbol{N}}\rangle_\blambda) &=\prod_{\sigma\in I_F}\langle \tau_{\lambda_\sigma}(\alpha\circ \sigma(g))\xi_{M_\sigma},\xi_{N_\sigma}\rangle_{\lambda_\sigma} =\prod_{\sigma\in I_F} \langle \tau_{{}^\alpha \lambda_{\alpha\circ \sigma}} (\alpha\circ \sigma(g))\xi_{{}^\alpha M_{\alpha\circ\sigma}}, \xi_{{}^{\alpha} N_{\alpha \circ \sigma}}\rangle_{{}^\alpha \lambda_{\alpha\circ \sigma}} \\
&=\prod_{\sigma \in I_F}\langle \tau_{{}^\alpha \lambda_\sigma}(\sigma(g))\xi_{{}^\alpha M_\sigma},\xi_{{}^\alpha N_\sigma}\rangle_{{}^\alpha \lambda_\sigma}=\langle \tau_{{}^\alpha \blambda}(g)\xi_{{}^\alpha \boldsymbol{M}},\xi_{{}^\alpha \boldsymbol{N}}\rangle_{{}^\alpha \blambda}=\langle \tau_{ \blambda}(g)\xi_{\boldsymbol{M}},\xi_{\boldsymbol{N}}\rangle_{ \blambda}.\qedhere 
\end{align*}
\end{proof}

\begin{proof}[Proof of Lemma~$\ref{lem:pairing_rational}$]
As we have already remarked, we may assume that $A$ is a subfield of $\mC$ containing both $\mQ(\blambda)$ and $\mQ(\bmu)$. By the definitions of the pairing $\langle \cdot,\cdot\rangle_{\lambda,\mu}^{(l)}$ \eqref{eq:pairingLambdaMu} and the $A$-rational structure (\ref{eq:rationalstr}), it suffices to verify that the morphisms $\rI^{\det}_{\bmu^\vee,m}=\bigotimes_{\sigma\in I_F}\rI^{\det}_{\mu_\sigma^\vee,m}$ and $\rI^{\blambda}_{\bmu^\vee-m}=\bigotimes_{\sigma\in I_F} \rI^{\lambda_\sigma}_{\mu_\sigma^\vee-m}$ preserve the $A$-rational structure, and the pairing $\langle \cdot,\cdot\rangle_{\blambda^\vee}$ restricted to $\widetilde{V}(\blambda^\vee)_A$ and $\widetilde{V}(\blambda)_A$ takes values in $A$. First of all, we have
\begin{align*}
 \rI^{\mathrm{det}}_{\bmu^\vee,m}(\tau_\bmu^\vee(g) \xi_{H(\bmu^\vee)})&=\lVert\det(g)\rVert_{F_{\mA,\infty}}^m \tau_{\bmu^\vee-m}(g)\xi_{H(\bmu^\vee-m)} \in \widetilde{V}(\bmu^\vee-m)_A
\end{align*}
by (\ref{eq:detl_shift}), which implies that $\rI^{\mathrm{det}}_{\bmu^\vee,m}$ preserves the $A$-rational structure. 
Next we verify that $\rI_{\bmu^\vee-m}^{\blambda}$ preserves the $A$-rational structure. 
Note that, due to our assumption on $A$, the $A$-representation $\widetilde{V}(\blambda )_A$ is regarded as a subspace of $\widetilde{V}(\blambda )$. 
Set $H(\bmu^\vee-m)[\blambda ]:=(H(\mu_\sigma^\vee-m)[\lambda_\sigma ])_{\sigma \in I_F}$. 
Then the weight vector $\xi_{H(\bmu^\vee-m)[\blambda]}$ is stable under the $\alpha$-twisting 
$\mathrm{tw}_\alpha \colon \widetilde{V}(\blambda)\to {}^\alpha\widetilde{V}(\blambda)=\widetilde{V}(\blambda)$ for any $\alpha \in \mathrm{Aut}(\mC)$ fixing $\blambda$ and $\bmu$; namely we have 
\begin{align*}
\mathrm{tw}_\alpha (\xi_{H(\bmu^\vee-m)[\blambda ]})&=\xi_{H({}^\alpha(\bmu^\vee-m))[{}^\alpha\blambda ]}=\xi_{H(\bmu^\vee-m)[\blambda ]}.
\end{align*}
Then we can show that $\langle \tau_\blambda(g) \xi_{H(\bmu^\vee-m)[\blambda ]} \mid g\in \rGL_n(F)\rangle_{A}$ 
defines another $A$-rational structure on $\widetilde{V}(\blambda )$ 
by an argument similar to the proof of the former part of \cite[Lemme I.1]{Waldspurger}. 
Furthermore, it is homothetic to $\widetilde{V}(\blambda )_A$ by the latter part of \cite[Lemme I.1]{Waldspurger}. 
Hence there exist $a_1,a_2,\dotsc,a_s\in A$, $g_1,g_2,\dotsc,g_s\in \rGL_n(F)$ and $c\in \mC^\times $ such that
\begin{align*}
 c\, \xi_{H(\bmu^\vee-m)[\blambda]} &=a_1\tau_{\blambda}(g_1)\xi_{H(\blambda)}+a_2\tau_{\blambda}(g_2)\xi_{H(\blambda)}+\cdots +a_s\tau_{\blambda}(g_s)\xi_{H(\blambda)}.
\end{align*}
Then, by taking $\langle -,\xi_{H(\bmu+m)[\blambda^\vee]}\rangle_{\blambda}$, we have
\begin{align*}
 (-1)^{\rrq(H(\bmu^\vee-m)[\blambda])}\rr(H(\bmu^\vee-m)[\blambda])c=\sum_{i=1}^s a_i\langle \tau_{\blambda}(g_i)\xi_{H(\blambda)},\xi_{H(\bmu+m)[\blambda^\vee]}\rangle_{\blambda}. 
\end{align*}
Note that $(-1)^{\rrq(H(\bmu^\vee-m)[\blambda])}\rr(H(\bmu^\vee-m)[\blambda])$ is a nonzero rational number, and the right hand side of the equality above is contained in $A$ by Lemma~\ref{lem:pairing_inv}. Thus the scalar multiple $c$ is indeed an element of $A^\times$, which implies that $\rI_{\bmu^\vee-m}^{\blambda}(\tau_{\bmu^\vee-m}(g)\xi_{H(\bmu^\vee-m)})=\tau_{\blambda}(\iota_n(g))\xi_{H(\bmu^\vee-m)[\blambda]}$ is contained in $\widetilde{V}(\blambda)_{A}$ for all $g\in \rGL_n(F)$. Consequently $\rI_{\bmu^\vee-m}^\blambda$ also preserves the $A$-rational structure.

Finally Lemma~\ref{lem:pairing_inv} implies that $\langle \tau_{\blambda^\vee}(g)\xi_{H(\blambda^\vee)},\tau_{\blambda}(g')\xi_{H(\blambda)}\rangle_{\blambda^\vee}=\langle \tau_{\blambda^\vee}((g')^{-1}g)\xi_{H(\blambda^\vee)},\xi_{H(\blambda)}\rangle_{\blambda^\vee}$ is indeed an element of $A$ for any $g,g'\in \rGL_n(F)$. Hence the pairing $\langle \cdot,\cdot\rangle_{\blambda^\vee}$ restricted to $\widetilde{V}(\blambda^\vee)_A$ and $\widetilde{V}(\blambda)_A$ takes values in $A$. 
\end{proof}

\subsection{Integral representation of $\rGL_n$}  \label{subsec:integral_rep} 

Now we consider integral representations of $\rGL_n$. Let $F_{\mathrm{nc}}$ denote the normal closure of $F$ in $\mC$, $\mathfrak{r}_{F_{\mathrm{nc}}}$ the ring of integers of it, and $\cO_{\mathrm{nc}}$ the $p$-adic closure of the image of $F_{\mathrm{nc}}\subset \mC \xrightarrow{\, \boldsymbol{i}\, }\mC_p$. We always assume that $\cA$ is a subring of $\mC_p$ containing $\cO_{\mathrm{nc}}$. Now consider the tensor product $\widetilde{V}(\blambda)^{(p)}_{\cA}:=\bigotimes_{\sigma \in I_F} V_{\lambda_{\sigma}}(\cA )$ for $\blambda \in \Lambda_n^{I_F}$, where $(\tau_{\lambda_{\sigma}}, V_{\lambda_{\sigma}}(\cA))$ is a finite rank representation of $\rGL_n(\cA)$ over $\cA$ defined as in Section~\ref{subsec:realisation}. As before, we define the action $\tau_{\blambda}^{(p)}$ of $\rGL_n(\mathfrak{r}_F)$  on $\widetilde{V}(\blambda)^{(p)}$ as
\begin{align*}
 \tau_{\blambda}^{(p)}(g) \left(\bigotimes_{\sigma\in I_F} \rv_\sigma\right)&:=\bigotimes_{\sigma \in I_F} \bigl( \tau_{\lambda_\sigma}(\bi \circ \sigma (g)) \rv_\sigma\bigr) &\text{for }g\in \rGL_n(\mathfrak{r}_F).
\end{align*}
Note that, since each $\bi\circ \sigma$ takes values in $\cO_{\mathrm{nc}}$ by construction, the pair $(\tau_{\blambda}^{(p)},\widetilde{V}(\blambda)^{(p)}_{\cA})$  indeed defines an $\cA$-representation of $\rGL_n(\mathfrak{r}_F)$ of finite rank, that is, $\widetilde{V}(\blambda)^{(p)}_{\cA}$ is a free $\cA$-module of finite rank on which $\rGL_n(\mathfrak{r}_F)$ acts via $\tau_\blambda^{(p)}$.  
The action $\tau_{\blambda}^{(p)}$ of $\rGL_n(\mathfrak{r}_F)$ is naturally extended to that of $\rGL_n(\mathfrak{r}_F \otimes_{\mZ} \mZ_p)=\prod_{v\mid (p)}\rGL_n(\mathfrak{r}_{F,v})$ by
\begin{align*}
 \tau_{\blambda}^{(p)}\bigl((g_v)_{v\mid (p)}\bigr) \left(\bigotimes_{v\mid (p)} \bigotimes_{\sigma \in I_{F,v}} \rv_\sigma\right) &= \bigotimes_{v\mid (p)} \bigotimes_{\sigma \in I_{F,v}} \tau_{\lambda_\sigma}(\sigma_v (g_v))\rv_\sigma &  \text{for }(g_v)_{v\mid (p)}\in \rGL_n(\mathfrak{r}_F\otimes_{\mZ} \mZ_p).
\end{align*}
Here the subset $I_{F,v}$ of $I_F$ is defined for every place $v$ of $F$ lying above $(p)$ as 
\begin{align*}
 I_{F,v}:=\{ \sigma \in I_F \mid \text{$v$ is induced by } \bi\circ \sigma \colon  F\hookrightarrow \mC_p  \},
\end{align*}
and $\sigma_v$ (for $\sigma \in I_{F,v}$) denotes the automorphism of $\mathfrak{r}_{F,v}$ induced by $\bi\circ \sigma$. 
If $\mathcal{A}$ is a field, we also define an action of $\mathrm{GL}_n(F\otimes_\mathbf{Q}\mathbf{Q}_p)$ 
on $\widetilde{V}(\blambda)^{(p)}_{\cA}$
in the same way and we denote it by $\tau^{(p)}_{\boldsymbol{\lambda}}$ again.

Define the $\rGL_n(\mathfrak{r}_F)$-equivariant pairing $\langle \cdot,\cdot\rangle_{\blambda}^{p}$ on $\widetilde{V}(\blambda)^{(p)}_{\cA}$ by 
\begin{align*}
 \langle \cdot, \cdot\rangle_\blambda^p&:=\prod_{\sigma\in I_F}\langle \cdot,\cdot\rangle_{\lambda_\sigma}\colon \widetilde{V}(\blambda)^{(p)}_{\cA}\otimes_{\cA}\widetilde{V}(\blambda^\vee)^{(p)}_{\cA}\rightarrow \cA.
\end{align*}
The definition of the (local) pairing $\langle \cdot,\cdot\rangle_{\lambda_\sigma}$ \eqref{eq:pairing_naive} implies that $\widetilde{V}(\blambda)_{\cA}^{(p)}$ is a {\em self dual lattice} in $\widetilde{V}(\blambda)_{\cA}^{(p)}\otimes_{\cA} \mathrm{Frac}(\cA)$ with respect to $\langle \cdot,\cdot\rangle^p_{\blambda}$ if one assumes $p>\max_{\sigma\in I_F} \{\lambda_{\sigma,1}-\lambda_{\sigma,n}+n-2\}$. We also define the global pairing 
\begin{align} \label{eq:global_pairing2}
 [\cdot,\cdot ]_{\blambda,\bmu}^{(m),p}&:=\prod_{\sigma\in I_F} \langle \cdot, \cdot\rangle_{\lambda_\sigma^\vee, \mu_\sigma^\vee}^{(m)}  \in  \mathrm{Hom}_{\rGL_n(F)}( \widetilde{V}(\blambda^\vee)^{(p)}_{\cA} \otimes_{\cA} \widetilde{V}(\bmu^\vee)^{(p)}_{\cA}, \cA_{ \lVert \det(\cdot)\rVert_{F_{\mA,\infty}}^m})
\end{align}
for $\blambda\in \Lambda_n^{I_F}$,  $\bmu\in\Lambda_{n-1}^{I_F}$ and $m\in \mZ$ satisfying $\blambda^\vee \succeq \bmu+m$.

\subsection{Local systems}\label{sec:locsys}

In this subsection we discuss the local system $\widetilde{\cV}(\blambda)_{\mathsf{E}}$ and the $p$-adic local system $\widetilde{\cV}(\blambda)^{(p)}_{\cA}$ on the symmetric space
\begin{align*}
   Y^{(n)}_{\mathcal K} 
     = {\rm GL}_n(F)  \backslash  {\rm GL}_n(F_{\mathbf A})  / \widetilde{K}_n  {\mathcal K}_n,
\end{align*}
which are respectively defined by $\widetilde{V}(\blambda)_{\mathsf{E}}$ and $\widetilde{V}(\blambda)^{(p)}_{\cA}$.   
For any field $\mathsf{E}$ containing $\mQ(\blambda)$, consider the diagonal left action of $\mathrm{GL}_n(F)$ 
on the direct product  ${\rm GL}_n(F_{\mathbf A})  / \widetilde{K}_n  {\mathcal K}_n     \times \widetilde{V}(\blambda)_{\mathsf{E}}$, which is denoted by $\gamma\cdot ([g]\!],\mathbf{v})=([\gamma g]\!],\tau_{\blambda}(\gamma)\mathbf{v})$ for $\gamma\in \rGL_n(F)$.  Here $[g]\!]$ denotes the left coset $g\widetilde{K}_n\mathcal{K}_n$ containing $g \in \rGL_n(F_{\mA})$. We use the symbol ${}_{\rGL_n(F)}([g]\!],\mv)$ for the left $\rGL_n(F)$-orbit containing $([g]\!],\mv)$.
Let $\widetilde{\cV}(\blambda)_{\mathsf{E}}$ be the local system on $Y^{(n)}_{\mathcal K}$ defined as the sheaf of locally constant sections of the 
first projection
\begin{align*}
    \mathrm{GL}_n(F) \backslash 
    \left( {\rm GL}_n(F_{\mathbf A})  / \widetilde{K}_n {\mathcal K}_n \times \widetilde{V}(\blambda)_\mathsf{E} \right)
    \longrightarrow 
   Y^{(n)}_{\mathcal K}.
\end{align*}
Note that we equip $\widetilde{V}(\blambda)_\mathsf{E}$ with the discrete topology.
For each double coset $[\![g]\!]=\rGL_n(F)g\widetilde{K}_n\cK_n \in Y^{(n)}_{\mathcal K}$, we write a local section around $[\![g]\!]$ as ${}_{\rGL_n(F)}([g]\!], \mathbf{v}_{[g]\!]})$ for $[g]\!]\in \rGL_n(F_{\mA})/\widetilde{K}_n\mathcal{K}_n$ and $\mathbf{v}_{[g]\!]} \in \widetilde{V}(\blambda)_{\mathsf{E}}$. Since  $[\![g]\!]$ and $[\![\gamma g]\!]$ define the same point on $Y^{(n)}_{\cK}$ for $\gamma\in \rGL_n(F)$ and $g\in \rGL_n(F_{\mA})$, local sections ${}_{\rGL_n(F)}([g]\!], \mv_{[g]\!]})$ and ${}_{\rGL_n(F)}([\gamma g]\!], \mv_{[\gamma g]\!]})$ must coincide. Comparing the first component, we observe that this coincidence implies  $\gamma\cdot ([g]\!],\mv_{[g]\!]})=([\gamma g]\!], \mv_{[\gamma g]\!]})$, and obtain the relation of local sections
\begin{align} \label{eq:local_relation}
 \mv_{[\gamma g]\!]}&=\tau_\blambda(\gamma) \mv_{[g]\!]} & \text{for every \;} \gamma \in \rGL_n(F) \text{\; and \;} g\in \rGL_n(F_{\mA}).  
\end{align}   
For $m\in \mZ$, let $\widetilde{\mathsf{E}}(m)$ be the local system defined by the character $\tau^{\det}_m =\lVert\det(\cdot) \rVert^m_{F_{\mA,\infty}} \! =\prod_{\sigma\in I_F}\sigma\circ \det^m\colon \rGL_n(F)\rightarrow \mathsf{E}^\times$. Then there exists a twisting isomorphism 
\begin{align} \label{eq:tw_e}
 \mathrm{Tw}_m \colon \widetilde{\mathsf{E}}(m) \xrightarrow{\, \sim \,} \widetilde{\mathsf{E}}_{\mathrm{triv}} \, ; {}_{\rGL_n(F)}([g]\!], e_{[g]\!]}) \longmapsto {}_{\rGL_n(F)}\left([g]\!], \prod_{v\in \Sigma_{F,\mathrm{fin}}} \lvert \det(g_v)\rvert_v^m e_{[g]\!]}\right),
\end{align}
where $\widetilde{\mathsf{E}}_\mathrm{triv}$ denotes the trivial local system with coefficients in $\mathsf{E}$.

Next, for any flat $\cO_{\mathrm{nc}}$-algebra $\cA$ containing $\cO_{\mathrm{nc}}$, consider the  right diagonal action of $\cK_n$ 
on the direct product  $\rGL_n(F)\backslash{\rm GL}_n(F_{\mathbf A})  / \widetilde{K}_n  
    \times \widetilde{V}(\blambda)_{\cA}^{(p)}$, which is denoted by $([\![g],\mathbf{v}_p)\cdot u=([\![gu],\tau_{\blambda}^{(p)}(u^{-1}_p)\mathbf{v}_p)$. Here $[\![g]$ denotes the double coset $\rGL_n(F)g\widetilde{K}_n$ containing $g \in \rGL_n(F_{\mA})$ and $u_p=(u_v)_{v\mid (p)}$ denotes the $p$-component of $u$.  We use the symbol $([\![g],\mv_p)_{\cK_n}$ for the right $\cK_n$-orbit containing $([\![g],\mv_p)$.  
Let $\widetilde{\cV}(\blambda)^{(p)}_{\cA}$ denote the local system on $Y^{(n)}_{\mathcal K}$ defined as the sheaf of locally constant sections of the 
first projection
\begin{align*}
    \left(\mathrm{GL}_n(F) \backslash 
     {\rm GL}_n(F_{\mathbf A})  / \widetilde{K}_n \times \widetilde{V}(\blambda)^{(p)}_{\cA} \right)/\mathcal{K}_n
    \longrightarrow 
   Y^{(n)}_{\mathcal K}.
\end{align*}
Note that we equip $\widetilde{V}(\blambda)^{(p)}_{\cA}$ with the discrete topology. A local section around $[\![g]\!] \in Y^{(n)}_{\mathcal{K}}$ is denoted as $([\![g], \mathbf{v}_{[\![g]})_{\cK_n}$ for $[\![g]\in \rGL_n(F)\backslash \rGL_n(F_{\mA})/\widetilde{K}_n$ and $\mathbf{v}_{[\![g]}\in \widetilde{V}(\blambda)_{\cA}^{(p)}$. 
By the same argument as before, we obtain the relation of local sections 
\begin{align} \label{eq:p_local_relation}
 \mv_{[\![gu]}&=\tau_{\blambda}^{(p)}(u_p^{-1})\mv_{[\![g]} & \text{for  every \;} u\in \cK_n \text{\; and \;} g\in \rGL_n(F_{\mA}).
\end{align}
For $m\in \mZ$, let $\widetilde{\cA}(m)^{(p)}$ be the local system defined by the character $\tau_m^{\det,(p)}=\prod_{\sigma\in I_F} \bi\circ \sigma \circ \det^m\colon \rGL_n(\mathfrak{r}_F)\rightarrow \cA^\times$. Then there exists a twisting isomorphism 
\begin{align} \label{eq:tw_a}
 \mathrm{Tw}_m^{(p)} \colon \widetilde{\cA}(m)^{(p)} \xrightarrow{\, \sim \,} \widetilde{\cA}_{\mathrm{triv}} \, ; ([\![g], a_{[\![g]})_{\cK_n} \longmapsto \left([\![g], \tau_m^{\det,(p)}(g_p) \prod_{v\in \Sigma_{F,\mathrm{fin}}} \lvert \det(g_v)\rvert^m_v a_{[\![g]}\right)_{\cK_n},
\end{align}
where $\widetilde{\cA}_\mathrm{triv}$ denotes the trivial local system with coefficients in $\cA$.

Now let $E^*$ be a subfield of $\mC$ which contains $F_{\mathrm{nc}}$ (note that $F_{\mathrm{nc}}$  contains $\mQ(\blambda)$ by definition). Let $\cE^*$ be the $p$-adic closure of the image of $E^*\subset \mC \xrightarrow{\, \boldsymbol{i}\,}\mC_p$. 
For $\cE^*$, we can define two local systems $\widetilde{\cV}(\blambda)_{\cE^*}$ and $\widetilde{\cV}(\blambda)^{(p)}_{\cE^*}$, which is identified in the following sense. Note that, on $\widetilde{V}(\blambda)_{\cE^*}=\widetilde{V}(\blambda)^{(p)}_{\cE*}$, the action $\tau_\blambda^{(p)}$ of $\rGL_n(\mathfrak{r}_F)$ is naturally extended to that of $\rGL_n(F)$ and $\rGL_n(F\otimes_{\mQ} \mQ_p)$, and we have $\tau_{\blambda}(\gamma)=\tau_{\blambda}^{(p)}(\gamma)=\tau_{\blambda}^{(p)}(\gamma_p)$ for $\gamma\in \rGL_n(F)$ where $\gamma_p$ denotes the $p$-component of the principal id\`ele generated by $\gamma$, that is, $\gamma_p=(\gamma)_{v\mid (p)}\in \prod_{v\mid (p)} F_v$.

\begin{lem}\label{lem:p-compatible}
The morphism $\Phi$ of local systems  on $Y^{(n)}_{\mathcal{K}}$ defined as
\begin{align*}
   \Phi\colon   \widetilde{\cV}(\blambda)_{\cE^*} \longrightarrow \widetilde{\cV}(\blambda)_{\cE^*}^{(p)}; \;  
     {}_{\rGL_n(F)}([g]\!], \mathbf{v}_{[g]\!]}) \longmapsto ([\![g],  \tau_\blambda^{(p)}(g_p^{-1}) \mathbf{v}_{[g]\!]})_{\cK_n}
\end{align*}
is an isomorphism, where $g_p=(g_v)_{v\mid (p)}$ denotes the $p$-component of $g=(g_v)_{v\in \Sigma_F}\in \rGL_n(F_{\mA})$.
\end{lem}

\begin{proof}
We first check that $\Phi$ is well defined.  
For each $\gamma \in \mathrm{GL}_n(F)$, $g\in \mathrm{GL}_n(F_{\mathbf{A}})$, $k\in \widetilde{K}_n$, and  $u \in \mathcal{K}_n$, 
it suffices to check that 
$([\![\gamma gku], \tau^{(p)}_\blambda((\gamma gk u)^{-1}_p) \mathbf{v}_{[\gamma gku]\!]})_{\cK_n}$ and 
$([\![g], \tau^{(p)}_\blambda(g^{-1}_p) \mathbf{v}_{[g]\!]})_{\cK_n}$ define the same section of $\widetilde{\cV}(\blambda)_{\cE^*}^{(p)}$. Note that $[\gamma gku]\!]=[\gamma g]\!]\in \rGL_n(F_{\mA})/\widetilde{K}_n\mathcal{K}_n$ and  $[\![\gamma gku]=[\![gu]\in \rGL_n(F)\backslash \rGL_n(F_{\mA})/\widetilde{K}_\infty$ hold. Then we can calculate as
\begin{align*}
([\![\gamma gku], \tau^{(p)}_\blambda((\gamma gk u)^{-1}_p) \mathbf{v}_{[\gamma gk u]\!]})_{\cK_n}
  &= ([\![gu], \tau^{(p)}_\blambda ( u^{-1}_p) \tau_{\blambda}^{(p)} (g^{-1}_p) \tau_{\blambda}(\gamma^{-1} )
     \mathbf{v}_{[\gamma g]\!]})_{\cK_n}     =( [\![gu], \tau_\blambda^{(p)}(u^{-1}_p)  \tau_\blambda^{(p)} ( g^{-1}_p ) \mathbf{v}_{[g]\!]})_{\cK_n} \\
&=([\![g], \tau_\blambda^{(p)}(g_p^{-1})\mathbf{v}_{[g]\!]})_{\cK_n}\cdot u= ([\![g], \tau_\blambda^{(p)}(g_p^{-1})\mathbf{v}_{[g]\!]})_{\cK_n} 
\end{align*}
as desired (the second equality follows from \eqref{eq:local_relation}). Hence $\Phi$ is well defined. 

To verify the isomorphy of $\Phi$, we define a morphism $\Psi\colon \widetilde{\cV}(\blambda)^{(p)}_{\cE^*} \to \widetilde{\cV}(\blambda)_{\cE^*}$ between local systems on $Y^{(n)}_{\mathcal{K}}$ as $([\![g], \mathbf{v}_{[\![g]})\mapsto ([g]\!],\tau_\blambda^{(p)}(g_p)\mathbf{v}_{[\![g]})$, and prove that $\Psi$ is the inverse of $\Phi$. Let us check that $\Psi$ is well defined. 
As before,
it suffices to check that 
${}_{\rGL_n(F)}([\gamma gk u]\!],   \tau^{(p)}((\gamma gk u)_p) \mathbf{v}_{[\![\gamma gku]})$ and 
${}_{\rGL_n(F)}([g]\!],  \tau^{(p)}(g_p) \mathbf{v}_{[\![g]})$ define the same section of $\widetilde{\cV}(\blambda)_{\cE^*}$ for each $\gamma \in \mathrm{GL}_n(F)$, $g\in \mathrm{GL}_n(F_{\mathbf{A}})$, $k\in \widetilde{K}_n$ and $u \in \mathcal{K}_n$. But we can calculate as
\begin{align*}
{}_{\rGL_n(F)}([\gamma gku]\!], \tau^{(p)}_\blambda((\gamma gk u)_p) \mathbf{v}_{[\![\gamma gku]})
  &= {}_{\rGL_n}([\gamma g]\!], \tau_\blambda (\gamma)  \tau_{\blambda}^{(p)} (g_p) \tau_{\blambda}^{(p)}(u_p )   
     \mathbf{v}_{[\![gu]})  ={}_{\rGL_n(F)}( [\gamma g]\!], \tau_\blambda(\gamma) \tau_\blambda^{(p)}(g_p)    \mathbf{v}_{[\![g]}) \\
&=\gamma\cdot {}_{\rGL_n(F)}([g]\!], \tau_\blambda^{(p)}(g_p)\mathbf{v}_{[\![g]}) ={}_{\rGL_n(F)}([g]\!], \tau_\blambda^{(p)}(g_p)\mathbf{v}_{[\![g]})  
\end{align*}
as desired (the second equality follows from \eqref{eq:p_local_relation}).
It is immediate to check that $\Phi$ and $\Psi$ give the inverse of each other, and hence the statement follows.
\end{proof}

We finally remark that, if we assume that $n>1$ and $\blambda^\vee \succeq \bmu+m$ holds for $\blambda\in \Lambda_n^{I_F}$, $\bmu\in \Lambda_{n-1}^{I_F}$ and $m\in \mZ$, we readily observe from the definitions (\ref{eq:global_pairing1}), \eqref{eq:global_pairing2}, (\ref{eq:tw_e}) and (\ref{eq:tw_a}) that the following diagram commutes.
\begin{align} \label{eq:pairing_compatible}
 \xymatrix{
j_n^*\widetilde{\cV}(\blambda^\vee)_{\cE^*} \times_{\mathcal{Y}^{(n-1)}_{\mathcal{K}}} \mathrm{p}_{n-1}^* \widetilde{\cV}(\bmu^\vee)_{\cE^*} \ar[rr]^-{[\cdot,\cdot]^{(m)}_{\blambda,\bmu}} \ar[d]_-{\Phi \times \Phi}^-{\rotatebox{90}{$\sim$}}     && \widetilde{\cE^*}(m) \ar[d]_-{\Phi}^{\rotatebox{90}{$\sim$}} \ar[rr]^-{\mathrm{Tw}_m} && \widetilde{\cE^*}_\mathrm{triv} \ar@{=}[d]  \\
j_n^*\widetilde{\cV}(\blambda^\vee)_{\cE^*}^{(p)} \times_{\mathcal{Y}^{(n-1)}_{\mathcal{K}}} \mathrm{p}_{n-1}^*\widetilde{\cV}(\bmu^\vee)_{\cE^*}^{(p)} \ar[rr]_-{[\cdot,\cdot]^{(m),p}_{\blambda,\bmu}} && \widetilde{\cE^*}(m)^{(p)} \ar[rr]_-{\mathrm{Tw}_m^{(p)}} && \widetilde{\cE^*}_\mathrm{triv}.
}
\end{align}

\subsection{Rational and integral structures of various cohomology groups} \label{subsec:ratint}

For a homomorphism of fields $\iota \colon \mathsf{F}\rightarrow \mathsf{E}$ and a finite dimensional $\mathsf{E}$-vector space $\mathsf{W}$ (also regarded as an $\mathsf{F}$-vector space via $\iota$), we say that an $\mathsf{F}$-subspace $\mathsf{W}_\mathsf{F}$ of $\mathsf{W}$ defines an {\em $\mathsf{F}$-rational structure} of $\mathsf{W}$ if $\mathsf{W}_{\mathsf{F}}\otimes_{\mathsf{F},\iota} \mathsf{E}$ is isomorphic to $\mathsf{W}$ as an $\mathsf{E}$-vector space. Similarly, for an embedding $\mathsf{\j}\colon \mathsf{R}\rightarrow \mathsf{E}$ of a principal ideal domain $\mathsf{R}$ into a field $\mathsf{E}$ and a finite dimensional $\mathsf{E}$-vector space $\mathsf{W}$, we say that a free $\mathsf{R}$-submodule $\mathsf{W}_{\mathsf{R}}$ of $\mathsf{W}$ defines an {\em $\mathsf{R}$-integral structure} of $\mathsf{W}$ if $\mathsf{W}_{\mathsf{R}}\otimes_{\mathsf{R},\mathsf{\j}} \mathsf{E}$ is isomorphic to $\mathsf{W}$ as an $\mathsf{E}$-vector space. In this subsection, we equip appropriate rational and integral structures on various cohomology groups of $Y^{(n)}_{\cK}$ by using local systems introduced in Section~\ref{sec:locsys}.

Let $E$ be a subextension $E$ of $\mC/\mQ(\blambda)$ of finite degree over $\mQ(\blambda)$, and let us take an arbitrary subextension $E^*$ of $\mC/E$ of finite degree over $E$ so that $E^*$ contains the normal closure $F_{\mathrm{nc}}$ of $F$ over $\mQ$. Let $\cE^*$ be the $p$-adic closure of the image of $E^* \subset \mC \xrightarrow{\, \boldsymbol{i}\,} \mC_p$, and $\cO^*$ the ring of integers of $\cE^*$. For each $\alpha \in \mathrm{Aut}(\mC)$, set ${}^\alpha E:=\alpha(E)$ and let ${}^\alpha \mathfrak{P}_0$ be the prime ideal of ${}^\alpha E$ induced by $\boldsymbol{i}\circ \alpha^{-1} \colon {}^\alpha E \subset \mC \xrightarrow{\, \alpha^{-1}\,} \mC \xrightarrow{\, \boldsymbol{i} \,} \mC_p$. Define $\mathfrak{r}_{{}^\alpha E,({}^\alpha \mathfrak{P}_0)}$ as the localisation of the ring of integers $\mathfrak{r}_{{}^\alpha E}$ of ${}^\alpha E$ at the prime ideal ${}^\alpha \mathfrak{P}_0$.  Note that $\boldsymbol{i}\circ \alpha^{-1}$ induces a field embedding $i_p^\alpha \colon {}^\alpha E\hookrightarrow \cE^*$, and we have a canonical isomorphism of $\cE^*$-representations of $\rGL_n(F)$
\begin{align} \label{eq:p_twist_isom}
\widetilde{V}({}^\alpha \blambda^\vee)_{{}^\alpha E}\otimes_{{}^\alpha E, i_p^\alpha} \cE^*\xrightarrow{\, \Upsilon_\alpha^{-1} \otimes \mathrm{id}\,} \bigl(\widetilde{V}(\blambda^\vee)_E \otimes_{E,\alpha} {}^\alpha E\bigr) \otimes_{{}^\alpha E, i_p^\alpha} \cE^* \xrightarrow{\, \sim \,} \widetilde{V}(\blambda^\vee)_{\cE^*}.  
\end{align}
Mimicking \eqref{eq:twist_a}, define $\mathrm{tw}_{i_p^\alpha} \colon \widetilde{V}({}^\alpha \blambda^\vee)_{{}^\alpha E}\rightarrow \widetilde{V}({}^\alpha \blambda^\vee)_{{}^\alpha E}\otimes_{{}^\alpha E, i_p^\alpha} \cE^* \cong \widetilde{V}(\blambda^\vee)_{\cE^*}$ as an $i^\alpha_p$-semilinear isomorphism induced by $\mv \mapsto \mv \otimes 1$. We always replace the superscript $\mathrm{id}$ whenever $\alpha$ is the identity map $\mathrm{id}\colon \mC\xrightarrow{\, =\,}\mC$. Then, by construction, we have $\mathfrak{r}_{{}^\alpha E}=\alpha(\mathfrak{r}_E)$ and ${}^\alpha \mathfrak{P}_0=\alpha(\mathfrak{P}_0)$. 

For $\widetilde{\cV}=\widetilde{\cV}({}^\alpha \blambda^\vee)$, $\widetilde{\cV}({}^\alpha \blambda^\vee)_{{}^\alpha E}$, $\widetilde{\cV}(\blambda^\vee)_{\cE^*}$, $\widetilde{\cV}({\blambda^\vee})_{\cE^*}^{(p)}$ or $\widetilde{\cV}(\blambda^\vee)_{\cO^*}^{(p)}$, let $H^*(Y^{(n)}_{\mathcal{K}}, \widetilde{\cV})$ be the cohomology group of $Y^{(n)}_{\cK}$ with coefficients in the local system $\widetilde{\cV}$, and $H^*_{\mathrm{c}}(Y^{(n)}_{\cK}, \widetilde{\cV})$ the compactly supported cohomology group with coefficients in $\widetilde{\cV}$. Then, due to flatness of $\mC$ over ${}^\alpha E$, the natural isomorphism $\widetilde{V}({}^\alpha \blambda^\vee)_{{}^\alpha E}\otimes_{{}^\alpha E}\mC \xrightarrow{\, \sim \,} \widetilde{V}({}^\alpha \blambda^\vee)$ induces an isomorphism of $\mC$-vector spaces $H^*_?(Y^{(n)}_{\cK}, \widetilde{\cV}({}^\alpha \blambda^\vee)_{{}^\alpha E})\otimes_{{}^\alpha E}\mC \xrightarrow{\, \sim \,} H^*_?(Y^{(n)}_{\cK}, \widetilde{\cV}({}^\alpha \blambda^\vee))$ for $?\in \{ \emptyset, \mathrm{c} \}$; in other words, $H^*_?(Y^{(n)}_{\cK}, \widetilde{\cV}({}^\alpha \blambda^\vee)_{{}^\alpha E})$ defines an ${}^\alpha E$-rational structure of $H^*_?(Y^{(n)}_{\cK}, \widetilde{\cV}({}^\alpha \blambda^\vee))$. Similarly, due to \eqref{eq:p_twist_isom}, $H^*_?(Y^{(n)}_{\cK}, \widetilde{\cV}({}^\alpha \blambda^\vee)_{{}^\alpha E})$ also defines an ${}^\alpha E$-rational structure of $H^*_?(Y^{(n)}_{\cK}, \widetilde{\cV}(\blambda)_{\cE^*})$ via the functorial map induced by $\mathrm{tw}_{i_p^\alpha}$, for which we use the same symbol $\mathrm{tw}_{i^\alpha_p}$ by abuse of notation. Finally, due to the natural isomorphism $\widetilde{V}(\blambda^\vee)_{\cO^*}^{(p)}\otimes_{\cO^*}\cE^* \xrightarrow{\, \sim \,} \widetilde{V}(\blambda^\vee)_{\cE^*}^{(p)}$ and flatness of $\cE^*$ over $\cO^*$, we obtain an isomorphism of $\cE^*$-vector spaces $H^*_?(Y^{(n)}_{\cK}, \widetilde{\cV}(\blambda^\vee)_{\cO^*}^{(p)})\otimes_{\cO^*}\cE^* \xrightarrow{\, \sim \,} H^*_?(Y^{(n)}_{\cK}, \widetilde{\cV}(\blambda^\vee)_{\cE^*}^{(p)})$; thus the image $H^*_?(Y^{(n)}_{\cK}, \widetilde{\cV}(\blambda^\vee)_{\cO^*}^{(p)})'$ of the natural functorial morphism $H^*_?(Y^{(n)}_{\cK}, \widetilde{\cV}(\blambda^\vee)_{\cO^*}^{(p)})\rightarrow H^*_?(Y^{(n)}_{\cK}, \widetilde{\cV}(\blambda^\vee)_{\cE^*}^{(p)})$, which is isomorphic to the $\cO^*$-torsion-free part of $H^*_?(Y^{(n)}_{\cK}, \widetilde{\cV}(\blambda^\vee)_{\cO^*}^{(p)})$ due to the first isomorphism theorem, defines an $\cO^*$-integral structure of $H^*_?(Y^{(n)}_{\cK}, \widetilde{\cV}(\blambda^\vee)_{\cE^*}^{(p)})$. Since $H^*_?(Y^{(n)}_\mathcal{K},\widetilde{\cV}(\blambda^\vee)_{\cE^*})$ is isomorphic to $H^*_?(Y^{(n)}_\mathcal{K},\widetilde{\cV}(\blambda^\vee)_{\cE^*}^{(p)})$ (as an $\cE^*$-vector space) by Lemma~\ref{lem:p-compatible}, we can define an  $\mathfrak{r}_{{}^\alpha E,({}^\alpha \mathfrak{P}_0)}$-integral structure of $H^*_?(Y^{(n)}_{\mathcal{K}}, \widetilde{\cV}(\blambda^\vee))$ by

\begin{equation*} \label{eq:integral_str}
\begin{aligned}
  H^*_{?}(Y^{(n)}_{\mathcal{K}}, \widetilde{\cV}({}^\alpha \blambda^\vee)_{\mathfrak{r}_{{}^\alpha E,({}^\alpha \mathfrak{P}_0)}}):=H^*_{?}(Y^{(n)}_{\mathcal{K}}, \widetilde{\cV}({}^\alpha \blambda^\vee)_{{}^\alpha E})\; \cap\; & H^*_?(Y^{(n)}_\mathcal{K}, \widetilde{\cV}(\blambda^\vee)_{\cO^*}^{(p)})' \\
& \subset H^*_{?}(Y^{(n)}_{\mathcal{K}}, \widetilde{\cV}({}^\alpha \blambda^\vee)_{{}^\alpha E})\subset H^*_{?}(Y^{(n)}_{\mathcal{K}}, \widetilde{\cV}({}^\alpha \blambda^\vee)).
\end{aligned} 
\end{equation*}
Here the intersection is taken in $H^*_?(Y^{(n)}_\mathcal{K},\widetilde{\cV}(\blambda^\vee)_{\cE^*})$. 

\begin{lem} \label{lem:integral_structure}
 Retain the notation. Then $H^*_?(Y^{(n)}_{\cK},\widetilde{\cV}({}^\alpha \blambda^\vee)_{\mathfrak{r}_{{}^\alpha E,({}^\alpha \mathfrak{P}_0)}})$ defines $\mathfrak{r}_{{}^\alpha E,({}^\alpha \mathfrak{P}_0)}$-integral structures of the cohomology groups $H^*_?(Y^{(n)}_{\cK},\widetilde{\cV}(\blambda^\vee)_{\cE^*})$, $H^*_?(Y^{(n)}_{\cK},\widetilde{\cV}({}^\alpha \blambda^\vee)_{{}^\alpha E})$ and $H^*_?(Y^{(n)}_{\cK},\widetilde{\cV}({}^\alpha \blambda^\vee))$. 
\end{lem}

\begin{proof}
Put $\mathsf{r}=\dim_{\mC}H^*_?(Y^{(n)}_{\cK},\widetilde{\cV}({}^\alpha \blambda^\vee))$. By construction, $H^*_?(Y^{(n)}_{\cK},\widetilde{\cV}({}^\alpha \blambda^\vee)_{\mathfrak{r}_{{}^\alpha E,({}^\alpha \mathfrak{P}_0)}})$ is a torsion-free submodule of $H^*_?(Y^{(n)}_{\cK},\widetilde{\cV}( \blambda^\vee)_{\cE^*})$
 over $(i_p^{\alpha})^{-1}(\cO^*)=\mathfrak{r}_{{}^\alpha E,({}^\alpha \mathfrak{P}_0)}$. Therefore one readily checks by using elementary divisor theory that $H^*_?(Y^{(n)}_{\cK},\widetilde{\cV}({}^\alpha \blambda^\vee)_{\mathfrak{r}_{{}^\alpha E,({}^\alpha \mathfrak{P}_0)}})$ is a free $\mathfrak{r}_{{}^\alpha E,({}^\alpha \mathfrak{P}_0)}$-module of rank at most $\mathsf{r}$. To verify the assertion, it suffices to prove that $\mathrm{rank}_{\mathfrak{r}_{{}^\alpha E,({}^\alpha \mathfrak{P}_0)}} H^*_?(Y^{(n)}_{\cK},\widetilde{\cV}({}^\alpha \blambda^\vee)_{\mathfrak{r}_{{}^\alpha E,({}^\alpha \mathfrak{P}_0)}})=\mathsf{r}$. Let $\{\mathsf{b}_1,\mathsf{b}_2,\dotsc, \mathsf{b}_{\mathsf{r}}\}$ be an ordered ${}^\alpha E$-basis of $H^{*}_?(Y^{(n)}_{\cK},\widetilde{\cV}({}^\alpha \blambda^\vee)_{{}^\alpha E})$ and $\{\mathsf{c}_1,\mathsf{c}_2,\dotsc, \mathsf{c}_{\mathsf{r}}\}$ an ordered $\cO^*$-basis of $H^{*}_?(Y^{(n)}_{\cK},\widetilde{\cV}(\blambda^\vee)_{\cO^*}^{(p)})'$. Since these two basis also form $\cE^*$-basis of $H^*_?(Y^{(n)}_{\cK},\widetilde{\cV}(\blambda^\vee)_{\cE^*})$, we can define the change-of-basis matrix $\mathsf{P}\in \rGL_{\mathsf{r}}(\cE^*)$ by 
\begin{align*}
 \begin{pmatrix}
  \mathrm{tw}_{i^\alpha_p} (\mathsf{b}_1) & \mathrm{tw}_{i^\alpha_p} (\mathsf{b}_2) & \cdots & \mathrm{tw}_{i^\alpha_p} (\mathsf{b}_\mathsf{r})
 \end{pmatrix}=  \begin{pmatrix}
  \mathsf{c}_1 & \mathsf{c}_2 & \cdots & \mathsf{c}_{\mathsf{r}}
 \end{pmatrix} \mathsf{P}.
\end{align*}
Here recall that we regard $H^*_?(Y^{(n)}_{\cK},\widetilde{\cV}({}^\alpha \blambda^\vee)_{{}^\alpha E})$ as a subspace of $H^*_?(Y^{(n)}_{\cK},\widetilde{\cV}(\blambda^\vee)_{\cE^*})$ via $\mathrm{tw}_{i^\alpha_p}$.
Let us take $\gamma\in \mathfrak{r}_{{}^\alpha E} \setminus \{0\}$ so that $i_p^{\alpha}(\gamma)$ is contained in the prime ideal of $\cO^*$. Then if we replace $\{\mathsf{b}_j\}_{1\leq j\leq \mathsf{r}}$ by  $\{ \gamma^m\mathsf{b}_j\}_{1\leq j\leq \mathsf{r}}$ for $m\in \mN$, the change-of-basis matrix $\mathsf{P}$ is replaced by $i_p^{\alpha}(\gamma)^m \mathsf{P}$. Therefore we can choose $m_0\in \mN$ so that all entries of $i_p^\alpha(\gamma)^{m_0}\mathsf{P}$ are in $\cO^*$. This implies that, for $1\leq j\leq \mathsf{r}$, the $j$-th basis vector $i_p^{\alpha}(\gamma)^{m_0}\mathrm{tw}_{i_p^\alpha}(\mathsf{b}_j)$ is represented as an $\cO^*$-linear combination of $\mathsf{c}_1,\mathsf{c}_2,\dotsc,\mathsf{c}_{\mathsf{r}}$, and thus $\gamma^{m_0}\mathsf{b}_j$ is contained in the intersection of $H^{*}_?(Y^{(n)}_{\cK},\widetilde{\cV}({}^\alpha \blambda^\vee)_{{}^\alpha E})$ and $H^{*}_?(Y^{(n)}_{\cK},\widetilde{\cV}( \blambda^\vee)_{\cO^*}^{(p)})'$ in $H^*_?(Y^{(n)}_{\cK},\widetilde{\cV}(\blambda^\vee)_{\cE^*})$; namely it is an element of $H^{*}_?(Y^{(n)}_{\cK},\widetilde{\cV}({}^\alpha \blambda^\vee)_{\mathfrak{r}_{{}^\alpha E, ({}^\alpha \mathfrak{P}_0)}})$. Obviously $\gamma^{m_0}\mathsf{b}_1, \gamma^{m_0}\mathsf{b}_2, \dotsc, \gamma^{m_0}\mathsf{b}_\mathsf{r}$ are linearly independent over $\mathfrak{r}_{{}^\alpha E,({}^\alpha \mathfrak{P}_0)}$, for so they are even over ${}^\alpha E$ by definition. Consequently $H^{*}_?(Y^{(n)}_{\cK},\widetilde{\cV}({}^\alpha \blambda^\vee)_{\mathfrak{r}_{{}^\alpha E, ({}^\alpha \mathfrak{P}_0)}})$ contains $\mathsf{r}$ elements $\gamma^{m_0}\mathsf{b}_1,\gamma^{m_0}\mathsf{b}_2,\dotsc, \gamma^{m_0}\mathsf{b}_{\mathsf{r}}$ linearly independent over $\mathfrak{r}_{{}^\alpha E,({}^\alpha \mathfrak{P}_0)}$, and thus its $\mathfrak{r}_{{}^\alpha E,({}^\alpha\mathfrak{P}_0)}$-rank is at least $\mathsf{r}$ as desired.
\end{proof}

As is well known, the cuspidal cohomology group $H^*_{\mathrm{cusp}}(Y^{(n)}_{\mathcal{K}}, \widetilde{\cV}({}^\alpha \blambda^\vee))$  is naturally regarded as a $\mC$-subspace of $H^*_{\mathrm{c}}(Y^{(n)}_{\mathcal{K}}, \widetilde{\cV}({}^\alpha \blambda^\vee))$ (see \cite[5.5~Corollary]{bor81},  \cite[page 123]{clo90} or \cite[Section 3.2.1]{mah05} for details). We thus equip the cuspidal cohomology group with a rational/integral structure defined as
\begin{align} \label{eq:ratint_cusp}
\begin{aligned}
 H^*_\mathrm{cusp}(Y^{(n)}_{\mathcal{K}},\widetilde{\cV}({}^\alpha \blambda^\vee)_{{}^\alpha E})&:=H^*_{\mathrm{c}}(Y^{(n)}_\mathcal{K},\widetilde{\cV}({}^\alpha\blambda^\vee)_{{}^\alpha E})\cap H^*_\mathrm{cusp}(Y^{(n)}_\mathcal{K},\widetilde{\cV}({}^\alpha \blambda^\vee)), \\
 H^*_\mathrm{cusp}(Y^{(n)}_{\mathcal{K}},\widetilde{\cV}({}^\alpha \blambda^\vee)_{\mathfrak{r}_{{}^\alpha E,({}^\alpha \mathfrak{P}_0)}})&:=H^*_{\mathrm{c}}(Y^{(n)}_\mathcal{K},\widetilde{\cV}({}^\alpha \blambda^\vee)_{\mathfrak{r}_{{}^\alpha E,({}^\alpha \mathfrak{P}_0)}})\cap H^*_\mathrm{cusp}(Y^{(n)}_\mathcal{K},\widetilde{\cV}({}^\alpha \blambda^\vee)).
\end{aligned}
\end{align}   
Here the intersections are taken in $H^*_{\mathrm{c}}(Y^{(n)}_\mathcal{K},\widetilde{\cV}({}^\alpha \blambda^\vee))$. Note that the rational structure $H^*_\mathrm{cusp}(Y^{(n)}_\mathcal{K}, \widetilde{\cV}({}^\alpha \blambda^\vee)_{{}^\alpha E})$ defined as (\ref{eq:ratint_cusp}) is compatible with the one which Clozel introduced in \cite[Th\'eor\`eme~3.19]{clo90} (see also \cite[Section 3.2.1]{mah05}). We introduce the integral structure compatibly with this rational structure, which should be regarded as an evidence of reasonability of our integral structure. 

We finally verify several basic properties of our integral structure.

\begin{lem} \label{lem:property_int_str}
 For $?\in \{\emptyset, \mathrm{c}, \mathrm{cusp}\}$, define $H^*_?(Y^{(n)}_{\cK}, \widetilde{\cV}({}^\alpha \blambda^\vee)_{\mathfrak{r}_{{}^\alpha E,({}^\alpha \mathfrak{P}_0)}})$ as above. 
\begin{enumerate}[label=$(\arabic*)$]
 \item The definition of $H^*_?(Y^{(n)}_{\cK}, \widetilde{\cV}({}^\alpha \blambda^\vee)_{\mathfrak{r}_{{}^\alpha E,({}^\alpha \mathfrak{P}_0)}})$ does not depend on the auxiliary choice of the field extension $E^*/E$. 
 \item For each $\alpha\in \mathrm{Aut}(\mC)$, the $\alpha$-twisting map $\mathrm{tw}_{\alpha}$ $\eqref{eq:twist_a}$ induces an isomorphism between $H^*_?(Y^{(n)}_{\cK}, \widetilde{\cV}( \blambda^\vee)_{\mathfrak{r}_{ E,( \mathfrak{P}_0)}})$ and $H^*_?(Y^{(n)}_{\cK}, \widetilde{\cV}({}^\alpha \blambda^\vee)_{\mathfrak{r}_{{}^\alpha E,({}^\alpha \mathfrak{P}_0)}})$.
\end{enumerate}
\end{lem}

\begin{proof}
By the definition of $\mathfrak{r}_{{}^\alpha E,({}^\alpha \mathfrak{P}_0)}$-integral structure of the cuspidal cohomology group, it suffices to show the statements for $?\in \{\emptyset, \mathrm{c}\}$. First note that $H^*_?(Y^{(n)}_{\cK}, \widetilde{\cV}({}^\alpha \blambda^\vee)_{\mathfrak{r}_{{}^\alpha E,({}^\alpha \mathfrak{P}_0)}})$ is characterised as a subset of $H^*_?(Y^{(n)}_{\cK}, \widetilde{\cV}({}^\alpha \blambda^\vee)_{{}^\alpha E})$ consisting of all cohomology classes $\xi$ satisfying $\mathrm{tw}_{i_p^\alpha}(\xi)\in H^*_?(Y^{(n)}_{\cK}, \widetilde{\cV}( \blambda^\vee)_{\cO^*}^{(p)})'$. 
To prove (1), let $E^*$ and $\widetilde{E}^*$ be subfields of $\mC$ both of which are finite over $E$ and contain $F_{\mathrm{nc}}$. We may further assume that $\widetilde{E}^*$ contains $E^*$ without loss of generality. Using $\widetilde{E}^*$, let us define $\widetilde{\cE}^*$ and $\widetilde{\cO}^*$ similarly to $\cE^*$ and $\cO^*$. Thus, since we define $\widetilde{V}(\blambda^\vee)_{\cO^*}$ (resp.\ $\widetilde{V}(\blambda^\vee)_{\widetilde{\cO}^*}$) as a free module over $\cO^*$ (resp.\ $\widetilde{\cO}^*$) generated by the Gel'fand--Tsetlin basis $\{ \xi_{\boldsymbol{N}} \mid \boldsymbol{N} \in \rG(\blambda^\vee)\}$, we readily observe that $\widetilde{V}(\blambda^\vee)_{\cA}\cong \widetilde{V}(\blambda^\vee)_{\cO^*}\otimes_{\cO^*}\cA$ holds for $\cA=\cE^*,\widetilde{\cE}^*$ and $\widetilde{\cO}^*$. Then, using flatness of $\cA$ over $\cO^*$, we have a canonical isomorphism $L_{\cA}\cong L\otimes_{\cO^*}\cA$ for  $L=H^*_?(Y^{(n)}_{\cK}, \widetilde{\cV}(\blambda^\vee)_{\cO^*}^{(p)})$ and $L_{\cA}=H^*_?(Y^{(n)}_{\cK}, \widetilde{\cV}(\blambda^\vee)_{\cA}^{(p)})$. Let $L'$ be the image of the composite map $L\rightarrow L\otimes_{\cO^*}\cE^*\cong L_{\cE^*}$, and $L'_{\widetilde{\cO}^*}$ the image of the composite map $L_{\widetilde{\cO^*}}\rightarrow L_{\widetilde{\cO}^*}\otimes_{\widetilde{\cO}^*}\widetilde{\cE}^*\cong L_{\widetilde{\cE}^*}$, respectively. Note that $L'\otimes_{\cO^*}\widetilde{\cO}^*\cong L'_{\widetilde{\cO}^*}$, $L'\otimes_{\cO^*}\cE^*\cong L_{\cE^*}$ and $L'\otimes_{\cO^*}\widetilde{\cE}^*\cong L_{\widetilde{\cE}^*}$ hold by construction. Then the proof of (1) is reduced to the following claim, which obviously holds:
\begin{quotation}
 for a free $\cO^*$-module $L'$ of finite rank, a vector $x\in L'\otimes_{\cO^*}\cE^*$ is an element of $L'$ if and only if $x\otimes 1\in L'\otimes_{\cO*}\widetilde{\cE}^*=(L'\otimes_{\cO*}\cE^*)\otimes_{\cE^*}\widetilde{\cE}^*$ is contained in $L'\otimes_{\cO^*}\widetilde{\cO}^*$.
\end{quotation} 
Indeed, we may complete the proof of (1) by applying the claim for $L'=H^*_?(Y^{(n)}_{\cK}, \widetilde{\cV}(\blambda^\vee)_{\cO^*}^{(p)})'$ and $x=\mathrm{tw}_{i_p^\alpha}(\xi)$. To prove (2), note that $\mathrm{tw}_{i_p^\alpha}\circ \mathrm{tw}_\alpha =\mathrm{tw}_{i_p}$ formally holds (here $i_p=i_p^{\mathrm{id}}\colon E\xrightarrow{\,\subset\,} \mC \xrightarrow{\, \boldsymbol{i}\,}\mC_p$). Therefore, for any cohomology class $\xi\in H^*_?(Y^{(n)}_{\cK}, \widetilde{\cV}( \blambda^\vee)_{\mathfrak{r}_{E,(\mathfrak{P}_0)}})$, we have $\mathrm{tw}_{i_p^\alpha}(\mathrm{tw}_\alpha (\xi))=\mathrm{tw}_{i_p}(\xi)\in H^*_?(Y^{(n)}_{\cK},\widetilde{\cV}(\blambda^\vee)_{\cO^*}^{(p)})'$, which implies that $\mathrm{tw}_\alpha(\xi)$ indeed belongs to $H^*_?(Y^{(n)}_{\cK}, \widetilde{\cV}( {}^\alpha \blambda^\vee)_{\mathfrak{r}_{{}^\alpha E,({}^\alpha \mathfrak{P}_0)}})$.
\end{proof}

\subsection*{Acknowledgements}
 The authors would like to express their sincere gratitude to Shih-Yu Chen for his useful comments.    
 The first named author was supported by Grant-in-Aid for Scientific Research (C) Grand Number JP22K03237 and Grant-in-Aid for Scientific Research (B) Grand Number JP23H01066.  
The second named author was supported by Grand-in-Aid for Scientific Research (C) Grand Number JP18K03252.
The third named author was supported by Grant-in-Aid for Scientific Research (C) Grand Number JP21K03207.

\def\cprime{$'$}

\end{document}